\documentclass[reqno,english]{amsart}
\usepackage{amsfonts,amsmath,latexsym,verbatim,amscd,mathrsfs,color,array}
\usepackage{amssymb,amsthm,graphicx }
\usepackage{amsfonts}
\usepackage[colorlinks=true]{hyperref}
\usepackage[normalem]{ulem}

\allowdisplaybreaks

\usepackage[hmargin=2.3cm, vmargin=2.3cm]{geometry}
\usepackage{bbm}
\usepackage{pdfsync}
\usepackage{epstopdf}
\usepackage{cite}
\usepackage{graphicx}
\usepackage{multirow}
\usepackage[font=bf,aboveskip=15pt]{caption}
\usepackage[toc,page]{appendix}
\usepackage[colorlinks=true]{hyperref}
\usepackage{bookmark}
\DeclareFontShape{OT1}{cmr}{bx}{sc}{<-> cmbcsc10}{}
\usepackage[square,numbers]{natbib}
\everymath{\displaystyle}
\usepackage{float}
\usepackage{ulem}

\usepackage{syntonly}
\usepackage{mathtools}

\usepackage{bm}
\usepackage{amsfonts,amsmath,latexsym,verbatim,amscd,mathrsfs,color,array}
\usepackage[colorlinks=true]{hyperref}

\usepackage{amsmath,amssymb,amsthm,amsfonts,graphicx,color}
\usepackage{amssymb}

\usepackage{epstopdf}

\allowdisplaybreaks

\newcommand{\R}{ {\mathbb{R}}}

\newcommand{\pp}{ {\partial} }

\newcommand{\cuad}{{\sqcap\kern-.68em\sqcup}}

\newcommand{\be}{\begin{equation}}
\newcommand{\ee}{\end{equation}}

\theoremstyle{plain}
\newtheorem{theorem}{Theorem}[section]
\newtheorem{lemma}[theorem]{Lemma}
\newtheorem{prop}[theorem]{Proposition}

\newtheorem{remark}{Remark}[theorem]
\newcommand{\bremark}{\begin{remark} \em}
\newcommand{\eremark}{\end{remark} }

\numberwithin{equation}{section}
\title[Multi-bubble solutions for the $H$-system]{Multi-bubble solutions for the Dirichlet problem of the $H$-system with higher degree}
\author[X. Fang]{Xiang Fang}
\address{\noindent School of Mathematics, Tianjin University, Tianjin 300072, P. R. China}
\email{fangx@tju.edu.cn}
\author[J. Wei]{Juncheng Wei}
\address{\noindent Department of Mathematics, Chinese University of Hong Kong, Shatin, NT, Hong Kong}
\email{wei@math.cuhk.edu.hk}
\author[Y. Zheng]{Youquan Zheng}
\address{\noindent School of Mathematics, Tianjin University, Tianjin 300072, P. R. China}
\email{zhengyq@tju.edu.cn}
\author[Y. Zhou]{Yifu Zhou}
\address{\noindent
School of Mathematics and Statistics, Wuhan University, Wuhan 430072, China}
\email{yifuzhou@whu.edu.cn}
\begin{document}
\begin{abstract}
We consider a Dirichlet problem of the $H$-system
\begin{equation*}
\begin{cases}
\Delta v =  2v_x\wedge v_y ~& \text{ in }\mathcal{D},\\
v=\varepsilon \tilde g ~& \text{ on }\partial{\mathcal{D}},
\end{cases}
\end{equation*}
where $\mathcal D\subset \R^2$ is the unit disk, $v:\mathcal D\to \R^3$, and $\tilde g:\pp \mathcal D\to \R^3$ is a given smooth map. As $\varepsilon\to 0^+$, we construct multi-bubble solutions concentrating at distinct points, taking around each point the profile of degree 2 $H$-bubble. This gives a partial answer to a conjecture due to Brezis-Coron \cite{BrezisCoron} and Chanillo-Malchiodi \cite{chanillomalchiodi2005cagasymptotic} concerning the  limiting configuration in the case of higher degrees. This seems to be the first construction in employing higher-degree harmonic maps as the primary configurations.

\end{abstract}
\maketitle

{
  \hypersetup{linkcolor=}
  \tableofcontents
}

\section{Introduction}
Let $\Omega$ be a smooth bounded domain in $\mathbb{R}^2$, and maps $v:\Omega\rightarrow \mathbb{R}^3$, $\tilde{g} :\partial {\Omega} \rightarrow \mathbb{R}^3$ with $\tilde g$ being smooth. We consider the Dirichlet boundary problem
\begin{equation}\label{e:main}
\begin{cases}
\Delta v =  2H(v)v_x\wedge v_y ~ &\text{ in }\Omega,\\
v=\tilde g ~ &\text{ on }\partial{\Omega},
\end{cases}
\end{equation}
where $H(v)$ is a given scalar function, $v_x$, $v_y$ denote the partial derivatives of $v$, and $\wedge$ denotes the cross-product in $\mathbb{R}^3$. The system \eqref{e:main} arises from a classical geometric variational problem of finding parametric surfaces in $\R^3$ with prescribed mean curvature $H(v)$ in the conformal representation. The existence and optimal estimates for the Dirichlet problem have been intensively investigated in many seminal works, such as Heinz \cite{Heinz1954}, Hildebrandt \cite{Hildebrandt1969,Hildebrandt1970}, Gulliver-Spruck \cite{Spruck1971,Spruck1972}, Steffen \cite{Steffen1976-1,Steffen1976-2} and Wente \cite{Wente1969}. See also Struwe \cite{StruweActa} and Duzaar-Steffen \cite{DuzaarSteffen} for existence results of the associated Plateau problem.

\medskip

Via Pohozaev-type identity, it was proved by Wente \cite{Wente1975} and Chanillo-Malchiodi \cite{chanillomalchiodi2005cagasymptotic} that there exists no non-trivial solution on simply connected domain when  $\tilde g = 0$. When $H(v) = 1$ and $\|\tilde g\|_{L^\infty} < 1$, Hildebrandt constructed a solution with minimal energy in \cite{Hildebrandt1970},  while Brezis-Coron \cite{BrezisCoroncpam}, Steffen \cite{SteffenARMA}, and Struwe \cite{Struwe1988} proved the existence of large energy solutions. A non-existence result was shown by Heinz \cite{HeinzARMA} for $\|\tilde g\|_{L^\infty}>1$, and thus the condition $\|\tilde g\|_{L^\infty} < 1$ is sharp.

\medskip

For general $H(v)$, the existence of solutions was studied in a series of papers by Caldiroli-Musina \cite{CaldiroliMusinaCCM,Caldiroli2004,CaldiroliMusinaRMI,MusinaJAM2004,caldiroli2011bubbles}, and their method is based on the variational perturbative method, which was introduced by Ambrosetti and Badiale in \cite{AmbrosettiBadiale}. Concentration phenomenon, such as bubbling and multi-bubble solutions were constructed in Caldiroli-Musina
\cite{caldirolimusinaduke2004} and Chanillo-Malchiodi \cite{chanillomalchiodi2005cagasymptotic}. Under assumptions on general $H(v)$, interior and boundary regularity were investigated by Bethuel \cite{bethuel1992resultat}, Strzelecki  \cite{strzelecki2003new}, Musina \cite{MusinaAALLMA}, M\"uller-Schikorra \cite{muller2009boundary} and their references. Furthermore, the asymptotic behavior of the solutions to \eqref{e:main} was studied, for example, in Sasahara \cite{Yasuhiro1995}, Isobe \cite{Takeshi2000,Takeshi2001a,Takeshi2001b}, and Caldiroli-Musina \cite{CaldiroliMusina2006arma,CaldiroliMusinaJFA2007} and the references therein, where some important quantities were computed and found. In \cite{Riviere-2008}, Rivi\`ere found a remarkable structure for two-dimensional conformally invariant nonlinear elliptic PDEs, including harmonic map equation and the prescribed mean curvature equation. Various regularity results were proved based on this structure; see Schikorra-Strzelecki \cite{Schikorra-Strzelecki} for the regularity theory for higher dimensional $H$-systems as well as $n$-harmonic maps into compact Riemannian manifolds. Geometric objects such as the CMC surfaces in related settings have been investigated, for instance, in \cite{Fall2012,Laurain2012}. We refer the interested readers to the comprehensive monographs of Struwe \cite{Struwe1988,Struwebook}, Duzaar-Steffen \cite{DuzaarSteffen1}, Steffen \cite{Steffen1}, Bethuel-Caldiroli-Guida \cite{Bethuel1} and the references therein for more geometrical motivations and general formulations.

\medskip

In the seminal work \cite{BrezisCoron}, Brezis and Coron
considered \eqref{e:main} with $H\equiv1$ and $\Omega = \mathcal{D}$ being the unit disk and proved that the sequence $u_n$, solving \eqref{e:main} with vanishing boundary data $\tilde g_n$ in $H^{1/2}(\partial\mathcal{D})\cap L^\infty(\partial\mathcal{D})$, converges to a finite and connected union of unit spheres. No further precise information about the limiting configuration was obtained, so they asked whether every configuration of spheres can be obtained as a limit of solutions $u_n$ for a suitable sequence of vanishing boundary data $\tilde g_n$. This question was later answered partially by Chanillo and Malchiodi \cite{chanillomalchiodi2005cagasymptotic} in the case of degree $1$, where they investigated the singularly perturbed problem
\begin{equation}\label{e:main2}
\begin{cases}
\Delta v =  2v_x\wedge v_y ~& \text{ in }\mathcal{D},\\
v=\varepsilon \tilde g ~& \text{ on }\partial{\mathcal{D}},
\end{cases}
\end{equation}
as $\varepsilon \to 0^+$. Above problem can be formulated variationally as follows. Let $g$ be the harmonic extension of $\tilde g$, i.e.
\begin{equation*}
\begin{cases}
\Delta g =  0 ~& \text{ in }\mathcal{D},\\
g=\tilde g ~& \text{ on }\partial{\mathcal{D}}.
\end{cases}
\end{equation*}
Let $v=u+\varepsilon  g$ be a solution of (\ref{e:main2}), then $u$ solves
\begin{equation*}
\begin{cases}
\Delta u =   2(u_x+\varepsilon  g_x)\wedge (u_y+\varepsilon  g_y )~& \text{ in }\mathcal{D},\\
u=0 ~& \text{ on }\partial{\mathcal{D}}.
\end{cases}
\end{equation*}
Its associated Euler functional $I_{\varepsilon}:H_0^1(\mathcal{D}, \mathbb{R}^3)\rightarrow \mathbb{R}$ is given by
\begin{equation}\label{Euler-functional}
I_{\varepsilon}(u)=\frac{1}{2}\int_{\mathcal{D}}|\nabla u|^2+\frac{2}{3}\int_{\mathcal{D}} u\cdot (u_x\wedge u_y)+\varepsilon \int_{\mathcal{D}} u\cdot(u_x\wedge g_y+g_x\wedge u_y )+2\varepsilon^2 \int_{\mathcal{D}} u\cdot (g_x\wedge g_y).
\end{equation}
By analyzing $I_{\varepsilon}(u)$, Chanillo and Malchiodi \cite{chanillomalchiodi2005cagasymptotic} established an asymptotic Morse theory and constructed solutions in the limit with multi-bubble configuration. More precisely, given any configuration of unit spheres, {\it all} passing through the origin in $\R^3$, there exists a sequence of boundary data $\tilde g_n$ such that the solutions $u_n$, with degree $1$ bubbling profile,  converge to the configuration in the Hausdorff sense, and their construction thus gives a partial answer to the question raised by Brezis-Coron.

\medskip

A natural question is the possible scenario in the case of higher degree. In this paper, we are interested in the bubbling phenomena of \eqref{e:main2} as $\varepsilon\to 0^+$ in the degree 2 case. Our aim is to construct bubbling solutions with multiple degree $2$ profiles concentrating at different points in $\mathcal D$. Recall that the fundamental solution, known as $H$-bubble, of the equation
\begin{equation}\label{intro-5}
\Delta u =  2u_x\wedge u_y ~\text{ in }~ \mathbb{R}^2
\end{equation}
is the stereographic projection $\pi:\mathbb{R}^2\rightarrow \mathbb{S}^2 \hookrightarrow \mathbb{R}^3$
\begin{equation*}
\pi(z)=\pi(x,y)=\frac{1}{1+x^2+y^2}(2x,~2y,~x^2+y^2-1),\qquad z=(x,y)\in \mathbb{R}^2.
\end{equation*}
Indeed, for the $H$-system in $\R^2$ with constant mean curvature $H=1$, finite-energy solutions were classified in the classical paper \cite{BrezisCoron} by Brezis-Coron as
\begin{equation*}
u(z) = \pi\left(\frac{P(z)}{Q(z)}\right)+C, \quad z = (x, y) \equiv x+iy,
\end{equation*}
where $P$, $Q$ are polynomials of $z$, and $C$ is a constant vector in $\mathbb{R}^3$.
From this result, we know that a typical class of solutions to \eqref{intro-5}
 are $\mathcal W^{(m)}(z) = \pi\left(z^m\right)$ for $m\in \mathbb{Z}$, where $m$ is the degree/homotopy class of the map, defined as
 $$
 {\rm deg}(u)=\frac1{4\pi}\int_{\R^2} u\cdot (u_y \wedge u_x).
 $$
 The degree 2 $H$-bubble $\mathcal W:=\pi(z^2): \mathbb{R}^2\to\mathbb{S}^2$ can be then written as
\begin{equation}\label{e:degree2Hbubble}
z = (x,y)\to \left(
\frac{2(x-y)(x+y)}{\left(x^2+y^2\right)^2+1},
\frac{4xy}{\left(x^2+y^2\right)^2+1},
\frac{\left(x^2+y^2\right)^2-1}{\left(x^2+y^2\right)^2+1}\right)
.
\end{equation}

\medskip

The degree 2 bubble defined in (\ref{e:degree2Hbubble}) of equation (\ref{intro-5}) has more invariance under rigid motions than the degree $1$ case. In fact, it is invariant under the following transformations. If $u(z)$ is a degree 2 solution of (\ref{intro-5}) and identifying $z$ as a complex number, then all of these maps
\begin{align*}
&\mbox{ (1) } u\left(\frac{z}{\mu}\right) ~\mbox{ for all }~ \mu >0, \quad \mbox{ (2) } u(z+\xi) ~\mbox{ for }~\xi\in\mathbb{C},\quad \mbox{ (3) } u\left(\frac{z-a|z|^2}{1-2a\cdot z + |a|^2|z|^2}\right) ~\mbox{ for }~a\in \mathbb{C},\\
&\mbox{ (4) } u(z^2+p) ~\mbox{ for }~p\in \mathbb{C}, \quad \mbox{ (5) } \mathcal Q u ~\mbox{ for a rotation}~ \mathcal Q\in SO(3) , \quad \mbox{ (6) }  u+C ~\mbox{ for any constant vector}~ C
\end{align*}
are still degree 2 solutions. That is to say, the function
\begin{equation}\label{approximateformfullform}
\mathcal Q \delta_{\mu,\xi,a,p}(z)+C=\mathcal Q \pi \left(\frac{\left(\frac{z-\xi}{\mu}-a\left|\frac{z-\xi}{\mu}\right|^2\right)^2}{\left(1-2a\cdot \frac{z-\xi}{\mu}+|a|^2\left|\frac{z-\xi}{\mu}\right|^2\right)^2}+p\right)+C
\end{equation}
is a $H$-bubble with degree 2, forming $13$-parameter family of solutions. Here, $z=(x,y)$, $\mu>0$, $\xi=(\xi_1,\xi_2)$, $a=(a_1,a_2)$, $p=(p_1,p_2)$. The functions $\mathcal Q  \delta_{\mu,\xi,a,p}$ are mountain-pass critical points of the functional
\begin{equation}\label{functional}
\bar{I}(u)=\frac{1}{2}\int_{\mathbb{R}^2}|\nabla u|^2+\frac{2}{3}\int_{\mathbb{R}^2} u\cdot (u_x\wedge u_y),\quad u\in \mathcal{X},
\end{equation}
where $\mathcal{X}$ denotes the functional space
\begin{equation*}
\mathcal{X}=\left\{ u\in L_{{\rm loc}}^2(\mathbb{R}^2,\mathbb{R}^3):\|u\|_{\mathcal{X}}^2=\int_{\mathbb{R}^2}|\nabla u|^2+\int_{\mathbb{R}^2}\frac{|u|^2}{(1+|z|^2)^2}<+\infty\right\}.
\end{equation*}
The space $\mathcal{X}$ coincides with $H^1(\mathbb{S}^2,\mathbb{R}^3)$ after inverse stereographic projection. When $a=0$, $p=0$, the function $\delta_{\mu,\xi,a,p}$ has the explicit form
\begin{equation}\label{A18-1}
\delta_{\mu,\xi,0,0}(z)=\pi\left(\left(\frac{z-\xi}{\mu}\right)^2\right)=\left(\frac{2\left[\left(\frac{x-\xi_1}{\mu}\right)^2-\left(\frac{y-\xi_2}{\mu}\right)^2\right]}{\left(\left|\frac{z-\xi}{\mu}\right|^2\right)^2+1},\frac{4\frac{(x-\xi_1)(y-\xi_2)}{\mu^2}}{\left(\left|\frac{z-\xi}{\mu}\right|^2\right)^2+1},\frac{\left(\left|\frac{z-\xi}{\mu}\right|^2\right)^2-1}{\left(\left|\frac{z-\xi}{\mu}\right|^2\right)^2+1}\right).
\end{equation}
\medskip
In the following, we simply denote $\delta_{\mu,\xi,0,0}$ as $\delta_{\mu,\xi}$.
We now consider the linearized operator around the $H$-bubble with degree $2$. The rotation matrix can be represented by $\mathcal Q_\alpha$, $\mathcal Q_\beta$ and $\mathcal Q_\gamma$ as
\begin{equation*}\mathcal Q_\alpha=\begin{pmatrix}
\cos \alpha & -\sin \alpha & 0\\
\sin \alpha & \cos \alpha & 0\\
 0 & 0 & 1
\end{pmatrix},\quad
\mathcal Q_\beta=\begin{pmatrix}
1&0&0\\
0&\cos \beta & -\sin \beta \\
0&\sin \beta & \cos \beta
\end{pmatrix}, \quad
\mathcal Q_\gamma=\begin{pmatrix}
\cos\gamma & 0 & \sin\gamma\\
0&1&0\\
-\sin\gamma&0&\cos\gamma
\end{pmatrix},
\end{equation*}
where $\alpha$,$\beta$, $\gamma$ are the angle of rotation around the coordinate $z$-axis, $x$-axis, $y$-axis in $\R^3$, respectively, and $\mathcal Q =\mathcal Q_\gamma \mathcal Q_\beta \mathcal Q_\alpha\in SO(3)$.
Then modulo constant vectors in $\R^3$,
\begin{equation}\label{explicitform}
\begin{aligned}
& U_{A}(x,y)= \mathcal Q_\gamma \mathcal Q_\beta \mathcal Q_\alpha \pi\left(\begin{pmatrix}
\frac{\left(\frac{x-\xi_1}{\mu}-a_1\left(\frac{(x-\xi_1)^2}{\mu^2}+\frac{(y-\xi_2)^2}{\mu^2}\right)\right)^2-
\left(\frac{y-\xi_2}{\mu}-a_2\left(\frac{(x-\xi_1)^2}{\mu^2}+\frac{(y-\xi_2)^2}{\mu^2}\right)\right)^2}
{\left(1-2\left(\frac{a_1(x-\xi_1)}{\mu}+\frac{a_2(y-\xi_2)}{\mu}\right)+\left(a_1^2+a_2^2\right)\left(\frac{(x-\xi_1)^2}
{\mu^2}+\frac{(y-\xi_2)^2}{\mu^2}\right)\right)^2}\\
\frac{2\left(\frac{x-\xi_1}{\mu}-a_1\left(\frac{(x-\xi_1)^2}{\mu^2}+\frac{(y-\xi_2)^2}{\mu^2}\right)\right)\left(\frac{y-\xi_2}{\mu}-a_2\left(\frac{(x-\xi_1)^2}
{\mu^2}+\frac{(y-\xi_2)^2}{\mu^2}\right)\right)}{\left(1-2\left(\frac{a_1(x-\xi_1)}{\mu}+\frac{a_2(y-\xi_2)}{\mu}\right)+\left(a_1^2+a_2^2\right)\left(\frac{(x-\xi_1)^2}
{\mu^2}+\frac{(y-\xi_2)^2}{\mu^2}\right)\right)^2}
\end{pmatrix}+\begin{pmatrix}p_1\\ p_2\end{pmatrix}\right)
\end{aligned}
\end{equation}
gives a ten-parameter family of solutions with degree $2$ to the $H$-system (\ref{intro-5}) in $\R^2$. Here we denote the tuple of real numbers by
$$A := (\mu, ~\alpha, ~\xi_1, ~\xi_2,~ p_1,~ p_2, ~a_1,~ a_2,~\beta,~\gamma).$$
The first variations with respect to the parameters $\mu$, $\alpha$, $\xi_1$, $\xi_2$, $p_1$, $p_2$, $a_1$, $a_2$, $\beta$, $\gamma$ are given respectively by
\begin{align}
\notag
Z_{0,1}(x,y) =&
\begin{pmatrix}
     \frac{4(x-y)(x+y)\left(x^2+y^2-1\right)\left(x^2+y^2+1\right)}{\left(\left(x^2+y^2\right)^2+1\right)^2}\\
     \frac{8x y\left(\left(x^2+y^2\right)^2-1\right)}{\left(\left(x^2+y^2\right)^2+1\right)^2}\\
     -\frac{8\left(x^2+y^2\right)^2}{\left(\left(x^2+y^2\right)^2+1\right)^2}
\end{pmatrix},\qquad
Z_{0,2}(x,y) =
\begin{pmatrix}
-\frac{4xy}{\left(x^2+y^2\right)^2+1}\\
\frac{2(x-y)(x+y)}{\left(x^2+y^2\right)^2+1}\\
0
\end{pmatrix},\\
\notag
Z_{1,1}(x,y) =&
\begin{pmatrix}
     \frac{4x\left(x^4-2x^2y^2-3y^4-1\right)}{\left(\left(x^2+y^2\right)^2+1\right)^2}\\
     -\frac{4y\left(-3x^4-2x^2y^2+y^4+1\right)}{\left(\left(x^2+y^2\right)^2+1\right)^2}\\
     -\frac{8x\left(x^2+y^2\right)}{\left(\left(x^2+y^2\right)^2+1\right)^2}
\end{pmatrix},\quad\qquad\quad
Z_{1,2}(x,y) = \begin{pmatrix}
     \frac{4y\left(3x^4+2x^2y^2-y^4+1\right)}{\left(\left(x^2+y^2\right)^2+1\right)^2}\\
     -\frac{4x\left(x^4-2x^2y^2-3y^4+1\right)}{\left(\left(x^2+y^2\right)^2+1\right)^2}\\
     -\frac{8y\left(x^2+y^2\right)}{\left(\left(x^2+y^2\right)^2+1\right)^2}
\end{pmatrix},
\\
\notag
Z_{2,1}(x,y) = &\begin{pmatrix}
     \frac{-2(x^4-6x^2y^2+y^4-1)}{\left(\left(x^2+y^2\right)^2+1\right)^2}\\
     \frac{-8xy(x-y)(x+y)}{\left(\left(x^2+y^2\right)^2+1\right)^2}\\
     \frac{4(x-y)(x+y)}{\left(\left(x^2+y^2\right)^2+1\right)^2}
\end{pmatrix},\quad\qquad\qquad
Z_{2,2}(x,y) =
\begin{pmatrix}
     \frac{-8xy(x-y)(x+y)}{\left(\left(x^2+y^2\right)^2+1\right)^2}\\
     \frac{2(x^4-6x^2y^2+y^4+1)}{\left(\left(x^2+y^2\right)^2+1\right)^2}\\
     \frac{8xy}{\left(\left(x^2+y^2\right)^2+1\right)^2}
\end{pmatrix},
\\
\notag
Z_{-1,1}(x,y) =&
\begin{pmatrix}
   \frac{-4x\left(3\left(x^4+1\right)y^2+3 x^2y^4+x^2\left(x^4-1\right)+y^6\right)}{\left(\left(x^2+y^2\right)^2+1\right)^2}\\
   \frac{-4y\left(x^6+3x^4y^2+3x^2\left(y^4-1\right)+y^6+y^2\right)}{\left(\left(x^2+y^2\right)^2+1\right)^2}\\
   \frac{8x\left(x^2+y^2\right)^2}{\left(\left(x^2+y^2\right)^2+1\right)^2}
\end{pmatrix},
\\
\notag
Z_{-1,2}(x,y) =&
\begin{pmatrix}
    \frac{4y\left(x^6+3x^4y^2+3x^2\left(y^4+1\right)+y^6-y^2\right)}{\left(\left(x^2+y^2\right)^2+1\right)^2}\\
    \frac{-4x\left(x^6+3x^4y^2+3x^2y^4+x^2+y^6-3y^2\right)}{\left(\left(x^2+y^2\right)^2+1\right)^2}\\
    \frac{8y\left(x^2+y^2\right)^2}{\left(\left(x^2+y^2\right)^2+1\right)^2}
\end{pmatrix},
\\
\label{A15-Z}
Z_{-2,1}(x,y) =& \begin{pmatrix}
0\\
-\frac{\left(x^2+y^2\right)^2-1}{\left(x^2+y^2\right)^2+1}\\
\frac{4xy}{\left(x^2+y^2\right)^2+1}
\end{pmatrix},\qquad\qquad\qquad\qquad
Z_{-2,2}(x,y) =
\begin{pmatrix}
-\frac{\left(x^2+y^2\right)^2-1}{\left(x^2+y^2\right)^2+1}\\
0\\
\frac{2(x-y)(x+y)}{\left(x^2+y^2\right)^2+1}
\end{pmatrix}.
\end{align}
In other words, $Z_{k,j}$, $k=0,\pm1,\pm2$, $j=1,2$ satisfy the linearized equation of \eqref{intro-5} around $\mathcal W$ defined by
\begin{eqnarray}\label{e:linearized1}
L_{\mathcal W}[v]:= \Delta v -2 \mathcal W_x\wedge v_y - 2 v_x\wedge \mathcal W_y = 0
\end{eqnarray}
for $v\in \mathcal{X}$. Here the first subscript $k$ stands for the Fourier mode using complex notation, and the second subscript $j$ is to distinguish two independent kernel functions in each mode. The explicit expressions of the infinitesimal generators \eqref{A15-Z} of rigid motions will be rather important in the analysis later as they will be projected onto the disk $\mathcal D$ due to the presence of boundary.

\medskip

In the perturbative process, a key ingredient is the $L^\infty$-nondegeneracy of $L_{\mathcal W}$, which is one of the crucial elements in the construction of Chanillo-Malchiodi \cite{chanillomalchiodi2005cagasymptotic} in the case of degree $\pm1$. By isoperimetric inequality, Isobe \cite{Takeshi2001a} also proved earlier a non-degeneracy property for the energy functional \eqref{functional}. The non-degeneracy for higher degree $H$-bubble $\pi(z^m)$ with $|m|\geq2$ was conjectured to be true by Chanillo and Malchiodi and was confirmed only recently by Sire and the last three authors in \cite{SireWeiZhengZhou-H-system}. Degenerate bubbles of a general form were shown recently to exist and were characterized by Guerra-Lamy-Zemas in \cite{guerra2024existence}.
See also Lenzmann-Schikorra \cite{LenzmanSchikorra}, Sire-Wei-Zheng \cite{SireWeiZheng}, and Chen-Liu-Wei \cite{chenliuwei} for related non-degeneracy properties of half-harmonic maps and higher degree harmonic maps.  From \cite{SireWeiZhengZhou-H-system} with degree $m\in\mathbb Z$, there are $4|m|+5$ bounded kernel functions to the linearization around $H$-bubble $\mathcal W^{(m)}(z) =\pi(z^m)$, and in particular, the linearized problem \eqref{e:linearized1} only admits bounded solutions as linear combination of $Z_{k,j}$, $k=0,\pm1,\pm2$, $j=1,2$ in \eqref{A15-Z}, and
\begin{equation}\label{kernels-WW}
\begin{aligned}
\begin{pmatrix}
\frac{2(x^2-y^2)((x^2+y^2)^2-1)}{((x^2+y^2)^2+1)^2}\\
\frac{4xy((x^2+y^2)^2-1)}{((x^2+y^2)^2+1)^2}\\
\frac{((x^2+y^2)^2-1)^2}{((x^2+y^2)^2+1)^2}
\end{pmatrix},\quad
\begin{pmatrix}
\frac{2(x^2-y^2)^2}{((x^2+y^2)^2+1)^2}\\
\frac{4xy(x^2-y^2)}{((x^2+y^2)^2+1)^2}\\
\frac{(x^2-y^2)((x^2+y^2)^2-1)}{((x^2+y^2)^2+1)^2}
\end{pmatrix},\quad
\begin{pmatrix}
\frac{4xy(x^2-y^2)}{((x^2+y^2)^2+1)^2}\\
\frac{8x^2y^2}{((x^2+y^2)^2+1)^2}\\
\frac{2xy((x^2+y^2)^2-1)}{((x^2+y^2)^2+1)^2}
\end{pmatrix}.
\end{aligned}
\end{equation}
Here we note that the nondegeneracy also holds in the $D^{1,2}$ sense, as can be readily verified in the precise expressions given in \cite[Section 3]{SireWeiZhengZhou-H-system}, where $$D^{1,2}: =\left\{ u\in L_{{\rm loc}}^2(\mathbb{R}^2,\mathbb{R}^3):\|u\|_{D^{1,2}}^2=\int_{\mathbb{R}^2}|\nabla u|^2<+\infty\right\}.$$

\medskip

To analyze the Dirichlet problem \eqref{e:main2}, we consider the projected bubble $$P\delta_{\mu, \xi, a, p}=\delta_{\mu, \xi, a, p}-\varphi_{\mu, \xi, a, p},$$
where $\varphi_{\mu, \xi, a, p}$ is defined in (\ref{2expansion-near-boundary}). $P\delta_{\mu, \xi, a, p}$ is the element of $H_0^1(\mathcal{D},\mathbb{R}^3)$ closest to $\delta_{\mu, \xi, a, p}$ in the Dirichlet norm. We may write
\begin{equation}\label{form-of-solution}
u=\sum_{i=1}^kP\mathcal Q_i\delta_{\mu_i,\xi^{(i)}, a_i, p_i}+w,
\end{equation}
where $k\in\mathbb N_+$, $\mathcal Q_i\in SO(3)$, $\mu_i>0$, $\xi^{(i)} = (\xi^{(i)}_1, \xi^{(i)}_2)\in \mathbb{R}^2$, $a_i = (a_{i, 1}, a_{i, 2})\in \mathbb{R}^2$, $p_i = (p_{i, 1}, p_{i, 2})\in \mathbb{R}^2$ for all $i$, and $w$ is a perturbation term. Here we remark that the extra three kernel functions in \eqref{kernels-WW} correspond to translation invariance in $\R^3$, and these are not included in the ansatz \eqref{form-of-solution} due to the presence of boundary.

\medskip

Our main theorem gives another partial answer to the question raised by Brezis and Coron \cite{BrezisCoron} in the case of degree 2.

\begin{theorem}\label{thm}
Let $\mathcal{S} = \{S_1\cup S_2\cup\cdots \cup S_k\}$ be any configuration of unit spheres, each passing through the origin of $\mathbb{R}^3$. Then there exist a sequence of real numbers $\varepsilon_n\to 0^+$, functions $\tilde g_n: \pp\mathcal D\to \mathbb{R}^3$ and a sequence of $u_n$ with form (\ref{form-of-solution}) solving the H-system (\ref{e:main2}) on the unit disk $\mathcal D$ with boundary datum $\varepsilon_n\tilde g_n$ such that the image of $u_n$ converges to $\mathcal{S}$ in the Hausdorff sense.
\end{theorem}
The limiting configuration of the spheres, {\it e.g.} $k=5$, is sketched in Figure \ref{figure1}.
\begin{figure}[H]
  \centering
  \includegraphics[width=0.45\textwidth]{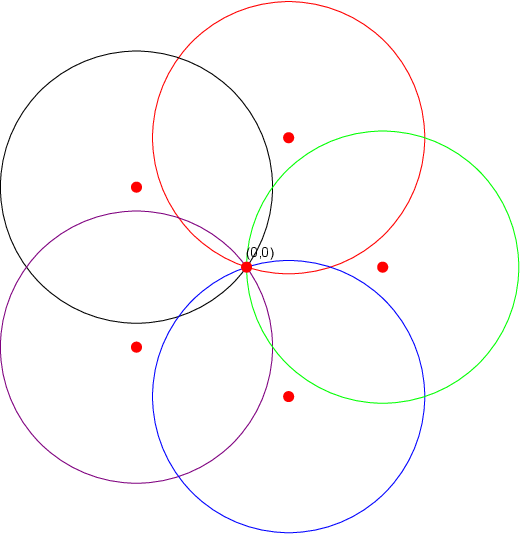}\\
  \caption{The configuration of 5 spheres (for clearity, we use circles instead), each passing through the origin. As $\varepsilon\to 0^+$, the image of the solution map covers each sphere twice.}\label{figure1}
\end{figure}

\medskip

The proof is based on the finite-dimensional reduction method together with the nondegeneracy property of the degree $2$ $H$-bubble. See also Caldiroli-Musina \cite{caldirolimusinaduke2004} and Felli \cite{felli2005note} for a related problem with perturbation in the mean curvature. Compared to the degree $1$ case in \cite{chanillomalchiodi2005cagasymptotic}, one of the main difficulties in the construction of Theorem \ref{thm} is the dealing of the extra bounded kernel functions in the linearization. The presence of these extra invariances results in the requirement of adjusting corresponding parameters in such a way that a critical point of the energy functional can be found. In particular, elaborate analysis is devoted to finding the extra four parameters $a_1$, $a_2$, $p_1$ and $p_2$ in \eqref{explicitform} for each of the $k$ bubbles in the degree 2 case. These parameters interact in the energy expansion through different combinations of kernel functions and their projections, and careful and delicate analysis is needed to make the interactions between different parameters precise. In fact, due to the Dirichlet boundary, each term in the approximation has to be projected. These projections enter the energy expansion as different Robin functions, whose explicit expressions depend on the bubble \eqref{A18-1} and kernel functions \eqref{A15-Z}. We derive in Lemma \ref{evalues-at-xi-2} and rely crucially on the explicit Robin functions to find the critical points of the reduced energy. In this process, several cancellation properties in the expansion are rather important.

\medskip

One of the key steps in the proof is the choice of the boundary data. As a building block, we find a non-trivial degree 2 analogue of the harmonic extension compared to \cite{Yasuhiro1995,chanillomalchiodi2005cagasymptotic}; see Lemma \ref{lemma7.1a}, in which its explicit form plays a rather important role in canceling out several large terms, measured in $\varepsilon$, in the energy expansion.
In the process of gluing $k$ single bubbles, up to rotations, we choose $g$ as a linear combination of the harmonic extension for the boundary condition
$$\tilde g_{\rho}(x,y) = 2\left(\frac{(x-\rho)^2-y^2}{\left((x-\rho)^2 + y^2\right)^2}, \frac{2\left(x-\rho\right) y}{\left((x-\rho)^2 + y^2\right)^2}, \frac{-\varepsilon}{\left((x-\rho)^2 + y^2\right)^2}\right)
$$
for $(x,y)\in \partial \mathcal{D}$, $\rho\in (0,1)$ is a parameter to be chosen. The motivation of such choice is to balance the leading part of the projection for $\delta_{\mu_i,\xi^{(i)},0,0}$, assuming that the parameters $a_i$, $p_i$ are suitably small. An important observation is that when the concentration points are sufficiently close to the boundary ($\rho$ is close to 1), the interaction between the boundary datum and the bubbles will be stronger than the interaction among different bubbles. This near-boundary concentration is pictured in Figure \ref{figure2}, where the smallness in the distance $d=1-\rho$ also serves a crucial role when solving for the parameters $a_i$ and $p_i$.
By choosing carefully $d$,
we capture the desired asymptotics of $a_i$ and $p_i$ measured in terms of $\varepsilon$; see \eqref{asymp-d}, \eqref{asymp-ap} and more detailed discussions in Section \ref{preliminary}.

\medskip

The proof in this paper suggests that a similar (and much more complicated) construction might work for even higher degrees, $|m|\geq 3$. The core part reflects in the dealing of more parameters as degree increases.
On the other hand, the construction of a more general configuration that not all the unit spheres pass through the origin in $\R^3$ seems to be more involved, where the ansatz of concentration pattern needs to be modified significantly. This will be the main object of a future work.

\medskip

The rest of this paper is devoted to the proof of Theorem \ref{thm}. In Section \ref{preliminary}, we introduce some preliminary facts, including the properties of the Robin function with degree 2 boundary condition and some basic integral and pointwise estimates for the degree 2 $H$-bubble. In Section \ref{finite-dimensional-reduction}, we reduce the problem into finding critical points of a function of the parameters using the finite-dimensional reduction method. To this aim, in Section \ref{expansion-for-one-bubble}, we compute the expansion of energy functional $I_{\varepsilon}\left(z_{\mathcal{A}}\right)$ for one bubble. In Section \ref{expansion-multi-bubbles}, we expand the energy functional in the case of multiple bubbles. We postpone part of the lengthy but necessary computations in Appendices \ref{Computations of mixed terms} and \ref{Appendix-F}. In Section \ref{proof-of-the-main-theorem}, we give the proof of Theorem \ref{thm}. In Appendix \ref{Appendix-G}, we provide a general method for computing the Robin function based on the Poisson integral formula and proving Lemma \ref{evalues-at-xi-2}. This general approach might be of independent interest and useful in further studies of higher degree harmonic maps and $H$-bubbles.

\medskip

\noindent {\bf Notations.} Here we list some notations used throughout this paper.
\begin{enumerate}
\item We write $z = (x, y)$ and identify the vector $z$ as a complex number.
\item We use the symbol ``$\lesssim$'' to denote the inequality ``$\leq C$'' for a positive constant $C$ that is independent of $\varepsilon$, while the constant $C$ may vary from line to line.
\item Write $\mathcal{A}_i=\left(\mu_i,\xi^{(i)}_1,\xi^{(i)}_2,\mathcal Q_i,a_{i,1},a_{i,2},p_{i,1},p_{i,2}\right)$, and $\mathcal{A}=\left(\mathcal{A}_1^{\mathrm{T}},\cdots,\mathcal{A}_k^{\mathrm{T}}\right)$.
\item $\xi^{(i)} = \left(\xi^{(i)}_1, \xi^{(i)}_2\right)$, $a_i = (a_{i, 1}, a_{i, 2})$, $p_i = (p_{i, 1}, p_{i, 2})$.
\item $Pf = f-f_0$ denotes the projection of $f$ onto $\mathcal D$, {\it i.e.}, $f_0$ is the harmonic function that satisfies $f_0 = f$ on $\partial \mathcal{D}$.
\end{enumerate}

\bigskip

\section{Preliminary facts}\label{preliminary}

\medskip

In this section, we introduce some preliminary facts. Let $\bar{I}:\mathcal{X}\rightarrow \mathbb{R}$ be defined in (\ref{functional}). As $\varepsilon\to 0^+$, the functional $I_{\varepsilon}$ is a perturbation of $\bar{I}$ and we will find a critical point of $I_{\varepsilon}$ by the finite-dimensional reduction method. The critical points of the functional $\bar{I}$ will play a critical role in our construction, namely, the solutions of
\begin{equation}\label{pre-1}
\Delta u =  2u_x\wedge u_y \quad \text{ in } \mathbb{R}^2,\quad u\in \mathcal{X}.
\end{equation}
Observe that $\pi(z^m)$ is a solution of (\ref{pre-1}), which we call degree $m$ $H$-bubble.

For $m=2$, (\ref{e:degree2Hbubble}) is indeed a solution of (\ref{pre-1}). By invariance, the function given in (\ref{approximateformfullform}) is also  a solution of (\ref{pre-1}). Define
\begin{equation*}
\bar{Z}=\left\{\mathcal Q\delta_{\mu,\xi,a,p}(z)=\mathcal Q\pi\left(\frac{\left(\frac{z-\xi}{\mu}-a\left|\frac{z-\xi}{\mu}\right|^2\right)^2}{\left(1-a\cdot \frac{z-\xi}{\mu}+|a|^2\left|\frac{z-\xi}{\mu}\right|^2 \right)^2}+p\right)\Bigg|\ \mu>0,\ \xi,\ a,\ p\in \mathbb{R}^2,\ \mathcal Q\in SO(3)\right \}.
\end{equation*}

We now list some useful expressions. For $\delta_{\mu, \xi, a, p}$, we have the following expansion,
$$
\begin{aligned}
\delta_{\mu, \xi, a, p}:=\delta_{\mu,\xi}+\mathcal{L}_{\mathcal{A}}+\mathcal{R}_{\mathcal{A}},
\end{aligned}
$$
where $\mathcal{L}_{\mathcal{A}}:=\sum_{l=1}^2\left(a_lZ_{-1,l}+p_l Z_{2,l}\right)$ is the linear part, $\mathcal{R}_{\mathcal{A}}$ denotes the higher order terms. These higher order terms include, for example, the following expressions
$$a_ja_kZ^{-1,-1}_{j,k}:=a_ja_k\frac{\pp^2 \delta_{\mu, \xi, a, p} }{\pp a_j\pp a_k}\Bigg|_{a=0,p=0},\ a_jp_kZ^{-1,2}_{j,k}:=a_jp_k\frac{\pp^2\delta_{\mu, \xi, a, p} }{\pp a_j\pp p_k}\Bigg|_{a=0,p=0}, \ p_jp_kZ^{2,2}_{j,k}:=p_jp_k\frac{\pp^2\delta_{\mu, \xi, a, p} }{\pp p_j\pp p_k}\Bigg|_{a=0,p=0}.$$

From (\ref{A18-1}), (\ref{explicitform}) and (\ref{A15-Z}), if $|z-\xi|\geq c>0$, for $\mu$ small and $a_i = o(\mu)$, $p_i = o(\mu)$, $i = 1, 2$, the following expansion holds in the case $m=2$,
\begin{equation}\label{expansion-near-boundary}
\begin{aligned}
\delta_{\mu,\xi}(z)& \sim\left(\frac{2\mu^2\left[(x-\xi_1)^2-(y-\xi_2)^2\right]}{\left|z-\xi\right|^4},\frac{4\mu^2(x-\xi_1)(y-\xi_2)}{\left|z-\xi\right|^4},1-\frac{2\mu^4}{\left|z-\xi\right|^4}\right)+ O(\mu^5),
\\\mathcal{L}_{\mathcal{A}}(z)&\sim  a_1\left(\frac{-4\mu(x-\xi_1)}{\left|z-\xi\right|^2},\frac{-4\mu(y-\xi_2)}{\left|z-\xi\right|^2},\frac{8\mu^3(x-\xi_1)}{\left|z-\xi\right|^4}\right)+ a_1O(\mu^5)\\
&\quad + a_2\left(\frac{4\mu(y-\xi_2)}{\left|z-\xi\right|^2},\frac{-4\mu(x-\xi_1)}{\left|z-\xi\right|^2},\frac{8\mu^3(y-\xi_2)}{\left|z-\xi\right|^4}\right)+ a_2O(\mu^5)\\
&\quad + p_1\left(\frac{-2\mu^4\left[(x-\xi_1)^4-6(x-\xi_1)^2(y-\xi_2)^2+(y-\xi_2)^4\right]}{\left|z-\xi\right|^8},\right.\\
&\quad\quad\quad\quad\
\left.\frac{-8\mu^4(x-\xi_1)(y-\xi_2)\left[(x-\xi_1)^2-(y-\xi_2)^2\right]}{\left|z-\xi\right|^8},\frac{4\mu^6\left[(x-\xi_1)^2-(y-\xi_2)^2\right]}{\left|z-\xi\right|^8}\right)+ p_1O(\mu^7)\\
&\quad + p_2\left(\frac{-8\mu^4(x-\xi_1)(y-\xi_2)\left[(x-\xi_1)^2-(y-\xi_2)^2\right]}{\left|z-\xi\right|^8},\right.\\
&\quad\quad\quad\quad\
\left.\frac{2\mu^4\left[(x-\xi_1)^4-6(x-\xi_1)^2(y-\xi_2)^2+(y-\xi_2)^4\right]}{\left|z-\xi\right|^8},\frac{8\mu^6(x-\xi_1)(y-\xi_2)}{\left|z-\xi\right|^8}\right)+ p_2O(\mu^7).
\end{aligned}
\end{equation}
The terms in $\mathcal{R}_{\mathcal{A}}(z)$ can also be expanded in a similar manner, but we omit the details here.

Let $\varphi_{\mu, \xi, a, p}:\mathcal{D}\rightarrow \mathbb{R}^3$ be the unique solution of the following problem
\begin{equation}\label{pre-3}
\begin{cases}
\Delta \varphi_{\mu, \xi, a, p}=0  ~& \text{in} \ \mathcal{D},\\
\varphi_{\mu, \xi, a, p}=\delta_{\mu, \xi, a, p}  ~& \text{on} \  \partial \mathcal{D},
\end{cases}
\end{equation}
and set $P\delta_{\mu, \xi, a, p}=\delta_{\mu, \xi, a, p}-\varphi_{\mu, \xi, a, p}$. The regular part $\varphi_{\mu, \xi, a, p}$ can be expressed as
$$\varphi_{\mu, \xi, a, p} := \varphi_1 + a_1\varphi_{-1,1} + a_2\varphi_{-1,2} +p_1 \varphi_{2,1} + p_2\varphi_{2,2}+\varphi_{\mathcal{R}_{\mathcal{A}}},$$
here $\varphi_{\mathcal{R}_{\mathcal{A}}}$ denotes the regular part associated with $\mathcal{R}_{\mathcal{A}}$. Using (\ref{expansion-near-boundary}), we obtain the following expansion
\begin{equation}\label{2expansion-near-boundary}
\begin{aligned}
\varphi_1(z,\xi)=&\left(2\mu^2h_1^{(1)}(z,\xi)+O(\mu^6),2\mu^2h_2^{(1)}(z,\xi)+O(\mu^6),1-2\mu^4h_3^{(1)}(z,\xi)+O(\mu^8)\right),
\\ \varphi_{-1,1}(z,\xi)=&~\left(2\mu h_1^{(-1,1)}(z,\xi)+O(\mu^5), 2\mu h_2^{(-1,1)}(z,\xi)+O(\mu^5), -2\mu^3h_3^{(-1,1)}(z,\xi)+O(\mu^7)\right),
\\\varphi_{-1,2}(z,\xi)=&~\left(2\mu h_1^{(-1,2)}(z,\xi)+O(\mu^5), 2\mu h_2^{(-1,2)}(z,\xi)+O(\mu^5),-2\mu^3h_3^{(-1,2)}(z,\xi)+O(\mu^7)\right),
\\
\varphi_{2,1}(z,\xi)=&~\left(2\mu^4 h_1^{(2,1)}(z,\xi)+O(\mu^8),2\mu^4 h_2^{(2,1)}(z,\xi)+O(\mu^8),-2\mu^6h_3^{(2,1)}(z,\xi)+O(\mu^{10})\right),
\\
\varphi_{2,2}(z,\xi)=&~\left(2\mu^4 h_1^{(2,2)}(z,\xi)+O(\mu^8),2\mu^4 h_2^{(2,2)}(z,\xi)+O(\mu^8),-2\mu^6h_3^{(2,2)}(z,\xi)+O(\mu^{10})\right).
\end{aligned}
\end{equation}

For given $(z,\xi)=((x,y),(\xi_1,\xi_2))\in \mathcal{D}\times \mathcal{D}$, $h_1^{(1)}$, $h_2^{(1)}$, $h_3^{(1)}:\mathcal{D}\times \mathcal{D} \rightarrow \mathbb{R}$ are the solutions to the following problems with boundary data determined by the explicit profile of degree $2$,
\begin{equation}\label{pre-4}
\begin{cases}
\Delta_z h_1^{(1)}(z,\xi)=0,\qquad\qquad\qquad\qquad\ \Delta_z h_2^{(1)}(z,\xi)=0 ,\qquad\qquad\qquad\quad \Delta_z h_3^{(1)}(z,\xi)=0 ~& \text{in} \ \mathcal{D},\\
h_1^{(1)}(z,\xi)=\frac{(x-\xi_1)^2-(y-\xi_2)^2}{\left|z-\xi\right|^4},\quad h_2^{(1)}(z,\xi)=\frac{2(x-\xi_1)(y-\xi_2)}{\left|z-\xi\right|^4},\quad h_3^{(1)}(z,\xi)=\frac{1}{\left|z-\xi\right|^4} ~& \text{on} \  \partial \mathcal{D}.
\end{cases}
\end{equation}
And $h_i^{(-1,1)}$, $h_i^{(-1,2)}$, $h_i^{(2,1)}$, $h_i^{(2,2)}:\mathcal{D}\times \mathcal{D} \rightarrow \mathbb{R}$, $i=1,2,3$ are the solutions of the following problems,
\begin{equation*}
\begin{cases}
\Delta_z h_1^{(-1,1)}(z,\xi)=0,\qquad\qquad\quad \Delta_z h_2^{(-1,1)}(z,\xi)=0,\qquad\qquad\quad \Delta_z h_3^{(-1,1)}(z,\xi)=0  ~& \text{in} \ \mathcal{D},\\
h_1^{(-1,1)}(z,\xi)=\frac{-2(x-\xi_1)}{\left|z-\xi\right|^2},\qquad h_2^{(-1,1)}(z,\xi)=\frac{-2(y-\xi_2)}{\left|z-\xi\right|^2},\qquad h_3^{(-1,1)}(z,\xi)=\frac{-4(x-\xi_1)}{\left|z-\xi\right|^4}  ~& \text{on} \  \partial \mathcal{D},
\end{cases}
\end{equation*}

\begin{equation*}
\begin{cases}
\Delta_z h_1^{(-1,2)}(z,\xi)=0,\qquad\qquad\ \Delta_z h_2^{(-1,2)}(z,\xi)=0,\qquad\qquad\quad \Delta_z h_3^{(-1,2)}(z,\xi)=0  ~& \text{in} \ \mathcal{D},\\
h_1^{(-1,2)}(z,\xi)=\frac{2(y-\xi_2)}{\left|z-\xi\right|^2},\qquad h_2^{(-1,2)}(z,\xi)=\frac{-2(x-\xi_1)}{\left|z-\xi\right|^2},\qquad h_3^{(-1,2)}(z,\xi)=\frac{-4(y-\xi_2)}{\left|z-\xi\right|^4}  ~& \text{on}  \  \partial \mathcal{D},
\end{cases}
\end{equation*}

\begin{equation*}
\begin{aligned}
&\begin{cases}
\Delta_z h_1^{(2,1)}(z,\xi)=0  ~& \text{in} \ \mathcal{D},\\
h_1^{(2,1)}(z,\xi)=\frac{-\left[(x-\xi_1)^4-6(x-\xi_1)^2(y-\xi_2)^2+(y-\xi_2)^4\right]}{\left|z-\xi\right|^8}~& \text{on} \  \partial \mathcal{D},
\end{cases}
\\&
\begin{cases}
\Delta_z h_2^{(2,1)}(z,\xi)=0  ~& \text{in} \ \mathcal{D},\\
h_2^{(2,1)}(z,\xi)=\frac{-4(x-\xi_1)(y-\xi_2)\left[(x-\xi_1)^2-(y-\xi_2)^2\right]}{\left|z-\xi\right|^8}  ~& \text{on} \  \partial \mathcal{D},
\end{cases}
\\&
\begin{cases}
\Delta_z h_3^{(2,1)}(z,\xi)=0   ~& \text{in} \ \mathcal{D},\\
h_3^{(2,1)}(z,\xi)=\frac{-2\left[(x-\xi_1)^2-(y-\xi_2)^2\right]}{\left|z-\xi\right|^8}  ~& \text{on} \  \partial \mathcal{D},
\end{cases}
\\&
\begin{cases}
\Delta_z h_1^{(2,2)}(z,\xi)=0  ~& \text{in} \ \mathcal{D},\\
h_1^{(2,2)}(z,\xi)=\frac{-4(x-\xi_1)(y-\xi_2)\left[(x-\xi_1)^2-(y-\xi_2)^2\right]}{\left|z-\xi\right|^8}~& \text{on} \  \partial \mathcal{D},
\end{cases}
\\&
\begin{cases}
\Delta_z h_2^{(2,2)}(z,\xi)=0  ~& \text{in} \ \mathcal{D},\\
h_2^{(2,2)}(z,\xi)=\frac{\left[(x-\xi_1)^4-6(x-\xi_1)^2(y-\xi_2)^2+(y-\xi_2)^4\right]}{\left|z-\xi\right|^8}  ~& \text{on} \  \partial \mathcal{D},
\end{cases}
\\&
\begin{cases}
\Delta_z h_3^{(2,2)}(z,\xi)=0  ~& \text{in} \ \mathcal{D},\\
h_3^{(2,2)}(z,\xi)=\frac{-4(x-\xi_1)(y-\xi_2)}{\left|z-\xi\right|^8} ~& \text{on} \  \partial \mathcal{D}.
\end{cases}
\end{aligned}
\end{equation*}
It is standard to deduce from (\ref{A18-1}), (\ref{A15-Z}) and (\ref{2expansion-near-boundary}) that
\begin{equation}\label{A15-expansion}
\begin{aligned}
P\delta_{\mu, \xi, a, p}=\delta_{\mu, \xi, a, p}-\varphi_{\mu, \xi, a, p}
&=\delta_{\mu, \xi}\left(z\right)-\varphi_1(z,\xi)+\sum_{l=1}^2a_l\left[Z_{-1,l}\left(\frac{z-\xi}{\mu}\right)-\varphi_{-1,l}(z,\xi)\right]
\\&\quad+\sum_{l=1}^2p_l\left[Z_{2,l}\left(\frac{z-\xi}{\mu}\right)-\varphi_{2,l}(z,\xi)\right]+\mathcal{R}_{\mathcal{A}}-\varphi_{\mathcal{R}_{\mathcal{A}}}.
\end{aligned}
\end{equation}

To present our results, we require additional information about $\varphi_{\mu, \xi, a, p}$. In \cite{chanillomalchiodi2005cagasymptotic}, let $G(z,\xi)$ be the Green's function of $ \mathcal{D}$ that satisfies $G(z,\xi)\sim -\log|z-\xi|$ for $z\sim \xi$. Denoting the regular part of $G(z,\xi)$ as $H(z,\xi)$, then  $G(z,\xi)=-\log|z-\xi|-H(z,\xi)$ and
\begin{equation*}
h^{(-1,1)}_1(z,\xi)=h^{(-1,2)}_2(z,\xi)=-2\frac{\partial H(z,\xi)}{\partial \xi_1},\quad h^{(-1,1)}_2(z,\xi)=-h^{(-1,2)}_1(z,\xi)=-2\frac{\partial H(z,\xi)}{\partial \xi_2}.
\end{equation*}
And from \cite[Proposition 6.1]{chanillomalchiodi2005cagasymptotic}, we know that
\begin{equation*}
\frac{\partial h^{(-1,1)}_1}{\partial x}(\xi,\xi)=\frac{\pp h^{(-1,1)}_2}{\partial y}(\xi,\xi)=\frac{\partial h^{(-1,2)}_2}{\partial x}(\xi,\xi)=-\frac{\partial h^{(-1,2)}_1}{\partial y}(\xi,\xi)=\frac{-2}{\left(1-|\xi|^2\right)^2}.
\end{equation*}
Nontrivial observations are the following relations,
\begin{equation*}
\begin{aligned}
& \frac{\pp h_2^{(-1,1)}}{\pp \xi_2}(z, \xi) = 2 h^{(1)}_1(z, \xi),\qquad -\frac{\pp h_1^{(-1,1)}}{\pp \xi_2}(z, \xi) = 2h^{(1)}_2(z, \xi).
\end{aligned}
\end{equation*}
Then we have the following result.

\begin{lemma}\label{lemma7.1a}
Assume $\xi\in \mathcal D$, $\tilde h(\cdot, \xi):\partial \mathcal{D}\to \mathbb{R}^3$ is the function defined by
\begin{equation*}
\begin{aligned}
\tilde h(z, \xi) & =\left(\frac{(x-\xi_1)^2-(y-\xi_2)^2}{\left(\left(x-\xi_1\right)^2+(y-\xi_2)^2\right)^2}, \frac{2\left(x-\xi_1\right)\left(y-\xi_2\right)}{\left(\left(x-\xi_1\right)^2+(y-\xi_2)^2\right)^2}, 0\right),\quad z=(x,y)\in \partial \mathcal{D}.
\end{aligned}
\end{equation*}
Then the harmonic extension on $\mathcal{D}$ of $\tilde h(z, \xi)$ is given by
\begin{equation*}
\begin{aligned}
h(z, \xi) & = \left(\frac{\left((\xi_1-\xi_2)\left(x^2+y^2\right)-x+y\right)\left((\xi_1+\xi_2)\left(x^2+y^2\right)-x-y\right)}{\left(\left(\xi_1^2+\xi_2^2\right)\left(x^2+y^2\right)-2\xi_1x-2\xi_2y+1\right)^2},\right. \\
&\quad\quad\quad\quad\quad\quad\quad\quad\quad\quad\quad\quad\quad\quad\quad \left. \frac{2\left(\xi_1(x^2+y^2)-x\right)\left(\xi_2\left(x^2+y^2\right)-y\right)}{\left(\left(\xi_1^2+\xi_2^2\right)\left(x^2+y^2\right)-2\xi_1x-2\xi_2y+1\right)^2}, 0\right),\quad z=(x,y)\in \mathcal{D}.
\end{aligned}
\end{equation*}
\end{lemma}
Note that $h(z, \xi) = \left(h^{(1)}_1(z, \xi), h^{(1)}_2(z, \xi), 0\right)$. Now we can compute the values of the Robin functions $h^{(-1,l)}_1(\xi, \xi)$, $h^{(-1,l)}_2(\xi, \xi)$, $h^{(1)}_1(\xi, \xi)$, $h^{(1)}_2(\xi, \xi)$, $h^{(2,l)}_1(\xi, \xi)$, $h^{(2,l)}_2(\xi, \xi)$, $l=1,2$.

\begin{lemma}\label{evalues-at-xi-2}
For the Robin functions $h_1^{(-1,l)}(\xi,\xi)$, $h_2^{(-1,l)}(\xi,\xi)$, $l=1,2$, one has
\begin{equation*}
\begin{aligned}
&h_1^{(-1,1)}(\xi,\xi)=h_2^{(-1,2)}(\xi,\xi) = \frac{-2 \xi _1}{1-|\xi|^2}, \quad h_2^{(-1,1)}(\xi,\xi)=-h_1^{(-1,2)}(\xi,\xi) =\frac{-2 \xi _2}{1-|\xi|^2},
\\& \frac{\partial h_1^{(-1,1)}}{\partial x}(\xi,\xi)=\frac{\partial h_2^{(-1,1)}}{\partial y}(\xi,\xi)=\frac{-2}{(1-|\xi|^2)^2},\qquad\frac{\partial h_1^{(-1,1)}}{\partial y}(\xi,\xi)=-\frac{\partial h_2^{(-1,1)}}{\partial x}(\xi,\xi)=0,
\\&\frac{\partial h_1^{(-1,2)}}{\partial x}(\xi,\xi)=\frac{\partial h_2^{(-1,2)}}{\partial y}(\xi,\xi)=0,\qquad\ \frac{\partial h_1^{(-1,2)}}{\partial y}(\xi,\xi)=-\frac{\partial h_2^{(-1,2)}}{\partial x}(\xi,\xi)=\frac{2}{(1-|\xi|^2)^2},
\end{aligned}
\end{equation*}
and
\begin{equation*}
\begin{aligned}
\frac{\partial^2 h_1^{(-1,1)}}{\partial x^2}(\xi,\xi) = -\frac{\partial^2 h_1^{(-1,1)}}{\partial y^2}(\xi,\xi)=-\frac{4\xi_1}{(1-|\xi|^2)^3},\quad\ \frac{\partial^2 h_2^{(-1,1)}}{\partial x^2}(\xi,\xi) = -\frac{\partial^2 h_2^{(-1,1)}}{\partial y^2}(\xi,\xi) =\frac{4\xi_2}{(1-|\xi|^2)^3},
\end{aligned}
\end{equation*}
\begin{equation*}
\begin{aligned}
\frac{\partial^2 h_1^{(-1,2)}}{\partial x^2}(\xi,\xi) = -\frac{\partial^2 h_1^{(-1,2)}}{\partial y^2}(\xi,\xi)=-\frac{4\xi_2}{(1-|\xi|^2)^3},\quad\ \frac{\partial^2 h_2^{(-1,2)}}{\partial x^2}(\xi,\xi) = -\frac{\partial^2 h_2^{(-1,2)}}{\partial y^2}(\xi,\xi) =-\frac{4\xi_1}{(1-|\xi|^2)^3}.
\end{aligned}
\end{equation*}
For the functions $h_1^{(1)}(z,\xi)$ and $h_2^{(1)}(z,\xi)$, we have the following identities
\begin{align*}
&h_1^{(1)}(\xi,\xi) = \frac{\xi_1^2-\xi_2^2}{(1-|\xi|^2)^2}, \quad h_2^{(1)}(\xi,\xi) = \frac{2\xi_1\xi_2}{(1-|\xi|^2)^2},\quad
\frac{\partial h_1^{(1)}}{\partial x}(\xi,\xi)=\frac{\partial h_2^{(1)}}{\partial y}(\xi,\xi)=\frac{2\xi_1}{(1-|\xi|^2)^3},
\\&\frac{\partial h_1^{(1)}}{ \partial y}(\xi,\xi)=-\frac{\partial h_2^{(1)}}{ \partial x}(\xi,\xi)=-\frac{2\xi_2}{(1-|\xi|^2)^3},\qquad \frac{\partial^2 h_1^{(1)}}{\partial x^2}(\xi,\xi)=\frac{\partial^2 h_2^{(1)}}{\partial x\pp y}(\xi,\xi)=\frac{2(1+2|\xi|^2)}{(1-|\xi|^2)^4},
\\&
\frac{\partial^2 h_1^{(1)}}{\partial y^2}(\xi,\xi)=-\frac{\partial^2 h_2^{(1)}}{\pp x\partial y}(\xi,\xi)=\frac{-2(1+2|\xi|^2)}{(1-|\xi|^2)^4},
\quad
\frac{\partial^2 h_1^{(1)}}{\partial x \partial y}(\xi,\xi)=-\frac{\partial^2 h_2^{(1)}}{\partial x^2}(\xi,\xi)=\frac{\partial^2 h_2^{(1)}}{\pp y^2}(\xi,\xi)= 0,
\\&
\frac{\partial^3 h_1^{(1)}}{\partial x^3}(\xi,\xi)=\frac{\partial^3 h_2^{(1)}}{\partial x^2\pp y}(\xi,\xi)=-\frac{\partial^3 h_1^{(1)}}{\partial x\pp y^2}(\xi,\xi)=-\frac{\partial^3 h_2^{(1)}}{\pp y^3}(\xi,\xi)=\frac{12 \xi _1 \left(1+|\xi|^2\right)}{(1-|\xi|^2)^5},
\\&\frac{\partial^3 h_1^{(1)}}{\partial y^3}(\xi,\xi)=-\frac{\partial^3 h_2^{(1)}}{\pp x\partial y^2}(\xi,\xi)=-\frac{\partial^3 h_1^{(1)}}{\pp x^2\partial y}(\xi,\xi)=\frac{\partial^3 h_2^{(1)}}{\pp x^3}(\xi,\xi)=-\frac{-12 \xi _2 \left(1+|\xi|^2\right)}{(1-|\xi|^2)^5},
\\&
\frac{\partial^4 h_1^{(1)}}{\partial x^4}(\xi,\xi)=\frac{\partial^4 h_2^{(1)}}{\partial x^3\pp y}(\xi,\xi)=-\frac{\partial^4 h_1^{(1)}}{\partial x^2\pp y^2}(\xi,\xi)=-\frac{\partial^4 h_2^{(1)}}{\partial x\pp y^3}(\xi,\xi)=\frac{\partial^4 h_1^{(1)}}{\partial y^4}(\xi,\xi)=
\frac{24(\xi_1^2-\xi_2^2)(3+2|\xi|^2)}{(1-|\xi|^2)^6},
\\&
\frac{\partial^4 h_2^{(1)}}{\partial x^4}(\xi,\xi)=-\frac{\partial^4 h_1^{(1)}}{\partial x^3\pp y}(\xi,\xi)=-\frac{\partial^4 h_2^{(1)}}{\partial x^2\pp y^2}(\xi,\xi)=\frac{\partial^4 h_1^{(1)}}{\partial x\pp y^3}(\xi,\xi)=\frac{\partial^4 h_2^{(1)}}{\partial y^4}(\xi,\xi)=\frac{-48\xi_1\xi_2(3+2|\xi|^2)}{(1-|\xi|^2)^6},
\\&
\frac{\partial^6 h_1^{(1)}}{\partial x^6}(\xi,\xi)=- \frac{\partial^6 h_1^{(1)}}{\partial y^6}(\xi,\xi)=\frac{720 \left(2 \xi _1^6+\left(5-10 \xi _2^2\right) \xi _1^4-10 \xi _2^2 \left(\xi _2^2+3\right) \xi _1^2+\xi _2^4 \left(2 \xi _2^2+5\right)\right)}{(1-|\xi|^2)^8}.
\end{align*}

For the functions $h^{(2,l)}$, $l=1,2$, one has
\begin{align*}
&\frac{\pp h_1^{(2,1)}}{\pp x}(\xi,\xi)=\frac{\pp h_2^{(2,1)}}{\pp y}(\xi,\xi)=-\frac{4\xi_1 \left(\xi_1^2-3\xi_2^2\right)}{\left(1-|\xi|^2\right)^5},\qquad \frac{\pp h_1^{(2,1)}}{\pp y}(\xi,\xi)=-\frac{\pp h_2^{(2,1)}}{\pp x}(\xi,\xi)=-\frac{4\xi_2 \left(-3\xi_1^2+\xi_2^2\right)}{\left(1-|\xi|^2\right)^5},
\\&\frac{\pp^4 h_1^{(2,1)}}{\pp x^4}(\xi,\xi)=\frac{\pp^4 h_1^{(2,1)}}{\pp y^4}(\xi,\xi)=-\frac{24 \left(35+4 \left(|\xi|^2-1\right)^3+30 \left(|\xi|^2-1\right)^2+60 \left(|\xi|^2-1\right)\right)}{\left(1-|\xi|^2\right)^8},
\\&\frac{\pp^4 h_2^{(2,1)}}{\pp x^4}(\xi,\xi)=\frac{\pp^4 h_2^{(2,1)}}{\pp y^4}(\xi,\xi)=0,
\\&\frac{\pp h_1^{(2,2)}}{\pp x}(\xi,\xi)=\frac{\pp h_2^{(2,2)}}{\pp y}(\xi,\xi)=\frac{4\xi_2 \left(-3\xi_1^2+\xi_2^2\right)}{\left(1-|\xi|^2\right)^5}
,\qquad \frac{\pp h_1^{(2,2)}}{\pp y}(\xi,\xi)=-\frac{\pp h_2^{(2,2)}}{\pp x}(\xi,\xi)=-\frac{4\xi_1 \left(\xi_1^2-3\xi_2^2\right)}{\left(1-|\xi|^2\right)^5},
\\&
\frac{\pp^4 h_1^{(2,2)}}{\pp x^4}(\xi,\xi)=\frac{\pp^4 h_1^{(2,2)}}{\pp y^4}(\xi,\xi)=0,
\\&
\frac{\pp^4 h_2^{(2,2)}}{\pp x^4}(\xi,\xi)=\frac{\pp^4 h_2^{(2,2)}}{\pp y^4}(\xi,\xi)=\frac{24 \left(35+4 \left(|\xi|^2-1\right)^3+30 \left(|\xi|^2-1\right)^2+60 \left(|\xi|^2-1\right)\right)}{\left(1-|\xi|^2\right)^8}.
\end{align*}
\end{lemma}
In Appendix \ref{Appendix-G}, we provide the proof of Lemma \ref{evalues-at-xi-2} using the Poisson integral formula. This approach does not depend on the explicit forms of the Robin functions. We believe it will be valuable for future studies on harmonic maps and $H$-bubbles of any degree.

Let us set up the problem as follows. For $\mu$ sufficiently small, we define
\begin{equation}\label{asymp-d}
d=d_\varepsilon=1-|\xi|^2=O\left(\mu^{\frac{2}{3}-\alpha}\right),
\end{equation}
where $\alpha$ is a small but fixed positive constant. Given a positive integer $k$, we place $k$ points on a circle centered at the origin with radius $1-d$, and they are parametrized as follows,
$$
\left((1-d)\cos\frac{2\pi j}{k}, (1-d)\sin\frac{2\pi j}{k}\right),\quad j = 1,2,\dots, k.
$$
The locations of bubbles are visualized as in Figure \ref{figure2} below.
\begin{figure}[H]
  \centering
  \includegraphics [width=0.45\textwidth]{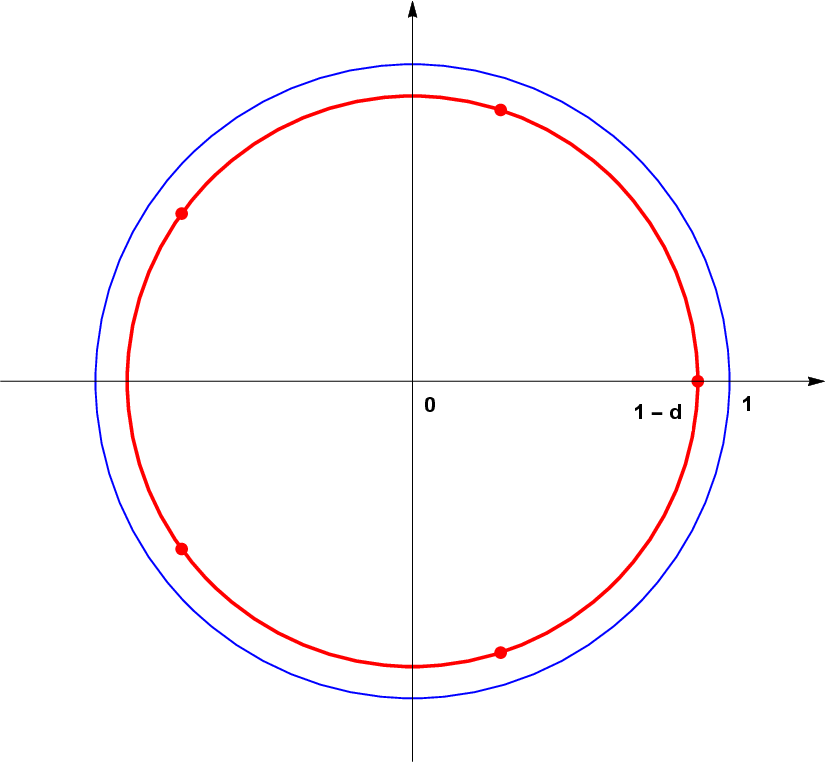}\\
  \caption{The approximate multiple spikes are located at $k$ given points, which are equidistributed on a circle of radius $1-d$ with center at origin, $k=5$ for example.}\label{figure2}
\end{figure}
As we will see in Subsection \ref{remark4.1}, the main orders of the parameters are
\begin{equation}\label{asymp-ap}
\mu_0 = \sqrt{\varepsilon},\quad a_1\sim \sqrt{\varepsilon} d^2,\quad a_2\sim \sqrt{\varepsilon} d^2,\quad
p_1\sim \frac{d^5}{\varepsilon}, \quad p_2\sim \frac{d^5}{\varepsilon}.
\end{equation}

We conclude this section with some estimates for further use. Let $\tau=\frac{d}{2}$ and $B_{\tau}:=B_{\tau}(\xi)$ be the ball of radius $\tau$ centered at $\xi$. Then for all $z \in \mathcal{D},$ $l=1, 2$, we have
\begin{equation}\label{estimates-1}
\begin{aligned}
&\|P\delta\|_{L^1\left(B_{\tau}\right)}= O\left(\mu^2\left|\log \mu\right|\right),\quad
\|\nabla P\delta\|_{L^1\left(B_{\tau}\right)}= O\left(\mu\right),\quad
\|P\delta\|_{L^2\left(B_{\tau}\right)}= O\left(\mu\right), \quad
\|\nabla P\delta\|_{L^2\left(B_{\tau}\right)}= O\left(1\right),
\\&  \|P\delta\|_{L^2\left(\mathcal{D}\backslash B_{\tau}\right)}+ \|\nabla P\delta\|_{L^2\left(\mathcal{D}\backslash B_{\tau}\right)}= O\left(\mu^2 d^{-2}\right),\quad
 |\varphi_1-(0,0,1)|(z)=O\left(\mu^2 d^{-2}\right),\quad |\nabla \varphi_1|(z)=O\left(\mu^2 d^{-3}\right),
\end{aligned}
\end{equation}
\begin{equation}\label{estimates-2}
\begin{aligned}
&\|PZ_{-1,l}\|_{L^1\left(B_{\tau}\right)}= O\left(\mu d\right),\quad
\|\nabla PZ_{-1,l}\|_{L^1\left(B_{\tau}\right)}= O\left(\mu \left|\log \mu\right|\right),\quad
\|PZ_{-1,l}\|_{L^2\left(B_{\tau}\right)}= O\left(\mu \left|\log \mu\right|\right),
 \\& \| PZ_{-1,l}\|_{L^2\left(\mathcal{D}\backslash B_{\tau}\right)}+ \|\nabla PZ_{-1,l}\|_{L^2\left(\mathcal{D}\backslash B_{\tau}\right)}= O\left(\mu d^{-1}\right),\quad |\varphi_{-1,l}|(z)=O(\mu d^{-1}),\quad |\nabla \varphi_{-1,l}|(z) =O(\mu d^{-2}).
 \end{aligned}
\end{equation}
\begin{equation}\label{estimates-3}
\begin{aligned}
&\|PZ_{2,l}\|_{L^1\left(B_{\tau}\right)}= O\left(\mu^2\right),\quad
\|\nabla PZ_{2,l}\|_{L^1\left(B_{\tau}\right)}= O\left(\mu\right),\quad
\|PZ_{2,l}\|_{L^2\left(B_{\tau}\right)}= O\left(\mu \right),
 \\& \|PZ_{2,l}\|_{L^2\left(\mathcal{D}\backslash B_{\tau}\right)}+\|\nabla PZ_{2,l}\|_{L^2\left(\mathcal{D}\backslash B_{\tau}\right)}= O\left(\mu^4 d^{-4}\right),\quad |\varphi_{2,l}|(z) =O(\mu^4 d^{-4}),\quad |\nabla \varphi_{2,l}|(z)=O(\mu^4 d^{-5}),
\end{aligned}
\end{equation}

\begin{equation}\label{estimates-4}
\begin{aligned}
&\|P\delta_{\mu, \xi, a, p}\|_{L^1\left(B_{\tau}\right)}= O\left(\mu^2\left|\log \mu\right|\right),\quad
\|\nabla P\delta_{\mu, \xi, a, p}\|_{L^1\left(B_{\tau}\right)}= O\left(\mu\right),\quad
\|P\delta_{\mu, \xi, a, p}\|_{L^2\left(B_{\tau}\right)}= O\left(\mu\right),
\\& \|\nabla P\delta_{\mu, \xi, a, p}\|_{L^2\left(B_{\tau}\right)}= O\left(1\right),\quad \|P\delta_{\mu, \xi, a, p}\|_{L^2\left(\mathcal{D}\backslash B_{\tau}\right)}+ \|\nabla P\delta_{\mu, \xi, a, p}\|_{L^2\left(\mathcal{D}\backslash B_{\tau}\right)}= O\left(\mu^2 d^{-2}\right).
\end{aligned}
\end{equation}

\section{The finite-dimensional reduction and $C^1$-estimates}\label{finite-dimensional-reduction}
In this section, we employ the finite-dimensional reduction method to reduce the problem into a finite-dimensional one on the parameters $\mathcal Q,\ \mu,\ \xi,\ a,\ p$.
Given $k\in \mathbb{N}$, $\bar C>0$, $\mathcal Q_i\in SO(3)$, set
\begin{equation}\label{solution-set}
\begin{aligned}
Z & =\left \{\sum_{i=1}^kP\mathcal Q_i\delta_{\mu_i,\xi^{(i)},a_i, p_i}\,\Bigg|\,\ \text{dist}\left(\xi^{(i)},\xi^{(j)}\right)\geq \bar{C}^{-1},\ \forall i\neq j, \ \frac{\varepsilon^{\frac{1}{2}}}{\mu_i}, \ \frac{\varepsilon^{\frac{1}{2}}d^2}{|a_{i}|}, \ \frac{\varepsilon^{-1}d^5}{|p_{i}|}\in \left[\bar{C}^{-1},\bar C\right] \right\}.
\end{aligned}
\end{equation}
We denote the elements in $Z$ as $z_{\mathcal{A}}=\sum_{i=1}^k z_{\mathcal{A}_i}$, where $z_{\mathcal{A}_i}=P\mathcal Q_i\delta_{\mu_i,\xi^{(i)},a_i, p_i}.$ The following result is a direct consequence of the  non-degeneracy of the degree 2 $H$-bubble proved in \cite{SireWeiZhengZhou-H-system} recently.
\begin{prop}\label{prop 3.1} There exists a constant $C_0>0$ such that
\begin{equation*}
\bar{I}''\left(Q\delta_{\mu,\xi,a,p}\right)\left[Q\delta_{\mu,\xi,a,p},Q\delta_{\mu,\xi,a,p}\right]\leq -C_0 \|\nabla(Q\delta_{\mu,\xi,a,p})\|_{L^2(\mathbb{R}^2)}^2 \quad \text {for all}\quad Q\delta_{\mu,\xi,a,p}\in \bar{Z},
\end{equation*}
and
\begin{equation*}
\bar{I}''\left(Q\delta_{\mu,\xi,a,p}\right)\left[v,v\right]\geq C_0 \|\nabla v\|_{L^2(\mathbb{R}^2)}^2 \quad \text {for all}\quad Q\delta_{\mu,\xi,a,p}\in \bar{Z}\ \text{and} \ v \perp \left(T_{Q\delta_{\mu,\xi,a,p}}\bar{Z}\oplus \{tQ\delta_{\mu,\xi,a,p}\}_t\right).
\end{equation*}
In particular, $\bar{I}''\left(Q\delta_{\mu,\xi,a,p}\right)\left[v\right]=0$ implies $v-c \in T_{Q\delta_{\mu,\xi,a,p}}\bar{Z}$ for some $c\in \mathbb{R}^3$.
\end{prop}

\begin{lemma}\label{lemma-3.2}
For any $z_{\mathcal{A}}\in Z$, there exists a positive constant $C$ such that if $v\in H_0^1(\mathcal{D})$, $v\perp T_{z_{\mathcal{A}}}Z$, $v\perp z_{\mathcal{A}_i},\ \forall i$, then
\begin{equation*}
I_{\varepsilon}''\left(z_{\mathcal{A}}\right)\left[v,v\right]\geq C^{-1} \| v\|_{ H_0^1(\mathcal{D})}^2.
\end{equation*}
\end{lemma}
\begin{proof}
For $i=1,\dots,k$, let $B_i$ and $\tilde{B}_i$ be the balls of radius $\frac{d}{2}$ and $\frac{d}{4}$ centered at $\xi^{(i)}$, respectively. For any $v\in H_0^1(\mathcal{D})$, set
\begin{equation}\label{Mar5-1}
v_1=v-\sum_{i=1}^{k}P_i v,
\end{equation}
where $P_i$ is the orthogonal projection of $H_0^1(\mathcal{D})$ onto $H_0^1(B_i)$. It follows that
\begin{equation}\label{reduction-2}
\|v\|_{H_0^1(\mathcal{D})}^2=\sum_{i=1}^{k}\|P_i v\|_{H_0^1(\mathcal{D})}^2+\| v_1\|_{H_0^1(\mathcal{D})}^2.
\end{equation}
Since $v_1$ is harmonic in each $B_i$ and coincides with $v$ on $\partial B_i$, by standard elliptic regularity theory, we have
\begin{equation}\label{reduction-3}
\|v_1\|_{C^2(\tilde{B}_i)}\leq C\|v\|_{ H_0^1(\mathcal{D})},\quad \forall~ i=1,\dots,k,
\end{equation}
where $C$ is a constant independent of $v$.

Now we compute $(P_i v,z_{\mathcal{A}_i})$. Using (\ref{Mar5-1}), we have
\begin{equation*}
    (P_i v,z_{\mathcal{A}_i})=\left(v-v_1-\sum_{i\neq j}^{k}P_i v,z_{\mathcal{A}_i}\right)=-(v_1,z_{\mathcal{A}_i})+\left(v-\sum_{i\neq j}^{k}P_i v,z_{\mathcal{A}_i}\right).
\end{equation*}
Since $v$ is orthogonal to $z_{\mathcal{A}_i}$, from (\ref{estimates-4}), and using H\"{o}lder's inequality, we deduce
\begin{equation*}
\begin{aligned}
\left(v-\sum_{j\neq i}P_j v,z_{\mathcal{A}_i}\right)=-\left(\sum_{j\neq i}P_j v,z_{\mathcal{A}_i}\right)
=-\sum_{j\neq i}\int_{B_j}\nabla P_j v\cdot \nabla z_{\mathcal{A}_i}=O(\mu_i^2)\| v\|_{ H_0^1(\mathcal{D})}.
\end{aligned}
\end{equation*}
Additionally, by using (\ref{estimates-4}) and (\ref{reduction-3}), we have
\begin{equation*}
\begin{aligned}
(v_1,z_{\mathcal{A}_i})&=\int_{\mathcal{D}}\nabla v_1\cdot \nabla z_{\mathcal{A}_i}
\\&=\int_{\tilde{B}_i}\nabla v_1\cdot \nabla z_{\mathcal{A}_i}+\int_{\mathcal{D} \backslash \tilde{B}_i}\nabla v_1\cdot \nabla z_{\mathcal{A}_i}
\\&\leq\|\nabla z_{\mathcal{A}_i}\|_{L^1{\left(\tilde{B}_i\right)}}\| v\|_{ H_0^1(\mathcal{D})}+\|\nabla z_{\mathcal{A}_i}\|\|_{L^2{\left(\mathcal{D} \backslash \tilde{B}_i\right)}}\| v\|_{ H_0^1(\mathcal{D})}
\\&=O(\mu_i)\| v\|_{ H_0^1(\mathcal{D})}+O(\mu_i^2d^{-2})\| v\|_{ H_0^1(\mathcal{D})}.
\end{aligned}
\end{equation*}
Combine the above estimates, and recall that $\frac{\varepsilon^{\frac{1}{2}}}{\mu_i}\in \left[\bar{C}^{-1},\bar C\right]$, $d=d_\varepsilon=1-|\xi^{(i)}|^2=O\left(\mu_i^{\frac{2}{3}-\alpha}\right)$, where $\alpha$ is a small but fixed positive constant, we obtain
\begin{equation}\label{Mar8-1}
|(P_iv,z_{\mathcal{A}_i})|=O(\mu_i^2d^{-2})\| v\|_{ H_0^1(\mathcal{D})}\leq C\varepsilon d^{-2}\| v\|_{ H_0^1(\mathcal{D})}.
\end{equation}

Next, we evaluate $\left(P_i v,\frac{\partial z_{\mathcal{A}_i}}{\partial a_{i,1}}\right)$. Similarly,
\begin{equation*}
\left(P_i v, \frac{\partial z_{\mathcal{A}_i}}{\partial a_{i,1}}\right)=\left(v-v_1-\sum_{j\neq i}P_j v, \frac{\partial z_{\mathcal{A}_i}}{\partial a_{i,1}}\right)=-\left(v_1, \frac{\partial z_{\mathcal{A}_i}}{\partial a_{i,1}}\right)+\left(v-\sum_{j\neq i}P_j v, \frac{\partial z_{\mathcal{A}_i}}{\partial a_{i,1}}\right).
\end{equation*}
Since $v$ is orthogonal to $\frac{\partial z_{\mathcal{A}_i}}{\partial a_{i,1}}$, we deduce from (\ref{estimates-2}) that
\begin{align*}
\left(v-\sum_{j\neq i}P_j v, \frac{\partial z_{\mathcal{A}_i}}{\partial a_{i,1}}\right)
&=-\left(\sum_{j\neq i}P_j v, \frac{\partial z_{\mathcal{A}_i}}{\partial a_{i,1}}\right)
\\&=-\sum_{j\neq i}\int_{B_j}\nabla P_j v\cdot \nabla \frac{\partial z_{\mathcal{A}_i}}{\partial a_{i,1}}
\\&\leq C \sum_{j\neq i}\left\|\nabla \left(P\mathcal Q_iZ_{-1,1}\left(\frac{z-\xi^{(i)}}{\mu_i}\right)\right)\right\|_{L^2{\left(B_j\right)}}\|v\|_{H_0^1(\mathcal{D})}
\\&=O(\mu_i)\|v\|_{H_0^1(\mathcal{D})}.
\end{align*}
Furthermore, using (\ref{estimates-2}) and (\ref{reduction-3}), we have
\begin{equation*}
\begin{aligned}
\left(v_1, \frac{\partial z_{\mathcal{A}_i}}{\partial a_{i,1}}\right)&=\int_{\mathcal{D}}\nabla v_1\cdot \nabla \frac{\partial z_{\mathcal{A}_i}}{\partial a_{i,1}}
\\&=\int_{\tilde{B}_i}\nabla v_1\cdot \nabla \frac{\partial z_{\mathcal{A}_i}}{\partial a_{i,1}}+\int_{\mathcal{D} \backslash \tilde{B}_i}\nabla v_1\cdot \nabla \frac{\partial z_{\mathcal{A}_i}}{\partial a_{i,1}}
\\&\leq\left\|\nabla \frac{\partial z_{\mathcal{A}_i}}{\partial a_{i,1}}\right\|_{L^1{\left(\tilde{B}_i\right)}}\|v\|_{H_0^1(\mathcal{D})}+\left\|\nabla \frac{\partial z_{\mathcal{A}_i}}{\partial a_{i,1}}\right\|_{L^2{\left(\mathcal{D} \backslash \tilde{B}_i\right)}}\|v\|_{H_0^1(\mathcal{D})}
\\&=O\left(\mu_i|\log \mu_i|\right)\|v\|_{H_0^1(\mathcal{D})}+O(\mu_id^{-1})\| v\|_{H_0^1(\mathcal{D})}.
\end{aligned}
\end{equation*}
Thus, we obtain
$$\left|\left(P_i v,\frac{\partial z_{\mathcal{A}_i}}{\partial a_{i,1}}\right)\right|\leq  C\varepsilon^{\frac{1}{2}} d^{-1}\| v\|_{H_0^1(\mathcal{D})}.$$

For $\left(P_i v,\frac{\partial z_{\mathcal{A}_i}}{\partial p_{i,1}}\right)$, using the orthogonality between $v$ and $\frac{\partial z_{\mathcal{A}_i}}{\partial p_{i,1}}$, along with the estimates provided in (\ref{estimates-3}) and (\ref{reduction-3}), there holds
\begin{equation*}
\begin{aligned}
\left(P_i v, \frac{\partial z_{\mathcal{A}_i}}{\partial p_{i,1}}\right)
&=\left(v-v_1-\sum_{j\neq i}P_j v, \frac{\partial z_{\mathcal{A}_i}}{\partial p_{i,1}}\right)
\\&=-\left(v_1, \frac{\partial z_{\mathcal{A}_i}}{\partial p_{i,1}}\right)+\left(v-\sum_{j\neq i}P_j v, \frac{\partial z_{\mathcal{A}_i}}{\partial p_{i,1}}\right)
\\&=-\int_{\mathcal{D}}\nabla v_1\cdot \nabla \frac{\partial z_{\mathcal{A}_i}}{\partial p_{i,1}}-\left(\sum_{j\neq i}P_j v, \frac{\partial z_{\mathcal{A}_i}}{\partial p_{i,1}}\right)
\\&=-\int_{\tilde{B}_i}\nabla v_1\cdot \nabla \frac{\partial z_{\mathcal{A}_i}}{\partial p_{i,1}}-\int_{\mathcal{D} \backslash \tilde{B}_i}\nabla v_1\cdot \nabla \frac{\partial z_{\mathcal{A}_i}}{\partial p_{i,1}}-\sum_{j\neq i}\int_{B_j}\nabla P_j v\cdot \nabla \frac{\partial z_{\mathcal{A}_i}}{\partial p_{i,1}}
\\&=O(\mu_i)\|v\|_{H_0^1(\mathcal{D})}+O(\mu_i^4d^{-4})\|v\|_{H_0^1(\mathcal{D})}.
\end{aligned}
\end{equation*}
Thus
$$\left|\left(P_i v,\frac{\partial z_{\mathcal{A}_i}}{\partial p_{i,1}}\right)\right|\leq  C\varepsilon^{\frac{1}{2}}\|v\|_{H_0^1(\mathcal{D})}.$$
Similarly, we can deduce
$$\left|\left(P_i v,\frac{\partial z_{\mathcal{A}_i}}{\partial Q_{\alpha_i}}+\mu_i\frac{\partial z_{\mathcal{A}_i}}{\partial\mu_i}+\mu_i\frac{\partial z_{\mathcal{A}_i}}{\partial \xi^{(i)}_1}+\mu_i\frac{\partial z_{\mathcal{A}_i}}{\partial \xi^{(i)}_2}+\frac{\partial z_{\mathcal{A}_i}}{\partial a_{i,2}}+\frac{\partial z_{\mathcal{A}_i}}{\partial p_{i,2}}\right)\right|\leq  C\varepsilon^{\frac{1}{2}} d^{-1}\|v\|_{H_0^1(\mathcal{D})}.$$

Let $\Pi_i$ denote the orthogonal projection onto the space spanned by $z_{\mathcal{A}_i}$, $\frac{\partial z_{\mathcal{A}_i}}{\partial \mathcal Q_i}$, $\mu_i\frac{\partial z_{\mathcal{A}_i}}{\partial\mu_i}$, $\mu_i\frac{\partial z_{\mathcal{A}_i}}{\partial \xi^{(i)}_1}$, $\mu_i\frac{\partial z_{\mathcal{A}_i}}{\partial \xi^{(i)}_2}$, $\frac{\partial z_{\mathcal{A}_i}}{\partial a_{i,1}}$, $\frac{\partial z_{\mathcal{A}_i}}{\partial a_{i,2}}$, $\frac{\partial z_{\mathcal{A}_i}}{\partial p_{i,1}}$ and $\frac{\partial z_{\mathcal{A}_i}}{\partial p_{i,2}}$. Then we obtain
$$\|\Pi_iP_i v\|_{H_0^1(\mathcal{D})}\leq C\varepsilon^{\frac{1}{2}} d^{-1} \|v\|_{H_0^1(\mathcal{D})}.$$

The functional $I_{\varepsilon}''\left(z_{\mathcal{A}}\right)$ is given by
$$
I_{\varepsilon}''\left(z_{\mathcal{A}}\right)[v,\tilde v]=\int_{\mathcal{D}}\nabla v\cdot \nabla \tilde v+2\int_{\mathcal{D}}z_{\mathcal{A}}\cdot (v_x \wedge \tilde v_y +\tilde v_x \wedge  v_y)+2\varepsilon \int_{\mathcal{D}}\tilde v \cdot (g_x \wedge v_y + v_x \wedge  g_y),\quad v, \tilde v \in H_0^1(\mathcal{D}).
$$
Hence, for an arbitrary function $v$, there holds
\begin{equation*}
I_{\varepsilon}''\left(z_{\mathcal{A}}\right)[v, v]=\sum_{i=1}^kI_{\varepsilon}''\left(z_{\mathcal{A}}\right)[P_i v,P_i v]+I_{\varepsilon}''\left(z_{\mathcal{A}}\right)[v_1,v_1]+2\sum_{i=1}^kI_{\varepsilon}''\left(z_{\mathcal{A}}\right)[P_i v,v_1].
\end{equation*}
For $\sum_{i=1}^kI_{\varepsilon}''\left(z_{\mathcal{A}}\right)[P_i v,P_i v]$, we have
$$I_{\varepsilon}''\left(z_{\mathcal{A}}\right)[P_i v,P_i v]=\bar I''\left(z_{\mathcal{A}}\right)[P_i v,P_i v]+2\varepsilon \int_{\mathcal{D}}P_i v\cdot \left(g_x \wedge (P_i v)_y + (P_i v)_x \wedge  g_y\right),$$
where $\bar I (u)$ is the functional given in (\ref{functional}). Furthermore, we have
$$\bar I''\left(z_{\mathcal{A}}\right)[P_i v,P_i v]=\bar I''\left(z_{\mathcal{A}}\right)[\Pi_iP_i v,\Pi_iP_i v]+\bar I''\left(z_{\mathcal{A}}\right)[P_i v-\Pi_iP_i v,P_i v-\Pi_iP_i v].$$
By Proposition \ref{prop 3.1} and (\ref{Mar8-1}), we have
$$
\bar I''\left(z_{\mathcal{A}}\right)[\Pi_iP_i v,\Pi_iP_i v]
=\left(P_i v,z_{\mathcal{A}_i}\right)^2\bar I''\left(z_{\mathcal{A}}\right)[z_{\mathcal{A}_i},z_{\mathcal{A}_i}]
\leq -C \varepsilon^2d^{-4} \| v\|^2_{ H_0^1(\mathcal{D})},
$$
and
\begin{align*}
\bar I''\left(z_{\mathcal{A}}\right)[P_i v-\Pi_iP_i v,P_i v-\Pi_iP_i v] &\geq C_0\| \nabla P_i v-\nabla\Pi_iP_i v \|^2_{L^2(\mathbb{R}^2)}
\\&\geq C_0\left(\| \nabla P_i v\|^2_{L^2(\mathcal{D})}-\|\nabla\Pi_iP_i v \|^2_{L^2(\mathcal{D})}\right)
\\&\geq C_0\left(\|  P_i v\|^2_{H_0^1(\mathcal{D})}-\|\Pi_iP_i v \|^2_{H_0^1(\mathcal{D})}\right)
\\&\geq C_0\left(\|  P_i v\|^2_{H_0^1(\mathcal{D})}-C \varepsilon d^{-2}\| v\|^2_{ H_0^1(\mathcal{D})}\right).
\end{align*}
Moreover,
$$\left|\varepsilon \int_{\mathcal{D}}P_i v\cdot [g_x \wedge (P_i v)_y + (P_i v)_x \wedge  g_y]\right|\leq C \varepsilon\|P_iv\|^2_{H_0^1(\mathcal{D})},$$
then
\begin{equation}\label{reduction-7}
I_{\varepsilon}''\left(z_{\mathcal{A}}\right)[P_i v,P_i v]\geq C_0\|  P_i v\|^2_{H_0^1(\mathcal{D})}-C \varepsilon d^{-2}\| v\|^2_{ H_0^1(\mathcal{D})}.
\end{equation}

Next, we compute $I_{\varepsilon}''\left(z_{\mathcal{A}}\right)[P_i v,v_1]$. From the orthogonality of $P_iv$ and $v_1$, we have
\begin{equation}\label{reduction-8}
\begin{aligned}
I_{\varepsilon}''\left(z_{\mathcal{A}}\right)[P_i v,v_1]
&=2\int_{\mathcal{D}}z_{\mathcal{A}}\cdot [(P_i v)_x \wedge  (v_1)_y +(v_1)_x \wedge  (P_i v)_y]+2\varepsilon \int_{\mathcal{D}}P_i v\cdot [g_x \wedge (v_1)_y + (v_1)_x \wedge  g_y]
\\&=O\left(\int_{\mathcal{D}}\left|z_{\mathcal{A}}\right||\nabla P_i v| |\nabla v_1|\right)+O(\varepsilon)\|v\|^2_{H_0^1(\mathcal{D})},
\end{aligned}\end{equation}
and by (\ref{estimates-4}), (\ref{reduction-2}) and (\ref{reduction-3}), we have
\begin{equation}\label{reduction-9}
\begin{aligned}
\int_{\mathcal{D}}\left|z_{\mathcal{A}}\right||\nabla P_i v| |\nabla v_1|
&=\int_{\tilde{B}_i}\left|z_{\mathcal{A}}\right||\nabla P_i v| |\nabla v_1|+\int_{\mathcal{D} \backslash \tilde{B}_i}\left|z_{\mathcal{A}}\right||\nabla P_i v| |\nabla v_1|
\\&=\|z_{\mathcal{A}}\|_{L^2(\tilde{B}_i)}\|v\|^2_{H_0^1(\mathcal{D})}+ \|z_{\mathcal{A}}\|_{L^{\infty}(\mathcal{D} \backslash \tilde{B}_i)}\|v\|^2_{H_0^1(\mathcal{D})}
\\&=O(\mu_i)\|v\|^2_{ H_0^1(\mathcal{D})}+O(\mu_i^2d^{-2})\|v\|^2_{ H_0^1(\mathcal{D})}.
\end{aligned}
\end{equation}
Similarly, we have
\begin{align}
\notag I_{\varepsilon}''\left(z_{\mathcal{A}}\right)[v_1,v_1]
&=\int_{\mathcal{D}}|\nabla v_1|^2+2\int_{\mathcal{D}}z_{\mathcal{A}}\cdot [(v_1)_x \wedge (v_1)_y +(v_1)_x \wedge  (v_1)_y]+2\varepsilon\int_{\mathcal{D}}v_1\cdot [g_x \wedge (v_1)_y + (v_1)_x \wedge  g_y]
\\ \label{reduction-10} &= \|v_1\|^2_{H_0^1(\mathcal{D})}+O(\mu_i^2d^{-2})\|v\|^2_{ H_0^1(\mathcal{D})}+O(\varepsilon)\|v\|^2_{H_0^1(\mathcal{D})}.
\end{align}
From (\ref{reduction-2}), (\ref{reduction-7}), (\ref{reduction-8}), (\ref{reduction-9}) and (\ref{reduction-10}), we can find a positive constant $C$ such that
$$I_{\varepsilon}''\left(z_{\mathcal{A}}\right)[v, v] \geq  C^{-1}\|v\|^2_{ H_0^1(\mathcal{D})}.$$
Thus, we complete the proof of Lemma \ref{lemma-3.2}.
\end{proof}

\begin{prop}\label{prop-3.3}
If $\varepsilon>0$ is sufficiently small, for every $z_{\mathcal{A}}\in Z$, there exists a function $w_{\varepsilon}\left(z_{\mathcal{A}}\right)\in H^1(\mathcal{D},\mathbb{R}^3)$ and a constant $C>0$ such that for all $z_{\mathcal{A}}\in Z$, $w_{\varepsilon}\left(z_{\mathcal{A}}\right)$ is orthogonal to $T_{z_{\mathcal{A}}} Z$, $I_{\varepsilon}'\left(z_{\mathcal{A}}+w_{\varepsilon}\left(z_{\mathcal{A}}\right)\right)\in T_{z_{\mathcal{A}}} Z$ and $\|w_{\varepsilon}\left(z_{\mathcal{A}}\right)\| \leq C\|I_{\varepsilon}'\left(z_{\mathcal{A}}\right)\|$. Therefore, the manifold
$$Z_{\varepsilon}=\{z_{\mathcal{A}}+w_{\varepsilon}\left(z_{\mathcal{A}}\right): z_{\mathcal{A}}\in Z\}$$ is a natural constraint for the functional $I_{\varepsilon}'$. That is to say, if $u\in Z_{\varepsilon}$ and $I_{\varepsilon}'|_{Z_{\varepsilon}}(u)=0$, then $I_{\varepsilon}'(u)=0$.
\end{prop}
\begin{proof}
The proof is standard, and we present it here for convenience; see also  \cite{chanillomalchiodi2005cagasymptotic}. Let $q_i=q_i(z_{\mathcal{A}})$, $i=1,\dots, 10$ denote the orthogonal basis of $T_{z_{\mathcal{A}}} Z$ (note that there are three kernel functions corresponding to the rotation matrix). Define a functional $$F_{\varepsilon}: Z\times H^1(\mathcal{D},\mathbb{R}^3)\times T_{z_{\mathcal{A}}} Z\rightarrow H^1(\mathcal{D},\mathbb{R}^3)\times \mathbb{R}^{10}$$
by
\begin{equation}\label{reduction-11:2}
 F_{\varepsilon}(z_{\mathcal{A}},w_{\varepsilon},\beta):=
 \begin{pmatrix}
 F^{(1)}_{\varepsilon}(z_{\mathcal{A}},w_{\varepsilon},\beta) \\F^{(2)}_{\varepsilon}(z_{\mathcal{A}},w_{\varepsilon},\beta)
\end{pmatrix}
=\begin{pmatrix}I_{\varepsilon}'\left(z_{\mathcal{A}}+w_{\varepsilon}\right)-\sum_{i=1}^{10}\beta_i q_i\\ \left(\langle w_{\varepsilon},q_1\rangle, \cdots, \langle w_{\varepsilon},q_{10}\rangle\right)
 \end{pmatrix}.
 \end{equation}
Thus, the unknown $(w_{\varepsilon},\beta)$ can be implicitly defined by the following equation
$$F_{\varepsilon}\left(z_{\mathcal{A}},w_{\varepsilon},\beta\right)=(0,0)^{\mathrm{T}}.$$
Let
$$G_{\varepsilon}\left(z_{\mathcal{A}},w_{\varepsilon},\beta\right)=F_{\varepsilon}\left(z_{\mathcal{A}},w_{\varepsilon},\beta\right)-\frac{\partial F_{\varepsilon}}{\partial (w_{\varepsilon},\beta)}\left(z_{\mathcal{A}},0,0\right)[(w_{\varepsilon},\beta)],$$
we have
$$F_{\varepsilon}\left(z_{\mathcal{A}},w_{\varepsilon},\beta\right)=0\quad \Longleftrightarrow \quad  \frac{\partial F_{\varepsilon}}{\partial (w_{\varepsilon},\beta)}\left(z_{\mathcal{A}},0,0\right)[(w_{\varepsilon},\beta)]+G_{\varepsilon}\left(z_{\mathcal{A}},w_{\varepsilon},\beta\right)=0.$$
Note that $ F_{\varepsilon}(z_{\mathcal{A}},w_{\varepsilon},\beta)=0$ implies that for all $z_{\mathcal{A}}\in Z$, $I_{\varepsilon}'\left(z_{\mathcal{A}}+w_{\varepsilon}\right)\in T_{z_{\mathcal{A}}} Z$ and $w_{\varepsilon}\left(z_{\mathcal{A}}\right)\perp T_{z_{\mathcal{A}}} Z$.

Now, for a fixed $z_{\mathcal{A}}\in Z$, we consider the derivatives of $F_{\varepsilon}$ evaluated on $(z_{\mathcal{A}},0,0)$.
For $\forall\  (v_{\varepsilon},\gamma)\in H^1(\mathcal{D},\mathbb{R}^3)\times \mathbb{R}^{10}$, there holds
\begin{equation*}
 \frac{\partial F_{\varepsilon}}{\partial (w_{\varepsilon},\beta)}\left(z_{\mathcal{A}},0,0\right)[v_{\varepsilon},\gamma]=
 \begin{pmatrix}
\frac{\partial F^{(1)}_{\varepsilon}}{\partial (w_{\varepsilon},\beta)}\left(z_{\mathcal{A}},0,0\right)[v_{\varepsilon},\gamma]
\\\frac{\partial F^{(2)}_{\varepsilon}}{\partial (w_{\varepsilon},\beta)}\left(z_{\mathcal{A}},0,0\right)[v_{\varepsilon},\gamma]
\end{pmatrix}
=\begin{pmatrix}I_{\varepsilon}''\left(z_{\mathcal{A}}\right)[v_{\varepsilon}]-\sum_{i=1}^{10}\gamma_i q_i
\\ \left(\langle v_{\varepsilon},q_1\rangle, \cdots, \langle v_{\varepsilon},q_{10}\rangle\right)
 \end{pmatrix}.
 \end{equation*}
We now show that $\frac{\partial F_{\varepsilon}}{\partial (w_{\varepsilon},\beta)}\left(z_{\mathcal{A}},0,0\right)$ is a one-to-one and onto map.

Firstly, suppose that
$$ \frac{\partial F_{\varepsilon}}{\partial (w_{\varepsilon},\beta)}\left(z_{\mathcal{A}},0,0\right)[v_{\varepsilon},\gamma]=(0,0)^{\mathrm{T}},$$
which implies
$$I_{\varepsilon}''\left(z_{\mathcal{A}}\right)[v_{\varepsilon}]=\sum_{i=1}^{10}\gamma_i q_i,\quad \langle v_{\varepsilon},q_i\rangle=0,~ i=1,\dots,10,$$
thus $v_{\varepsilon}\left(z_{\mathcal{A}}\right)$ is orthogonal to $T_{z_{\mathcal{A}}} Z$. From Lemma \ref{lemma-3.2}, we deduce that $v_{\varepsilon} = 0$ and consequently,
$\gamma_i=0$. Hence, $\frac{\partial F_{\varepsilon}}{\partial (w_{\varepsilon},\beta)}\left(z_{\mathcal{A}},0,0\right)$ is injective.

Secondly, we show that $ \frac{\partial F_{\varepsilon}}{\partial (w_{\varepsilon},\beta)}\left(z_{\mathcal{A}},0,0\right)$ is surjective. For given $(\tilde w, \tilde\gamma)\in H^1(\mathcal{D},\mathbb{R}^3)\times \mathbb{R}^{10}$, we need to find $(v_{\varepsilon}, \gamma)\in H^1(\mathcal{D},\mathbb{R}^3)\times \mathbb{R}^{10}$ such that
\begin{equation}\label{1-invertible}
I_{\varepsilon}''\left(z_{\mathcal{A}}\right)[v_{\varepsilon}]-\sum_{i=1}^{10}\gamma_i q_i= \tilde w,\quad (\langle v_{\varepsilon},q_1\rangle, \cdots, \langle v_{\varepsilon},q_{10}\rangle)= \tilde \gamma.
\end{equation}
Let
$$v_{\varepsilon}=v_{\varepsilon}^{(1)}+v_{\varepsilon}^{(2)},$$
where $v_{\varepsilon}^{(1)} \perp T_{z_{\mathcal{A}}} Z$ and $v_{\varepsilon}^{(2)}\in T_{z_{\mathcal{A}}} Z$. Then,
$$\left(\langle v_{\varepsilon}^{(2)},q_1\rangle, \cdots, \langle v_{\varepsilon}^{(2)},q_{10}\rangle\right)= \tilde \gamma.$$
Choosing $\gamma_i=-\langle\tilde w,q_i\rangle$, we obtain
$$\left(\tilde w+\sum_{i=1}^{10}\gamma_i q_i\right)\perp T_{z_{\mathcal{A}}} Z.$$
By Lemma \ref{lemma-3.2}, we have $$v_{\varepsilon}^{(1)}=\left(I_{\varepsilon}''\left(z_{\mathcal{A}}\right)\right)^{-1}\left(\tilde w+\sum_{i=1}^{10}\gamma_i q_i\right),$$
therefore, (\ref{1-invertible}) has a solution and $\frac{\partial F_{\varepsilon}}{\partial (w_{\varepsilon},\beta)}\left(z_{\mathcal{A}},0,0\right)$ is onto. Therefore, $\frac{\partial F_{\varepsilon}}{\partial (w_{\varepsilon},\beta)}\left(z_{\mathcal{A}},0,0\right)$ is uniformly reversible for $z_{\mathcal{A}}\in Z$ and $\varepsilon$ sufficiently small. Consequently, we can write
$$\begin{aligned}
F_{\varepsilon}\left(z_{\mathcal{A}},w_{\varepsilon},\beta\right)=0\quad \Longleftrightarrow \quad (w_{\varepsilon},\beta)&=-\left( \frac{\partial F_{\varepsilon}}{\partial (w_{\varepsilon},\beta)}\left(z_{\mathcal{A}},0,0\right)\right)^{-1}[G_{\varepsilon}\left(z_{\mathcal{A}},w_{\varepsilon},\beta\right)]
\\&=-\left( \frac{\partial F_{\varepsilon}}{\partial (w_{\varepsilon},\beta)}\left(z_{\mathcal{A}},0,0\right)\right)^{-1}[F_{\varepsilon}\left(z_{\mathcal{A}},0,0\right)+H_{\varepsilon}\left(z_{\mathcal{A}},w_{\varepsilon},\beta\right)]
\\&:=W_{\varepsilon}\left(z_{\mathcal{A}},w_{\varepsilon},\beta\right),
\end{aligned}$$
where
$$H_{\varepsilon}\left(z_{\mathcal{A}},w_{\varepsilon},\beta\right)=F_{\varepsilon}\left(z_{\mathcal{A}},w_{\varepsilon},\beta\right) -F_{\varepsilon}\left(z_{\mathcal{A}},0,0\right)-\frac{\partial F_{\varepsilon}}{\partial (w_{\varepsilon},\beta)}\left(z_{\mathcal{A}},0,0\right)[(w_{\varepsilon},\beta)].$$
Then $H_{\varepsilon}\left(z_{\mathcal{A}},w_{\varepsilon},\beta\right)$ satisfies
\begin{equation}\label{reduction-11}
\begin{cases}
\|H_{\varepsilon}\left(z_{\mathcal{A}},w_{\varepsilon},\beta\right)\|\leq C\|(w_{\varepsilon},\beta)\|^2,\\
\|H_{\varepsilon}\left(z_{\mathcal{A}},w_{\varepsilon},\beta\right)-H_{\varepsilon}\left(z_{\mathcal{A}},\bar w_{\varepsilon},\bar \beta\right)\| \leq C\left(\|(w_{\varepsilon},\beta)\|+\|(\bar w_{\varepsilon},\bar \beta)\|\right)\|(w_{\varepsilon},\beta)-(\bar w_{\varepsilon},\bar \beta)\|,
\end{cases}
\end{equation}
where $C=C(\mathcal{D},g,\bar C)$ is a constant.

Using (\ref{reduction-11}), we can prove that the map $W_{\varepsilon}$ is a contraction
in a ball of radius $\tilde C \|F_{\varepsilon}(z_{\mathcal{A}},0,0)\|$ for some positive constant $\tilde C$. Hence, there exists a fixed point $(w_{\varepsilon},\beta)$ such that
$$(w_{\varepsilon},\beta)=W_{\varepsilon}\left(z_{\mathcal{A}},w_{\varepsilon},\beta\right).$$
That is, $F_{\varepsilon}\left(z_{\mathcal{A}},w_{\varepsilon},\beta\right)=0$. Since $\|F_{\varepsilon}\left(z_{\mathcal{A}},0,0\right)\|\leq C\|I_{\varepsilon}'\left(z_{\mathcal{A}}\right)\|$ for some constant $C$ and $\left(z_{\mathcal{A}},w_{\varepsilon},\beta\right)\in B_{\tilde C \|F_{\varepsilon}\left(z_{\mathcal{A}},0,0\right)\|}$, we have
$$\|(w_{\varepsilon},\beta)\|\leq \tilde C \|F_{\varepsilon}\left(z_{\mathcal{A}},0,0\right)\| \leq C\|I_{\varepsilon}'\left(z_{\mathcal{A}}\right)\|.$$
Thus, we conclude that
$$\|w_{\varepsilon}\left(z_{\mathcal{A}}\right)\|\leq C\|I_{\varepsilon}'\left(z_{\mathcal{A}}\right)\|.$$
\end{proof}

In order to control the norm of $w_{\varepsilon}\left(z_{\mathcal{A}}\right)$ in Proposition \ref{prop-3.3}, we estimate $\left\|I_{\varepsilon}'\left(z_{\mathcal{A}}\right)\right\|$ for $z_{\mathcal{A}}\in Z$.

 \begin{lemma}\label{lemma-3.4}
For $\varepsilon$ sufficiently small and for all $z_{\mathcal{A}}\in Z$, there holds
$$\|I_{\varepsilon}'\left(z_{\mathcal{A}}\right)\|\leq e(\varepsilon),$$
where $e(\varepsilon) :=O\left(\sum_{i,j=1}^k\left(\mu_i^2\mu_j^{1-\frac{\alpha}{2}}+\sum_{l=1}^2\left(|a_{i, l}|\mu_i\mu_j^{1-\frac{\alpha}{2}}d^{-2}+|p_{i, l}|\mu_i^4\mu_j^{1-\frac{\alpha}{2}}d^{-5}\right)\right)\right)=O\left(\varepsilon^{\frac{3}{2}-\frac{\alpha}{4}}\right)$.
\end{lemma}
\begin{proof}
Let $v\in H_0^1(\mathcal{D},\mathbb{R}^3)$ and $z_{\mathcal{A}}\in Z$. The functional $I_{\varepsilon}'\left(z_{\mathcal{A}}\right)$ is given by
\begin{equation}\label{reduction-12}
I_{\varepsilon}'\left(z_{\mathcal{A}}\right)[v]=\int_{\mathcal{D}}\nabla z_{\mathcal{A}}\cdot \nabla v+2\int_{\mathcal{D}}v\cdot \left(\left(z_{\mathcal{A}}\right)_x \wedge \left(z_{\mathcal{A}}\right)_y \right)+2\varepsilon \int_{\mathcal{D}}z_{\mathcal{A}}\cdot (v_x \wedge g_y + g_x \wedge v_y)+2\varepsilon^2 \int_{\mathcal{D}}v\cdot (g_x \wedge g_y).
\end{equation}
For the first two terms in (\ref{reduction-12}),  integrating by parts, we have
\begin{equation}\label{reduction-13}
\int_{\mathcal{D}}\nabla z_{\mathcal{A}}\cdot \nabla v+2\int_{\mathcal{D}}v\cdot \left[\left(z_{\mathcal{A}}\right)_x \wedge \left(z_{\mathcal{A}}\right)_y \right]=\int_{\mathcal{D}}v\cdot \left[-\Delta z_{\mathcal{A}} +2\left(z_{\mathcal{A}}\right)_x \wedge \left(z_{\mathcal{A}}\right)_y \right].
\end{equation}
Since
$$\Delta z_{\mathcal{A}}=\Delta (P\mathcal Q_i\delta_{\mu_i,\xi^{(i)},a_i, p_i})=\Delta \mathcal Q_i\delta_{\mu_i,\xi^{(i)},a_i, p_i}=2\left(\mathcal Q_i\delta_{\mu_i,\xi^{(i)},a_i, p_i}\right)_x \wedge \left(\mathcal Q_i\delta_{\mu_i,\xi^{(i)},a_i, p_i}\right)_y,$$
we have
\begin{equation}\label{e:3.12}
\begin{aligned}
&\int_{\mathcal{D}}v\cdot \left[-\Delta z_{\mathcal{A}} +2\left(z_{\mathcal{A}}\right)_x \wedge \left(z_{\mathcal{A}}\right)_y \right]
\\&=2\sum_{i\neq j}\int_{\mathcal{D}}v\cdot\left[\left(z_{\mathcal{A}_i}\right)_x
\wedge \left(z_{\mathcal{A}_j}\right)_y\right]+2\sum_{i=1}^k  \int_{\mathcal{D}}v\cdot\left[ \left(\mathcal Q_i\varphi_{\mu_i,\xi^{(i)},a_i, p_i}\right)_x
\wedge \left(\mathcal Q_i\varphi_{\mu_i,\xi^{(i)},a_i, p_i}\right)_y\right]
\\&\quad -2\sum_{i=1}^k  \int_{\mathcal{D}}v\cdot \left[ \left(\mathcal Q_i\delta_{\mu_i,\xi^{(i)},a_i, p_i}\right)_x\wedge\left(\mathcal Q_i\varphi_{\mu_i,\xi^{(i)},a_i, p_i}\right)_y+\left(\mathcal Q_i\varphi_{\mu_i,\xi^{(i)},a_i, p_i}\right)_x
\wedge \left(\mathcal Q_i\delta_{\mu_i,\xi^{(i)},a_i, p_i}\right)_y \right] .
\end{aligned}
\end{equation}
Moreover, for the terms containing the boundary function $g=\sum_{i=1}^k g^{(i)}(z,\omega^{(i)})$ in (\ref{reduction-12}), integrating by parts, one has
\begin{equation}\label{Mar10-1}
\begin{aligned}
&2\varepsilon \int_{\mathcal{D}}z_{\mathcal{A}}\cdot (v_x \wedge g_y + g_x \wedge v_y)+2\varepsilon^2 \int_{\mathcal{D}}v\cdot (g_x \wedge g_y)
\\&=2\varepsilon \sum_{i\neq j}\int_{\mathcal{D}}\left(z_{\mathcal{A}_i}\right)\cdot \left[v_x \wedge \left(g^{(j)}\right)_y + \left(g^{(j)}\right)_x \wedge v_y\right]+2\varepsilon^2 \sum_{i\neq j}\int_{\mathcal{D}}v\cdot \left[\left(g^{(i)}\right)_x \wedge \left(g^{(j)}\right)_y\right]
\\&\quad+2\varepsilon \sum_{i=1}^{k}\int_{\mathcal{D}}\left(z_{\mathcal{A}_i}\right)\cdot \left[v_x \wedge \left(g^{(i)}\right)_y + \left(g^{(i)}\right)_x \wedge v_y\right]+2\varepsilon^2 \sum_{i=1}^{k}\int_{\mathcal{D}}v\cdot \left[\left(g^{(i)}\right)_x \wedge \left(g^{(i)}\right)_y\right].
\end{aligned}
\end{equation}

We first consider the following terms in (\ref{e:3.12}) and (\ref{Mar10-1}),
\begin{equation}\label{Mar10-7}
\begin{aligned}
&2\sum_{i\neq j}\int_{\mathcal{D}}v\cdot\left[\left(z_{\mathcal{A}_i}\right)_x
\wedge \left(z_{\mathcal{A}_j}\right)_y\right]
\\&+2\varepsilon \sum_{i\neq j}\int_{\mathcal{D}}\left(z_{\mathcal{A}_i}\right)\cdot \left[v_x \wedge \left(g^{(j)}\right)_y + \left(g^{(j)}\right)_x \wedge v_y\right]+2\varepsilon^2 \sum_{i\neq j}\int_{\mathcal{D}}v\cdot \left[\left(g^{(i)}\right)_x \wedge \left(g^{(j)}\right)_y\right].
\end{aligned}
\end{equation}
Recall that
$$z_{\mathcal{A}_i}=P\mathcal Q_i\delta_{\mu_i,\xi^{(i)},a_i, p_i}= P\mathcal Q_i\delta_{\mu_i,\xi^{(i)}}+P\mathcal Q_i\mathcal{L}_{\mathcal{A}_i}+P\mathcal Q_i\mathcal{R}_{\mathcal{A}_i},$$
where $\mathcal{L}_{\mathcal{A}_i}:=\sum_{l=1}^2\left(a_{i,l}Z_{-1,l}+p_{i,l} Z_{2,l}\right)$ represents the linear part, $\mathcal{R}_{\mathcal{A}_i}$ denotes the higher order terms.

For any given $i\neq j$, we now focus on the terms involving $P\mathcal Q_i\delta_{\mu_i,\xi^{(i)}}$ and $P\mathcal Q_j\delta_{\mu_j,\xi^{(j)}}$ in (\ref{Mar10-7}). For convenience, we denote $P\mathcal Q_i\delta_{\mu_i,\xi^{(i)}}$ and $P\mathcal Q_j\delta_{\mu_j,\xi^{(j)}}$ as $P\mathcal Q_i\delta_i$ and $P\mathcal Q_j\delta_j$, respectively.
There holds
\begin{equation*}
\begin{aligned}
&2\int_{\mathcal{D}}v\cdot\left[\left(P\mathcal{Q}_i\delta_i\right)_x
\wedge \left(P\mathcal{Q}_j\delta_j\right)_y+\left(P\mathcal{Q}_j\delta_j\right)_x\wedge \left(P\mathcal{Q}_i\delta_i\right)_y \right]
+2\varepsilon \int_{\mathcal{D}}v\cdot \left[\left(P\mathcal{Q}_i\delta_i\right)_x \wedge \left(g^{(j)}\right)_y + \left(g^{(j)}\right)_x \wedge \left(P\mathcal{Q}_i\delta_i\right)_y\right]
\\&+2\varepsilon \int_{\mathcal{D}}v\cdot \left[\left(P\mathcal{Q}_j\delta_j\right)_x \wedge \left(g^{(i)}\right)_y + \left(g^{(i)}\right)_x \wedge \left(P\mathcal{Q}_j\delta_j\right)_y\right]
+2\varepsilon^2 \int_{\mathcal{D}}v\cdot \left[\left(g^{(i)}\right)_x \wedge \left(g^{(j)}\right)_y+\left(g^{(j)}\right)_x\wedge \left(g^{(i)}\right)_y\right]
\\&=2\int_{\mathcal{D}}v\cdot\left[\left(P\mathcal{Q}_i\delta_i+\varepsilon g^{(i)}\right)_x
\wedge \left(P\mathcal{Q}_j\delta_j+\varepsilon g^{(j)}\right)_y+\left(P\mathcal{Q}_j\delta_j+\varepsilon g^{(j)}\right)_x\wedge \left(P\mathcal{Q}_i\delta_i+\varepsilon g^{(i)}\right)_y \right].
\end{aligned}
\end{equation*}
Note that the boundary function $g^{(i)}(z,\omega^{(i)})=\left(g^{(i)}_1(z,\omega^{(i)}),g^{(i)}_2(z,\omega^{(i)}),g^{(i)}_3(z,\omega^{(i)})\right)$ is defined as follows
$$g^{(i)}_1(z,\omega^{(i)})=2h_1^{(1)}(z,\omega^{(i)}),\quad g^{(i)}_2(z,\omega^{(i)})=2h_2^{(1)}(z,\omega^{(i)}),\quad g^{(i)}_3(z,\omega^{(i)})=-2\varepsilon h_3^{(1)}(z,\omega^{(i)}), $$
where $h_1^{(1)}(z,\omega^{(i)}),~h_2^{(1)}(z,\omega^{(i)}),~h_3^{(1)}(z,\omega^{(i)})$ are the functions given in (\ref{pre-4}).

Let $\gamma =\frac{d}{2}$ and divide the integral domain into three regions $B_i:=B_{\gamma}(\xi^{(i)})$, $B_j:=B_{\gamma}(\xi^{(j)})$ and $\mathcal{D} \setminus\left(B_i \cup B_j\right)$. From the definition of $Z$, see (\ref{solution-set}), it is clear that $B_i$ and $B_j$ are disjoint. Using (\ref{estimates-1}) and H\"{o}lder's inequality, along with the fact that
$$\left\|\nabla \left(P\mathcal{Q}_i\delta_i+\varepsilon g^{(i)}\right)\right\|_{L^p(B_i)}=O\left(\mu_i^{\frac{2-p}{p}}\right), \quad
\left\|\nabla \left(P\mathcal{Q}_j\delta_j+\varepsilon g^{(j)} \right)\right\|_{L^{\infty}(B_i)}=O\left(\mu_j^2\right),$$
for $p=\frac{4}{4-\alpha}$, where $\alpha$ is the constant in \eqref{asymp-d}, it is straightforward to verify that
\begin{equation*}
\begin{aligned}
&\left|\int_{B_i}v\cdot\left[\left(P\mathcal{Q}_i\delta_i+\varepsilon g^{(i)}\right)_x
\wedge \left(P\mathcal{Q}_j\delta_j+\varepsilon g^{(j)}\right)_y+\left(P\mathcal{Q}_j\delta_j+\varepsilon g^{(j)}\right)_x\wedge \left(P\mathcal{Q}_i\delta_i+\varepsilon g^{(i)}\right)_y \right]\right|
\\&=O\left(\mu_j^2\right)\left\|\nabla \left(P\mathcal{Q}_i\delta_i+\varepsilon g^{(i)}\right)\right\|_{L^p(B_i)}\| v\|_{L^q(B_i)}
\\&=O\left(\mu_i^{1-\frac{\alpha}{2}}\mu_j^2\right)\| v\|_{ H_0^1(\mathcal{D})},
\end{aligned}
\end{equation*}
where $\frac{1}{p}+\frac{1}{q}=1$. Similarly, we have
\begin{equation*}
\begin{aligned}
&\left|\int_{B_j}v\cdot\left[\left(P\mathcal{Q}_i\delta_i+\varepsilon g^{(i)}\right)_x
\wedge \left(P\mathcal{Q}_j\delta_j+\varepsilon g^{(j)}\right)_y+\left(P\mathcal{Q}_j\delta_j+\varepsilon g^{(j)}\right)_x\wedge \left(P\mathcal{Q}_i\delta_i+\varepsilon g^{(i)}\right)_y \right]\right|
\\&=O\left(\mu_i^2\mu_j^{1-\frac{\alpha}{2}}\right)\| v\|_{ H_0^1(\mathcal{D})}.
\end{aligned}
\end{equation*}
And since $\left\|\nabla \left(P\mathcal{Q}_j\delta_j+\varepsilon g^{(j)} \right)\right\|_{L^{\infty}(\mathcal{D} \backslash \left(B_i \cup B_j\right))}=O\left(\mu_j^2+(\mu_j^2-\varepsilon)d^{-3}\right)$, there holds
\begin{equation*}
\begin{aligned}
&\left|\int_{\mathcal{D} \setminus\left(B_i \cup B_j\right)}v\cdot\left[\left(P\mathcal{Q}_i\delta_i+\varepsilon g^{(i)}\right)_x
\wedge \left(P\mathcal{Q}_j\delta_j+\varepsilon g^{(j)}\right)_y+\left(P\mathcal{Q}_j\delta_j+\varepsilon g^{(j)}\right)_x\wedge \left(P\mathcal{Q}_i\delta_i+\varepsilon g^{(i)}\right)_y \right]\right|
\\&=O\left(\left(\mu_i^2+(\mu_i^2-\varepsilon)d^{-3}\right)\left(\mu_j^2+(\mu_j^2-\varepsilon)d^{-3}\right)\right)\| v\|_{ H_0^1(\mathcal{D})}.
\end{aligned}
\end{equation*}
Thus, we conclude that
\begin{equation}\label{Mar10-3}
\begin{aligned}
&\left|\int_{\mathcal{D}}v\cdot\left[\left(P\mathcal{Q}_i\delta_i+\varepsilon g^{(i)}\right)_x
\wedge \left(P\mathcal{Q}_j\delta_j+\varepsilon g^{(j)}\right)_y+\left(P\mathcal{Q}_j\delta_j+\varepsilon g^{(j)}\right)_x\wedge \left(P\mathcal{Q}_i\delta_i+\varepsilon g^{(i)}\right)_y \right]\right|
\\&=O\left(\mu_i^{1-\frac{\alpha}{2}}\mu_j^2+\mu_i^2\mu_j^{1-\frac{\alpha}{2}}\right)\| v\|_{ H_0^1(\mathcal{D})}.
\end{aligned}
\end{equation}

In (\ref{Mar10-7}), for given $i\neq j$, we also need to consider terms of the following form.
\begin{equation}\label{Mar10-8}
\begin{aligned}
&2\int_{\mathcal{D}}v\cdot\left[\left(P\mathcal Q_i\mathcal{L}_{\mathcal{A}_i}+P\mathcal Q_i\mathcal{R}_{\mathcal{A}_i}\right)_x
\wedge \left(z_{\mathcal{A}_j}+\varepsilon g^{(j)}\right)_y+\left(z_{\mathcal{A}_j}+\varepsilon g^{(j)}\right)_x\wedge \left(P\mathcal Q_i\mathcal{L}_{\mathcal{A}_i}+P\mathcal Q_i\mathcal{R}_{\mathcal{A}_i}\right)_y\right]
\\&=2\int_{\mathcal{D}}\left(P\mathcal Q_i\mathcal{L}_{\mathcal{A}_i}+P\mathcal Q_i\mathcal{R}_{\mathcal{A}_i}\right)\cdot\left[v_x
\wedge \left(z_{\mathcal{A}_j}+\varepsilon g^{(j)}\right)_y+\left(z_{\mathcal{A}_j}+\varepsilon g^{(j)}\right)_x\wedge v_y\right].
\end{aligned}
\end{equation}
Since
\begin{equation}\label{Apr30-1}
\begin{aligned}
&\left\|\nabla \left(P\mathcal Q_i\mathcal{L}_{\mathcal{A}_i}+P\mathcal Q_i\mathcal{R}_{\mathcal{A}_i}\right)\right\|_{L^p(B_i)}=O\left(\sum_{l=1}^2a_{i,l}\mu_i d^{-2}+p_{i,l}\mu_i^4d^{-5}\right), \quad
\left\|\nabla \left(z_{\mathcal{A}_j}+\varepsilon g^{(j)}\right)\right\|_{L^{\infty}(B_i)}=O(\mu_j^2),
\\&\left\|\nabla \left(P\mathcal Q_i\mathcal{L}_{\mathcal{A}_i}+P\mathcal Q_i\mathcal{R}_{\mathcal{A}_i}\right)\right\|_{L^{\infty}(B_j)}=O\left(\sum_{l=1}^2a_{i,l}\mu_i+p_{i,l}\mu_i^4\right),
\\&
\left\|\nabla \left(z_{\mathcal{A}_j}+\varepsilon g^{(j)}\right)\right\|_{L^{p}(B_j)}=O\left(\mu_j^{1-\frac{\alpha}{2}}+\sum_{l=1}^2a_{j,l}\mu_j d^{-2}+p_{j,l}\mu_j^4d^{-5}\right),
\end{aligned}
\end{equation}
following a similar approach as in the estimate of (\ref{Mar10-3}), we can estimate (\ref{Mar10-8}) as
\begin{equation}\label{Mar10-9}
\begin{aligned}
&O\left(\sum_{l=1}^2a_{i,l}\mu_i\mu_j^2 d^{-2}+p_{i,l}\mu_i^4\mu_j^2d^{-5}+\sum_{l=1}^2a_{i,l}\mu_i\mu_j^{1-\frac{\alpha}{2}}+p_{i,l}\mu_i^4\mu_j^{1-\frac{\alpha}{2}}\right).
\end{aligned}
\end{equation}

For the remaining terms in (\ref{e:3.12}) and (\ref{Mar10-1}), we have
\begin{equation}\label{Mar10-2}
\begin{aligned}
&\left|2  \int_{\mathcal{D}}v\cdot\left[ \left(\mathcal Q_i\varphi_{\mu_i,\xi^{(i)},a_i, p_i}\right)_x
\wedge \left(\mathcal Q_i\varphi_{\mu_i,\xi^{(i)},a_i, p_i}\right)_y\right]\right.
\\&~-2\int_{\mathcal{D}}v\cdot \left[ \left(\mathcal Q_i\delta_{\mu_i,\xi^{(i)},a_i, p_i}\right)_x\wedge\left(\mathcal Q_i\varphi_{\mu_i,\xi^{(i)},a_i, p_i}\right)_y+\left(\mathcal Q_i\varphi_{\mu_i,\xi^{(i)},a_i, p_i}\right)_x
\wedge \left(\mathcal Q_i\delta_{\mu_i,\xi^{(i)},a_i, p_i}\right)_y \right]
\\&~\left.+2\varepsilon \int_{\mathcal{D}}v\cdot \left[\left(z_{\mathcal{A}_i}\right)_x \wedge \left(g^{(i)}\right)_y + \left(g^{(i)}\right)_x \wedge \left(z_{\mathcal{A}_i}\right)_y\right]+2\varepsilon^2 \int_{\mathcal{D}}v\cdot \left[\left(g^{(i)}\right)_x \wedge \left(g^{(i)}\right)_y\right]\right|
\\&=\left|2 \int_{\mathcal{D}}v\cdot\left[ \left(\mathcal Q_i\varphi_{\mu_i,\xi^{(i)},a_i, p_i}-\varepsilon g^{(i)}\right)_x
\wedge \left(\mathcal Q_i\varphi_{\mu_i,\xi^{(i)},a_i, p_i}-\varepsilon g^{(i)}\right)_y\right]\right.
\\&~\left.\quad-2\int_{\mathcal{D}}v\cdot \left[ \left(\mathcal Q_i\delta_{\mu_i,\xi^{(i)},a_i, p_i}\right)_x\wedge\left(\mathcal Q_i\varphi_{\mu_i,\xi^{(i)},a_i, p_i}-\varepsilon g^{(i)}\right)_y+\left(\mathcal Q_i\varphi_{\mu_i,\xi^{(i)},a_i, p_i}-\varepsilon g^{(i)}\right)_x
\wedge \left(\mathcal Q_i\delta_{\mu_i,\xi^{(i)},a_i, p_i}\right)_y \right]\right|.
\end{aligned}
\end{equation}
Since
\begin{equation}\label{Apr30-2}
\begin{aligned}
&\left\|\nabla \left(\mathcal Q_i\varphi_{\mu_i,\xi^{(i)},a_i, p_i}-\varepsilon g^{(i)}\right)\right\|_{L^{\infty}(\mathcal{D})}=O\left((\mu_i^2-\varepsilon)d^{-3}+\sum_{l=1}^2\left(|a_{i, l}|\mu_id^{-2}+|p_{i, l}|\mu_i^4d^{-5}\right)\right),
\\&\|\nabla \left(\mathcal Q_i\delta_{\mu_i,\xi^{(i)},a_i, p_i}\right)\|_{L^p(B_i)}=O\left(\mu_i^{1-\frac{\alpha}{2}}+\sum_{l=1}^2a_{i,l}\mu_i^{1-\frac{\alpha}{2}}+p_{i,l}\mu_i^{1-\frac{\alpha}{2}}\right),
\end{aligned}
\end{equation}
we can estimate (\ref{Mar10-2}) as
\begin{equation}\label{Mar10-5}
\begin{aligned}
O\left(\mu_i^{1-\frac{\alpha}{2}}\left((\mu_i^2-\varepsilon)d^{-3}+\sum_{l=1}^2\left(|a_{i, l}|\mu_id^{-2}+|p_{i, l}|\mu_i^4d^{-5}\right)\right)\right)\| v\|_{ H_0^1(\mathcal{D})}.
\end{aligned}
\end{equation}
Combining (\ref{Mar10-3}), (\ref{Mar10-9}), (\ref{Mar10-5}), we obtain the conclusion.
\end{proof}

\begin{lemma}\label{lemma2-3.4}
Let $Z$ be as in (\ref{solution-set}). Then for $z_{\mathcal{A}}\in Z$, $l=1,2$, we have the following estimates
\begin{equation*}
\begin{aligned}
&\left\|I_{\varepsilon}''\left(z_{\mathcal{A}}\right)\left[\frac{\partial z_{\mathcal{A}}}{\partial \xi_l^{(i)}}\right]\right\|
=O\left(\sum_{i=1}^k\mu_i^{3-\frac{\alpha}{2}}d^{-4}\right),\qquad
\left\|I_{\varepsilon}''\left(z_{\mathcal{A}}\right)\left[\frac{\partial z_{\mathcal{A}}}{\partial \mathcal Q_{i}}\right]\right\|=O\left(\sum_{i=1}^k\mu_i^{3-\frac{\alpha}{2}}d^{-3}\right),
\\&
\left\|I_{\varepsilon}''\left(z_{\mathcal{A}}\right)\left[\frac{\partial z_{\mathcal{A}}}{\partial \mu_i}\right]\right\|
=O\left(\sum_{i=1}^k\mu_i^{2-\frac{\alpha}{2}}d^{-3}\right),\qquad
\left\|I_{\varepsilon}''\left(z_{\mathcal{A}}\right)\left[\frac{\partial z_{\mathcal{A}}}{\partial a_{i,l}}\right]\right\|
=O\left(\sum_{i=1}^k\mu_i^{2-\frac{\alpha}{2}}d^{-2}\right),
\\&
\left\|I_{\varepsilon}''\left(z_{\mathcal{A}}\right)\left[\frac{\partial z_{\mathcal{A}}}{\partial p_{i,l}}\right]\right\|
=O\left(\sum_{i=1}^k\mu_i^{5-\frac{\alpha}{2}}d^{-5}\right).
\end{aligned}
\end{equation*}
\end{lemma}
\begin{proof}
The functional $I_{\varepsilon}''\left(z_{\mathcal{A}}\right)$ is defined as follows,
$$
I_{\varepsilon}''\left(z_{\mathcal{A}}\right)[v,\tilde v]=\int_{\mathcal{D}}\nabla v\cdot \nabla \tilde v+2\int_{\mathcal{D}}\tilde v\cdot \left(\left(z_{\mathcal{A}}\right)_x \wedge v_y +v_x \wedge  \left(z_{\mathcal{A}}\right)_y\right)+2\varepsilon \int_{\mathcal{D}}\tilde v \cdot (g_x \wedge v_y + v_x \wedge  g_y),\quad \forall \ v, \tilde v \in H_0^1(\mathcal{D},\mathbb{R}^3).
$$
Now, we choose $\tilde v=\frac{\partial z_{\mathcal{A}}}{\partial \xi_1^{(i)}}$ and let $v\in H_0^1(\mathcal{D},\mathbb{R}^3)$ be an arbitrary test function. Clearly, we have
\begin{equation*}
\frac{\partial z_{\mathcal{A}}}{\partial \xi_1^{(i)}}=\frac{\partial \left(\mathcal Q_i\delta_{\mu_i,\xi^{(i)},a_i, p_i}\right)}{\partial \xi_1^{(i)}}-\frac{\partial \left(\mathcal Q_i\varphi_{\mu_i,\xi^{(i)},a_i, p_i}\right)}{\partial \xi_1^{(i)}},
\end{equation*}
where $\varphi_{\mu_i,\xi^{(i)},a_i, p_i}$ is defined in (\ref{pre-3}) corresponding to $\delta_{\mu_i,\xi^{(i)},a_i, p_i}$. We have
\begin{equation}\label{reduction-16}
\begin{aligned}
&\int_{\mathcal{D}}\nabla v\cdot \nabla \frac{\partial z_{\mathcal{A}}}{\partial \xi_1^{(i)}}+2\int_{\mathcal{D}}\frac{\partial z_{\mathcal{A}}}{\partial \xi_1^{(i)}}\cdot \left(\left(z_{\mathcal{A}}\right)_x \wedge v_y +v_x \wedge  \left(z_{\mathcal{A}}\right)_y\right)+2\varepsilon \int_{\mathcal{D}}\frac{\partial z_{\mathcal{A}}}{\partial \xi_1^{(i)}} \cdot (g_x \wedge v_y + v_x \wedge  g_y)
\\&=\int_{\mathcal{D}}\nabla v\cdot \nabla\left(\frac{\partial\left(\mathcal Q_i\delta_{\mu_i,\xi^{(i)},a_i, p_i}\right) }{\partial \xi_1^{(i)}} \right)\\
&\quad +2\int_{\mathcal{D}}\frac{\partial\left(\mathcal Q_i\delta_{\mu_i,\xi^{(i)},a_i, p_i}\right)}{\partial \xi_1^{(i)}} \cdot\left[\Big(\mathcal Q_i\delta_{\mu_i,\xi^{(i)},a_i, p_i}\Big)_x \wedge v_y +v_x \wedge  \Big(\mathcal Q_i\delta_{\mu_i,\xi^{(i)},a_i, p_i}\Big)_y\right]
\\& \quad+2\sum_{j\neq i}\int_{\mathcal{D}}\frac{\partial\left(\mathcal Q_i\delta_{\mu_i,\xi^{(i)},a_i, p_i}\right)}{\partial \xi_1^{(i)}}\cdot\left[\left(z_{\mathcal{A}_j}+\varepsilon g^{(j)}\right)_x \wedge v_y +v_x \wedge  \left(z_{\mathcal{A}_j}+\varepsilon g^{(j)}\right)_y\right]
\\&\quad-\int_{\mathcal{D}}\nabla v\cdot \nabla \frac{\partial \left(\mathcal Q_i\varphi_{\mu_i,\xi^{(i)},a_i, p_i}\right)}{\partial \xi_1^{(i)}}-2\int_{\mathcal{D}}\frac{\partial \left(\mathcal Q_i\varphi_{\mu_i,\xi^{(i)},a_i, p_i}\right)}{\partial \xi_1^{(i)}}\cdot \left[\left(z_{\mathcal{A}}+\varepsilon g\right)_x \wedge v_y
+v_x \wedge  \left(z_{\mathcal{A}}+\varepsilon g\right)_y\right]
\\&\quad-2 \int_{\mathcal{D}}\frac{\partial\left(\mathcal Q_i\delta_{\mu_i,\xi^{(i)},a_i, p_i}\right)}{\partial \xi_1^{(i)}} \cdot \left[\left(\mathcal Q_i\varphi_{\mu_i,\xi^{(i)},a_i, p_i}-\varepsilon g^{(i)}\right)_x \wedge v_y +v_x \wedge  \left(\mathcal Q_i\varphi_{\mu_i,\xi^{(i)},a_i, p_i}-\varepsilon g^{(i)}\right)_y\right].
\end{aligned}
\end{equation}
Since $
\Delta (\mathcal Q_i\delta_{\mu_i,\xi^{(i)},a_i, p_i}) = 2(\mathcal Q_i\delta_{\mu_i,\xi^{(i)},a_i, p_i})_x\wedge (\mathcal Q_i\delta_{\mu_i,\xi^{(i)},a_i, p_i})_y,
$
integrating by parts, we have
\begin{equation*}
\begin{aligned}
&\int_{\mathcal{D}}\nabla v\cdot  \nabla\left(\frac{\partial\left(\mathcal Q_i\delta_{\mu_i,\xi^{(i)},a_i, p_i}\right)}{\partial \xi_1^{(i)}} \right)
\\&+2\int_{\mathcal{D}}\frac{\partial \left(\mathcal Q_i\delta_{\mu_i,\xi^{(i)},a_i, p_i}\right)}{\partial \xi_1^{(i)}} \cdot\left[\Big(\mathcal Q_i\delta_{\mu_i,\xi^{(i)},a_i, p_i}\Big)_x \wedge v_y +v_x \wedge  \Big(\mathcal Q_i\delta_{\mu_i,\xi^{(i)},a_i, p_i}\Big)_y\right] = 0.
\end{aligned}
\end{equation*}
For the third term in (\ref{reduction-16}), let $B_i$ and $B_j$ denote the balls of radius $\frac{d}{2}$ centered at $\xi^{(i)}$ and $\xi^{(j)}$, respectively. By (\ref{estimates-4}), we obtain
\begin{align*}
&\Bigg|2\sum_{j\neq i}\int_{\mathcal{D}}\frac{\partial  \left(\mathcal Q_i\delta_{\mu_i,\xi^{(i)},a_i, p_i}\right)}{\partial \xi_1^{(i)}}\cdot\left[\left(z_{\mathcal{A}_j}+\varepsilon g^{(j)}\right)_x \wedge v_y +v_x \wedge  \left(z_{\mathcal{A}_j}+\varepsilon g^{(j)}\right)_y\right]\Bigg|
\\&=\Bigg|2\sum_{j\neq i}\left(\int_{B_i}+\int_{B_j}+\int_{\mathcal{D} \backslash \left(B_i \cup B_j\right)}\right)\frac{\partial \left(\mathcal Q_i\delta_{\mu_i,\xi^{(i)},a_i, p_i}\right)}{\partial \xi_1^{(i)}} \cdot\left[\left(z_{\mathcal{A}_j}+\varepsilon g^{(j)}\right)_x \wedge v_y +v_x \wedge  \left(z_{\mathcal{A}_j}+\varepsilon g^{(j)}\right)_y\right]\Bigg|
\\&\lesssim \sum_{j\neq i}\left\|\frac{\partial \left(\mathcal Q_i\delta_{\mu_i,\xi^{(i)},a_i, p_i}\right)}{\partial \xi_1^{(i)}}\right\|_{L^2(B_i)}\left\|\left(z_{\mathcal{A}_j}+\varepsilon g^{(j)}\right)_x\right\|_{L^{\infty}(B_i)}\| v\|_{ H_0^1(\mathcal{D})}
\\&\quad+\sum_{j\neq i}\left\|\frac{\partial \left(\mathcal Q_i\delta_{\mu_i,\xi^{(i)},a_i, p_i}\right)}{\partial \xi_1^{(i)}}\right\|_{L^{\infty}(B_j)}\left\|\left(z_{\mathcal{A}_j}+\varepsilon g^{(j)}\right)_x\right\|_{L^2(B_j)}\| v\|_{ H_0^1(\mathcal{D})}
\\&\quad+\sum_{j\neq i}\left\|\frac{\partial \left(\mathcal Q_i\delta_{\mu_i,\xi^{(i)},a_i, p_i}\right)}{\partial \xi_1^{(i)}}\right\|_{L^{\infty}(\mathcal{D} \backslash \left(B_i \cup B_j\right))}\left\|\left(z_{\mathcal{A}_j}+\varepsilon g^{(j)}\right)_x\right\|_{L^{\infty}(\mathcal{D} \backslash \left(B_i \cup B_j\right))}\| v\|_{ H_0^1(\mathcal{D})}
\\&=O\left(\sum_{j\neq i}\mu_j^2+\mu_i^2+\mu_i^2\mu_j^2\right)\| v\|_{ H_0^1(\mathcal{D})}.
\end{align*}
We proceed to estimate the terms in (\ref{reduction-16}) involving $\mathcal Q_i\varphi_{\mu_i,\xi^{(i)},a_i, p_i}$. Since
$$
\Delta (\mathcal Q_i\varphi_{\mu_i,\xi^{(i)},a_i, p_i}) = 0\quad \text{in}~ \mathcal{D},$$
we apply integration by parts, which gives
$$\int_{\mathcal{D}}\nabla v\cdot \nabla \frac{\partial (\mathcal Q_i\varphi_{\mu_i,\xi^{(i)},a_i, p_i})}{\partial \xi_1^{(i)}}=0.$$
Furthermore, by \eqref{Apr30-1} and
$$\left\|\nabla\frac{\partial (\mathcal Q_i\varphi_{\mu_i,\xi^{(i)},a_i, p_i})}{\partial \xi_1^{(i)}}\right\|_{L^{\infty}(\mathcal{D})}=O(\mu_i^2d^{-4}),$$
we get
\begin{equation*}
\begin{aligned}
&\left|2\int_{\mathcal{D}}\frac{\partial (\mathcal Q_i\varphi_{\mu_i,\xi^{(i)},a_i, p_i})}{\partial \xi_1^{(i)}}\cdot \left[\left(z_{\mathcal{A}}+\varepsilon g\right)_x \wedge v_y
+v_x \wedge  \left(z_{\mathcal{A}}+\varepsilon g\right)_y\right]\right|
\\&=\left|2\int_{\mathcal{D}}v\cdot \left[\left(z_{\mathcal{A}}+\varepsilon g\right)_x \wedge \left(\frac{\partial (\mathcal Q_i\varphi_{\mu_i,\xi^{(i)},a_i, p_i})}{\partial \xi_1^{(i)}}\right)_y
+\left(\frac{\partial (\mathcal Q_i\varphi_{\mu_i,\xi^{(i)},a_i, p_i})}{\partial \xi_1^{(i)}}\right)_x \wedge  \left(z_{\mathcal{A}}+\varepsilon g\right)_y\right]\right|
\\&=O\left(\sum_{j\neq i}\mu_i^{3-\frac{\alpha}{2}}d^{-4}+\mu_i^2\mu_j^{1-\frac{\alpha}{2}}+\mu_i^2d^{-4}\mu_j^{2}\right)\| v\|_{ H_0^1(\mathcal{D})}.
\end{aligned}
\end{equation*}
Additionally, through \eqref{Apr30-2} and the boundedness of $\left\|\frac{\partial\left(\mathcal Q_i\delta_{\mu_i,\xi^{(i)},a_i, p_i}\right)}{\partial \xi_1^{(i)}}\right\|_{L^{2}(\mathcal{D})}$, we obtain
\begin{equation*}
\begin{aligned}
&\left|2 \int_{\mathcal{D}}\frac{\partial\left(\mathcal Q_i\delta_{\mu_i,\xi^{(i)},a_i, p_i}\right)}{\partial \xi_1^{(i)}} \cdot \left[\left(\mathcal Q_i\varphi_{\mu_i,\xi^{(i)},a_i, p_i}-\varepsilon g^{(i)}\right)_x \wedge v_y +v_x \wedge  \left(\mathcal Q_i\varphi_{\mu_i,\xi^{(i)},a_i, p_i}-\varepsilon g^{(i)}\right)_y\right]\right|
\\&=O\left((\mu_i^2-\varepsilon)d^{-3}+\sum_{l=1}^2\left(|a_{i, l}|\mu_id^{-2}+|p_{i, l}|\mu_i^4d^{-5}\right)\right)\| v\|_{ H_0^1(\mathcal{D})}.
\end{aligned}
\end{equation*}
Combine the above estimates, and recall that for all $i=1,\dots,k$, $\frac{\varepsilon^{\frac{1}{2}}}{\mu_i}, \ \frac{\varepsilon^{\frac{1}{2}}d^2}{|a_{i}|}, \ \frac{\varepsilon^{-1}d^5}{|p_{i}|}\in \left[\bar{C}^{-1},\bar C\right]$, we conclude that
$$
\left\|I_{\varepsilon}''\left(z_{\mathcal{A}}\right)\left[\frac{\partial z_{\mathcal{A}}}{\partial \xi_1^{(i)}}\right]\right\| =O\left(\sum_{i=1}^k\mu_i^{3-\frac{\alpha}{2}}d^{-4}\right).
$$
From similar arguments, we can obtain the results of this lemma.
\end{proof}
\begin{prop}\label{prop-3.6}
Let $w_{\varepsilon}\left(z_{\mathcal{A}}\right)$ be as in Proposition \ref{prop-3.3}, then
\begin{equation}\label{reduction-19}
\left|I_{\varepsilon}(z_{\mathcal{A}}+w_{\varepsilon}\left(z_{\mathcal{A}}\right))-I_{\varepsilon}\left(z_{\mathcal{A}}\right)\right|\lesssim e^2(\varepsilon),\quad \forall z_{\mathcal{A}}\in Z,
\end{equation}
where $e(\varepsilon)$ is given in Lemma \ref{lemma-3.4}. Moreover, for any $z_{\mathcal{A}}\in Z$, $l=1,2$, there holds
\begin{equation}\label{reduction-20}
\begin{aligned}
&\left|\frac{\partial I_{\varepsilon}(z_{\mathcal{A}}+w_{\varepsilon}\left(z_{\mathcal{A}}\right))}{\partial \xi_l^{(i)}}-\frac{\partial I_{\varepsilon}\left(z_{\mathcal{A}}\right)}{\partial \xi_l^{(i)}}\right|=O\left(\sum_{i=1}^k\mu_i^{3-\frac{\alpha}{2}}d^{-4}\right) e(\varepsilon)+\frac{e^2(\varepsilon)}{\mu_i},
\\&\left|\frac{\partial I_{\varepsilon}(z_{\mathcal{A}}+w_{\varepsilon}\left(z_{\mathcal{A}}\right))}{\partial \mathcal Q_{i}}-\frac{\partial I_{\varepsilon}\left(z_{\mathcal{A}}\right)}{\partial \mathcal Q_{i}}\right|=O\left(\sum_{i=1}^k\mu_i^{3-\frac{\alpha}{2}}d^{-3}\right) e(\varepsilon)+e^2(\varepsilon),
\\&\left|\frac{\partial I_{\varepsilon}(z_{\mathcal{A}}+w_{\varepsilon}\left(z_{\mathcal{A}}\right))}{\partial \mu_i}-\frac{\partial I_{\varepsilon}\left(z_{\mathcal{A}}\right)}{\partial \mu_i}\right|=O\left(\sum_{i=1}^k\mu_i^{2-\frac{\alpha}{2}}d^{-3}\right)e(\varepsilon)+\frac{e^2(\varepsilon)}{\mu_i},
\\&\left|\frac{\partial I_{\varepsilon}(z_{\mathcal{A}}+w_{\varepsilon}\left(z_{\mathcal{A}}\right))}{\partial a_{i,l}}-\frac{\partial I_{\varepsilon}\left(z_{\mathcal{A}}\right)}{\partial a_{i,l}}\right|=O\left(\sum_{i=1}^k\mu_i^{2-\frac{\alpha}{2}}d^{-2}\right)e(\varepsilon)+e^2(\varepsilon),
\\&\left|\frac{\partial I_{\varepsilon}(z_{\mathcal{A}}+w_{\varepsilon}\left(z_{\mathcal{A}}\right))}{\partial p_{i,l}}-\frac{\partial I_{\varepsilon}\left(z_{\mathcal{A}}\right)}{\partial p_{i,l}}\right|=O\left(\sum_{i=1}^k\mu_i^{5-\frac{\alpha}{2}}d^{-5}\right)e(\varepsilon)+e^2(\varepsilon).
\end{aligned}
\end{equation}
\end{prop}
\begin{proof}
From the mean value theorem, we have
\begin{equation}\label{reduction-21}
\begin{aligned}
I_{\varepsilon}(z_{\mathcal{A}}+w_{\varepsilon}\left(z_{\mathcal{A}}\right))-I_{\varepsilon}\left(z_{\mathcal{A}}\right)&=\int_0^1 I'_{\varepsilon}(z_{\mathcal{A}}+sw_{\varepsilon}\left(z_{\mathcal{A}}\right))[w_{\varepsilon}\left(z_{\mathcal{A}}\right)]ds
\\&=I'_{\varepsilon}\left(z_{\mathcal{A}}\right)[w_{\varepsilon}\left(z_{\mathcal{A}}\right)]+\int_0^1 \left(I'_{\varepsilon}\left(z_{\mathcal{A}}+sw_{\varepsilon}\left(z_{\mathcal{A}}\right)\right)-I'_{\varepsilon}\left(z_{\mathcal{A}}\right)\right)[w_{\varepsilon}\left(z_{\mathcal{A}}\right)]ds.
\end{aligned}\end{equation}
Since the functional $I''_{\varepsilon}$ is locally bounded, we have
$$\begin{aligned}
\left|I'_{\varepsilon}\left(z_{\mathcal{A}}+sw_{\varepsilon}\left(z_{\mathcal{A}}\right)\right)-I'_{\varepsilon}\left(z_{\mathcal{A}}\right)\right|&\leq \left|\int_0^1 I''_{\varepsilon}\left(z_{\mathcal{A}}+tsw_{\varepsilon}\left(z_{\mathcal{A}}\right)\right)[w_{\varepsilon}\left(z_{\mathcal{A}}\right)]dt\right|
\\& \leq \sup_{t,s\in[0,1]}\|I''_{\varepsilon}\left(z_{\mathcal{A}}+tsw_{\varepsilon}\left(z_{\mathcal{A}}\right)\right)\|\|w_{\varepsilon}\left(z_{\mathcal{A}}\right)\|
\\&\lesssim \|w_{\varepsilon}\left(z_{\mathcal{A}}\right)\|.
\end{aligned}$$
In the sequel, we abbreviate $\|\cdot\|_{H_0^{1}(\mathcal{D})}$ as $\|\cdot\|$. By Proposition \ref{prop-3.3} and Lemma \ref{lemma-3.4}, we obtain
$$\begin{aligned}
\left|I_{\varepsilon}\left(z_{\mathcal{A}}+w_{\varepsilon}\left(z_{\mathcal{A}}\right)\right)-I_{\varepsilon}\left(z_{\mathcal{A}}\right)\right|\lesssim \|I'_{\varepsilon}\left(z_{\mathcal{A}}\right)\|\|w_{\varepsilon}\left(z_{\mathcal{A}}\right)\|+\|w_{\varepsilon}\left(z_{\mathcal{A}}\right)\|^2\lesssim e^2(\varepsilon).
\end{aligned}$$
Thus (\ref{reduction-19}) holds.

Next, we prove (\ref{reduction-20}). Differentiating the equation $F_{\varepsilon}\left(z_{\mathcal{A}},w_{\varepsilon},\beta \right)$ with respect to $\xi_1^{(i)}$, which is defined in (\ref{reduction-11:2}), we obtain
$$0=\frac{\partial F_{\varepsilon}}{\partial z_{\mathcal{A}}}\left[\frac{\partial z_{\mathcal{A}}}{\partial \xi_1^{(i)}}\right]+\frac{\partial F_{\varepsilon}}{\partial (w_{\varepsilon},\beta)}\frac{\partial(w_{\varepsilon},\beta)}{\partial \xi_1^{(i)}}=I''_{\varepsilon}\left(z_{\mathcal{A}}+w_{\varepsilon}\left(z_{\mathcal{A}}\right)\right)\left[\frac{\partial z_{\mathcal{A}}}{\partial \xi_1^{(i)}}\right]+\frac{\partial F_{\varepsilon}}{\partial (w_{\varepsilon},\beta)}\frac{\partial(w_{\varepsilon},\beta)}{\partial \xi_1^{(i)}}.$$
From Lemma \ref{lemma-3.2}, $\frac{\partial F_{\varepsilon}}{\partial (w_{\varepsilon},\beta)}$ is uniformly invertible. Combine the fact that $I''_{\varepsilon}$ is locally Lipschitz, we have
\begin{equation*}
\frac{\partial(w_{\varepsilon},\beta)}{\partial \xi_1^{(i)}}=-\left(\frac{\partial F_{\varepsilon}}{\partial (w_{\varepsilon},\beta)}\right)^{-1}\left[I''_{\varepsilon}\left(z_{\mathcal{A}}+w_{\varepsilon}\left(z_{\mathcal{A}}\right)\right)\left[\frac{\partial z_{\mathcal{A}}}{\partial \xi_1^{(i)}}\right]\right],
\end{equation*}
and then
\begin{equation}\label{reduction-22}
\left\|\frac{\partial(w_{\varepsilon},\beta)}{\partial \xi_1^{(i)}}\right\|\lesssim \left\|I''_{\varepsilon}\left(z_{\mathcal{A}}+w_{\varepsilon}\left(z_{\mathcal{A}}\right)\right)\left[\frac{\partial z_{\mathcal{A}}}{\partial \xi_1^{(i)}}\right]\right\|\lesssim \left(\left\|I''_{\varepsilon}\left(z_{\mathcal{A}}\right)\left[\frac{\partial z_{\mathcal{A}}}{\partial \xi_1^{(i)}}\right]\right\|+\left\|w_{\varepsilon}\left(z_{\mathcal{A}}\right)\right\|\left\|\frac{\partial z_{\mathcal{A}}}{\partial \xi_1^{(i)}}\right\|\right).
\end{equation}
By (\ref{reduction-21}), we have
\begin{equation*}
\begin{aligned}
&\frac{\partial I_{\varepsilon}\left(z_{\mathcal{A}}+w_{\varepsilon}\left(z_{\mathcal{A}}\right)\right)}{\partial \xi_1^{(i)}}-\frac{\partial I_{\varepsilon}\left(z_{\mathcal{A}}\right)}{\partial \xi_1^{(i)}}
\\&= I''_{\varepsilon}\left(z_{\mathcal{A}}\right)\left[\frac{\partial z_{\mathcal{A}}}{\partial \xi_1^{(i)}},w_{\varepsilon}\left(z_{\mathcal{A}}\right)\right]
+I'_{\varepsilon}\left(z_{\mathcal{A}}\right)\left[\frac{\partial w_{\varepsilon}\left(z_{\mathcal{A}}\right)}{\partial \xi_1^{(i)}}\right]
\\&\quad+\int_0^1 \left(I''_{\varepsilon}\left(z_{\mathcal{A}}+sw_{\varepsilon}\left(z_{\mathcal{A}}\right)\right)-I''_{\varepsilon}\left(z_{\mathcal{A}}\right)\right)\left[\frac{\partial z_{\mathcal{A}}}{\partial \xi_1^{(i)}},w_{\varepsilon}\left(z_{\mathcal{A}}\right)\right]ds
\\&\quad+\int_0^1 s\left(I''_{\varepsilon}\left(z_{\mathcal{A}}+sw_{\varepsilon}\left(z_{\mathcal{A}}\right)\right)-I''_{\varepsilon}\left(z_{\mathcal{A}}\right)\right)\left[\frac{\partial w_{\varepsilon}\left(z_{\mathcal{A}}\right)}{\partial \xi_1^{(i)}},w_{\varepsilon}\left(z_{\mathcal{A}}\right)\right]ds+\int_0^1 sI''_{\varepsilon}\left(z_{\mathcal{A}}\right)\left[\frac{\partial w_{\varepsilon}\left(z_{\mathcal{A}}\right)}{\partial \xi_1^{(i)}},w_{\varepsilon}\left(z_{\mathcal{A}}\right)\right]ds
\\&\quad+\int_0^1 \left(I'_{\varepsilon}\left(z_{\mathcal{A}}+sw_{\varepsilon}\left(z_{\mathcal{A}}\right)\right)-I'_{\varepsilon}\left(z_{\mathcal{A}}\right)\right)\left[\frac{\partial w_{\varepsilon}\left(z_{\mathcal{A}}\right)}{\partial \xi_1^{(i)}}\right]ds.
\end{aligned}
\end{equation*}
Thus
\begin{equation*}
\begin{aligned}
\left|\frac{\partial I_{\varepsilon}\left(z_{\mathcal{A}}+w_{\varepsilon}\left(z_{\mathcal{A}}\right)\right)}{\partial \xi_1^{(i)}}-\frac{\partial I_{\varepsilon}\left(z_{\mathcal{A}}\right)}{\partial \xi_1^{(i)}}\right|
&\lesssim \left\| I''_{\varepsilon}\left(z_{\mathcal{A}}\right)\left[\frac{\partial z_{\mathcal{A}}}{\partial \xi_1^{(i)}}\right]\right\|\|w_{\varepsilon}\left(z_{\mathcal{A}}\right)\|
+\left\|I'_{\varepsilon}\left(z_{\mathcal{A}}\right)\right\|\left\|\frac{\partial w_{\varepsilon}\left(z_{\mathcal{A}}\right)}{\partial \xi_1^{(i)}}\right\|
\\&\quad+\|w_{\varepsilon}\left(z_{\mathcal{A}}\right)\| \left\|\frac{\partial w_{\varepsilon}\left(z_{\mathcal{A}}\right)}{\partial \xi_1^{(i)}}\right\|
+\|w_{\varepsilon}\left(z_{\mathcal{A}}\right)\|^2\left(\left\|\frac{\partial z_{\mathcal{A}}}{\partial \xi_1^{(i)}}\right\|+\left\|\frac{\partial w_{\varepsilon}\left(z_{\mathcal{A}}\right)}{\partial \xi_1^{(i)}}\right\|\right).
\end{aligned}
\end{equation*}
From Lemma \ref{lemma-3.4}, (\ref{reduction-22}), Lemma \ref{lemma2-3.4}, Proposition \ref{prop-3.3} and the fact that $\left\|\frac{\partial z_{\mathcal{A}}}{\partial \xi_1^{(i)}}\right\|=O\left(\frac{1}{\mu_i}\right)$, we obtain
$$\left|\frac{\partial I_{\varepsilon}\left(z_{\mathcal{A}}+w_{\varepsilon}\left(z_{\mathcal{A}}\right)\right)}{\partial \xi_1^{(i)}}-\frac{\partial I_{\varepsilon}\left(z_{\mathcal{A}}\right)}{\partial \xi_1^{(i)}}\right|\lesssim \left(\sum_{i=1}^k\mu_i^{3-\frac{\alpha}{2}}d^{-4}\right)e(\varepsilon)+\frac{e^2(\varepsilon)}{\mu_i}.$$
The remaining parts of (\ref{reduction-20}) can be proved in a similar way.
\end{proof}

\section{Energy expansion for one bubble}\label{expansion-for-one-bubble}
In this section, we evaluate the expansion of $I_{\varepsilon}\left(z_{\mathcal{A}}\right)$ for sufficiently small $\varepsilon$ and $z_{\mathcal{A}}\in Z$ in the case where $k=1$. To simplify the notation, we denote $\mu_1,\  \xi^{(1)},\ \mathcal Q_{1},\ a_{1,1},\ a_{1,2},\ p_{1,1},\ p_{1,2}$ as
$ \mu,\  \xi,\  \mathcal Q,\ a_1,\ a_2,\ p_1,\ p_2$.

For a given boundary condition $\tilde g$ and its harmonic extension $g$ on $\mathcal D$, we aim to calculate the expansion of the functional $I_{\varepsilon}\left(P\mathcal{Q}\delta_{\mu,\xi, a, p}\right)$ as a function of $\mathcal{Q}\in SO(3)$, $\mu$, $\xi$, $a_l$, $p_l$, $l=1, 2$ and $\varepsilon$. For a fixed point $\omega\in \mathcal D$ to be chosen, we define the boundary function $g(z,\omega)=\left(g_1(z,\omega),g_2(z,\omega),g_3(z,\omega)\right)$ as follows
$$g_1(z,\omega)=2h_1^{(1)}(z,\omega),\quad g_2(z,\omega)=2h_2^{(1)}(z,\omega),\quad g_3(z,\omega)=-2\varepsilon h_3^{(1)}(z,\omega), $$
where $\left(h_1^{(1)}(z,\omega),h_2^{(1)}(z,\omega),h_3^{(1)}(z,\omega)\right)$ are the functions given in (\ref{pre-4}). The main result of this section is the following proposition.

\begin{prop}\label{prop 4.1} Let $ \mu, \ \xi,\ \mathcal Q,\  a_1,\ a_2,\ p_1,\ p_2$ be such that $P\mathcal{Q}\delta_{\mu,\xi,a,p}\in Z$, and $d=d_{\varepsilon}=1-|\xi|^2$.
Then there holds
\begin{equation*}
I_{\varepsilon}\left(P\mathcal{Q}\delta_{\mu,\xi,a,p}\right)=\frac{8}{3}\pi+F_{\mathcal{D},g}(\mu,\xi,\mathcal{Q},a,p)+\tilde{e}^{(1)}(\varepsilon,\mu,a,p)+\tilde{e}^{(2)}(\varepsilon,\mu,a,p),
\end{equation*}
where
\begin{align*}
&F_{\mathcal{D},g}(\mu,\xi,\mathcal{Q},a,p)
\\&:=8\pi\mu^4\frac{\partial^2 h_2^{(1)}}{\partial x \partial y}(\xi,\xi)
-4\pi\varepsilon \mu^2\left[\frac{\partial^2 \left(\mathcal Q^{-1} g\right)_1}{\partial x^2}(\xi,\omega)+\frac{\partial^2 \left(\mathcal Q^{-1} g\right)_2}{\partial x \partial y}(\xi,\omega)\right]
\\&\quad-16\pi a_1^2\mu^2\frac{\partial h_1^{(-1,1)}}{\partial x}(\xi,\xi)
-16\pi a_2^2\mu^2\frac{\partial h_2^{(-1,2)}}{\partial x}(\xi,\xi)
-\frac{4\pi}{3}p_1^2\mu^8\frac{\partial^4 h^{(2, 1)}_1}{\partial x^4}(\xi, \xi) +\frac{4\pi}{3}p_2^2\mu^8\frac{\partial^4 h^{(2, 2)}_2}{\partial x^4}(\xi, \xi)
\\&\quad-32\pi p_1a_1\mu^5\frac{\partial h_1^{(2,1)}}{\partial x}(\xi,\xi)-32\pi p_1a_2\mu^5\frac{\partial h_2^{(2,1)}}{\partial x}(\xi,\xi)
-32\pi p_2a_1\mu^5\frac{\partial h_1^{(2,2)}}{\partial x}(\xi,\xi)
-32\pi p_2a_2\mu^5\frac{\partial h_2^{(2,2)}}{\partial x}(\xi,\xi),
\end{align*}

\begin{equation*}
\begin{aligned}
\tilde{e}^{(1)}(\mu,a,p):=&~O\left(\mu^8d^{-10}+\varepsilon\mu^4d^{-6}+\varepsilon^3\mu^2|\log \mu|+\varepsilon^2\mu^2+\varepsilon^3\mu^2d^{-2}\right)
\\&+O\left(\left(a_1^2+a_2^2\right)\mu^4d^{-4}\right)
+ O\left((p_1^2+p_2^2)\mu^{10}d^{-10}\right)+O\left((a_1+a_2)(p_1+p_2)\mu^7d^{-7}\right),
\end{aligned}
\end{equation*}
and
\begin{align}
&\tilde{e}^{(2)}(\varepsilon,\mu,a,p)
\\&:\notag =O\left(\left(p_1+p_2\right)(\mu^2-\varepsilon)\left(\mu^4d^{-6}+(\mu^8+\varepsilon\mu^6+\varepsilon^2\mu^4)d^{-12}\right)\right)
\\ \notag&~\quad+O\left((a_1+a_2)(\mu^2-\varepsilon)\left(\mu d^{-3}+\left(\mu^3\left(\mu^2+\varepsilon\right)+\varepsilon^2\mu\right)d^{-9}\right)\right)
\\ \notag&~\quad+O\left(p_1p_2\mu^{10}d^{-10}+\sum_{k,l=1}^2p_kp_l(\mu^2-\varepsilon)\left(\left((\mu^2+\varepsilon)\right)\left(\mu^6d^{-12}+\mu^2d^{-6}\right)+\varepsilon^2\mu^6d^{-14}\right)\right)
\\ \notag&~\quad+O\left(a_1a_2\mu^4d^{-4}+\sum_{k,l=1}^2a_ka_l\left(\mu^2-\varepsilon\right)\left(\mu^2(\mu^2+\varepsilon)d^{-8}\right)\right)
\\ \notag&~\quad+O\left(\sum_{k,l=1}^2a_kp_l(\mu^2-\varepsilon)\left(\mu^3d^{-5}+(\mu^5(\mu^2+\varepsilon)++\varepsilon^2 \mu^3)d^{-11}\right)\right)
\\ \notag&~\quad+O\left(\sum_{k,l,m=1}^2a_ka_la_m\left(\mu^5d^{-7}+\left(\mu^2-\varepsilon\right)\mu(\mu^2+\varepsilon)d^{-7}\right)+a_ka_lp_m\left(\mu^2d^{-2}+\left(\mu^2-\varepsilon\right)\mu(\mu^2+\varepsilon)d^{-7}\right)\right)
\\ \notag&~\quad+O\left(\sum_{k,l,m=1}^2a_kp_lp_m\left(\mu^5d^{-5}+\left(\mu^2-\varepsilon\right)\mu(\mu^2+\varepsilon)d^{-7}\right)+p_kp_lp_m\left(\mu^6d^{-6}+\left(\mu^2-\varepsilon\right)\mu(\mu^2+\varepsilon)d^{-7}\right)\right)
\\ \notag&~\quad+O\left(\left(a_1+a_2+p_1+p_2\right)^4\left(\mu d^{-1}+(\mu^2-\varepsilon)\mu^2d^{-6}\right)\right)
\\ \label{Ju28-error}&~\quad+O\left(\left(a_1+a_2+p_1+p_2\right)^3\left(a_1+a_2\right)\mu^4 d^{-6}\right)
+O\left(\left(a_1+a_2+p_1+p_2\right)^5\right).
\end{align}
Moreover, we have
\begin{align*}
& \frac{\partial I_{\varepsilon}\left(P\mathcal{Q}\delta_{\mu,\xi,a,p}\right)}{\partial \mu}=\frac{\partial F_{\mathcal{D},g}}{\partial \mu}+\frac{1}{\mu}\tilde{e}^{(1)}(\mu,a,p)+\frac{1}{\mu}\tilde{e}^{(2)}(\varepsilon,\mu,a,p),
\\&\frac{\partial I_{\varepsilon}\left(P\mathcal{Q}\delta _{\mu,\xi,a,p}\right)}{\partial \xi_l}=\frac{\partial F_{\mathcal{D},g}}{\partial \xi_l}+\frac{\pp\tilde{e}^{(1)}(\mu,a,p)}{\pp \xi_l}+\frac{\pp\tilde{e}^{(2)}(\varepsilon,\mu,a,p)}{\pp\xi_l},\quad l = 1,2,
\\&\frac{\partial I_{\varepsilon}\left(P\mathcal{Q}\delta _{\mu,\xi,a,p}\right)}{\partial \mathcal{Q}}=\frac{\partial F_{\mathcal{D},g}}{\partial \mathcal{Q}}+\tilde{e}^{(1)}(\mu,a,p)+\tilde{e}^{(2)}(\varepsilon,\mu,a,p),
\\&\frac{\partial I_{\varepsilon}\left(P\mathcal{Q}\delta _{\mu,\xi,a,p}\right)}{\partial a_l}=\frac{\partial F_{\mathcal{D},g}}{\partial a_l}+\frac{\pp\tilde{e}^{(1)}(\mu,a,p)}{\pp a_l}+\frac{\pp\tilde{e}^{(2)}(\varepsilon,\mu,a,p)}{\pp a_l},\quad l = 1,2,
\\&\frac{\partial I_{\varepsilon}\left(P\mathcal{Q}\delta _{\mu,\xi,a,p}\right)}{\partial p_l}=\frac{\partial F_{\mathcal{D},g}}{\partial p_l}+\frac{\pp\tilde{e}^{(1)}(\mu,a,p)}{\pp p_l}+\frac{\pp\tilde{e}^{(2)}(\varepsilon,\mu,a,p)}{\pp p_l},\quad l = 1,2.
\end{align*}
\end{prop}
\begin{proof}
Assume $\mathcal{Q}=Id$, and let $\varphi_{\mu, \xi, a, p}$ be the solution of (\ref{pre-3}). Then $P\delta_{\mu, \xi, a, p}=\delta_{\mu, \xi, a, p}-\varphi_{\mu, \xi, a, p}$ and we have
\begin{equation}\label{mainA24-1}
\begin{aligned}
I_{\varepsilon}\left(P\delta_{\mu, \xi, a, p}\right)
=&\frac{1}{2}\int_{\mathcal{D}}\left|\nabla P\delta_{\mu, \xi, a, p}\right|^2
+\frac{2}{3}\int_{\mathcal{D}}P\delta_{\mu, \xi, a, p} \cdot\left[(P\delta_{\mu, \xi, a, p})_x\wedge (P\delta _{\mu, \xi, a, p})_y\right]
\\&+\varepsilon \int_{\mathcal{D}}\left(P\delta_{\mu, \xi, a, p}\right)
\cdot\left[(P\delta _{\mu, \xi, a, p})_x\wedge g_y+g_x\wedge (P\delta_{\mu, \xi, a, p})_y\right]+2\varepsilon^2\int_{\mathcal{D}}\left(P\delta_{\mu, \xi, a, p}\right)\cdot (g_x\wedge g_y).
\end{aligned}
\end{equation}
Write $\delta_{\mu, \xi, 0, 0}$ as $\delta$, recall that we have the following expansion
$$
\begin{aligned}
\delta_{\mu, \xi, a, p}:=\delta+\mathcal{L}_{\mathcal{A}}+\mathcal{R}_{\mathcal{A}},
\end{aligned}
$$
where $\mathcal{L}_{\mathcal{A}}:=\sum_{l=1}^2\left(a_lZ_{-1,l}+p_l Z_{2,l}\right)$ is the linear part, $\mathcal{R}_{\mathcal{A}}$ denotes the higher order terms.

Now, we highlight some evident relations between the terms in (\ref{mainA24-1}) that arise through integration by parts. Firstly, for the following terms in (\ref{mainA24-1}), there holds
\begin{equation}\label{A30-1}
\begin{aligned}
&\frac{1}{2}\int_{\mathcal{D}}\nabla\left(P\mathcal{L}_{\mathcal{A}}\right)\cdot\nabla\left(P\mathcal{L}_{\mathcal{A}}\right)
+\frac{2}{3}\int_{\mathcal{D}}P\delta\cdot\left[\left(P\mathcal{L}_{\mathcal{A}}\right)_x
 \wedge \left(P\mathcal{L}_{\mathcal{A}}\right)_y\right]
\\&+\frac{2}{3}\int_{\mathcal{D}}\left(P\mathcal{L}_{\mathcal{A}}\right)\cdot\left[\left(P\delta\right)_x
 \wedge \left(P\mathcal{L}_{\mathcal{A}}\right)_y+
\left(P\mathcal{L}_{\mathcal{A}}\right)_x\wedge \left(P\delta\right)_y\right]
\\&=-\frac{1}{2}\int_{\mathcal{D}}\Delta\left(P\mathcal{L}_{\mathcal{A}}\right)\cdot\left(P\mathcal{L}_{\mathcal{A}}\right)+\int_{\mathcal{D}}\left(P\mathcal{L}_{\mathcal{A}}\right)\cdot\left[\delta_x\wedge \left(P\mathcal{L}_{\mathcal{A}}\right)_y+\left(P\mathcal{L}_{\mathcal{A}}\right)_x\wedge\delta_y\right]
\\&\quad-\int_{\mathcal{D}}\left(P\mathcal{L}_{\mathcal{A}}\right)\cdot\left[(\varphi_1)_x\wedge \left(P\mathcal{L}_{\mathcal{A}}\right)_y+\left(P\mathcal{L}_{\mathcal{A}}\right)_x\wedge (\varphi_1)_y\right]
\\&=-\frac{1}{2}\int_{\mathcal{D}}L_\delta\left[P\mathcal{L}_{\mathcal{A}}\right]\cdot\left(P\mathcal{L}_{\mathcal{A}}\right)-2\int_{\mathcal{D}}\varphi_1
\cdot\left[\left(P\mathcal{L}_{\mathcal{A}}\right)_x\wedge\left(P\mathcal{L}_{\mathcal{A}}\right)_y\right],
\end{aligned}
\end{equation}
where $L_\delta[\phi] := \Delta\phi-2\delta_x\wedge\phi_y-2\phi_x\wedge\delta_y$.

\begin{equation*}
\begin{aligned}
&\frac{1}{2}\int_{\mathcal{D}}\nabla \left(P\mathcal{R}_{\mathcal{A}}\right)\cdot\nabla \left(P\mathcal{R}_{\mathcal{A}}\right)+\frac{2}{3}\int_{\mathcal{D}}P\delta\cdot \left[\left(P\mathcal{R}_{\mathcal{A}}\right)_x\wedge\left(P\mathcal{R}_{\mathcal{A}}\right)_y\right]
\\&+\frac{2}{3}\int_{\mathcal{D}}\left(P\mathcal{R}_{\mathcal{A}}\right)\cdot\left[\left(P\delta\right)_x
 \wedge \left(P\mathcal{R}_{\mathcal{A}}\right)_y+
\left(P\mathcal{R}_{\mathcal{A}}\right)_x\wedge\left(P\delta\right)_y\right]
\\&=\frac{1}{2}\int_{\mathcal{D}}\nabla\left(P\mathcal{R}_{\mathcal{A}}\right)\cdot\nabla \left(P\mathcal{R}_{\mathcal{A}}\right)+2\int_{\mathcal{D}}P\delta\cdot \left[\left(P\mathcal{R}_{\mathcal{A}}\right)_x
 \wedge \left(P\mathcal{R}_{\mathcal{A}}\right)_y\right].
\end{aligned}
\end{equation*}
Secondly,
\begin{equation*}
\begin{aligned}
&\int_{\mathcal{D}}\nabla P\delta\cdot\nabla \left(P(\mathcal{L}_{\mathcal{A}}+\mathcal{R}_{\mathcal{A}})\right)
+\frac{2}{3}\int_{\mathcal{D}}\left(P(\mathcal{L}_{\mathcal{A}}+\mathcal{R}_{\mathcal{A}}))\right)\cdot \left[\left(P\delta\right)_x
 \wedge \left(P\delta\right)_y\right]
\\&+\frac{2}{3}\int_{\mathcal{D}}P\delta\cdot \left[\left(P\delta\right)_x
 \wedge \left(P(\mathcal{L}_{\mathcal{A}}+\mathcal{R}_{\mathcal{A}}))\right)_y+
\left(P(\mathcal{L}_{\mathcal{A}}+\mathcal{R}_{\mathcal{A}}))\right)_x\wedge \left(P\delta\right)_y\right]
\\&=-2\int_{\mathcal{D}}\left(P(\mathcal{L}_{\mathcal{A}}+\mathcal{R}_{\mathcal{A}})\right)
\cdot (\delta_x\wedge \delta_y)+2\int_{\mathcal{D}}\left(P(\mathcal{L}_{\mathcal{A}}+\mathcal{R}_{\mathcal{A}})\right)
\cdot \left[(P\delta)_x\wedge (P\delta)_y)\right]
\\&=-2 \int_{\mathcal{D}}\left(P(\mathcal{L}_{\mathcal{A}}+\mathcal{R}_{\mathcal{A}}))\right)
\cdot \left[\delta_x\wedge (\varphi_1)_y+(\varphi_1)_x\wedge \delta_y-(\varphi_1)_x\wedge (\varphi_1)_y\right]
\\&=-2\int_{\mathcal{D}}\varphi_1
\cdot \left[\delta_x\wedge \left(P(\mathcal{L}_{\mathcal{A}}+\mathcal{R}_{\mathcal{A}})\right)_y+\left(P(\mathcal{L}_{\mathcal{A}}+\mathcal{R}_{\mathcal{A}})\right)_x\wedge \delta_y\right]+2\int_{\mathcal{D}} \left(P(\mathcal{L}_{\mathcal{A}}+\mathcal{R}_{\mathcal{A}})\right)\cdot\left[(\varphi_1)_x\wedge (\varphi_1)_y\right].
\end{aligned}
\end{equation*}
Moreover,
\begin{align*}
&\int_{\mathcal{D}}\nabla \left(P\mathcal{L}_{\mathcal{A}}\right)\cdot\nabla \left(P\mathcal{R}_{\mathcal{A}}\right)
+\frac{2}{3}\int_{\mathcal{D}}P\delta\cdot \left[\left(P\mathcal{L}_{\mathcal{A}}\right)_x
 \wedge \left(P\mathcal{R}_{\mathcal{A}}\right)_y+
\left(P\mathcal{R}_{\mathcal{A}}\right)_x\wedge \left(P\mathcal{L}_{\mathcal{A}}\right)_y\right]
\\&+\frac{2}{3}\int_{\mathcal{D}}\left(P\mathcal{L}_{\mathcal{A}}\right)\cdot \left[\left(P\delta\right)_x
 \wedge \left(P\mathcal{R}_{\mathcal{A}}\right)_y+
\left(P\mathcal{R}_{\mathcal{A}}\right)_x\wedge \left(P\delta\right)_y\right]
+\frac{2}{3}\int_{\mathcal{D}}\left(P\mathcal{R}_{\mathcal{A}}\right)
\cdot \left[\left(P\delta\right)_x\wedge \left(P\mathcal{L}_{\mathcal{A}}\right)_y+\left(P\mathcal{L}_{\mathcal{A}}\right)_x\wedge \left(P\delta\right)_y \right]
\\&=-2\int_{\mathcal{D}}\left(P\mathcal{R}_{\mathcal{A}}\right)
\cdot \left[\delta_x\wedge \left(\mathcal{L}_{\mathcal{A}}\right)_y +\left(\mathcal{L}_{\mathcal{A}}\right)_x\wedge\delta_y\right]+2\int_{\mathcal{D}}\left(P\mathcal{R}_{\mathcal{A}}\right)
\cdot \left[\left(P\delta\right)_x\wedge \left(P\mathcal{L}_{\mathcal{A}}\right)_y+\left(P\mathcal{L}_{\mathcal{A}}\right)_x\wedge \left(P\delta\right)_y \right].
\end{align*}
Integrating by parts again, we derive
\begin{equation}\label{A30-5-2}
\begin{aligned}
\varepsilon \int_{\mathcal{D}}\left(P\delta _{\mu, \xi, a, p}\right)
\cdot \left[(P\delta _{\mu, \xi, a, p})_x\wedge g_y
+g_x\wedge (P\delta _{\mu, \xi, a, p})_y\right]
=2\varepsilon\int_{\mathcal{D}}g\cdot  \left[\left(P\delta _{\mu, \xi, a, p}\right)_x\wedge \left(P\delta _{\mu, \xi, a, p}\right)_y\right].
\end{aligned}
\end{equation}
For the sake of clarity and brevity, we expand the terms in (\ref{mainA24-1}), but only list some typical terms here, and the remaining terms which are of higher order will be estimated in Appendix \ref{Computations of mixed terms}. Using the relations derived in (\ref{A30-1})-(\ref{A30-5-2}), we have
\begin{equation}\label{Mar22-2}
\begin{aligned}
 I_{\varepsilon}\left(P\delta_{\mu, \xi, a, p}\right)
&=\frac{1}{2}\int_{\mathcal{D}}|\nabla P\delta|^2 +\frac{2}{3}\int_{\mathcal{D}}P\delta
\cdot \left[(P\delta)_x\wedge (P\delta)_y\right]
\\&\quad-\frac{1}{2}\int_{\mathcal{D}}L_\delta\left[P\mathcal{L}_{\mathcal{A}}\right]\cdot\left(P\mathcal{L}_{\mathcal{A}}\right)-2\int_{\mathcal{D}}\varphi_1
\cdot \left[\left(P\mathcal{L}_{\mathcal{A}}\right)_x\wedge \left(P\mathcal{L}_{\mathcal{A}}\right)_y\right]
\\&\quad+\frac{1}{2}\int_{\mathcal{D}}\nabla \left(P\mathcal{R}_{\mathcal{A}}\right)\cdot\nabla \left(P\mathcal{R}_{\mathcal{A}}\right)+2\int_{\mathcal{D}}P\delta\cdot \left[\left(P\mathcal{R}_{\mathcal{A}}\right)_x
 \wedge \left(P\mathcal{R}_{\mathcal{A}}\right)_y\right]
\\&\quad-2\int_{\mathcal{D}}\varphi_1
\cdot \left[\delta_x\wedge \left(P(\mathcal{L}_{\mathcal{A}}+\mathcal{R}_{\mathcal{A}})\right)_y+\left(P(\mathcal{L}_{\mathcal{A}}+\mathcal{R}_{\mathcal{A}})\right)_x\wedge \delta_y\right]
+2\int_{\mathcal{D}} \left(P(\mathcal{L}_{\mathcal{A}}+\mathcal{R}_{\mathcal{A}})\right)\cdot\left[(\varphi_1)_x\wedge (\varphi_1)_y\right]
\\&\quad
-2\int_{\mathcal{D}}\left(P\mathcal{R}_{\mathcal{A}}\right)
\cdot \left[\delta_x\wedge \left(\mathcal{L}_{\mathcal{A}}\right)_y +\left(\mathcal{L}_{\mathcal{A}}\right)_x\wedge\delta_y\right]
+2\int_{\mathcal{D}}\left(P\mathcal{R}_{\mathcal{A}}\right)
\cdot \left[\left(P\delta\right)_x\wedge \left(P\mathcal{L}_{\mathcal{A}}\right)_y+\left(P\mathcal{L}_{\mathcal{A}}\right)_x\wedge \left(P\delta\right)_y \right]
\\&\quad+\frac{2}{3}\int_{\mathcal{D}}\left(P\left(\mathcal{L}_{\mathcal{A}}+\mathcal{R}_{\mathcal{A}}\right)\right)\cdot \left[\left(P\left(\mathcal{L}_{\mathcal{A}}+\mathcal{R}_{\mathcal{A}}\right)\right)_x
 \wedge \left(P\left(\mathcal{L}_{\mathcal{A}}+\mathcal{R}_{\mathcal{A}}\right)\right)_y\right]
\\&\quad+2\varepsilon \int_{\mathcal{D}}g
\cdot \left[\left(P\delta _{\mu, \xi, a, p}\right)_x\wedge \left(P\delta _{\mu, \xi, a, p}\right)_y\right]+2\varepsilon^2\int_{\mathcal{D}}\left(P\delta _{\mu, \xi, a, p}\right)\cdot (g_x\wedge g_y)
\\&:=I^{(1)}_{\varepsilon}\left(P\delta_{\mu, \xi, a, p}\right)+\mathcal{R}_{\mu,\xi,a,p,\varepsilon},
\end{aligned}
\end{equation}
where
\begin{equation}\label{Au10-I1}
\begin{aligned}
I^{(1)}_{\varepsilon}\left(P\delta _{\mu, \xi, a, p}\right)
&=\frac{1}{2}\int_{\mathcal{D}}|\nabla P\delta|^2+\frac{2}{3}\int_{\mathcal{D}}P\delta
\cdot \left[(P\delta)_x\wedge (P\delta)_y\right]+2\varepsilon \int_{\mathcal{D}}g
\cdot \left[(P\delta)_x\wedge (P\delta)_y\right]
\\&\quad-\sum_{l=1}^2\frac{a_l^2}{2}\int_{\mathcal{D}}L_\delta[P Z_{-1,l}]\cdot(P Z_{-1,l})
-\sum_{l=1}^2\frac{p_l^2}{2}\int_{\mathcal{D}}L_\delta[P Z_{2,l}]\cdot(P Z_{2,l})
\\&\quad-\int_{\mathcal{D}}L_\delta\left[\sum_{l=1}^2p_l PZ_{2,l}\right]\cdot \left(\sum_{l=1}^2a_lP Z_{-1,l}\right).
\end{aligned}
\end{equation}
And $\mathcal{R}_{\mu,\xi,a,p,\varepsilon}$ is given in Appendix \ref{Computations of mixed terms}, see (\ref{Au12-R1}).

We claim that when the boundary function $g(z,\omega)=\left(2h_1^{(1)}(z,\omega),2h_2^{(1)}(z,\omega),-2\varepsilon h_3^{(1)}(z,\omega)\right)$, the terms in
$\mathcal{R}_{\mu,\xi,a,p,\varepsilon}$ become higher order after certain cancellations, which will be discussed in Appendix \ref{Computations of mixed terms}.

\subsection{Estimates of the terms in (\ref{Au10-I1}).}\label{subsection4.1}
In this subsection, we estimate the terms in (\ref{Au10-I1}).

{\bf $\bullet$ Estimates of the terms containing only $\mu$ in (\ref{Au10-I1}).}
Integrating by parts, we can write
\begin{equation*}
\begin{aligned}
\frac{1}{2}\int_{\mathcal{D}}|\nabla P\delta|^2 = -\int_{\mathcal{D}}P\delta \cdot (\delta_x\wedge \delta_y) = \int_{\mathcal{D}}(\varphi_1-\delta)\cdot (\delta_x\wedge \delta_y).
\end{aligned}
\end{equation*}
First, we calculate $\int_{\mathcal{D}}(\varphi_1-\delta)_1\cdot (\delta_x\wedge \delta_y)_1$. By (\ref{cross}) and (\ref{A15-expansion}), we deduce
\begin{equation}\label{expansion-7}
\begin{aligned}
\int_{\mathcal{D}}(\varphi_1-\delta)_1\cdot (\delta_x\wedge \delta_y)_1
=\int_{\mathcal{D}}\left[\frac{2\left[\left(\frac{x-\xi_1}{\mu}\right)^2-\left(\frac{y-\xi_2}{\mu}\right)^2\right]}{1+\left(\left|\frac{z-\xi}{\mu}\right|^2\right)^2}-2\mu^2h_1^{(1)}(z,\xi)+O(\mu^5)\right]\frac{32\left[\left(\frac{x-\xi_1}{\mu}\right)^4-\left(\frac{y-\xi_2}{\mu}\right)^4\right]}{\mu^2\left[1+\left(\left|\frac{z-\xi}{\mu}\right|^2\right)^2\right]^3}.
\end{aligned}
\end{equation}
Using a change of variable, we obtain
\begin{equation}\label{expansion-8}
\begin{aligned}
\frac{1}{\mu^2}\int_{\mathcal{D}}\frac{\left[\left(\frac{x-\xi_1}{\mu}\right)^2-\left(\frac{y-\xi_2}{\mu}\right)^2\right]\left[\left(\frac{x-\xi_1}{\mu}\right)^4-\left(\frac{y-\xi_2}{\mu}\right)^4\right]}{\left[1+\left(\left|\frac{z-\xi}{\mu}\right|^2\right)^2\right]^4}
&=\int_{\mathbb{R}^2}\frac{(x^2-y^2)(x^4-y^4)}{\left(1+|z|^4\right)^4}+O(\mu^8)=\frac{\pi}{24}+O(\mu^8).
\end{aligned}
\end{equation}
From the smoothness of $h_1^{(1)}$ and $\frac{\pp^3h_1^{(1)}}{\pp x^l\pp y^{3-l}}(\xi,\xi)=O\left(d^{-5}\right)$ for $l=0,1,2,3$, where $d=d_{\varepsilon}=1-|\xi|^2$, we obatin
\begin{equation*}
\begin{aligned}
&\left|\int_{\mathcal{D}}\left[h_1^{(1)}(z,\xi)-h_1^{(1)}(\xi,\xi)-(z-\xi) \cdot \nabla h_1^{(1)}(\xi,\xi)-\frac{1}{2}(z-\xi)\cdot \nabla^2h_1^{(1)}(\xi,\xi)\cdot(z-\xi)^{\mathrm{T}} \right] \frac{\left[\left(\frac{x-\xi_1}{\mu}\right)^4-\left(\frac{y-\xi_2}{\mu}\right)^4\right]}{\mu^2\left[1+\left(\left|\frac{z-\xi}{\mu}\right|^2\right)^2\right]^3}\right|
\\&= O\left(\int_{\mathcal{D}}\frac{\left|\left(\frac{x-\xi_1}{\mu}\right)^4-\left(\frac{y-\xi_2}{\mu}\right)^4\right|d^{-5}\left|z-\xi\right|^3}{\left[1+\left(\left|\frac{z-\xi}{\mu}\right|^2\right)^2\right]^3}\right),
\end{aligned}
\end{equation*}
where
$$(z-\xi)\cdot \nabla^2h_1^{(1)}(\xi,\xi)\cdot(z-\xi)^{\mathrm{T}}=(x-\xi_1)^2\frac{\partial^2 h_1^{(1)}}{\partial x^2}(\xi,\xi)+2(x-\xi_1)(y-\xi_2)\frac{\partial^2 h_1^{(1)}}{\partial x \partial y}(\xi,\xi)+(y-\xi_2)^2\frac{\partial^2 h_1^{(1)}}{\partial y^2}(\xi,\xi).
$$
As a consequence, we deduce
\begin{align}
&\notag\int_{\mathcal{D}}h_1^{(1)}(z,\xi) \frac{\left(\frac{x-\xi_1}{\mu}\right)^4-\left(\frac{y-\xi_2}{\mu}\right)^4}{\left[1+\left(\left|\frac{z-\xi}{\mu}\right|^2\right)^2\right]^3}
\\ \notag&= \int_{\mathcal{D}}h_1^{(1)}(\xi,\xi)\frac{\left(\frac{x-\xi_1}{\mu}\right)^4-\left(\frac{y-\xi_2}{\mu}\right)^4}{\left[1+\left(\left|\frac{z-\xi}{\mu}\right|^2\right)^2\right]^3}
+ \int_{\mathcal{D}}\left((z-\xi) \cdot \nabla h_1^{(1)}(\xi,\xi)\right)\frac{\left(\frac{x-\xi_1}{\mu}\right)^4-\left(\frac{y-\xi_2}{\mu}\right)^4}{\left[1+\left(\left|\frac{z-\xi}{\mu}\right|^2\right)^2\right]^3}
\\ \notag&\quad + \frac{1}{2}\int_{\mathcal{D}}\left[(z-\xi)\cdot \nabla^2h_1^{(1)}(\xi,\xi)\cdot(z-\xi)^{\mathrm{T}} \right]\frac{\left(\frac{x-\xi_1}{\mu}\right)^4-\left(\frac{y-\xi_2}{\mu}\right)^4}{\left[1+\left(\left|\frac{z-\xi}{\mu}\right|^2\right)^2\right]^3}
+ O\left(\int_{\mathcal{D}}\frac{\left[\left(\frac{x-\xi_1}{\mu}\right)^4-\left(\frac{y-\xi_2}{\mu}\right)^4\right]\frac{\left|z-\xi\right|^3}{\left(1-|\xi|^2\right)^5}}{\left[1+\left(\left|\frac{z-\xi}{\mu}\right|^2\right)^2\right]^3}\right)
\\ \notag&=\mu^2h_1^{(1)}(\xi,\xi)\int_{\mathbb{R}^2}\frac{x^4-y^4}{\left(1+|z|^4\right)^3}
+\frac{1}{2}\mu^4\int_{\mathbb{R}^2}\frac{x^4-y^4}{\left(1+|z|^4\right)^3}\left[x^2\frac{\partial^2 h_1^{(1)}}{\partial x^2}(\xi,\xi)+y^2\frac{\partial^2 h_1^{(1)}}{\partial y^2}(\xi,\xi)\right]
\\ \notag&\quad+O\left(\mu^5d^{-5}\int_{\mathbb{R}^2}\frac{(x^4-y^4)|z|^3}{\left(1+|z|^4\right)^3}\right)
+O\left(\mu^8d^{-5}\right)
\\ \label{expansion-9}&=\frac{\pi}{32}\mu^4\left[\frac{\partial^2 h_1^{(1)}}{\partial x^2}(\xi,\xi)-\frac{\partial^2 h_1^{(1)}}{\partial y^2}(\xi,\xi)\right]+O\left(\mu^8d^{-5}\right).
\end{align}
In (\ref{expansion-8}) and (\ref{expansion-9}), we use the fact that
\begin{align}
&\int_{\mathbb{R}^2}\frac{(x^2-y^2)(x^4-y^4)}{\left(1+|z|^4\right)^4}=\frac{\pi}{24},\quad \int_{\mathbb{R}^2}\frac{x^4-y^4}{\left(1+|z|^4\right)^3}=0,\quad
\int_{\mathbb{R}^2}\frac{(x^4-y^4)|z|^3}{\left(1+|z|^4\right)^3}=0,
\\&\int_{\mathbb{R}^2}\frac{x^2(x^4-y^4)}{\left(1+|z|^4\right)^3}=-\int_{\mathbb{R}^2}\frac{y^2(x^4-y^4)}{\left(1+|z|^4\right)^3}=\frac{\pi}{16}.
\end{align}
Combining (\ref{expansion-7}), (\ref{expansion-8}) and (\ref{expansion-9})
we obtain
\begin{equation}\label{expansion-10}
\begin{aligned}
\int_{\mathcal{D}}(\varphi_1-\delta)_1\cdot (\delta_x\wedge \delta_y)_1
=\frac{8\pi}{3}-2\pi\mu^4\left[\frac{\partial^2 h_1^{(1)}}{\partial x^2}(\xi,\xi)-\frac{\partial^2 h_1^{(1)}}{\partial y^2}(\xi,\xi)\right]+O\left(\mu^8d^{-5}\right).
\end{aligned}
\end{equation}
Now we consider $\int_{\mathcal{D}}(\varphi_1-\delta)_2\cdot (\delta_x\wedge \delta_y)_2$. From (\ref{cross}) and (\ref{A15-expansion}), we have
\begin{equation}\label{expansion-11}
\begin{aligned}
\int_{\mathcal{D}}(\varphi_1-\delta)_2\cdot (\delta_x\wedge \delta_y)_2
=\int_{\mathcal{D}}\left[\frac{4\frac{(x-\xi_1)(y-\xi_2)}{\mu^2}}{1+\left(\left|\frac{z-\xi}{\mu}\right|^2\right)^2}-2\mu^2h_2^{(1)}(z,\xi)+O(\mu^5)\right]\frac{64\left(\frac{x-\xi_1}{\mu}\right)\left(\frac{y-\xi_2}{\mu}\right)\left|\frac{z-\xi}{\mu}\right|^2}{\mu^2\left[1+\left(\left|\frac{z-\xi}{\mu}\right|^2\right)^2\right]^3}.
\end{aligned}
\end{equation}
And
\begin{equation}\label{expansion-12}
\begin{aligned}
\frac{1}{\mu^2}\int_{\mathcal{D}}\frac{\frac{(x-\xi_1)^2(y-\xi_2)^2}{\mu^4}\left|\frac{z-\xi}{\mu}\right|^2}{\left[1+\left(\left|\frac{z-\xi}{\mu}\right|^2\right)^2\right]^4}=\int_{\mathbb{R}^2}\frac{x^2y^2|z|^2}{\left(1+|z|^4\right)^4}+O(\mu^8)=\frac{\pi}{96}+O\left(\mu^8\right).
\end{aligned}
\end{equation}
And similar as (\ref{expansion-9}), we have
\begin{align}
&\notag\int_{\mathcal{D}}h_2^{(1)}(z,\xi) \frac{\left(\frac{x-\xi_1}{\mu}\right)\left(\frac{y-\xi_2}{\mu}\right)\left|\frac{z-\xi}{\mu}\right|^2}{\left[1+\left(\left|\frac{z-\xi}{\mu}\right|^2\right)^2\right]^3}
\\ \notag&= \int_{\mathcal{D}}h_2^{(1)}(\xi,\xi)\frac{\left(\frac{x-\xi_1}{\mu}\right)\left(\frac{y-\xi_2}{\mu}\right)\left|\frac{z-\xi}{\mu}\right|^2}{\left[1+\left(\left|\frac{z-\xi}{\mu}\right|^2\right)^2\right]^3}
+ \int_{\mathcal{D}}\left[(z-\xi) \cdot \nabla h_2^{(1)}(\xi,\xi)\right]\frac{\left(\frac{x-\xi_1}{\mu}\right)\left(\frac{y-\xi_2}{\mu}\right)\left|\frac{z-\xi}{\mu}\right|^2}{\left[1+\left(\left|\frac{z-\xi}{\mu}\right|^2\right)^2\right]^3}
\\ \notag&\quad + \frac{1}{2}\int_{\mathcal{D}}\left[(z-\xi)\cdot \nabla^2h_2^{(1)}(\xi,\xi)\cdot(z-\xi)^{\mathrm{T}}\right] \frac{\left(\frac{x-\xi_1}{\mu}\right)\left(\frac{y-\xi_2}{\mu}\right)\left|\frac{z-\xi}{\mu}\right|^2}{\left[1+\left(\left|\frac{z-\xi}{\mu}\right|^2\right)^2\right]^3}
+ O\left(\int_{\mathcal{D}}\frac{\left(\frac{x-\xi_1}{\mu}\right)\left(\frac{y-\xi_2}{\mu}\right)\left|\frac{z-\xi}{\mu}\right|^2d^{-5}|z-\xi|^3}{\left[1+\left(\left|\frac{z-\xi}{\mu}\right|^2\right)^2\right]^3}\right)
\\ \notag&=\mu^4\int_{\mathbb{R}^2}\frac{x^2y^2|z|^2}{\left(1+|z|^4\right)^3}\frac{\partial^2 h_2^{(1)}}{\partial x \partial y}(\xi,\xi)+O\left(\frac{\mu^5}{\left(1-|\xi|^2\right)^5}\int_{\mathbb{R}^2}\frac{xy|z|^5}{\left(1+|z|^4\right)^3}\right)+O\left(\mu^8d^{-5}\right)
\\ \label{expansion-13}&=\frac{\pi}{32}\mu^4\frac{\partial^2 h_2^{(1)}}{\partial x \partial y}(\xi,\xi)+O\left(\mu^8d^{-5}\right).
\end{align}
In (\ref{expansion-12}) and (\ref{expansion-13}), we use the fact that $$\int_{\mathbb{R}^2}\frac{x^2y^2|z|^2}{\left(1+|z|^4\right)^4}=\frac{\pi}{96},\quad \int_{\mathbb{R}^2}\frac{x^2y^2|z|^2}{\left(1+|z|^4\right)^3}=\frac{\pi}{32} ,\quad \int_{\mathbb{R}^2}\frac{xy|z|^5}{\left(1+|z|^4\right)^3}=0.$$
Combining (\ref{expansion-11}), (\ref{expansion-12}) and (\ref{expansion-13}), we obtain
\begin{equation}\label{expansion-14}
\begin{aligned}
\int_{\mathcal{D}}(\varphi-\delta)_2\cdot (\delta_x\wedge \delta_y)_2
=\frac{8\pi}{3}-4\pi\mu^4\frac{\partial^2 h_2^{(1)}}{\partial x \partial y}(\xi,\xi)+O\left(\mu^8d^{-5}\right).
\end{aligned}
\end{equation}
Similarly, for $\int_{\mathcal{D}}(\varphi_1-\delta)_3\cdot (\delta_x\wedge \delta_y)_3$, we have
\begin{equation}\label{expansion-3}
\begin{aligned}
\int_{\mathcal{D}}(\varphi_1-\delta)_3\cdot (\delta_x\wedge \delta_y)_3
=\int_{\mathcal{D}}\left(\frac{2}{1+\left(\left|\frac{z-\xi}{\mu}\right|^2\right)^2}-2\mu^4h_3^{(1)}(z,\xi)+O(\mu^5)\right)\frac{16\left[1-\left(\left|\frac{z-\xi}{\mu}\right|^2\right)^2\right]\left|\frac{z-\xi}{\mu}\right|^2}{\mu^2\left[1+\left(\left|\frac{z-\xi}{\mu}\right|^2\right)^2\right]^3}.
\end{aligned}
\end{equation}
Reasoning as (\ref{expansion-8}) and (\ref{expansion-9}), we have
\begin{equation}\label{expansion-4}
\begin{aligned}
\frac{1}{\mu^2}\int_{\mathcal{D}}\frac{\left[1-\left(\left|\frac{z-\xi}{\mu}\right|^2\right)^2\right]\left|\frac{z-\xi}{\mu}\right|^2}{\left[1+\left(\left|\frac{z-\xi}{\mu}\right|^2\right)^2\right]^4}
=\int_{\mathbb{R}^2}\frac{\left(1-|z|^4\right)\left|z\right|^2}{\left(1+|z|^4\right)^4}+O\left(\mu^8\right)
=\frac{\pi}{12}+O(\mu^8),
\end{aligned}
\end{equation}
and
\begin{align}
&\notag\int_{\mathcal{D}}h_3^{(1)}(z,\xi) \frac{\mu^2\left[1-\left(\left|\frac{z-\xi}{\mu}\right|^2\right)^2\right]\left|\frac{z-\xi}{\mu}\right|^2}{\left[1+\left(\left|\frac{z-\xi}{\mu}\right|^2\right)^2\right]^3}
\\ \notag&= \int_{\mathcal{D}}h_3^{(1)}(\xi,\xi)\frac{\mu^2\left[1-\left(\left|\frac{z-\xi}{\mu}\right|^2\right)^2\right]\left|\frac{z-\xi}{\mu}\right|^2}{\left[1+\left(\left|\frac{z-\xi}{\mu}\right|^2\right)^2\right]^3}
+ \int_{\mathcal{D}}\left[(z-\xi) \cdot \nabla h_3^{(1)}(\xi,\xi)\right]\frac{\mu^2\left[1-\left(\left|\frac{z-\xi}{\mu}\right|^2\right)^2\right]\left|\frac{z-\xi}{\mu}\right|^2}{\left[1+\left(\left|\frac{z-\xi}{\mu}\right|^2\right)^2\right]^3}
\\ \notag&\quad + \frac{1}{2}\int_{\mathcal{D}}\left[(z-\xi)\cdot \nabla^2h_3^{(1)}(\xi,\xi)\cdot(z-\xi)^{\mathrm{T}} \right]\frac{\mu^2\left[1-\left(\left|\frac{z-\xi}{\mu}\right|^2\right)^2\right]\left|\frac{z-\xi}{\mu}\right|^2}{\left[1+\left(\left|\frac{z-\xi}{\mu}\right|^2\right)^2\right]^3}
\\ \notag&\quad+ O\left(\int_{\mathcal{D}}\frac{\mu^2\left[1-\left(\left|\frac{z-\xi}{\mu}\right|^2\right)^2\right]d^{-7}\left|z-\xi\right|^3}{\left[1+\left(\left|\frac{z-\xi}{\mu}\right|^2\right)^2\right]^3}\right)
\\ \notag&=\mu^4h_3^{(1)}(\xi,\xi)\int_{\mathbb{R}^2}\frac{\left(1-|z|^4\right)\left|z\right|^2}{\left(1+|z|^4\right)^3}
+\frac{1}{2}\mu^6\int_{\mathbb{R}^2}\frac{\left(1-|z|^4\right)\left|z\right|^2}{\left(1+|z|^4\right)^3}\left[x^2\frac{\partial^2 h_3^{(1)}}{\partial x^2}(\xi,\xi)+y^2\frac{\partial^2 h_3^{(1)}}{\partial y^2}(\xi,\xi)\right]
\\ \notag&\quad+O\left(\mu^7d^{-7}\int_{\mathbb{R}^2}\frac{\left(1-|z|^4\right)\left|z\right|^5}{\left(1+|z|^4\right)^3}\right)+O\left(\mu^8d^{-7}\right)
\\ \label{expansion-5}&=-\frac{\pi^2}{32}\mu^6\left[\frac{\partial^2 h_3^{(1)}}{\partial x^2}(\xi,\xi)+\frac{\partial^2 h_3^{(1)}}{\partial y^2}(\xi,\xi)\right]+O\left(\mu^7d^{-7}\right).
\end{align}
In (\ref{expansion-4}) and (\ref{expansion-5}), we use the fact that \begin{equation*}
\begin{aligned}
&\int_{\mathbb{R}^2}\frac{\left(1-|z|^4\right)\left|z\right|^2}{\left(1+|z|^4\right)^4}=\frac{\pi}{12}, \quad\int_{\mathbb{R}^2}\frac{\left(1-|z|^4\right)\left|z\right|^2}{\left(1+|z|^4\right)^3}=0,\quad \int_{\mathbb{R}^2}\frac{\left(1-|z|^4\right)\left|z\right|^5}{\left(1+|z|^4\right)^3}=-\frac{9\pi^2}{16\sqrt{2}},
\\& \int_{\mathbb{R}^2}\frac{x^2\left(1-|z|^4\right)\left|z\right|^2}{\left(1+|z|^4\right)^3}=\int_{\mathbb{R}^2}\frac{y^2\left(1-|z|^4\right)\left|z\right|^2}{\left(1+|z|^4\right)^3}=-\frac{\pi^2}{16}.
\end{aligned}
\end{equation*}
From (\ref{expansion-3}), (\ref{expansion-4}) and (\ref{expansion-5}), we obtain
\begin{equation}\label{expansion-6}
\begin{aligned}
\int_{\mathcal{D}}(\varphi_1-\delta)_3\cdot (\delta_x\wedge \delta_y)_3&=\frac{8\pi}{3}-\pi^2\mu^6\left[\frac{\partial^2 h_3^{(1)}}{\partial x^2}(\xi,\xi)+\frac{\partial^2 h_3^{(1)}}{\partial y^2}(\xi,\xi)\right]+O(\mu^7d^{-7})=\frac{8\pi}{3}+O\left(\mu^6d^{-6}\right).
\end{aligned}
\end{equation}
From (\ref{expansion-10}), (\ref{expansion-14}) and (\ref{expansion-6}), it follows that
\begin{equation}\label{expansion-15}
\begin{aligned}
\frac{1}{2}\int_{\mathcal{D}}|\nabla P\delta|^2
&=\int_{\mathcal{D}}(\varphi_1-\delta)\cdot (\delta_x\wedge \delta_y)
\\&=8\pi-2\pi\mu^4\left[\frac{\partial^2 h_1^{(1)}}{\partial x^2}(\xi,\xi)+2\frac{\partial^2 h_2^{(1)}}{\partial x \partial y}(\xi,\xi)-\frac{\partial^2 h_1^{(1)}}{\partial y^2}(\xi,\xi)\right]
+O\left(\mu^6d^{-6}\right)
\\&=8\pi-8\pi\mu^4\frac{\partial^2 h_2^{(1)}}{\partial x \partial y}(\xi,\xi)
+O\left(\mu^6d^{-6}\right),
\end{aligned}
\end{equation}
where we use the fact that $\frac{\partial^2 h_1^{(1)}}{\partial x^2}(\xi,\xi)-\frac{\partial^2 h_1^{(1)}}{\partial y^2}(\xi,\xi)=2\frac{\partial^2 h_2^{(1)}}{\partial x \partial y}(\xi,\xi)$ for $\xi\in \mathcal{D}$.

In fact, from the above computation, we can see that
\begin{equation}\label{A30-delta1}
\begin{aligned}
\int_{\mathcal{D}}\delta\cdot (\delta_x\wedge \delta_y)
=-8\pi+O\left(\mu^8\right),
\end{aligned}
\end{equation}
and
\begin{equation}\label{A30-delta2}
\begin{aligned}
\int_{\mathcal{D}}\varphi_1\cdot (\delta_x\wedge \delta_y)
=-8\pi\mu^4\frac{\partial^2 h_2^{(1)}}{\partial x \partial y}(\xi,\xi)
+O\left(\mu^6d^{-6}\right).
\end{aligned}
\end{equation}

We estimate the second integral in (\ref{Au10-I1}). It holds that
\begin{equation}\label{Jan5-1}
\begin{aligned}
&\frac{2}{3}\int_{\mathcal{D}}P\delta
\cdot \left[(P\delta)_x\wedge (P\delta)_y\right]
\\&= \frac{2}{3}\int_{\mathcal{D}}(\delta-\varphi_1)\cdot (\delta_x\wedge \delta_y)
-\frac{2}{3}\int_{\mathcal{D}}(\delta-\varphi_1)\cdot [\delta_x\wedge (\varphi_1)_y+(\varphi_1)_x \wedge \delta_y]+\frac{2}{3} \int_{\mathcal{D}}(\delta-\varphi_1)\cdot [(\varphi_1)_x\wedge (\varphi_1)_y]
\\&=\frac{2}{3}\int_{\mathcal{D}}(\delta-\varphi_1)\cdot (\delta_x\wedge \delta_y)
-\frac{4}{3}\int_{\mathcal{D}}\varphi_1 \cdot (\delta_x\wedge \delta_y)+\frac{2}{3}\int_{\mathcal{D}}\varphi_1 \cdot [\delta_x\wedge (\varphi_1)_y+(\varphi_1)_x \wedge \delta_y]
\\&\quad+\frac{2}{3} \int_{\mathcal{D}}(\delta-\varphi_1)\cdot [(\varphi_1)_x\wedge (\varphi_1)_y].
\end{aligned}
\end{equation}
Since
$$|(\varphi_1)_x\wedge (\varphi_1)_y|_1=O\left(\mu^6d^{-8}\right), ~|(\varphi_1)_x\wedge (\varphi_1)_y|_2=O\left(\mu^6d^{-8}\right),~|(\varphi_1)_x\wedge (\varphi_1)_y|_3=O\left(\mu^4d^{-6}\right),$$
we obtain
\begin{equation}\label{Jan5-3}
\begin{aligned}
\frac{2}{3}\int_{\mathcal{D}}(\delta-\varphi_1)\cdot [(\varphi_1)_x\wedge (\varphi_1)_y]=O\left(\mu^8|\log\mu|d^{-8}+\mu^6d^{-6}+\mu^8d^{-10}\right).
\end{aligned}
\end{equation}
Using (\ref{derivative}), we deduce
\begin{align*}
&\frac{2}{3}\int_{\mathcal{D}}\left(\varphi_1\right)_1[\delta_x\wedge (\varphi_1)_y+(\varphi_1)_x \wedge \delta_y]_1
\\&=-\frac{8}{3}\mu^7\int_{\mathbb{R}^2}\left[h_1^{(1)}(\xi,\xi)+\mu y \frac{\pp h_1^{(1)}}{\pp y}(\xi,\xi)\right]\frac{4 y \left(-3 x^4-2 x^2 y^2+y^4+1\right)}{\left(\left(x^2+y^2\right)^2+1\right)^2}
\left[\frac{\pp h_3^{(1)}}{\pp y}(\xi,\xi)+\mu y \frac{\pp^2 h_3^{(1)}}{\pp y^2}(\xi,\xi)\right]
\\&\quad-\frac{8}{3}\mu^5\int_{\mathbb{R}^2}\left[h_1^{(1)}(\xi,\xi)+\mu x \frac{\pp h_1^{(1)}}{\pp x}(\xi,\xi)\right]\frac{8 x
   \left(x^2+y^2\right)}{\left(\left(x^2+y^2\right)^2+1\right)^2}
\left[\frac{\pp h_2^{(1)}}{\pp y}(\xi,\xi)+\mu x \frac{\pp^2 h_2^{(1)}}{\pp x\pp y}(\xi,\xi)\right]
\\&\quad+\frac{8}{3}\mu^5\int_{\mathbb{R}^2}\left[h_1^{(1)}(\xi,\xi)+\mu y \frac{\pp h_1^{(1)}}{\pp y}(\xi,\xi)\right]\left[\frac{\pp h_2^{(1)}}{\pp x}(\xi,\xi)+\mu y \frac{\pp^2 h_2^{(1)}}{\pp x\pp y}(\xi,\xi)\right]
\frac{8 y
   \left(x^2+y^2\right)}{\left(\left(x^2+y^2\right)^2+1\right)^2}
\\&\quad+\frac{8}{3}\mu^7\int_{\mathbb{R}^2}\left[h_1^{(1)}(\xi,\xi)+\mu x \frac{\pp h_1^{(1)}}{\pp x}(\xi,\xi)\right]\left[\frac{\pp h_3^{(1)}}{\pp x}(\xi,\xi)+\mu x \frac{\pp^2 h_3^{(1)}}{\pp x^2}(\xi,\xi)\right]
\frac{4 x \left(x^4-2 x^2 y^2-3 y^4+1\right)}{\left(\left(x^2+y^2\right)^2+1\right)^2}
+ O\left(\mu^8d^{-8}\right)
\\&=O\left(\mu^6d^{-6}\right).
\end{align*}
For $l=2,3$, $\frac{2}{3}\int_{\mathcal{D}}\left(\varphi_1\right)_l[\delta_x\wedge (\varphi_1)_y+(\varphi_1)_x \wedge \delta_y]_l$ can be estimated similarly, we conclude that
\begin{equation}\label{Jan5-2}
\begin{aligned}
\frac{2}{3}\int_{\mathcal{D}}\varphi_1 \cdot [\delta_x\wedge (\varphi_1)_y+(\varphi_1)_x \wedge \delta_y]=O\left(\mu^8|\log\mu|d^{-8}+\mu^6d^{-6}\right).
\end{aligned}
\end{equation}
Thus, combining (\ref{A30-delta1}), (\ref{A30-delta2}), (\ref{Jan5-1}), (\ref{Jan5-3}) and (\ref{Jan5-2}), we get
\begin{equation}\label{A30-delta3}
\begin{aligned}
\frac{2}{3}\int_{\mathcal{D}}P\delta
\cdot \left[(P\delta)_x\wedge (P\delta)_y\right]
=-\frac{16}{3}\pi+16\pi\mu^4\frac{\partial^2 h_2^{(1)}}{\partial x \partial y}(\xi,\xi)+O\left(\mu^8|\log\mu|d^{-8}+\mu^6d^{-6}+\mu^8d^{-10}\right).
\end{aligned}
\end{equation}

Next, we compute the terms that contain the boundary function $g(z,\xi)$, i.e. $2\varepsilon\int_{\mathcal{D}}g\cdot\left[(P\delta)_x\wedge (P\delta)_y\right]$.
In Proposition \ref{prop 4.1}, the boundary function $g(z,\omega)=\left(g_1(z,\omega),g_2(z,\omega),g_3(z,\omega)\right)$ is given by
$$g_1(z,\omega)=2h_1^{(1)}(z,\omega),\quad g_2(z,\omega)=2h_2^{(1)}(z,\omega),\quad g_3(z,\omega)=-2\varepsilon h_3^{(1)}(z,\omega),$$
and $\left|g(z,\omega)-g(z,\xi)\right|=O(d^{-3}|\xi-\omega|).$
There holds
\begin{equation}\label{Jy18}
\begin{aligned}
2\varepsilon\int_{\mathcal{D}}g
\cdot \left[(P\delta)_x\wedge (P\delta)_y\right]
=2\varepsilon \int_{\mathcal{D}}g
\cdot (\delta_x\wedge \delta_y)-g
\cdot [\delta_x\wedge (\varphi_1)_y+(\varphi_1)_x\wedge \delta_y]+g
\cdot [(\varphi_1)_x\wedge (\varphi_1)_y].
\end{aligned}
\end{equation}
Similar to the calculation of (\ref{expansion-15}), using (\ref{cross}) and $\frac{\partial^2 g_1}{\partial x^2}(\xi,\omega)=-\frac{\partial^2 g_1}{\partial y^2}(\xi,\omega)$ for $\xi\in \mathcal{D}$ as well as the fact that $\frac{\partial^2 g_3}{\partial x^2}(\xi,\omega)+\frac{\partial^2 g_3}{\partial y^2}(\xi,\omega)=0$, we have
\begin{equation}\label{A30-delta4}
\begin{aligned}
2\varepsilon \int_{\mathcal{D}}g
\cdot (\delta_x\wedge \delta_y)
&=2\varepsilon \int_{\mathbb{R}^2}\left[\mu^2\frac{x^2}{2}\frac{\partial^2 g_1}{\partial x^2}(\xi,\omega)+\mu^2 \frac{y^2}{2}\frac{\partial^2 g_1}{\partial y^2}(\xi,\omega)
+O(\mu^4)\right]\frac{-32(x^4-y^4)}{\left(1+|z|^4\right)^3}
\\&\quad+2\varepsilon \int_{\mathbb{R}^2}\left[\mu^2xy\frac{\partial^2 g_2}{\partial x\partial y}(\xi,\omega)+O(\mu^4)\right]\frac{-64xy|z|^2}{\left(1+|z|^4\right)^3}
+O\left(\varepsilon\mu^4d^{-6}\right)
\\&=-4\pi\varepsilon \mu^2\left[\frac{\partial^2 g_1}{\partial x^2}(\xi,\omega)+\frac{\partial^2 g_2}{\partial x \partial y}(\xi,\omega)\right]+O\left(\varepsilon\mu^4d^{-6}\right).
\end{aligned}
\end{equation}
For the second term in (\ref{Jy18}), similar to (\ref{Jan5-2}), we get
\begin{equation}\label{Jy18-delta4}
\begin{aligned}
-2\varepsilon \int_{\mathcal{D}}g
\cdot [\delta_x\wedge (\varphi_1)_y+(\varphi_1)_x\wedge \delta_y]=O\left(\varepsilon\mu^4d^{-3}+\varepsilon\mu^6|\log\mu|d^{-5}+\varepsilon^2\mu^4|\log\mu|d^{-3}\right).
\end{aligned}
\end{equation}
And analogously with (\ref{Jan5-3}), there holds
\begin{equation}\label{Jan5-4}
\begin{aligned}
2\varepsilon \int_{\mathcal{D}}g
\cdot [(\varphi_1)_x\wedge (\varphi_1)_y]=O\left(\varepsilon\mu^6d^{-8}+\varepsilon^2\mu^4d^{-6}\right),
\end{aligned}
\end{equation}

\begin{equation}\label{Jy18-delta5}
\begin{aligned}
\varepsilon^2\int_{\mathcal{D}}(\delta-\varphi)\cdot(g_x\wedge g_y)=O\left(\varepsilon^3\mu^2|\log \mu|+\varepsilon^2\mu^2+\varepsilon^3\mu^2d^{-2}+\varepsilon^2\mu^4d^{-4}\right).
\end{aligned}
\end{equation}

{\bf $\bullet$ Estimates of the terms containing $a_l$ in (\ref{Au10-I1}).}
We start with the terms in (\ref{Au10-I1}) that contain $a_l^2$, $l=1,2$. Recall that $Z_{k,l}$, $k=-1,2$, $l=1,2$ satisfy the linearized equation of \eqref{intro-5} around $\mathcal W$ defined by
\begin{eqnarray*}
L_{\mathcal W}[v]:= \Delta v -2 \mathcal W_x\wedge v_y - 2 v_x\wedge \mathcal W_y = 0
\end{eqnarray*}
for $v\in \mathcal X$. Thus,
\begin{equation*}
\begin{aligned}
-\frac{a_1^2}{2}\int_{\mathcal{D}}L_\delta[P Z_{-1,1}]\cdot(P Z_{-1,1})
&=-a_1^2\int_{\mathcal{D}}\left[\delta_x\wedge (\varphi_{-1,1})_y + (\varphi_{-1,1})_x\wedge \delta_y\right]\cdot (P Z_{-1,1})
\\&=- a_1^2\int_{\mathcal{D}}\varphi_{-1,1}\cdot\left[\delta_x\wedge (P Z_{-1,1})_y + (P Z_{-1,1})_x\wedge \delta_y\right].
\end{aligned}
\end{equation*}
Using (\ref{expansion6-1}), we have
\begin{align}
\notag
&-a_1^2\int_{\mathcal{D}}\varphi_{-1,1}\cdot\left[\delta_x\wedge (Z_{-1,1})_y + (Z_{-1,1})_x\wedge \delta_y\right]
\\\notag &=-2a_1^2\mu\int_{\mathbb{R}^2}\left[\mu x\frac{\partial h_1^{(-1,1)}}{\partial x}(\xi,\xi)+O\left(\mu^3d^{-4}\right)\right]\frac{128 x \left(x^2+y^2\right) \left(x^2 \left(\left(x^2+y^2\right)^2-2\right)+3 y^2\right)}{\left(\left(x^2+y^2\right)^2+1\right)^4}
\\\notag &\quad-2a_1^2\mu\int_{\mathbb{R}^2}\left[\mu y\frac{\partial h_2^{(-1,1)}}{\partial y}(\xi,\xi)+O\left(\mu^3d^{-4}\right)\right]\frac{64 y \left(x^2+y^2\right) \left(3 x^6+7 x^4 y^2+x^2 \left(5y^4-9\right)+y^6+y^2\right)}{\left(\left(x^2+y^2\right)^2+1\right)^4}
\\\notag &\quad+O\left(a_1^2\mu^4d^{-4}\right)
\\\notag &=-8\pi a_1^2\mu^2\left[\frac{\partial h_1^{(-1,1)}}{\partial x}(\xi,\xi)+\frac{\partial h_2^{(-1,1)}}{\partial y}(\xi,\xi)\right]+O\left(a_1^2\mu^4d^{-4}\right)
\\\label{Mar22-1}&=-16\pi a_1^2\mu^2\frac{\partial h_1^{(-1,1)}}{\partial x}(\xi,\xi)+O\left(a_1^2\mu^4d^{-4}\right).
\end{align}
And similar to (\ref{Jan5-2}), there holds
\begin{equation*}
\begin{aligned}
a_1^2\int_{\mathcal{D}}\varphi_{-1,1}\cdot\left[\delta_x\wedge (\varphi_{-1,1})_y + (\varphi_{-1,1})_x\wedge \delta_y\right]=O\left(a_1^2\mu^6|\log\mu|d^{-6}+a_1^2\mu^4d^{-4}\right).
\end{aligned}
\end{equation*}
Therefore,
\begin{equation}\label{A30-a1}
\begin{aligned}
-\frac{a_1^2}{2}\int_{\mathcal{D}}L_\delta[P Z_{-1,1}]\cdot(P Z_{-1,1})
=-16\pi a_1^2\mu^2\frac{\partial h_1^{(-1,1)}}{\partial x}(\xi,\xi)+O\left(a_1^2\mu^4d^{-4}\right).
\end{aligned}
\end{equation}
And similarly, by  (\ref{J9-a2}), one has
\begin{align}
\notag
&-\frac{a_2^2}{2}\int_{\mathcal{D}}L_\delta[P Z_{-1,2}]\cdot(P Z_{-1,2})
 \\
\notag &=-a_2^2\int_{\mathcal{D}}\varphi_{-1,2}\cdot\left(\delta_x\wedge (Z_{-1,2})_y + (Z_{-1,2})_x\wedge \delta_y\right) +O\left(a_2^2\mu^6|\log\mu|d^{-6}+a_2^2\mu^4d^{-4}\right)
\\ \notag &=-2a_2^2\mu\int_{\mathbb{R}^2}\left[\mu y\frac{\partial h_1^{(-1,2)}}{\partial y}(\xi,\xi)+O\left(\mu^3d^{-4}\right)\right]\frac{-128 y \left(x^2+y^2\right) \left(\left(x^4-2\right) y^2+2 x^2 y^4+3 x^2+y^6\right)}{\left(\left(x^2+y^2\right)^2+1\right)^4}
 \\ \notag &\quad-2a_2^2\mu\int_{\mathbb{R}^2}\left[\mu x\frac{\partial h_2^{(-1,2)}}{\partial x}(\xi,\xi)+O\left(\mu^3d^{-4}\right)\right]\frac{64 x \left(x^2+y^2\right) \left(x^6+5 x^4 y^2+x^2 \left(7 y^4+1\right)+3 y^2
   \left(y^4-3\right)\right)}{\left(\left(x^2+y^2\right)^2+1\right)^4}
\\ \notag &\quad +O\left(a_2^2\mu^4d^{-4}\right)
\\ \notag &=8\pi a_2^2\mu^2\left[\frac{\partial h_1^{(-1,2)}}{\partial y}(\xi,\xi)-\frac{\partial h_2^{(-1,2)}}{\partial x}(\xi,\xi)\right]+O\left(a_2^2\mu^4d^{-4}\right)
\\ \label{M27-a2} &=-16\pi a_2^2\mu^2\frac{\partial h_2^{(-1,2)}}{\partial x}(\xi,\xi)+O\left(a_2^2\mu^4d^{-4}\right).
\end{align}

{\bf $\bullet$ Estimates of the terms containing $p_l$ in (\ref{Au10-I1}).}
Now, we consider the term of $p_l^2$, $l=1,2$.
Similar to the calculations of (\ref{A30-a1}), one has
\begin{align*}
&-\sum_{l=1}^2\frac{p_l^2}{2}\int_{\mathcal{D}}L_\delta[P Z_{2,l}]\cdot(P Z_{2,l})
\\&=-\sum_{l=1}^2 p_l^2\int_{\mathcal{D}}\left[\delta_x\wedge (\varphi_{2,l})_y + (\varphi_{2,l})_x\wedge \delta_y\right]\cdot (P Z_{2,l})
\\&=-\sum_{l=1}^2p_l^2\int_{\mathcal{D}}\varphi_{2,l}\cdot\left[\delta_x\wedge (P Z_{2,l})_y +(P Z_{2,l})_x\wedge \delta_y\right]
\\&=-\sum_{l=1}^2p_l^2\int_{\mathcal{D}}\varphi_{2,l}\cdot\left[\delta_x\wedge (Z_{2,l})_y + (Z_{2,l})_x\wedge \delta_y\right]+O\left((p_1^2+p_2^2)\left(\mu^{12}|\log\mu|d^{-12}+\mu^{10}d^{-10}\right)\right).
\end{align*}
Using (\ref{A26-5}), we obtain
\begin{equation}\label{A30-p3-a}
\begin{aligned}
& -p_1^2\int_{\mathcal{D}}\varphi_{2,1}\cdot\left[\delta_x\wedge (Z_{2,1})_y + (Z_{2,1})_x\wedge \delta_y\right]
\\& = -2p_1^2\mu^4\int_{\mathbb{R}^2}\left[h_1^{(2,1)}(\xi,\xi)+\frac{\mu^2}{2} x^2\frac{\partial^2 h^{(2, 1)}_1}{\partial x^2}(\xi, \xi)+\frac{\mu^2}{2}y^2\frac{\partial^2 h^{(2, 1)}_1}{\partial y^2}(\xi, \xi)+\frac{\mu^4}{24}x^4\frac{\partial^4 h^{(2, 1)}_1}{\partial x^4}(\xi, \xi) \right.
\\&\qquad\qquad\qquad\quad\left. +\frac{\mu^4}{24}y^4\frac{\partial^4 h^{(2, 1)}_1}{\partial y^4}(\xi, \xi) +\frac{\mu^4}{4}x^2y^2\frac{\partial^4 h^{(2, 1)}_1}{\partial x^2\partial y^2}(\xi, \xi)\right]\frac{32 \left(x^2+y^2\right) \left(5 x^4-14 x^2 y^2+5 y^4-1\right)}{\left(\left(x^2+y^2\right)^2+1\right)^4}
\\&\quad -2p_1^2\mu^4\int_{\mathbb{R}^2}\left[h_2^{(2,1)}(\xi,\xi)+\mu^2xy\frac{\partial^2 h^{(2, 1)}_2}{\partial x \partial y}(\xi, \xi) +\frac{\mu^4}{6}x^3y\frac{\partial^4 h^{(2, 1)}_2}{\partial x^3\partial y}(\xi, \xi) +\frac{\mu^4}{6}xy^3\frac{\partial^4 h^{(2, 1)}_2}{\partial x\partial y^3}(\xi, \xi)\right]
\\&\qquad\qquad\qquad \times \frac{384 x y \left(x^4-y^4\right)}{\left(\left(x^2+y^2\right)^2+1\right)^4}
\\&\quad+2 p_1^2\mu^6\int_{\mathbb{R}^2}\left[h_3^{(2,1)}(\xi,\xi)+\frac{\mu^2}{2}x^2\frac{\partial^2 h^{(2, 1)}_3}{\partial x^2}(\xi, \xi)+\frac{\mu^2}{2}y^2\frac{\partial^2 h^{(2, 1)}_3}{\partial y^2}(\xi, \xi)\right]
\\&\qquad\qquad\qquad \times \frac{64 (x-y) (x+y) \left(x^2+y^2\right)\left(\left(x^2+y^2\right)^2-2\right)}{\left(\left(x^2+y^2\right)^2+1\right)^4}+O\left(p_1^2\mu^{10}d^{-10}\right)
\\& = -2p_1^2\mu^4\left[\frac{\pi^2}{4}\mu^2\frac{\partial^2 h^{(2, 1)}_1}{\partial x^2}(\xi, \xi)+\frac{\pi^2}{4}\mu^2\frac{\partial^2 h^{(2, 1)}_1}{\partial y^2}(\xi, \xi) + \frac{\pi}{6}\mu^4\frac{\partial^4 h^{(2, 1)}_1}{\partial x^4}(\xi, \xi)+\frac{\pi}{6}\mu^4\frac{\partial^4 h^{(2, 1)}_1}{\partial y^4}(\xi, \xi)\right]
\\&\quad -2p_1^2\mu^4\left[\frac{\pi}{6}\mu^4\frac{\partial^4 h^{(2, 1)}_2}{\partial x^3\partial y}(\xi, \xi) -\frac{\pi}{6}\mu^4\frac{\partial^4 h^{(2, 1)}_2}{\partial x\partial y^3}(\xi, \xi)\right]+O\left(p_1^2\mu^{10}d^{-10}\right)
\\& = -\frac{4\pi}{3}p_1^2\mu^8\frac{\partial^4 h^{(2, 1)}_1}{\partial x^4}(\xi, \xi) +O\left(p_1^2\mu^{10}d^{-10}\right).
\end{aligned}
\end{equation}

And by (\ref{A26-6}), we have
\begin{align}\notag
& -p_2^2\int_{\mathcal{D}}\varphi_{2,2}\cdot\left[\delta_x\wedge (Z_{2,2})_y + (Z_{2,2})_x\wedge \delta_y\right]
\\\notag& = -2p_2^2\mu^4\int_{\mathbb{R}^2}\left[h_1^{(2,2)}(\xi,\xi)+\mu^2xy \frac{\partial^2 h^{(2, 2)}_1}{\partial x \partial y}(\xi, \xi)+\frac{\mu^4}{6}x^3y\frac{\partial^4 h^{(2, 2)}_1}{\partial x^3\partial y}(\xi, \xi)+\frac{\mu^4}{6}xy^3\frac{\partial^4 h^{(2, 2)}_1}{\partial x\partial y^3}(\xi, \xi)\right]
\\\notag&\qquad\qquad\qquad \times \frac{384 x y \left(x^4-y^4\right)}{\left(\left(x^2+y^2\right)^2+1\right)^4}
\\\notag&\quad -2p_2^2\mu^4\int_{\mathbb{R}^2}\left[h_2^{(2,2)}(\xi,\xi)+\frac{\mu^2}{2}x^2\frac{\partial^2 h^{(2, 2)}_2}{\partial x^2}(\xi, \xi) +\frac{\mu^2}{2}y^2\frac{\partial^2 h^{(2, 2)}_2}{\partial y^2}(\xi, \xi)+\frac{\mu^4}{24}x^4\frac{\partial^4 h^{(2, 2)}_2}{\partial x^4}(\xi, \xi) \right.
\\\notag&\qquad\qquad\qquad\quad\left. +\frac{\mu^4}{24}y^4\frac{\partial^4 h^{(2, 2)}_2}{\partial y^4}(\xi, \xi)  +\frac{\mu^4}{4}x^2y^2\frac{\partial^4 h^{(2, 2)}_2}{\partial x^2\partial y^2}(\xi, \xi)\right]\frac{-32\left(x^2+y^2\right) \left(x^4-22x^2 y^2+y^4+1\right)}{\left(\left(x^2+y^2\right)^2+1\right)^4}
\\\notag&\quad+2p_2^2\mu^6\int_{\mathbb{R}^2}\left[h_3^{(2,2)}(\xi,\xi)+\mu^2xy\frac{\partial^2 h^{(2, 2)}_3}{\partial x \partial y}(\xi, \xi)\right] \frac{128 x y \left(x^2+y^2\right)\left(\left(x^2+y^2\right)^2-2\right)}{\left(\left(x^2+y^2\right)^2+1\right)^4}+O\left(p_2^2\mu^{10}d^{-10}\right)
\\\notag& =-2p_2^2\mu^4\left[\frac{\pi}{6}\mu^4\left(\frac{\partial^4 h^{(2, 2)}_1}{\partial x^3\partial y}-\frac{\partial^4 h^{(2, 2)}_1}{\partial x\partial y^3}\right)(\xi, \xi)\right]
\\\notag &\quad -2p_2^2\mu^4\left[\frac{\pi^2}{4}\mu^2\left(\frac{\partial^2 h^{(2, 2)}_2}{\partial x^2}+\frac{\partial^2 h^{(2, 2)}_2}{\partial y^2}\right)(\xi, \xi) + \frac{\pi}{12}\mu^4\left(\frac{\partial^4 h^{(2, 2)}_2}{\partial x^4}+\frac{\partial^4 h^{(2, 2)}_2}{\partial y^4}\right)(\xi, \xi)+\frac{\pi}{2}\mu^4\frac{\partial^4 h^{(2, 2)}_2}{\partial x^2\partial y^2}(\xi, \xi)\right]
\\\notag&\quad+O\left(p_2^2\mu^{10}d^{-10}\right)
\\\label{A30-p3-b}& = \frac{4\pi}{3}p_2^2\mu^8\frac{\partial^4 h^{(2, 2)}_2}{\partial x^4}(\xi, \xi) +O\left(p_2^2\mu^{10}d^{-10}\right).
\end{align}
Thus
\begin{equation}\label{A30-p3}
\begin{aligned}
-\sum_{l=1}^2\frac{p_l^2}{2}\int_{\mathcal{D}}L_\delta[P Z_{2,l}]\cdot(P Z_{2,l})
=-\frac{4\pi}{3}p_1^2\mu^8\frac{\partial^4 h^{(2, 1)}_1}{\partial x^4}(\xi, \xi)+\frac{4\pi}{3}p_2^2\mu^8\frac{\partial^4 h^{(2, 2)}_2}{\partial x^4}(\xi, \xi) + O\left((p_1^2+p_2^2)\mu^{10}d^{-10}\right).
\end{aligned}
\end{equation}

Now, we estimate terms in (\ref{Au10-I1}) that involve $p_l$, $l=1,2$.
Following a similar approach as in (\ref{A30-a1}) and (\ref{M27-a2}), we have
\begin{equation}\label{Au11-1}
\begin{aligned}
&-\int_{\mathcal{D}}L_\delta\left[\sum_{l=1}^2p_l PZ_{2,l}\right]\cdot \left(\sum_{l=1}^2a_lP Z_{-1,l}\right)
\\&=-2\int_{\mathcal{D}}\left[\delta_x\wedge \left(\sum_{l=1}^2p_l \varphi_{2,l}\right)_y+\left(\sum_{l=1}^2p_l \varphi_{2,l}\right)_x \wedge  \delta_y\right]\cdot \left(\sum_{l=1}^2a_lP Z_{-1,l}\right)
\\&=-2\int_{\mathcal{D}}\left(\sum_{l=1}^2p_l \varphi_{2,l}\right)\cdot\left[\delta_x\wedge \left(\sum_{l=1}^2a_l Z_{-1,l}\right)_y+\left(\sum_{l=1}^2a_l Z_{-1,l}\right)_x \wedge  \delta_y\right]
\\&\quad+O\left((a_1+a_2)(p_1+p_2)\left(\mu^9|\log\mu|d^{-9}+\mu^7d^{-7}\right)\right)
\\&=-32\pi p_1a_1\mu^5\frac{\partial h_1^{(2,1)}}{\partial x}(\xi,\xi)-32\pi p_1a_2\mu^5\frac{\partial h_2^{(2,1)}}{\partial x}(\xi,\xi)
-32\pi p_2a_1\mu^5\frac{\partial h_1^{(2,2)}}{\partial x}(\xi,\xi)
-32\pi p_2a_2\mu^5\frac{\partial h_2^{(2,2)}}{\partial x}(\xi,\xi)
\\&\quad+O\left((a_1+a_2)(p_1+p_2)\left(\mu^9|\log\mu|d^{-9}+\mu^7d^{-7}\right)\right).
\end{aligned}
\end{equation}

\subsection{Conclusion}
In this subsection, we conclude the results calculated in Subsection \ref{subsection4.1} and Appendix \ref{Computations of mixed terms}.
Indeed, we have
\begin{equation*}
\begin{aligned}
 I^{(1)}_{\varepsilon}\left(P\delta_{\mu, \xi, a, p}\right)
&=(\ref{expansion-15})+(\ref{A30-delta3})+(\ref{A30-delta4})
+(\ref{Jy18-delta4})
+(\ref{Jan5-4})+(\ref{Jy18-delta5})
+(\ref{A30-a1})+(\ref{M27-a2})
+(\ref{A30-p3})+(\ref{Au11-1})
\\&=\frac{8\pi}{3}
+8\pi\mu^4\frac{\partial^2 h_2^{(1)}}{\partial x \partial y}(\xi,\xi)-4\pi\varepsilon \mu^2\left[\frac{\partial^2 g_1}{\partial x^2}(\xi,\omega)+\frac{\partial^2 g_2}{\partial x \partial y}(\xi,\omega)\right]
\\&\quad-16\pi a_1^2\mu^2\frac{\partial h_1^{(-1,1)}}{\partial x}(\xi,\xi)
-16\pi a_2^2\mu^2\frac{\partial h_2^{(-1,2)}}{\partial x}(\xi,\xi)
\\&\quad-\frac{4\pi}{3}p_1^2\mu^8\frac{\partial^4 h^{(2, 1)}_1}{\partial x^4}(\xi, \xi) +\frac{4\pi}{3}p_2^2\mu^8\frac{\partial^4 h^{(2, 2)}_2}{\partial x^4}(\xi, \xi)
\\&\quad-32\pi p_1a_1\mu^5\frac{\partial h_1^{(2,1)}}{\partial x}(\xi,\xi)-32\pi p_1a_2\mu^5\frac{\partial h_2^{(2,1)}}{\partial x}(\xi,\xi)
\\&\quad-32\pi p_2a_1\mu^5\frac{\partial h_1^{(2,2)}}{\partial x}(\xi,\xi)
-32\pi p_2a_2\mu^5\frac{\partial h_2^{(2,2)}}{\partial x}(\xi,\xi)
+\tilde{e}^{(1)}(\mu,a,p),
\end{aligned}
\end{equation*}
where
\begin{equation*}
\begin{aligned}
 \tilde{e}^{(1)}(\varepsilon,\mu,a,p)
=&~O\left(\mu^8|\log\mu|d^{-8}+\mu^6d^{-6}+\mu^8d^{-10}\right)+O\left(\varepsilon\mu^4d^{-6}+\varepsilon\mu^4d^{-3}+\varepsilon\mu^6|\log\mu|d^{-5}+\varepsilon^2\mu^4|\log\mu|d^{-3}\right)
\\&+O\left(\varepsilon\mu^6d^{-8}+\varepsilon^2\mu^4d^{-6}\right)
+O\left(\varepsilon^3\mu^2|\log \mu|+\varepsilon^2\mu^2+\varepsilon^3\mu^2d^{-2}+\varepsilon^2\mu^4d^{-4}\right)
\\&+O\left(\left(a_1^2+a_2^2\right)\mu^4d^{-4}\right)
+ O\left((p_1^2+p_2^2)\mu^{10}d^{-10}\right)
\\&+O\left((a_1+a_2)(p_1+p_2)\left(\mu^9|\log\mu|d^{-9}+\mu^7d^{-7}\right)\right)
 \\=&~O\left(\mu^8d^{-10}+\varepsilon\mu^4d^{-6}+\varepsilon^3\mu^2|\log \mu|+\varepsilon^2\mu^2+\varepsilon^3\mu^2d^{-2}\right)
\\&+O\left(\left(a_1^2+a_2^2\right)\mu^4d^{-4}\right)
+ O\left((p_1^2+p_2^2)\mu^{10}d^{-10}\right)+O\left((a_1+a_2)(p_1+p_2)\mu^7d^{-7}\right).
\end{aligned}
\end{equation*}

Recall that
$I_{\varepsilon}\left(P\delta_{\mu, \xi, a, p}\right)=I^{(1)}_{\varepsilon}\left(P\delta_{\mu, \xi, a, p}\right)+\mathcal{R}_{\mu,\xi,a,p,\varepsilon}$,
where $\mathcal{R}_{\mu,\xi,a,p,\varepsilon}$ is estimated in Appendix \ref{Computations of mixed terms}, see (\ref{Jan8-R}).
Thus we complete the proof of Proposition \ref{prop 4.1} in this case. For general rotation $\mathcal{Q}$, it is sufficient to consider the boundary datum $\mathcal{Q}^{-1}\tilde g$ and to replace $g_1$, $g_2$, $g_3$ by $\left(\mathcal{Q}^{-1}g\right)_1$, $\left(\mathcal{Q}^{-1}g\right)_2$ and $\left(\mathcal{Q}^{-1}g\right)_3$, respectively.
\end{proof}

\section{Energy expansion for multi-bubbles}\label{expansion-multi-bubbles}
In this section, we consider the case of multiple bubbles, starting with the case of two bubbles. We define
\begin{equation*}
z_{\mathcal{A}_1}:=P\mathcal Q_1\delta_{\mu_1,\xi^{(1)}, a_1, p_1}\approx P\mathcal Q_1\delta_{\mu_1,\xi^{(1)}}+P\mathcal Q_1\left(\mathcal{L}_{\mathcal{A}_1}+\mathcal{R}_{\mathcal{A}_1}\right)
\end{equation*}
and
\begin{equation*}
z_{\mathcal{A}_2}:=P\mathcal Q_2\delta_{\mu_2,\xi^{(2)}, a_2, p_2}\approx P\mathcal Q_2\delta_{\mu_2,\xi^{(2)}}+P\mathcal Q_2\left(\mathcal{L}_{\mathcal{A}_2}+\mathcal{R}_{\mathcal{A}_2}\right),
\end{equation*}
where
$$\mathcal{L}_{\mathcal{A}_1}=\sum_{l=1}^2\left(a_{1,l}Z_{-1,l}\left(\frac{z-\xi^{(1)}}{\mu_1}\right)+p_{1, l}Z_{2,l}\left(\frac{z-\xi^{(1)}}{\mu_1}\right)\right),$$
$$\mathcal{L}_{\mathcal{A}_2}=\sum_{l=1}^2a_{2,l}\left(Z_{-1,l}\left(\frac{z-\xi^{(2)}}{\mu_2}\right)+p_{2, l}Z_{2,l}\left(\frac{z-\xi^{(2)}}{\mu_2}\right)\right),$$
and
$\mathcal{R}_{\mathcal{A}_i}$, $i=1,2$ denotes the corresponding higher order terms, as discussed in Section \ref{expansion-for-one-bubble}.
For simplicity, we denote $\mathcal Q_1\delta_{\mu_1,\xi^{(1)}}$, $\mathcal Q_2\delta_{\mu_2,\xi^{(2)}}$, $\xi^{(1)}$ and $\xi^{(2)}$ as $\delta_1$, $\delta_2$, $\xi$ and $\zeta$, respectively. We assume $\mathcal Q_1=Id$, namely, the first bubble is not rotated. For $\bar C >0,\ k=2$ and
$\sum_{i=1}^2z_{\mathcal{A}_i}\in Z,$
we aim to expand the functional $I_{\varepsilon}$ on $Z$ in terms of the parameters $\mu_i$,  $\xi^{(i)}$, $\mathcal Q_i$,  $a_{i,l}$ and $p_{i, l}$ for $i, l =1, 2$. In the following, we define
\begin{equation}\label{def-sigmasigma}
\sigma:=\xi^{(2)}-\xi^{(1)}=\zeta-\xi=(\sigma_1,\sigma_2)
\end{equation}
and write $P\delta_1=\delta_1-\varphi_1(z,\xi)$, $P\mathcal Q_2\delta_2=\mathcal Q_2\left(\delta_2-\varphi_1(z,\zeta)\right)$. Recall that for $u\in H_0^1(\mathcal{D},\mathbb{R}^3)$, the explicit form of the functional $I_{\varepsilon}(u)$ is
\begin{equation}\label{2-functional}
I_{\varepsilon}(u)=\frac{1}{2}\int_{\mathcal{D}}|\nabla u|^2+\frac{2}{3}\int_{\mathcal{D}} u\cdot (u_x\wedge u_y)+\varepsilon\int_{\mathcal{D}} u\cdot(u_x\wedge g_y+g_x\wedge u_y )+2\varepsilon^2 \int_{\mathcal{D}} u\cdot (g_x\wedge g_y).
\end{equation}

For $u=\sum_{i=1}^2z_{\mathcal{A}_i}\in Z$, we provide the estimates for (\ref{2-functional}). By applying integration by parts, we decompose the functional as $I_{\varepsilon}(u)=I^{(1)}_{\varepsilon}(u)+I^{(2)}_{\varepsilon}(u),$
where
\begin{equation*}
\begin{aligned}
I^{(1)}_{\varepsilon}(u)
&=\frac{1}{2}\int_{\mathcal{D}}\left|\nabla z_{\mathcal{A}_1}\right|^2
+\frac{1}{2}\int_{\mathcal{D}}\left|\nabla z_{\mathcal{A}_2}\right|^2+\frac{2}{3}\int_{\mathcal{D}} z_{\mathcal{A}_1} \cdot \left[\left(z_{\mathcal{A}_1}\right)_x\wedge \left(z_{\mathcal{A}_1}\right)_y\right]+\frac{2}{3}\int_{\mathcal{D}}z_{\mathcal{A}_2} \cdot \left[\left(z_{\mathcal{A}_2}\right)_x\wedge \left(z_{\mathcal{A}_2}\right)_y\right]
\\&\quad+\varepsilon\int_{\mathcal{D}} z_{\mathcal{A}_1}\cdot\left[\left(z_{\mathcal{A}_1}\right)_x\wedge g^{(1)}_y+g^{(1)}_x\wedge \left(z_{\mathcal{A}_1}\right)_y \right]+\varepsilon\int_{\mathcal{D}} z_{\mathcal{A}_2}\cdot\left[\left(z_{\mathcal{A}_2}\right)_x\wedge \left(\mathcal Q_2g^{(2)}\right)_y+\left(\mathcal Q_2g^{(2)}\right)_x\wedge \left(z_{\mathcal{A}_2}\right)_y \right]
\\&\quad+2\varepsilon^2 \int_{\mathcal{D}} z_{\mathcal{A}_1}\cdot \left(g^{(1)}_x\wedge g^{(1)}_y\right)+2\varepsilon^2 \int_{\mathcal{D}} z_{\mathcal{A}_2}\cdot \left[\left(\mathcal Q_2g^{(2)}\right)_x\wedge\left(\mathcal Q_2g^{(2)}\right)_y\right],
\end{aligned}
\end{equation*}

\begin{equation}\label{22-2-interaction}
\begin{aligned}
I^{(2)}_{\varepsilon}(u)
&=\int_{\mathcal{D}}\nabla z_{\mathcal{A}_1}
\cdot \nabla z_{\mathcal{A}_2}+2\int_{\mathcal{D}}z_{\mathcal{A}_1} \cdot \left[\left(z_{\mathcal{A}_2}\right)_x\wedge \left(z_{\mathcal{A}_2}\right)_y\right]
+2\int_{\mathcal{D}}z_{\mathcal{A}_2} \cdot \left[\left(z_{\mathcal{A}_1}\right)_x\wedge \left(z_{\mathcal{A}_1}\right)_y\right]
\\&\quad+2\varepsilon\int_{\mathcal{D}}  g^{(1)} \cdot\left[\left(z_{\mathcal{A}_2}\right)_x\wedge \left(z_{\mathcal{A}_2}\right)_y\right]+2\varepsilon\int_{\mathcal{D}} \left(\mathcal Q_2g^{(2)}\right)\cdot\left[\left(z_{\mathcal{A}_1}\right)_x\wedge \left(z_{\mathcal{A}_1}\right)_y\right]
\\&\quad+2\varepsilon\int_{\mathcal{D}} z_{\mathcal{A}_1}\cdot\left[\left(z_{\mathcal{A}_2}\right)_x\wedge \left(\mathcal Q_2g^{(2)}\right)_y+\left(\mathcal Q_2g^{(2)}\right)_x\wedge \left(z_{\mathcal{A}_2}\right)_y \right]+2\varepsilon\int_{\mathcal{D}} z_{\mathcal{A}_2}\cdot\left[\left(z_{\mathcal{A}_1}\right)_x\wedge g^{(1)}_y+g^{(1)}_x\wedge \left(z_{\mathcal{A}_1}\right)_y \right]
\\&\quad+2\varepsilon^2 \int_{\mathcal{D}} g^{(1)}\cdot \left[\left(z_{\mathcal{A}_2}\right)_x\wedge \left(\mathcal Q_2g^{(2)}\right)_y+\left(\mathcal Q_2g^{(2)}\right)_x\wedge \left(z_{\mathcal{A}_2}\right)_y \right]
\\&\quad+2\varepsilon^2 \int_{\mathcal{D}} \mathcal Q_2g^{(2)}\cdot \left[\left(z_{\mathcal{A}_1}\right)_x\wedge g^{(1)}_y+g^{(1)}_x\wedge \left(z_{\mathcal{A}_1}\right)_y\right]
\\&\quad+2\varepsilon^2 \int_{\mathcal{D}} z_{\mathcal{A}_1}\cdot \left[\left(\mathcal Q_2g^{(2)}\right)_x\wedge\left(\mathcal Q_2g^{(2)}\right)_y\right]+2\varepsilon^2 \int_{\mathcal{D}} z_{\mathcal{A}_2}\cdot  \left(g^{(1)}_x\wedge g^{(1)}_y\right).
\end{aligned}
\end{equation}
The estimate of $I^{(1)}_{\varepsilon}(u)$ is provided in Section \ref{expansion-for-one-bubble}. As for $I^{(2)}_{\varepsilon}(u)$, we estimate it in Subsections \ref{subsection5.1}-\ref{subsection5.4}.

\subsection{Interaction between two bubbles $\delta_1$ and $\delta_2$}\label{subsection5.1}
 In this subsection, we address the terms in (\ref{22-2-interaction}) that involve both $P\delta_1$ and $P\mathcal Q_2\delta_2$, namely,
\begin{equation}\label{3-interaction}
\begin{aligned}
&N(\delta_1,\delta_2,Id,\mathcal Q_2)
\\&:=\int_{\mathcal{D}}\nabla P\delta_1\cdot \nabla P\mathcal Q_2\delta_2
+2 \int_{\mathcal{D}}P\delta_1\cdot \left[\left(P\mathcal Q_2\delta_2\right)_x\wedge (P\mathcal Q_2\delta_2)_y\right]
+2\int_{\mathcal{D}} P\mathcal Q_2\delta_2\cdot \left[(P\delta_1)_x\wedge (P\delta_1)_y\right]
\\&~\quad+2\varepsilon\int_{\mathcal{D}}  g^{(1)} \cdot\left[\left(P\mathcal Q_2\delta_2\right)_x\wedge \left(P\mathcal Q_2\delta_2\right)_y\right]+2\varepsilon\int_{\mathcal{D}} \left(\mathcal Q_2g^{(2)}\right)\cdot\left[\left(P\delta_1\right)_x\wedge \left(P\delta_1\right)_y\right]
\\&~\quad+2\varepsilon\int_{\mathcal{D}} P\delta_1\cdot\left[\left(P\mathcal Q_2\delta_2\right)_x\wedge \left(\mathcal Q_2g^{(2)}\right)_y+\left(\mathcal Q_2g^{(2)}\right)_x\wedge \left(P\mathcal Q_2\delta_2\right)_y \right]
\\&~\quad+2\varepsilon\int_{\mathcal{D}} P\mathcal Q_2\delta_2\cdot\left[\left(P\delta_1\right)_x\wedge g^{(1)}_y+g^{(1)}_x\wedge \left(P\delta_1\right)_y \right]
\\&~\quad+2\varepsilon^2 \int_{\mathcal{D}} g^{(1)}\cdot \left[\left(P\mathcal Q_2\delta_2\right)_x\wedge\left(\mathcal Q_2g^{(2)}\right)_y+\left(\mathcal Q_2g^{(2)}\right)_x\wedge \left(P\mathcal Q_2\delta_2\right)_y \right]
\\&~\quad+2\varepsilon^2 \int_{\mathcal{D}} \mathcal Q_2g^{(2)}\cdot \left[\left(P\delta_1\right)_x\wedge g^{(1)}_y+g^{(1)}_x\wedge \left(P\delta_1\right)_y\right]
\\&~\quad+2\varepsilon^2 \int_{\mathcal{D}} P\delta_1\cdot \left[\left(\mathcal Q_2g^{(2)}\right)_x\wedge\left(\mathcal Q_2g^{(2)}\right)_y\right]+2\varepsilon^2 \int_{\mathcal{D}} P\mathcal Q_2\delta_2\cdot  \left(g^{(1)}_x\wedge g^{(1)}_y\right).
\end{aligned}
\end{equation}
Integrating by parts, expression (\ref{3-interaction}) simplifies to
\begin{align*}
&N(\delta_1,\delta_2,Id,\mathcal Q_2)
\\&=2\int_{\mathcal{D}}P\mathcal Q_2\delta_2\cdot \left[(\delta_1)_x\wedge (\delta_1)_y\right]+2\varepsilon\int_{\mathcal{D}}\mathcal Q_2g^{(2)}\cdot \left[(\delta_1)_x\wedge (\delta_1)_y\right]+2\varepsilon\int_{\mathcal{D}}  g^{(1)} \cdot\left[\left(\mathcal Q_2\delta_2\right)_x\wedge \left(\mathcal Q_2\delta_2\right)_y\right]
\\&\quad-2\int_{\mathcal{D}} \left(P\delta_1+\varepsilon g^{(1)}\right)\cdot \left[(\mathcal{Q}_2\delta_2)_x\wedge \left(\mathcal{Q}_2\varphi_1(z,\zeta)-\varepsilon\mathcal Q_2g^{(2)}\right)_y+\left(\mathcal{Q}_2\varphi_1(z,\zeta)-\varepsilon\mathcal Q_2g^{(2)}\right)_x\wedge (\mathcal{Q}_2\delta_2)_y\right]
\\&\quad+2\int_{\mathcal{D}}P\delta_1\cdot \left[(\mathcal{Q}_2\varphi_1(z,\zeta))_x\wedge  \left(\mathcal{Q}_2\varphi_1(z,\zeta)-\varepsilon\mathcal Q_2g^{(2)}\right)_y-\left(\varepsilon\mathcal Q_2g^{(2)}\right)_x\wedge \left(\mathcal{Q}_2\varphi_1(z,\zeta)-\varepsilon \mathcal Q_2g^{(2)}\right)_y\right]
\\&\quad+2\int_{\mathcal{D}} \varepsilon g^{(1)}\cdot \left[(\mathcal{Q}_2\varphi_1(z,\zeta))_x\wedge  \left(\mathcal{Q}_2\varphi_1(z,\zeta)-\varepsilon\mathcal Q_2g^{(2)}\right)_y-\left(\varepsilon\mathcal Q_2g^{(2)}\right)_x\wedge \left(\mathcal{Q}_2\varphi_1(z,\zeta)\right)_y\right]
\\&\quad-2\int_{\mathcal{D}} \left(P\mathcal Q_2\delta_2+\varepsilon \mathcal Q_2g^{(2)}\right)\cdot \left[(\delta_1)_x\wedge \left(\varphi_1(z,\xi)-\varepsilon g^{(1)}\right)_y+\left(\varphi_1(z,\xi)-\varepsilon g^{(1)}\right)_x\wedge (\delta_1)_y\right]
\\&\quad+2\int_{\mathcal{D}} P\mathcal Q_2\delta_2\cdot \left[(\varphi_1(z,\xi))_x\wedge \left(\varphi_1(z,\xi)-\varepsilon g^{(1)}\right)_y-\left(\varepsilon g^{(1)}\right)_x\wedge \left(\varphi_1(z,\xi)-\varepsilon g^{(1)}\right)_y\right]
\\&\quad+2\int_{\mathcal{D}} \varepsilon \mathcal Q_2g^{(2)}\cdot \left[(\varphi_1(z,\xi))_x\wedge \left(\varphi_1(z,\xi)-\varepsilon g^{(1)}\right)_y-\left(\varepsilon g^{(1)}\right)_x\wedge \left(\varphi_1(z,\xi)\right)_y\right].
\end{align*}

\begin{lemma}\label{lemma 6.3} There holds
\begin{equation*}
\begin{aligned}
N(\delta_1,\delta_2,Id,\mathcal{Q}_2)=F_{\mathcal{D}}(\delta_1,\delta_2,Id,\mathcal{Q}_2)+\mathcal{E}\left(\mu_1,\mu_2\right),
\end{aligned}
\end{equation*}
where
\begin{align}
F_{\mathcal{D}}(\delta_1,\delta_2,Id,\mathcal{Q}_2):= \notag &~2\int_{\mathcal{D}}P\mathcal Q_2\delta_2\cdot \left[(\delta_1)_x\wedge (\delta_1)_y\right]+2\varepsilon\int_{\mathcal{D}}\mathcal Q_2g^{(2)}\cdot \left[(\delta_1)_x\wedge (\delta_1)_y\right]
\\ \notag &+2\varepsilon\int_{\mathcal{D}}  g^{(1)} \cdot\left[\left(\mathcal Q_2\delta_2\right)_x\wedge \left(\mathcal Q_2\delta_2\right)_y\right]
\\= \notag &~q_{11}\left[\frac{-48\pi\left(\sigma_1^4-6 \sigma_2^2 \sigma_1^2+\sigma_2^4\right)}{|\sigma|^8}\mu_1^2\mu_2^2-4\pi \varepsilon\mu_2^2\frac{\partial^2g_1^{(1)}}{\partial x^2}(\zeta,\omega^{(1)})\right]
\\ \label{interaction-23} &+q_{12}\left[\frac{-192\pi\sigma_1\sigma_2(\sigma_1^2-\sigma_2^2)\mu_1^2\mu_2^2}{|\sigma|^8}
-4\pi\varepsilon\mu_2^2\frac{\partial^2 g^{(1)}_2}{\partial x\pp y}(\zeta,\omega^{(1)})\right]
\\ \notag &
+q_{21}\left[\frac{-192\pi\sigma_1\sigma_2(\sigma_1^2-\sigma_2^2)\mu_1^2\mu_2^2}{|\sigma|^8}
-4\pi\varepsilon\mu_2^2\frac{\partial^2 g^{(1)}_1}{\partial x^2}(\zeta,\omega^{(1)})\right]
\\ \notag &+q_{22}\left[\frac{48\pi\left(\sigma_1^4-6 \sigma_2^2 \sigma_1^2+\sigma_2^4\right)}{|\sigma|^8}\mu_1^2\mu_2^2
-4\pi\varepsilon\mu_2^2\frac{\partial^2 g^{(1)}_2}{\partial x \partial y}(\zeta,\omega^{(1)}) \right],
\\ \notag
\mathcal{E}\left(\mu_1,\mu_2\right):=&~
O\left(\varepsilon^{3}d^{-3}
+\varepsilon^{2}\left(\mu_1^2-\varepsilon\right)d^{-6}
+\varepsilon^{2}\left(\mu_2^2-\varepsilon\right)d^{-6}
+\varepsilon(\mu_2^2-\varepsilon)\right).
\end{align}
\end{lemma}
\begin{proof}
Let us now compute the term $\int_{\mathcal{D}} P\mathcal Q_2\delta_2\cdot \left[(\delta_1)_x\wedge (\delta_1)_y\right]$. Up to an error of order $\varepsilon$, we have
\begin{equation*}
P\mathcal{Q}_2\delta_2=\begin{pmatrix}q_{11}\quad q_{12}\quad q_{13}\\
 q_{21}\quad q_{22}\quad q_{23}\\
 q_{31}\quad q_{32}\quad q_{33}
\end{pmatrix}
\begin{pmatrix}
\frac{2\left[\left(\frac{x-\zeta_1}{\mu_2}\right)^2-\left(\frac{y-\zeta_2}{\mu_2}\right)^2\right]}{1+\left(\left|\frac{z-\zeta}{\mu_2}\right|^2\right)^2}-2\mu_2^2h_1^{(1)}(z,\zeta)\\
\frac{4\frac{(x-\zeta_1)(y-\zeta_2)}{\mu_2^2}}{1+\left(\left|\frac{z-\zeta}{\mu_2}\right|^2\right)^2}-2\mu_2^2h_2^{(1)}(z,\zeta)\\
\frac{-2}{1+\left(\left|\frac{z-\zeta}{\mu_2}\right|^2\right)^2}+2\mu_2^4h_3^{(1)}(z,\zeta)
\end{pmatrix},
\end{equation*}
where $q_{ij},\ i,j=1,2,3$ are the entries of the matrix $\mathcal{Q}_2$.

Since
\begin{equation}\label{interaction-24}
\begin{aligned}
2\int_{\mathcal{D}} P\mathcal{Q}_2\delta_2\cdot \left[(\delta_1)_x\wedge (\delta_1)_y\right]=2\sum_{i,j=1}^{3}q_{ij}\int_{\mathcal{D}}(\delta_2-\varphi_2)_j\cdot \left[(\delta_1)_x\wedge (\delta_1)_y\right]_i,
\end{aligned}
\end{equation}
the terms in (\ref{interaction-24}) involving the coefficients $q_{ij}$ are given by $2\int_{\mathcal{D}}(\delta_2-\varphi_2)_j\cdot \left[(\delta_1)_x\wedge (\delta_1)_y\right]_i$, $i, j=1, 2, 3$, respectively. Recall that in the definition of $Z$, we assume that $\text{dist}~(\xi^{(i)},\xi^{(j)})\geq \bar C^{-1}$ for $i\neq j$ and $\frac{\varepsilon}{\mu^2_i}, \frac{\varepsilon}{\mu^2_j} \in[ \bar C^{-1}, \bar C]$.
Let $\gamma\leq \frac{ \bar C^{-1}}{2}$ be a fixed positive number, we may choose
$\gamma=\frac{|\sigma|}{8}$, where $\sigma$ is defined in \eqref{def-sigmasigma}. Then, we divide the integral domain into two regions,
$$\mathcal{D}_1:=B_{\gamma}(\zeta) ,\qquad\quad\mathcal{D}_2:=\mathcal{D} \backslash B_{\gamma}(\zeta),$$
where $B_{\gamma}(\zeta)$ is the ball of radius $\gamma$ centered at $\zeta$.

On the set $\mathcal{D}_1=B_{\gamma}(\zeta)$, using (\ref{estimates-1}) and recalling that for $i=1,2$,
$$d=d_{\varepsilon}=1-|\xi^{(i)}|^2=O\left(\mu_i^{\frac{2}{3}-\alpha}\right)=O\left(\varepsilon^{\frac{1}{3}-\frac{\alpha}{2}}\right),$$
we have
\begin{equation}\label{Au7-1}
\begin{aligned}
\left|2\int_{\mathcal{D}_1}(\delta_2-\varphi_2)_1\left[(\delta_1)_x\wedge (\delta_1)_y\right]_1\right|
\lesssim \mu_1^6\int_{\mathcal{D}_1} \left|(\delta_2-\varphi_2)_1\right|
\lesssim \mu_1^6\left(\mu_2^2|\log \mu_2|+\mu_2^2 d^{-2}\right).
\end{aligned}
\end{equation}
In $\mathcal{D}_2$, since $z$ is away from $\zeta$, for $\mu_2$ sufficiently small, $\delta_2\left(\frac{z-\zeta}{\mu_2}\right)$ has the following expansion
\begin{equation}\label{Ju25-6}
\begin{aligned}
\delta_2\left(\frac{z-\zeta}{\mu_2}\right)=\left(\frac{2\mu_2^2\left[(x-\zeta_1)^2-(y-\zeta_2)^2\right]}{|z-\zeta|^4}+O\left(\mu_2^6\right),\frac{4\mu_2^2(x-\zeta_1)(y-\zeta_2)}{|z-\zeta|^4}+O\left(\mu_2^6\right),1-\frac{2\mu_2^4}{|z-\zeta|^4}+O\left(\mu_2^8\right)\right).
\end{aligned}
\end{equation}
Then we deduce
\begin{align}
&\notag2\int_{\mathcal{D}_2}(\delta_2-\varphi_2)_1\left[(\delta_1)_x\wedge (\delta_1)_y\right]_1
\\\notag&=-128\int_{\frac{\mathcal{D}_2-\xi}{\mu_1}}\left[\mu_2^2\frac{(\mu_1x+\xi_1-\zeta_1)^2-(\mu_1y+\xi_2-\zeta_2)^2}{|\mu_1z+\xi-\zeta|^4}+O(\mu_2^6)\right]\frac{x^4-y^4}{\left(1+|z|^4\right)^3}
\\\notag&\quad+128\int_{\frac{\mathcal{D}_2-\xi}{\mu_1}}\left[\mu_2^2h_1^{(1)}(\mu_1z+\xi,\zeta)+O(\mu_2^5)\right]\frac{(x^4-y^4)}{\left(1+|z|^4\right)^3}
\\ \label{37-interaction}&=-64\mu_1^2\mu_2^2\int_{\frac{\mathcal{D}_2-\xi}{\mu_1}}\frac{6 \left(\sigma_1^4-6 \sigma_2^2 \sigma_1^2+\sigma_2^4\right)}{|\sigma|^8}\frac{x^2(x^4-y^4)}{\left(1+|z|^4\right)^3}+\frac{-6\left(\sigma_1^4-6 \sigma_2^2 \sigma_1^2+\sigma_2^4\right)}{|\sigma|^8}\frac{y^2(x^4-y^4)}{\left(1+|z|^4\right)^3}
\\\notag&\quad+64\mu_1^2\mu_2^2\int_{\frac{\mathcal{D}_2-\xi}{\mu_1}}\frac{x^2(x^4-y^4)}{\left(1+|z|^4\right)^3}\frac{\partial^2h_1}{\partial x^2}(\xi,\zeta)+\frac{y^2(x^4-y^4)}{\left(1+|z|^4\right)^3}\frac{\partial^2h_1^{(1)}}{\partial y^2}(\xi,\zeta)+O\left(\mu_1^4\mu_2^2\right)
\\ \notag&=\frac{-48\pi\left(\sigma _1^4-6 \sigma_2^2\sigma_1^2+\sigma_2^4\right)}{|\sigma|^8}\mu_1^2\mu_2^2+4\pi\mu_1^2\mu_2^2\left[\frac{\partial^2h_1^{(1)}}{\partial x^2}(\xi,\zeta)-\frac{\partial^2h_1^{(1)}}{\partial y^2}(\xi,\zeta)\right]+O\left(\mu_1^4\mu_2^2\right).
\end{align}
Hence, combining (\ref{Au7-1}) and (\ref{37-interaction}), and the fact that $\frac{\partial^2 h_1^{(1)}}{\partial x^2}(\xi,\zeta)=-\frac{\partial^2 h_1^{(1)}}{\partial y^2}(\xi,\zeta)$, we find
\begin{equation}\label{q11-interaction}
\begin{aligned}
2\int_{\mathcal{D}}(\delta_2-\varphi_2)_1 \left[(\delta_1)_x\wedge (\delta_1)_y\right]_1=\frac{-48\pi \left(\sigma _1^4-6 \sigma _2^2 \sigma _1^2+\sigma _2^4\right)}{|\sigma|^8}\mu_1^2\mu_2^2+8\pi\mu_1^2\mu_2^2\frac{\partial^2h_1^{(1)}}{\partial x^2}(\xi,\zeta) +O\left(\mu_1^4\mu_2^2\right).
\end{aligned}
\end{equation}

Next, we present the details of the terms involving $q_{22}$ , i.e., we need to compute $2\int_{\mathcal{D}}(\delta_2-\varphi_2)_2\left[(\delta_1)_x\wedge (\delta_1)_y\right]_2$. Similarly, in $\mathcal{D}_1$, it can be estimated as $O\left(\mu_1^6\mu_2^2|\log\mu_1|+\mu_1^6\mu_2^2d^{-2}\right)$.
And in $\mathcal{D}_2$, using (\ref{Ju25-6}), we have
\begin{align*}
&2\int_{\mathcal{D}_2}(\delta_2-\varphi_2)_2 \left[(\delta_1)_x\wedge (\delta_1)_y\right]_2
\\&=-512\int_{\frac{\mathcal{D}_2-\xi}{\mu_1}}\left[\mu_2^2\frac{(\mu_1x+\xi_1-\zeta_1)(\mu_1y+\xi_2-\zeta_2)}{|\mu_1z+\xi-\zeta|^4}+O(\mu_2^6)\right]
\frac{xy|z|^2}{\left(1+|z|^4\right)^3}
\\&\quad+256\int_{\frac{\mathcal{D}_2-\xi}{\mu_1}}\frac{\left[\mu_2^2h_2^{(1)}(\mu_1z+\xi,\zeta)+O(\mu_2^5)\right]xy|z|^2}{\left(1+|z|^4\right)^3}
\\&=-512\mu_1^2\mu_2^2 \int_{\frac{\mathcal{D}_2-\xi}{\mu_1}}\frac{-3 \left(\sigma _1^4-6 \sigma _2^2 \sigma _1^2+\sigma _2^4\right)}{|\sigma|^8}\frac{x^2y^2|z|^2}{\left(1+|z|^4\right)^3}
+256\mu_1^2\mu_2^2\int_{\frac{\mathcal{D}_2-\xi}{\mu_1}}\frac{x^2y^2|z|^2}{\left(1+|z|^4\right)^3}\frac{\partial^2 h_2^{(1)}}{\partial x\partial y}(\xi,\zeta)+O\left(\mu_1^4\mu_2^2\right)
\\&=\frac{48\pi\left(\sigma _1^4-6 \sigma _2^2 \sigma _1^2+\sigma _2^4\right)}{|\sigma|^8}\mu_1^2\mu_2^2+8\pi \mu_1^2\mu_2^2 \frac{\partial^2 h_2^{(1)}}{\partial x\partial y}(\xi,\zeta)+O\left(\mu_1^4\mu_2^2\right).
\end{align*}
 Thus,
\begin{equation}
\begin{aligned}
2\int_{\mathcal{D}}(\delta_2-\varphi_2)_2 \left[(\delta_1)_x\wedge (\delta_1)_y\right]_2=&\frac{48\pi\left(\sigma _1^4-6 \sigma _2^2 \sigma _1^2+\sigma _2^4\right)}{|\sigma|^8}\mu_1^2\mu_2^2
+8\pi\mu_1^2\mu_2^2\frac{\partial^2 h_2^{(1)}}{\partial x\partial y}(\xi,\zeta) +O\left(\mu_1^4\mu_2^2\right).
\end{aligned}
\end{equation}
The terms involving the other coefficients of the matrix $\mathcal{Q}_2$ can be handled similarly. We only present the results of these estimations without delving into the detailed calculations. It follows that
\begin{equation}
\begin{aligned}
q_{12}:&\ 2\int_{\mathcal{D}}(\delta_2-\varphi_2)_2 \left[(\delta_1)_x\wedge (\delta_1)_y\right]_1=\frac{-192\pi\sigma_1\sigma_2(\sigma_1^2-\sigma_2^2)\mu_1^2\mu_2^2}{|\sigma|^8}+8\pi\mu_1^2\mu_2^2\frac{\partial^2h_2^{(1)}}{\partial x^2}(\xi,\zeta)+O\left(\mu_1^4\mu_2^2\right),\\
q_{13}:&\ 2\int_{\mathcal{D}}(\delta_2-\varphi_2)_3 \left[(\delta_1)_x\wedge (\delta_1)_y\right]_1=O\left(\mu_1^4\mu_2^2\right),\\
q_{21}:&\ 2\int_{\mathcal{D}}(\delta_2-\varphi_2)_1 \left[(\delta_1)_x\wedge (\delta_1)_y\right]_2=\frac{-192\pi\sigma_1\sigma_2(\sigma_1^2-\sigma_2^2)\mu_1^2\mu_2^2}{|\sigma|^8}
+8\pi\mu_1^2\mu_2^2\frac{\partial^2h_1^{(1)}}{\partial x \partial y}(\xi,\zeta)+O\left(\mu_1^4\mu_2^2\right),\\
q_{23}:&\ 2\int_{\mathcal{D}}(\delta_2-\varphi_2)_3 \left[(\delta_1)_x\wedge (\delta_1)_y\right]_2=O\left(\mu_1^4\mu_2^2\right),\\
q_{31}:&\ 2\int_{\mathcal{D}}(\delta_2-\varphi_2)_1 \left[(\delta_1)_x\wedge (\delta_1)_y\right]_3=O\left(\mu_1^4\mu_2^2|\log\mu_2|+\mu_1^4\mu_2^2d^{-2}\right),\\
q_{32}:&\ 2\int_{\mathcal{D}}(\delta_2-\varphi_2)_2 \left[(\delta_1)_x\wedge (\delta_1)_y\right]_3=O\left(\mu_1^4\mu_2^2|\log\mu_2|+\mu_1^4\mu_2^2d^{-2}\right),\\
q_{33}:&\ 2\int_{\mathcal{D}}(\delta_2-\varphi_2)_3 \left[(\delta_1)_x\wedge (\delta_1)_y\right]_3=O\left(\mu_1^4\mu_2^2+\mu_1^4\mu_2^4d^{-4}\right).
\end{aligned}
\end{equation}
Similarly, one has
\begin{equation}
\begin{aligned}
&2\varepsilon \int_{\mathcal{D}} \left(\mathcal Q_2g^{(2)}\right)\cdot\left[\left(\delta_1\right)_x\wedge\left(\delta_1\right)_y\right]
\\&=-4\pi \varepsilon\mu_1^2\left[q_{11}\frac{\partial^2 g^{(2)}_1}{\partial x^2}(\xi,\omega^{(2)})+q_{12}\frac{\partial^2 g^{(2)}_2}{\partial x^2}(\xi,\omega^{(2)})+q_{13}\frac{\partial^2 g^{(2)}_3}{\partial x^2}(\xi,\omega^{(2)})\right.
\\&\qquad\qquad\qquad\left.
+q_{21}\frac{\partial^2 g^{(2)}_1}{\partial x \partial y}(\xi,\omega^{(2)})+q_{22}\frac{\partial^2 g^{(2)}_2}{\partial x \partial y}(\xi,\omega^{(2)})+q_{23}\frac{\partial^2 g^{(2)}_3}{\partial x \partial y}(\xi,\omega^{(2)})\right]+O\left(\varepsilon\mu_1^4\right),
\end{aligned}
\end{equation}
and
\begin{equation}
\begin{aligned}
&2\varepsilon\int_{\mathcal{D}}  g^{(1)} \cdot\left[\left(\mathcal Q_2\delta_2\right)_x\wedge \left(\mathcal Q_2\delta_2\right)_y\right]
\\&=-4\pi \varepsilon\mu_2^2\left[q_{11}\frac{\partial^2 g^{(1)}_1}{\partial x^2}(\zeta,\omega^{(1)})+q_{12}\frac{\partial^2 g^{(1)}_2}{\partial x\pp y}(\zeta,\omega^{(1)})
+q_{21}\frac{\partial^2 g^{(1)}_1}{\partial x^2}(\zeta,\omega^{(1)})\right.
\\&\qquad\qquad\qquad\left.
+q_{22}\frac{\partial^2 g^{(1)}_2}{\partial x \partial y}(\zeta,\omega^{(1)})
+q_{31}\frac{\partial^2 g^{(1)}_3}{\partial x^2}(\zeta,\omega^{(1)})+q_{32}\frac{\partial^2 g^{(1)}_3}{\partial x \partial y}(\zeta,\omega^{(1)})\right]+O\left(\varepsilon\mu_2^4\right).
\end{aligned}
\end{equation}

Now, we consider
\begin{equation}\label{J10-g1}
\begin{aligned}
-2\int_{\mathcal{D}} \left(P\delta_1+\varepsilon g^{(1)}\right)\cdot \left[(\mathcal{Q}_2\delta_2)_x\wedge \left(\mathcal{Q}_2\varphi_1(z,\zeta)-\varepsilon\mathcal Q_2g^{(2)}\right)_y+\left(\mathcal{Q}_2\varphi_1(z,\zeta)-\varepsilon\mathcal Q_2g^{(2)}\right)_x\wedge (\mathcal{Q}_2\delta_2)_y\right]
.
\end{aligned}
\end{equation}
Note that in $\mathcal{D} \backslash B_{\gamma}(\zeta)$, we have
\begin{equation}\label{F26-1}
\begin{aligned}
\varphi_1(z,\xi)-\varepsilon g^{(1)}(z,\omega^{(1)})=O\left(\left(\mu_1^2-\varepsilon\right)d^{-2}+ \varepsilon d^{-3}|\xi-\omega^{(1)}|\right)=O\left(\left(\mu_1^2-\varepsilon\right)d^{-2}\right),
\end{aligned}
\end{equation}
where the second equality follows from the asymptotic relation $\varepsilon d^{-1}|\xi-\omega^{(1)}|=O\left(\mu_1^2-\varepsilon\right)$. And in $B_{\gamma}(\zeta)$, one has
\begin{equation}\label{F26-2}
\begin{aligned}
\left|\nabla \left(\mathcal{Q}_2\varphi_1(z,\zeta)-\varepsilon\mathcal Q_2g^{(2)}(z,\omega^{(2)})\right)\right|= O\left(\left(\mu_2^2-\varepsilon\right)d^{-3}+ \varepsilon d^{-4}|\zeta-\omega^{(2)}|\right)=O\left(\left(\mu_2^2-\varepsilon\right)d^{-3}\right).
\end{aligned}
\end{equation}
In the following, we will frequently refer to (\ref{F26-1}) and (\ref{F26-2}) to simplify our estimates without further explanation.
It is then straightforward to verify that
\begin{equation}
\begin{aligned}
&\int_{\mathcal{D} \backslash B_{\gamma}(\zeta)}
\left(P\delta_1+\varepsilon g^{(1)}\right)\cdot \left[(\mathcal{Q}_2\delta_2)_x\wedge \left(\mathcal{Q}_2\varphi_1(z,\zeta)-\varepsilon\mathcal Q_2g^{(2)}\right)_y+\left(\mathcal{Q}_2\varphi_1(z,\zeta)-\varepsilon\mathcal Q_2g^{(2)}\right)_x\wedge (\mathcal{Q}_2\delta_2)_y\right]
\\&=O\left(\left(\mu_1^2|\log \mu_1|+\left(\mu_1^2-\varepsilon\right)d^{-2}\right)\mu_2^2\left(\mu_2^2-\varepsilon\right)\right).
\end{aligned}
\end{equation}
Using (\ref{derivative}), we have
\begin{equation}
\begin{aligned}
&\int_{B_{\gamma}(\zeta)}
\left(P\delta_1+\varepsilon g^{(1)}\right)\cdot \left[(\mathcal{Q}_2\delta_2)_x\wedge \left(\mathcal{Q}_2\varphi_1(z,\zeta)-\varepsilon\mathcal Q_2g^{(2)}\right)_y+\left(\mathcal{Q}_2\varphi_1(z,\zeta)-\varepsilon\mathcal Q_2g^{(2)}\right)_x\wedge (\mathcal{Q}_2\delta_2)_y\right]
\\&=O\left(\left(\mu_1^2+\varepsilon\right)\mu_2^2|\log \mu_2|\left(\left(\mu_2^2-\varepsilon\right)d^{-3}+\varepsilon d^{-4}|\zeta-\omega^{(2)}|\right)\right)
\\&=O\left(\left(\mu_1^2+\varepsilon\right)\mu_2^2|\log \mu_2|\left(\mu_2^2-\varepsilon\right)d^{-3}\right).
\end{aligned}
\end{equation}
Moreover,
\begin{equation}
\begin{aligned}
&2\int_{\mathcal{D}} P\delta_1\cdot \left[(\mathcal{Q}_2\varphi_1(z,\zeta))_x\wedge  \left(\mathcal{Q}_2\varphi_1(z,\zeta)-\varepsilon\mathcal Q_2g^{(2)}\right)_y-\left(\varepsilon\mathcal Q_2g^{(2)}\right)_x\wedge \left(\mathcal{Q}_2\varphi_1(z,\zeta)-\varepsilon \mathcal Q_2g^{(2)}\right)_y\right]
\\&=O\left(\left(\mu_1^2|\log \mu_1|+\mu_1^2d^{-2}\right)\mu_2^2\left(\mu_2^2-\varepsilon\right)+\mu_1^2\mu_2^2\left(\mu_2^2-\varepsilon\right)d^{-6}\right),
\end{aligned}
\end{equation}

\begin{equation}
\begin{aligned}
&2\int_{\mathcal{D}} \varepsilon g^{(1)}\cdot \left[(\mathcal{Q}_2\varphi_1(z,\zeta))_x\wedge  \left(\mathcal{Q}_2\varphi_1(z,\zeta)-\varepsilon\mathcal Q_2g^{(2)}\right)_y-\left(\varepsilon\mathcal Q_2g^{(2)}\right)_x\wedge \left(\mathcal{Q}_2\varphi_1(z,\zeta)\right)_y\right]
\\&=O\left(\varepsilon\mu_2^2\left(\mu_2^2-\varepsilon\right)+\varepsilon^2\mu_2^2+\varepsilon\mu_2^2\left(\mu_2^2-\varepsilon\right)d^{-6}+\varepsilon^2\mu_2^2d^{-3}\right).
\end{aligned}
\end{equation}
And similarly, one has
\begin{align*}
&-2\int_{\mathcal{D}} \left(P\mathcal Q_2\delta_2+\varepsilon \mathcal Q_2g^{(2)}\right)\cdot \left[(\delta_1)_x\wedge \left(\varphi_1(z,\xi)-\varepsilon g^{(1)}\right)_y+\left(\varphi_1(z,\xi)-\varepsilon g^{(1)}\right)_x\wedge (\delta_1)_y\right]
\\&+2\int_{\mathcal{D}} P\mathcal Q_2\delta_2\cdot \left[(\varphi_1(z,\xi))_x\wedge \left(\varphi_1(z,\xi)-\varepsilon g^{(1)}\right)_y-\left(\varepsilon g^{(1)}\right)_x\wedge \left(\varphi_1(z,\xi)-\varepsilon g^{(1)}\right)_y\right]
\\&+2\int_{\mathcal{D}} \varepsilon \mathcal Q_2g^{(2)}\cdot \left[(\varphi_1(z,\xi))_x\wedge \left(\varphi_1(z,\xi)-\varepsilon g^{(1)}\right)_y-\left(\varepsilon g^{(1)}\right)_x\wedge \left(\varphi_1(z,\xi)\right)_y\right]
\\&=O\left(\left(\mu_2^2|\log \mu_2|+\left(\mu_2^2-\varepsilon\right)d^{-2}\right)\mu_1^2\left(\mu_1^2-\varepsilon\right)\right)
+O\left(\left(\mu_2^2+\varepsilon\right)\mu_1^2|\log \mu_1|\left(\mu_1^2-\varepsilon\right)d^{-3}\right)
\\&\quad+O\left(\left(\mu_2^2|\log \mu_2|+\mu_2^2d^{-2}\right)\mu_1^2\left(\mu_1^2-\varepsilon\right)+\mu_2^2\mu_1^2\left(\mu_1^2-\varepsilon\right)d^{-6}\right)
\\&\quad+O\left(\varepsilon\mu_1^2\left(\mu_1^2-\varepsilon\right)+\varepsilon^2\mu_1^2+\varepsilon\mu_1^2\left(\mu_1^2-\varepsilon\right)d^{-6}+\varepsilon^2\mu_1^2d^{-3}\right).
\end{align*}

Finally, we use the connection between the function $h^{(1)}$ and the boundary function $g$, i.e.
$$g^{(2)}(z,\omega^{(2)})=\left(2h_1^{(1)}(z,\omega^{(2)}),2h_2^{(1)}(z,\omega^{(2)}),-2\varepsilon h_3^{(1)}(z,\omega^{(2)})\right),$$
\begin{equation}\label{F26-3}
\begin{aligned}
\left|h^{(1)}(\xi,\zeta)-h^{(1)}\left(\xi,\omega^{(2)}\right)\right|=O\left(|\zeta-\omega^{(2)}|\right)
\end{aligned}
\end{equation}
to obtain
\begin{equation}\label{Jan19-1}
\begin{aligned}
&8\pi\mu_1^2\mu_2^2\left[q_{11}\frac{\partial^2h_1^{(1)}}{\partial x^2}(\xi,\zeta)+q_{12}\frac{\partial^2h_2^{(1)}}{\partial x^2}(\xi,\zeta)
+q_{21}\frac{\partial^2h_1^{(1)}}{\partial x \partial y}(\xi,\zeta)+q_{22}\frac{\partial^2 h_2^{(1)}}{\partial x\partial y}(\xi,\zeta)  \right]
\\&-4\pi \varepsilon\mu_1^2\left[q_{11}\frac{\partial^2 g^{(2)}_1}{\partial x^2}(\xi,\omega^{(2)})+q_{12}\frac{\partial^2 g^{(2)}_2}{\partial x^2}(\xi,\omega^{(2)})+q_{13}\frac{\partial^2 g^{(2)}_3}{\partial x^2}(\xi,\omega^{(2)})\right.
\\&\qquad\qquad\quad\left.
+q_{21}\frac{\partial^2 g^{(2)}_1}{\partial x \partial y}(\xi,\omega^{(2)})+q_{22}\frac{\partial^2 g^{(2)}_2}{\partial x \partial y}(\xi,\omega^{(2)})+q_{23}\frac{\partial^2 g^{(2)}_3}{\partial x \partial y}(\xi,\omega^{(2)})\right]
\\&=O\left(\mu_1^2(\mu_2^2-\varepsilon)+\varepsilon\mu_1^2|\zeta-\omega^{(2)}|+\varepsilon^2\mu_1^2\right)
\\&=O\left(\mu_1^2(\mu_2^2-\varepsilon)+\varepsilon^2\mu_1^2\right).
\end{aligned}
\end{equation}
Thus, combining the fact that
$\mu_i=O\left(\sqrt{\varepsilon}\right),~i=1,2$,
Lemma \ref{lemma 6.3} follows from (\ref{q11-interaction})-(\ref{Jan19-1}).
\end{proof}

\subsection{Interaction with the linear and higher order expansions}\label{subsection5.2}
Referring to the energy expansion of one bubble in Section \ref{expansion-for-one-bubble}, along with the results from Lemma \ref{lemma 6.3}, it is natural to consider the mixed terms of $\delta_1,\ \delta_2,\ Z_{-1,l},\ Z_{2,l}$ for $l=1, 2$ in (\ref{22-2-interaction}). For simplicity, we express these terms as $$\mathcal{M}_{\delta_1,\delta_2,Z_{-1,l},Z_{2,l}}:=\mathcal{M}^1_{\delta_1,\delta_2,Z_{-1,l},Z_{2,l}}+\mathcal{M}^2_{\delta_1,\delta_2,Z_{-1,l},Z_{2,l}}+\mathcal{M}^3_{\delta_1,\delta_2,Z_{-1,l},Z_{2,l}}+\mathcal{M}^4_{\delta_1,\delta_2,Z_{-1,l},Z_{2,l}}+\mathcal{M}^5_{\delta_1,\delta_2,Z_{-1,l},Z_{2,l}},$$
and by applying integration by parts, each $\mathcal{M}^i_{\delta_1,\delta_2,Z_{-1,l},Z_{2,l}},\ i=1, \dots,5$, takes the following form
\begin{align}
\notag
&\mathcal{M}^1_{\delta_1,\delta_2,Z_{-1,l},Z_{2,l}}
\\ \notag &:=\int_{\mathcal{D}}\nabla P\mathcal Q_2\delta_2\cdot\nabla P\mathcal{L}_{\mathcal{A}_1}
+2\int_{\mathcal{D}} P\mathcal Q_2\delta_2\cdot \left[\left(P\delta_1\right)_x
\wedge \left(P\mathcal{L}_{\mathcal{A}_1}\right)_y+\left(P\mathcal{L}_{\mathcal{A}_1}\right)_x
\wedge \left(P\delta_1\right)_y\right]
\\ \notag &~\quad+2\int_{\mathcal{D}} P\mathcal{L}_{\mathcal{A}_1}
\cdot\left[ \left(P\mathcal Q_2\delta_2\right)_x
\wedge \left(P\mathcal Q_2\delta_2\right)_y\right]
\\ \notag&~\quad+2\varepsilon\int_{\mathcal{D}} \mathcal Q_2g^{(2)}\cdot \left[\left(P\delta_1\right)_x
\wedge\left(P\mathcal{L}_{\mathcal{A}_1}\right)_y+\left(P\mathcal{L}_{\mathcal{A}_1}\right)_x
\wedge \left(P\delta_1\right)_y\right]
\\ \notag&~\quad+2\varepsilon\int_{\mathcal{D}} P\mathcal{L}_{\mathcal{A}_1}\cdot\left[\left(P\mathcal Q_2\delta_2\right)_x\wedge \left(\mathcal Q_2g^{(2)}\right)_y+\left(\mathcal Q_2g^{(2)}\right)_x\wedge \left(P\mathcal Q_2\delta_2\right)_y \right]
\\ \notag&~\quad+2\varepsilon\int_{\mathcal{D}} P\mathcal Q_2\delta_2\cdot\left[\left(g^{(1)}\right)_x\wedge \left(P\mathcal{L}_{\mathcal{A}_1}\right)_y+\left(P\mathcal{L}_{\mathcal{A}_1}\right)_x\wedge \left(g^{(1)}\right)_y\right]
\\ \notag&~\quad+2\varepsilon^2\int_{\mathcal{D}} \mathcal Q_2g^{(2)}\cdot\left[\left(g^{(1)}\right)_x\wedge \left(P\mathcal{L}_{\mathcal{A}_1}\right)_y+\left(P\mathcal{L}_{\mathcal{A}_1}\right)_x\wedge \left(g^{(1)}\right)_y \right]
\\ \notag&~\quad+2\varepsilon^2\int_{\mathcal{D}} P\mathcal{L}_{\mathcal{A}_1} \cdot\left[\left(\mathcal Q_2g^{(2)}\right)_x\wedge \left(\mathcal Q_2g^{(2)}\right)_y\right]
\\ \notag&~=-2\int_{\mathcal{D}}P\mathcal{L}_{\mathcal{A}_1}\cdot\left[\left(\mathcal Q_2\delta_2\right)_x\wedge\left(\mathcal Q_2\varphi_1(z,\zeta)-\varepsilon \mathcal Q_2g^{(2)}\right)_y+\left(\mathcal Q_2\varphi_1(z,\zeta)-\varepsilon \mathcal Q_2g^{(2)}\right)_x\wedge\left(\mathcal Q_2\delta_2\right)_y\right]
\\ \notag&~\quad+2\int_{\mathcal{D}}P\mathcal{L}_{\mathcal{A}_1}\cdot\left[\left(\mathcal Q_2\varphi_1(z,\zeta)\right)_x\wedge \left(\mathcal Q_2\varphi_1(z,\zeta)-\varepsilon \mathcal Q_2g^{(2)}\right)_y-\left(-\varepsilon \mathcal Q_2g^{(2)}\right)_x\wedge\left(\mathcal Q_2\varphi_1(z,\zeta)-\varepsilon \mathcal Q_2g^{(2)}\right)_y\right]
\\ \label{1-remainder} &~\quad+2\int_{\mathcal{D}} \left(P\mathcal Q_2\delta_2+\varepsilon \mathcal Q_2g^{(2)} \right)\cdot \left[\left(P\delta_1+\varepsilon g^{(1)}\right)_x
\wedge \left(P\mathcal{L}_{\mathcal{A}_1}\right)_y+\left(P\mathcal{L}_{\mathcal{A}_1}\right)_x
\wedge \left(P\delta_1+\varepsilon g^{(1)}\right)_y\right],
\end{align}
and similarly
\begin{equation}\label{2-remainder}
\begin{aligned}
&\mathcal{M}^2_{\delta_1,\delta_2,Z_{-1,l},Z_{2,l}}
\\&:=-2\int_{\mathcal{D}}P\mathcal Q_2\mathcal{L}_{\mathcal{A}_2}\cdot\left[\left(\delta_1\right)_x\wedge\left(\varphi_1(z,\xi)-\varepsilon g^{(1)}\right)_y+\left(\varphi_1(z,\xi)-\varepsilon g^{(1)}\right)_x\wedge\left(\delta_1\right)_y\right]
\\&~\quad+2\int_{\mathcal{D}}P\mathcal Q_2\mathcal{L}_{\mathcal{A}_2}\cdot\left[\left(\varphi_1(z,\xi)\right)_x\wedge \left(\varphi_1(z,\xi)-\varepsilon g^{(1)}\right)_y-\left(-\varepsilon g^{(1)}\right)_x\wedge\left(\varphi_1(z,\xi)-\varepsilon g^{(1)}\right)_y\right]
\\&~\quad+2\int_{\mathcal{D}} \left(P\delta_1+\varepsilon g^{(1)}\right)\cdot \left[\left(P\mathcal Q_2\delta_2+\varepsilon \mathcal Q_2g^{(2)} \right)_x
\wedge \left(P\mathcal Q_2\mathcal{L}_{\mathcal{A}_2}\right)_y+\left(P\mathcal Q_2\mathcal{L}_{\mathcal{A}_2}\right)_x
\wedge\left(P\mathcal Q_2\delta_2+\varepsilon \mathcal Q_2g^{(2)} \right)_y\right].
\end{aligned}
\end{equation}
And
\begin{equation}\label{4-remainder}
\begin{aligned}
\mathcal{M}^3_{\delta_1,\delta_2,Z_{-1,l},Z_{2,l}}
:&=2\int_{\mathcal{D}} \left(P\mathcal Q_2\delta_2+\varepsilon \mathcal Q_2g^{(2)}\right)\cdot\left[\left(P\mathcal{L}_{\mathcal{A}_1}\right)_x\wedge\left(P\mathcal{L}_{\mathcal{A}_1}\right)_y\right]
\\&\quad+2\int_{\mathcal{D}} \left(P\delta_1+\varepsilon g^{(1)}\right)\cdot\left[\left(P\mathcal Q_2\mathcal{L}_{\mathcal{A}_2}\right)_x\wedge\left(P\mathcal Q_2\mathcal{L}_{\mathcal{A}_2}\right)_y\right],
\end{aligned}
\end{equation}

\begin{align}
\notag
\mathcal{M}^4_{\delta_1,\delta_2,Z_{-1,l},Z_{2,l}}
:=&\int_{\mathcal{D}}\nabla P\mathcal{L}_{\mathcal{A}_1}
\cdot \nabla P\mathcal Q_2\mathcal{L}_{\mathcal{A}_2}
+2\int_{\mathcal{D}} P\mathcal Q_2\mathcal{L}_{\mathcal{A}_2}
\cdot \left[\left(P\delta_1\right)_x
\wedge \left(P\mathcal{L}_{\mathcal{A}_1}\right)_y+\left(P\mathcal{L}_{\mathcal{A}_1}\right)_x
\wedge \left(P\delta_1\right)_y
\right]
\\\notag&+2\int_{\mathcal{D}} P\mathcal{L}_{\mathcal{A}_1}
\cdot \left[\left(P\mathcal{Q}_2\delta_2\right)_x
\wedge \left(P\mathcal Q_2\mathcal{L}_{\mathcal{A}_2}\right)_y+\left(P\mathcal Q_2\mathcal{L}_{\mathcal{A}_2}\right)_x
\wedge \left(P\mathcal{Q}_2\delta_2\right)_y
\right]
\\\notag&+2\varepsilon\int_{\mathcal{D}} P\mathcal Q_2\mathcal{L}_{\mathcal{A}_2}
\cdot \left[\left(g^{(1)}\right)_x
\wedge \left(P\mathcal{L}_{\mathcal{A}_1}\right)_y+\left(P\mathcal{L}_{\mathcal{A}_1}\right)_x
\wedge \left(g^{(1)}\right)_y
\right]
\\\notag&+2\varepsilon\int_{\mathcal{D}} P\mathcal{L}_{\mathcal{A}_1}
\cdot \left[\left(\mathcal Q_2g^{(2)}\right)_x
\wedge \left(P\mathcal Q_2\mathcal{L}_{\mathcal{A}_2}\right)_y+\left(P\mathcal Q_2\mathcal{L}_{\mathcal{A}_2}\right)_x
\wedge \left(\mathcal Q_2g^{(2)}\right)_y
\right]
\\\notag=&-2\int_{\mathcal{D}} P\mathcal Q_2\mathcal{L}_{\mathcal{A}_2}
\cdot \left[\left(\delta_1\right)_x
\wedge \left(\sum_{l=1}^2\left(a_{1,l}\varphi_{-1,l}\left(z,\xi\right)+p_{1, l}\varphi_{2,l}\left(z,\xi\right)\right)\right)_y\right.
\\\notag&\qquad\qquad\qquad\qquad\quad\left.
+\left(\sum_{l=1}^2\left(a_{1,l}\varphi_{-1,l}\left(z,\xi\right)+p_{1, l}\varphi_{2,l}\left(z,\xi\right)\right)\right)_x
\wedge \left(\delta_1\right)_y\right]
\\\notag&-2\int_{\mathcal{D}} P\mathcal Q_2\mathcal{L}_{\mathcal{A}_2}
\cdot \left[\left(\varphi_1-\varepsilon g^{(1)}\right)_x
\wedge \left(P\mathcal{L}_{\mathcal{A}_1}\right)_y+\left(P\mathcal{L}_{\mathcal{A}_1}\right)_x
\wedge \left(\varphi_1-\varepsilon g^{(1)}\right)_y
\right]
\\\label{5-remainder}&+2\int_{\mathcal{D}} P\mathcal{L}_{\mathcal{A}_1}
\cdot \left[\left(P\mathcal{Q}_2\delta_2+\varepsilon \mathcal Q_2g^{(2)}\right)_x
\wedge \left(P\mathcal Q_2\mathcal{L}_{\mathcal{A}_2}\right)_y+\left(P\mathcal Q_2\mathcal{L}_{\mathcal{A}_2}\right)_x
\wedge \left(P\mathcal{Q}_2\delta_2+\varepsilon \mathcal Q_2g^{(2)}\right)_y
\right],
\end{align}

\begin{equation}\label{6-remainder}
\begin{aligned}
\mathcal{M}^5_{\delta_1,\delta_2,Z_{-1,l},Z_{2,l}}
:&=~
2\int_{\mathcal{D}}  P\mathcal{L}_{\mathcal{A}_1}\cdot \left[\left(P\mathcal Q_2\mathcal{L}_{\mathcal{A}_2}\right)_x\wedge \left(P\mathcal Q_2\mathcal{L}_{\mathcal{A}_2}\right)_y\right]
+2\int_{\mathcal{D}}  P\mathcal Q_2\mathcal{L}_{\mathcal{A}_2}\cdot \left[\left(P\mathcal{L}_{\mathcal{A}_1}\right)_x\wedge \left(P\mathcal{L}_{\mathcal{A}_1}\right)_y\right].
\end{aligned}
\end{equation}

\begin{lemma}\label{lemma-5.2} There holds
\begin{equation*}
\begin{aligned}
&
\mathcal{M}^1_{\delta_1,\delta_2,Z_{-1,l},Z_{2,l}}=F^{(1)}_{\mathcal{D}}(\delta_1,\delta_2,Id,\mathcal{Q}_2,a_{1,l},p_{1,l})+\mathcal{E}^{(1)}\left(\varepsilon,a_{1,l},p_{1,l}\right),
\\&
\mathcal{M}^2_{\delta_1,\delta_2,Z_{-1,l},Z_{2,l}}=F^{(2)}_{\mathcal{D}}(\delta_1,\delta_2,Id,\mathcal{Q}_2,a_{2,l},p_{2,l})+\mathcal{E}^{(2)}\left(\varepsilon,a_{2,l},p_{2,l}\right),
\\&
\mathcal{M}^3_{\delta_1,\delta_2,Z_{-1,l},Z_{2,l}}=
\mathcal{E}^{(3)}\left(\varepsilon,a_{1,l},a_{2,l},p_{1,l},p_{2,l}\right),
\\&
\mathcal{M}^4_{\delta_1,\delta_2,Z_{-1,l},Z_{2,l}}=\mathcal{E}^{(4)}\left(\varepsilon,a_{1,l},a_{2,l},p_{1,l},p_{2,l}\right),
\\&
\mathcal{M}^5_{\delta_1,\delta_2,Z_{-1,l},Z_{2,l}}=\mathcal{E}^{(5)}\left(\varepsilon,a_{1,l},a_{2,l},p_{1,l},p_{2,l}\right),
\end{aligned}
\end{equation*}
where
\begin{equation}\label{J7-R1}
\begin{aligned}
F^{(1)}_{\mathcal{D}}(\delta_1,\delta_2,Id,\mathcal{Q}_2,a_{1,l},p_{1,l})
:&=a_{1,1}\mu_1\mu_2^2\left\{q_{11}\frac{-32\pi \sigma _1 \left(\sigma _1^2-3 \sigma _2^2\right)}{|\sigma|^6}
+q_{12}\frac{-32\pi \sigma _2 \left(3\sigma _1^2- \sigma _2^2\right)}{|\sigma|^6}\right.
\\&\qquad\qquad\qquad\left.+ q_{21}\frac{-32\pi \sigma _2 \left(3\sigma _1^2- \sigma _2^2\right)}{|\sigma|^6}
+q_{22}\frac{32\pi \sigma _1 \left(\sigma _1^2- 3\sigma _2^2\right)}{|\sigma|^6}\right\}
\\&\quad+a_{1,2}\mu_1\mu_2^2\left\{q_{11}\frac{32\pi \sigma _2 \left(3\sigma _1^2-\sigma _2^2\right)}{|\sigma|^6}+q_{12}\frac{-32\pi \sigma _1 \left(\sigma _1^2- 3\sigma _2^2\right)}{|\sigma|^6}\right.
\\&\qquad\qquad\qquad\quad\left.+q_{21}\frac{-32\pi \sigma _1 \left(\sigma _1^2-3 \sigma _2^2\right)}{|\sigma|^6}+q_{22}\frac{-32\pi \sigma _2 \left(3\sigma _1^2- \sigma _2^2\right)}{|\sigma|^6}\right\},
\end{aligned}
\end{equation}
\begin{equation}\label{J7-R2}
\begin{aligned}
\mathcal{E}^{(1)}\left(\varepsilon,a_{1,l},p_{1,l}\right)
:&=O\left((a_{1,1}+a_{1,2})\left(\varepsilon^{\frac{3}{2}}\left(\mu_2^2-\varepsilon\right)d^{-6}+\varepsilon^{\frac{5}{2}}|\log\varepsilon|d^{-2}+\varepsilon^{\frac{3}{2}}\left(\mu_1^2-\varepsilon\right)d^{-5}\right)\right)
\\&\quad+O\left((p_{1,1}+p_{1,2})\left(\varepsilon^{3}+\varepsilon^{3}\left(\mu_2^2-\varepsilon\right)d^{-6}+\varepsilon^{4}|\log\varepsilon|d^{-5}+\varepsilon^{3}\left(\mu_1^2-\varepsilon\right)d^{-8}\right)\right),
\end{aligned}
\end{equation}

\begin{equation}
\begin{aligned}
&F^{(2)}_{\mathcal{D}}(\delta_1,\delta_2,Id,\mathcal{Q}_2,a_{2,l},p_{2,l})
:=a_{2,1}\mu_1^2\mu_2\left\{q_{11}\frac{32\pi \sigma _1 \left(\sigma _1^2-3 \sigma _2^2\right)}{|\sigma|^6}
+q_{12}\frac{32\pi \sigma _2 \left(3\sigma _1^2- \sigma _2^2\right)}{|\sigma|^6}\right.
\\&\qquad\qquad\qquad\qquad\qquad\qquad\qquad\qquad\qquad\left.+ q_{21}\frac{32\pi \sigma _2 \left(3\sigma _1^2- \sigma _2^2\right)}{|\sigma|^6}
+q_{22}\frac{-32\pi \sigma _1 \left(\sigma _1^2- 3\sigma _2^2\right)}{|\sigma|^6}\right\}
\\&\qquad\qquad\qquad\qquad\qquad\qquad\quad+a_{2,2}\mu_1^2\mu_2\left\{q_{11}\frac{-32\pi \sigma _2 \left(3\sigma _1^2-\sigma _2^2\right)}{|\sigma|^6}+q_{12}\frac{32\pi \sigma _1 \left(\sigma _1^2- 3\sigma _2^2\right)}{|\sigma|^6}\right.
\\&\qquad\qquad\qquad\qquad\qquad\qquad\qquad\qquad\qquad\quad\left.+ q_{21}\frac{32\pi \sigma _1 \left(\sigma _1^2-3 \sigma _2^2\right)}{|\sigma|^6}+q_{22}\frac{32\pi \sigma _2 \left(3\sigma _1^2- \sigma _2^2\right)}{|\sigma|^6}\right\},
\\&\mathcal{E}^{(2)}\left(\varepsilon,a_{2,l},p_{2,l}\right)
:=O\left((a_{2,1}+a_{2,2})\left(\varepsilon^{\frac{3}{2}}\left(\mu_1^2-\varepsilon\right)d^{-6}+\varepsilon^{\frac{5}{2}}|\log\varepsilon|d^{-2}+\varepsilon^{\frac{3}{2}}\left(\mu_2^2-\varepsilon\right)d^{-5}\right)\right)
\\&\qquad\qquad\qquad\qquad+O\left((p_{2,1}+p_{2,2})\left(\varepsilon^{3}+\varepsilon^{3}\left(\mu_1^2-\varepsilon\right)d^{-6}+\varepsilon^{4}|\log\varepsilon|d^{-5}+\varepsilon^{3}\left(\mu_2^2-\varepsilon\right)d^{-8}\right)\right),
\end{aligned}
\end{equation}

\begin{equation}\label{J10-R3}
\begin{aligned}
&\mathcal{E}^{(3)}\left(\varepsilon,a_{1,l},a_{2,l},p_{1,l},p_{2,l}\right)
\\&:=O\left(\left(\left(a_{1,1}+a_{1,2}\right)^2+\left(a_{2,1}+a_{2,2}\right)^2\right)\varepsilon^2d^{-4}\right)
+O\left(\left(\left(p_{1,1}+p_{1,2}\right)^2+\left(p_{2,1}+p_{2,2}\right)^2\right)\left(\varepsilon^2+\varepsilon^5d^{-10}\right)\right)
\\&\quad~+O\left(\left(a_{1,1}+a_{1,2}\right)\left(p_{1,1}+p_{1,2}\right)\left(\varepsilon^{\frac{3}{2}}+\varepsilon^{\frac{7}{2}}d^{-7}\right)\right)
+O\left(\left(a_{2,1}+a_{2,2}\right)\left(p_{2,1}+p_{2,2}\right)\left(\varepsilon^{\frac{3}{2}}+\varepsilon^{\frac{7}{2}}d^{-7}\right)\right),
\end{aligned}
\end{equation}

\begin{equation}\label{J10-R4}
\begin{aligned}
&\mathcal{E}^{(4)}\left(\varepsilon,a_{1,l},a_{2,l},p_{1,l},p_{2,l}\right)
\\&:=O\left(\sum_{i,k,l=1}^2a_{1,k}a_{2,l}\left(\varepsilon+\varepsilon^2|\log \varepsilon|d^{-2}+\varepsilon \left(\mu_i^2-\varepsilon\right)d^{-5}\right)+p_{1,k}p_{2,l}\left(\varepsilon^{5}|\log \varepsilon|d^{-5}+\varepsilon^4\left(\mu_i^2-\varepsilon\right)d^{-8}\right)\right)
\\&~\quad+O\left(\sum_{i,k,l=1}^2\left(a_{1,k}p_{2,l}+p_{1,k}a_{2,l}\right)\left(\varepsilon^{\frac{7}{2}}|\log \varepsilon|d^{-5}+\varepsilon^{\frac{5}{2}} \left(\mu_i^2-\varepsilon\right)d^{-8}\right)\right),
\end{aligned}
\end{equation}

\begin{align}
\notag
&\mathcal{E}^{(5)}\left(\varepsilon,a_{1,l},a_{2,l},p_{1,l},p_{2,l}\right)
\\\notag&:=
O\left(\sum_{j,k,l=1}^2\left(a_{1,j}a_{2,k}a_{2,l}+a_{2,j}a_{1,k}a_{1, l}\right)\left(\varepsilon^{\frac{3}{2}}|\log \varepsilon|+\varepsilon^{\frac{3}{2}}d^{-4}\right)+\left(a_{1,j}a_{2,k}p_{2, l}+a_{2,j}a_{1,k}p_{1, l}\right)\left(\varepsilon+\varepsilon^{3}d^{-7}\right)\right)
\\\notag&~\quad+O\left(\sum_{j,k,l=1}^2\left(a_{1,j}p_{2,k}p_{2, l}+a_{2,j}p_{1,k}p_{1, l}\right)\left(\varepsilon^{\frac{3}{2}}+\varepsilon^{\frac{9}{2}}d^{-10}\right)+
\left(p_{1,j}a_{2,k}a_{2, l}+p_{2,j}a_{1,k}a_{1, l}\right)\left(\varepsilon^3|\log \varepsilon|+\varepsilon^{3}d^{-4}\right)\right)
\\ \label{Ju30-M5}&~\quad
+O\left(\sum_{j,k,l=1}^2\left(p_{1,j}a_{2,k}p_{2, l}++p_{2,j}a_{1,k}p_{1, l}\right)\left(\varepsilon^{\frac{5}{2}}+\varepsilon^{\frac{9}{2}}d^{-7}\right)+\left(p_{1,j}p_{2,k}p_{2, l}+p_{2,j}p_{1,k}p_{1, l}\right)\left(\varepsilon^{3}+\varepsilon^{6}d^{-10}\right)\right).
\end{align}
\end{lemma}
\begin{proof}
Firstly, we estimate $\mathcal{M}^1_{\delta_1,\delta_2,Z_{-1,l},Z_{2,l}}$, i.e. (\ref{1-remainder}). Similar to the calculation of (\ref{interaction-23}),
in $\mathcal{D}\backslash B_{\gamma}(\xi)$, since $z$ is away from $\xi$, for $\mu_1$ sufficiently small, the following expansions hold
\begin{equation}\label{Ju24-1}
\begin{aligned}
Z_{-1,1}\left(\frac{z-\xi}{\mu_1}\right)=\left(\frac{-4\mu_1(x-\xi_1)}{|z-\xi|^2}+O(\mu_1^5),~\frac{-4\mu_1(y-\xi_2)}{|z-\xi|^2}+O(\mu_1^5),~\frac{8\mu_1^3(x-\xi_1)}{|z-\xi|^4}+O(\mu_1^7)\right),
\end{aligned}
\end{equation}
\begin{equation}\label{Ju24-3}
\begin{aligned}
Z_{-1,2}\left(\frac{z-\xi}{\mu_1}\right)=\left(\frac{4\mu_1(y-\xi_2)}{|z-\xi|^2}+O(\mu_1^5),~\frac{-4\mu_1(x-\xi_1)}{|z-\xi|^2}+O(\mu_1^5),~\frac{8\mu_1^3(y-\xi_2)}{|z-\xi|^4}+O(\mu_1^7)\right),
\end{aligned}
\end{equation}

\begin{equation}\label{Ju24-4}
\begin{aligned}
Z_{2,1}\left(\frac{z-\xi}{\mu_1}\right)=
\begin{pmatrix}
\frac{-2\mu_1^4\left((x-\xi_1)^4-6(x-\xi_1)^2(y-\xi_2)^2+(y-\xi_2)^4\right)}{|z-\xi|^8}+O(\mu_1^8)\\
\frac{-8\mu_1^4(x-\xi_1)(y-\xi_2)\left((x-\xi_1)^2-(y-\xi_2)^2\right)}{|z-\xi|^8}+O(\mu_1^8)\\
\frac{4\mu_1^6\left((x-\xi_1)^2-(y-\xi_2)^2\right)}{|z-\xi|^8}+O(\mu_1^{10})
\end{pmatrix},
\end{aligned}
\end{equation}

\begin{equation}\label{Ju24-5}
\begin{aligned}
Z_{2,2}\left(\frac{z-\xi}{\mu_1}\right)=
\begin{pmatrix}
\frac{-8\mu_1^4(x-\xi_1)(y-\xi_2)\left((x-\xi_1)^2-(y-\xi_2)^2\right)}{|z-\xi|^8}+O(\mu_1^8)\\
\frac{2\mu_1^4\left((x-\xi_1)^4-6(x-\xi_1)^2(y-\xi_2)^2+(y-\xi_2)^4\right)}{|z-\xi|^8}+O(\mu_1^8)\\
\frac{8\mu_1^6(x-\xi_1)(y-\xi_2)}{|z-\xi|^8}+O(\mu_1^{10})
\end{pmatrix}.
\end{aligned}
\end{equation}
Using (\ref{F26-2}), (\ref{F26-3}) and (\ref{Ju24-1})-(\ref{Ju24-5}) , it is direct to check that
\begin{equation}\label{Jan17-1}
\begin{aligned}
&-2\int_{\mathcal{D}}P\mathcal{L}_{\mathcal{A}_1}\cdot\left[\left(\mathcal Q_2\delta_2\right)_x\wedge\left(\mathcal Q_2\varphi_1(z,\zeta)-\varepsilon \mathcal Q_2g^{(2)}\right)_y+\left(\mathcal Q_2\varphi_1(z,\zeta)-\varepsilon \mathcal Q_2g^{(2)}\right)_x\wedge\left(\mathcal Q_2\delta_2\right)_y\right]
\\&=O\left((a_{1,1}+a_{1,2})\left(\mu_1d^{-1}\mu_2^2\left(\mu_2^2-\varepsilon\right)+\mu_1\mu_2^2|\log\mu_2|\left(\mu_2^2-\varepsilon\right)d^{-3}\right)\right)
\\&\quad+O\left(
(p_{1,1}+p_{1,2})\left(\left(\mu_1^2+\mu_1^4d^{-4}\right)\mu_2^2\left(\mu_2^2-\varepsilon\right)+\mu_1^4\mu_2^2|\log\mu_2|\left(\mu_2^2-\varepsilon\right)d^{-3}\right)\right)
\\&=O\left((a_{1,1}+a_{1,2})\varepsilon^{\frac{3}{2}}|\log\varepsilon|\left(\mu_2^2-\varepsilon\right)d^{-3}+(p_{1,1}+p_{1,2})\varepsilon^3\left(\mu_2^2-\varepsilon\right)d^{-4}\right),
\end{aligned}
\end{equation}
and
\begin{equation}\label{Jan17-2}
\begin{aligned}
&2\int_{\mathcal{D}}P\mathcal{L}_{\mathcal{A}_1}\cdot\left[\left(\mathcal Q_2\varphi_1(z,\zeta)\right)_x\wedge \left(\mathcal Q_2\varphi_1(z,\zeta)-\varepsilon \mathcal Q_2g^{(2)}\right)_y-\left(-\varepsilon \mathcal Q_2g^{(2)}\right)_x\wedge\left(\mathcal Q_2\varphi_1(z,\zeta)-\varepsilon \mathcal Q_2g^{(2)}\right)_y\right]
\\&=O\left((a_{1,1}+a_{1,2})\left(\mu_1d^{-1}\mu_2^2\left(\mu_2^2-\varepsilon\right)+\mu_1\mu_2^2\left(\mu_2^2-\varepsilon\right)d^{-6})+\mu_1\varepsilon\left(\mu_2^2-\varepsilon\right)d^{-3}\right)\right)
\\&\quad+O\left((p_{1,1}+p_{1,2})\left(\left(\mu_1^2+\mu_1^4 d^{-4}\right) \mu_2^2\left(\mu_2^2-\varepsilon\right)
+\mu_1^4\mu_2^2\left(\mu_2^2-\varepsilon\right)d^{-6}+\mu_1^4\varepsilon\left(\mu_2^2-\varepsilon\right)d^{-3}\right)\right)
\\&=O\left((a_{1,1}+a_{1,2})\varepsilon^{\frac{3}{2}}\left(\mu_2^2-\varepsilon\right)d^{-6}+(p_{1,1}+p_{1,2})\varepsilon^3\left(\mu_2^2-\varepsilon\right)d^{-6}\right).
\end{aligned}
\end{equation}

Now, we compute the last term in (\ref{1-remainder}). There holds
\begin{equation*}
\begin{aligned}
&2\int_{\mathcal{D}} \left(P\mathcal Q_2\delta_2+\varepsilon \mathcal Q_2g^{(2)} \right)\cdot \left[\left(P\delta_1+\varepsilon g^{(1)}\right)_x
\wedge \left(P\mathcal{L}_{\mathcal{A}_1}\right)_y+\left(P\mathcal{L}_{\mathcal{A}_1}\right)_x
\wedge \left(P\delta_1+\varepsilon g^{(1)}\right)_y\right]
\\&=2\int_{\mathcal{D}}\left(P\mathcal Q_2\delta_2+\varepsilon \mathcal Q_2g^{(2)} \right)\cdot \left[\left(\delta_1\right)_x
\wedge \left(\mathcal{L}_{\mathcal{A}_1}\right)_y+\left(\mathcal{L}_{\mathcal{A}_1}\right)_x
\wedge \left(\delta_1\right)_y\right]
\\&\quad-2\int_{\mathcal{D}} \left(P\mathcal Q_2\delta_2+\varepsilon \mathcal Q_2g^{(2)} \right)
\cdot \left[\left(\delta_1\right)_x
\wedge \left(\sum_{l=1}^2\left(a_{1,l}\varphi_{-1,l}\left(z,\xi\right)+p_{1, l}\varphi_{2,l}\left(z,\xi\right)\right)\right)_y\right.
\\&\qquad\qquad\qquad\qquad\qquad\qquad\qquad\quad\left.
+\left(\sum_{l=1}^2\left(a_{1,l}\varphi_{-1,l}\left(z,\xi\right)+p_{1, l}\varphi_{2,l}\left(z,\xi\right)\right)\right)_x
\wedge \left(\delta_1\right)_y\right]
\\&\quad-2\int_{\mathcal{D}} \left(P\mathcal Q_2\delta_2+\varepsilon \mathcal Q_2g^{(2)} \right)
\cdot \left[\left(\varphi_1(z,\xi)-\varepsilon g^{(1)}\right)_x
\wedge \left(P\mathcal{L}_{\mathcal{A}_1}\right)_y
+\left(P\mathcal{L}_{\mathcal{A}_1}\right)_x
\wedge \left(\varphi_1(z,\xi)-\varepsilon g^{(1)}\right)_y\right].
\end{aligned}
\end{equation*}
Recall that
$$g^{(2)}(z,\omega^{(2)})=\left(2h_1^{(1)}(z,\omega^{(2)}),2h_2^{(1)}(z,\omega^{(2)}),-2\varepsilon h_3^{(1)}(z,\omega^{(2)})\right),$$
and since $\varepsilon d^{-1}|\zeta-\omega^{(2)}|=O\left(\mu_2^2-\varepsilon\right)$, in $B_{\gamma}(\zeta)$, we get
$$\left|\varphi_1\left(z,\zeta\right)-\varepsilon g^{(2)}(z,\omega^{(2)})\right|=O\left((\mu_2^2-\varepsilon)d^{-2}+\varepsilon d^{-3}|\zeta-\omega^{(2)}|\right)=O\left((\mu_2^2-\varepsilon)d^{-2}\right).$$
Then
\begin{equation*}
\begin{aligned}
&2\int_{B_{\gamma}(\zeta)} \left(\delta_2-\varphi_1\left(z,\zeta\right)+\varepsilon g^{(2)}(z,\omega^{(2)})\right)_1\cdot \left[\left(\delta_1\right)_x
\wedge \left(Z_{-1,1}\left(\frac{z-\xi}{\mu_1}\right)\right)_y+\left(Z_{-1,1}\left(\frac{z-\xi}{\mu_1}\right)\right)_x
\wedge \left(\delta_1\right)_y\right]_1
\\&=O\left(\mu_1^3\mu_2^{2}|\log\mu_2|+\mu_1^3(\mu_2^2-\varepsilon)d^{-2}\right).
\end{aligned}
\end{equation*}
And using (\ref{Ju25-6}) and (\ref{expansion6-1}), similar to the calculation of (\ref{interaction-23}), in $\mathcal{D}_2=\mathcal{D}\backslash B_{\gamma}(\zeta)$, we  obtain
\begin{align*}
&2\int_{\mathcal{D}_2} \left(\delta_2-\varphi_1+\varepsilon g^{(2)}(z,\omega^{(2)})\right)_1\cdot \left[\left(\delta_1\right)_x
\wedge \left(Z_{-1,1}\left(\frac{z-\xi}{\mu_1}\right)\right)_y+\left(Z_{-1,1}\left(\frac{z-\xi}{\mu_1}\right)\right)_x
\wedge \left(\delta_1\right)_y\right]_1
\\&=4\int_{\frac{\mathcal{D}_2-\xi}{\mu_1}}\left[\mu_2^2\frac{(\mu_1x+\xi_1-\zeta_1)^2-(\mu_1y+\xi_2-\zeta_2)^2}{|\mu_1z+\xi-\zeta|^2}+O(\mu_2^6)\right]\frac{128 x \left(x^2+y^2\right) \left(x^2 \left(\left(x^2+y^2\right)^2-2\right)+3 y^2\right)}{\left(\left(x^2+y^2\right)^2+1\right)^4}
\\&\quad-4\int_{\frac{\mathcal{D}_2-\xi}{\mu_1}}\left[\mu_2^2h_1^{(1)}(\mu_1z+\xi,\zeta)-\varepsilon g^{(2)}(\mu_1z+\xi,\omega^{(2)})+O(\mu_2^3(\mu_2^2-\varepsilon))\right]
\\&\qquad\qquad\qquad\times
\frac{128 x \left(x^2+y^2\right) \left(x^2 \left(\left(x^2+y^2\right)^2-2\right)+3 y^2\right)}{\left(\left(x^2+y^2\right)^2+1\right)^4}
\\&=4\mu_1\mu_2^2\int_{\frac{\mathcal{D}_2-\xi}{\mu_1}}\frac{-2 \sigma _1 \left(\sigma _1^2-3 \sigma _2^2\right)}{|\sigma|^6}\frac{128 x^2 \left(x^2+y^2\right) \left(x^2 \left(\left(x^2+y^2\right)^2-2\right)+3 y^2\right)}{\left(\left(x^2+y^2\right)^2+1\right)^4}
\\&\quad-4\mu_1\int_{\frac{\mathcal{D}_2-\xi}{\mu_1}}\frac{128 x^2 \left(x^2+y^2\right) \left(x^2 \left(\left(x^2+y^2\right)^2-2\right)+3 y^2\right)}{\left(\left(x^2+y^2\right)^2+1\right)^4}\left[\mu_2^2\frac{\partial h_1}{\partial x}(\xi,\zeta)-\varepsilon \frac{\partial g_1^{(1)}}{\partial x}(\xi,\omega^{(2)})\right]
\\&\quad+O\left(\mu_1^3\mu_2^2+\mu_1^3(\mu_2^2-\varepsilon)\right)
\\&=\frac{-32\pi \sigma _1 \left(\sigma _1^2-3 \sigma _2^2\right)}{|\sigma|^6}\mu_1\mu_2^2-16\pi \mu_1\left[\mu_2^2\frac{\partial h_1}{\partial x}(\xi,\zeta)-\varepsilon \frac{\partial g_1^{(1)}}{\partial x}(\xi,\omega^{(2)})\right]
+O\left(\mu_1^3\mu_2^{2}+\mu_1^3(\mu_2^2-\varepsilon)\right)
\\&=\frac{-32\pi \sigma _1 \left(\sigma _1^2-3 \sigma _2^2\right)}{|\sigma|^6}\mu_1\mu_2^2+O\left(\mu_1^3\mu_2^{2}+\mu_1(\mu_2^2-\varepsilon)\right),
\end{align*}
in the last equality, we use the fact that
$$\left|\mu_2^2\frac{\partial h_1}{\partial x}(\xi,\zeta)-\varepsilon \frac{\partial g_1^{(1)}}{\partial x}(\xi,\omega^{(2)})\right|=O\left(\mu_2^2-\varepsilon+\varepsilon|\zeta-\omega^{(2)}|\right)=O\left(\mu_2^2-\varepsilon\right).$$
Similarly, we can obtain
\begin{align*}
&2\int_{\mathcal{D}} \left(\delta_2-\varphi_1+\varepsilon g^{(2)}(z,\omega^{(2)})\right)_1\cdot \left[\left(\delta_1\right)_x
\wedge \left(Z_{-1,1}\left(\frac{z-\xi}{\mu_1}\right)\right)_y+\left(Z_{-1,1}\left(\frac{z-\xi}{\mu_1}\right)\right)_x
\wedge \left(\delta_1\right)_y\right]_2
\\&=4\int_{\frac{\mathcal{D}_2-\xi}{\mu_1}}\left[\mu_2^2\frac{(\mu_1x+\xi_1-\zeta_1)^2-(\mu_1y+\xi_2-\zeta_2)^2}{|\mu_1z+\xi-\zeta|^2}+O(\mu_2^6)\right]
\\&\qquad\qquad\quad\times
\frac{64 y \left(x^2+y^2\right) \left(3 x^6+7 x^4 y^2+x^2 \left(5
   y^4-9\right)+y^6+y^2\right)}{\left(\left(x^2+y^2\right)^2+1\right)^4}
\\&\quad-4\int_{\frac{\mathcal{D}_2-\xi}{\mu_1}}\left[\mu_2^2h_1^{(1)}(\mu_1z+\xi,\zeta)-\varepsilon g^{(2)}(\mu_1z+\xi,\omega^{(2)})+O(\mu_2^3(\mu_2^2-\varepsilon))\right]
\\&\qquad\qquad\qquad\times
\frac{64 y \left(x^2+y^2\right) \left(3 x^6+7 x^4 y^2+x^2 \left(5
   y^4-9\right)+y^6+y^2\right)}{\left(\left(x^2+y^2\right)^2+1\right)^4}
\\&\quad+O\left(\mu_1^3\mu_2^{2}|\log\mu_2|+\mu_1^3(\mu_2^2-\varepsilon)d^{-2}+\mu_1^3\varepsilon d^{-3}|\zeta-\omega^{(2)}|\right)
\\&=4\mu_1\mu_2^2\int_{\frac{\mathcal{D}_2-\xi}{\mu_1}}\frac{-2\sigma _2 \left(3\sigma _1^2- \sigma _2^2\right)}{|\sigma|^6}\frac{64 y^2 \left(x^2+y^2\right) \left(3 x^6+7 x^4 y^2+x^2 \left(5
   y^4-9\right)+y^6+y^2\right)}{\left(\left(x^2+y^2\right)^2+1\right)^4}
\\&\quad+O\left(\mu_1^3\mu_2^{2}+\mu_1(\mu_2^2-\varepsilon)+\mu_1^3\mu_2^{2}|\log\mu_2|+\mu_1^3(\mu_2^2-\varepsilon)d^{-2}\right)
\\&=\frac{-32\pi \sigma _2 \left(3\sigma _1^2- \sigma _2^2\right)}{|\sigma|^6}\mu_1\mu_2^2
+O\left(\mu_1^3\mu_2^{2}+\mu_1(\mu_2^2-\varepsilon)+\mu_1^3\mu_2^{2}|\log\mu_2|+\mu_1^3(\mu_2^2-\varepsilon)d^{-2}\right).
\end{align*}
And
\begin{align*}
&2\int_{\mathcal{D}} \left(\delta_2-\varphi_1+\varepsilon g^{(2)}(z,\omega^{(2)})\right)_1\cdot \left[\left(\delta_1\right)_x
\wedge \left(Z_{-1,1}\left(\frac{z-\xi}{\mu_1}\right)\right)_y+\left(Z_{-1,1}\left(\frac{z-\xi}{\mu_1}\right)\right)_x
\wedge \left(\delta_1\right)_y\right]_3
\\&=4\int_{\frac{\mathcal{D}_2-\xi}{\mu_1}}\left[\mu_2^2\frac{(\mu_1x+\xi_1-\zeta_1)^2-(\mu_1y+\xi_2-\zeta_2)^2}{|\mu_1z+\xi-\zeta|^2}+O(\mu_2^6)\right]\frac{32 x \left(x^2+y^2\right) \left(\left(x^2+y^2\right)^4-8
   \left(x^2+y^2\right)^2+3\right)}{\left(\left(x^2+y^2\right)^2+1\right)^4}
\\&\quad-4\int_{\frac{\mathcal{D}_2-\xi}{\mu_1}}\left[\mu_2^2h_1^{(1)}(\mu_1z+\xi,\zeta)-\varepsilon g^{(2)}(\mu_1z+\xi,\omega^{(2)})+O(\mu_2^3(\mu_2^2-\varepsilon))\right]
\\&\qquad\qquad\qquad\times
\frac{32 x \left(x^2+y^2\right) \left(\left(x^2+y^2\right)^4-8
   \left(x^2+y^2\right)^2+3\right)}{\left(\left(x^2+y^2\right)^2+1\right)^4}
\\&\quad+O\left(\mu_1^3\mu_2^{2}|\log\mu_2|+\mu_1^3(\mu_2^2-\varepsilon)d^{-2}\right)
\\&=O\left(\mu_1^3|\log\mu_1|\mu_2^{2}+\mu_1(\mu_2^2-\varepsilon)+\mu_1^3\mu_2^{2}|\log\mu_2|+\mu_1^3(\mu_2^2-\varepsilon)d^{-2}\right).
\end{align*}
Follow the same path, we conclude that
\begin{align}
\notag&2\int_{\mathcal{D}} \left(P\mathcal Q_2\delta_2+\varepsilon \mathcal Q_2g^{(2)} \right)\cdot \left[\left(\delta_1\right)_x
\wedge \left(a_{1,1}Z_{-1,1}\left(\frac{z-\xi}{\mu_1}\right)\right)_y+\left(a_{1,1}Z_{-1,1}\left(\frac{z-\xi}{\mu_1}\right)\right)_x
\wedge \left(\delta_1\right)_y\right]
\\\notag&=a_{1,1}\mu_1\mu_2^2\left\{q_{11}\frac{-32\pi \sigma _1 \left(\sigma _1^2-3 \sigma _2^2\right)}{|\sigma|^6}
+q_{12}\frac{-32\pi \sigma _2 \left(3\sigma _1^2- \sigma _2^2\right)}{|\sigma|^6}\right.
\\\notag&\qquad\qquad\qquad\left.+ q_{21}\frac{-32\pi \sigma _2 \left(3\sigma _1^2- \sigma _2^2\right)}{|\sigma|^6}
+q_{22}\frac{32\pi \sigma _1 \left(\sigma _1^2- 3\sigma _2^2\right)}{|\sigma|^6}\right\}
\\ \label{Jan17-3}&\quad+O\left(a_{1,1}\left(\varepsilon^{\frac{5}{2}}|\log\varepsilon|+\varepsilon^{\frac{1}{2}}(\mu_2^2-\varepsilon)+\varepsilon^{\frac{3}{2}}(\mu_2^2-\varepsilon)d^{-2}\right)\right).
\end{align}
Using (\ref{J9-a2}), (\ref{A26-5}) and (\ref{A26-6}), the remaining terms in
$$2\int_{\mathcal{D}} \left(P\mathcal Q_2\delta_2+\varepsilon \mathcal Q_2g^{(2)} \right)\cdot \left[\left(\delta_1\right)_x
\wedge \left(\mathcal{L}_{\mathcal{A}_1}\right)_y+\left(\mathcal{L}_{\mathcal{A}_1}\right)_x
\wedge \left(\delta_1\right)_y\right]$$
can be estimated in a similar manner. Indeed, reasoning as (\ref{M27-a2}), (\ref{A30-p3-a}) and (\ref{A30-p3-b}), we have
\begin{align}\notag
&2\int_{\mathcal{D}} \left(P\mathcal Q_2\delta_2+\varepsilon \mathcal Q_2g^{(2)} \right)\cdot \left[\left(\delta_1\right)_x
\wedge \left(a_{1,2}Z_{-1,2}\left(\frac{z-\xi}{\mu_1}\right)\right)_y+\left(a_{1,2}Z_{-1,2}\left(\frac{z-\xi}{\mu_1}\right)\right)_x
\wedge \left(\delta_1\right)_y\right]
\\\notag&=a_{1,2}\mu_1\mu_2^2\left\{q_{11}\frac{32\pi \sigma _2 \left(3\sigma _1^2-\sigma _2^2\right)}{|\sigma|^6}+q_{12}\frac{-32\pi \sigma _1 \left(\sigma _1^2- 3\sigma _2^2\right)}{|\sigma|^6}\right.
\\\notag&\qquad\qquad\qquad\left.+q_{21}\frac{-32\pi \sigma _1 \left(\sigma _1^2-3 \sigma _2^2\right)}{|\sigma|^6}+q_{22}\frac{-32\pi \sigma _2 \left(3\sigma _1^2- \sigma _2^2\right)}{|\sigma|^6}\right\}
\\\label{Jan17-4}&\quad+O\left(a_{1,2}\left(\varepsilon^{\frac{5}{2}}|\log\varepsilon|+\varepsilon^{\frac{1}{2}}(\mu_2^2-\varepsilon)+\varepsilon^{\frac{3}{2}}(\mu_2^2-\varepsilon)d^{-2}\right)\right),
\end{align}

\begin{equation}\label{Jan17-5}
\begin{aligned}
&2\int_{\mathcal{D}} \left(P\mathcal Q_2\delta_2+\varepsilon \mathcal Q_2g^{(2)} \right)\cdot \left[\left(\delta_1\right)_x
\wedge \left(\sum_{l=1}^2p_{1,l}Z_{2,l}\left(\frac{z-\xi}{\mu_1}\right)\right)_y+\left(\sum_{l=1}^2p_{1,l}Z_{2,l}\left(\frac{z-\xi}{\mu_1}\right)\right)_x
\wedge \left(\delta_1\right)_y\right]
\\&=O\left((p_{1,1}+p_{1,2})\left(\varepsilon^3+\varepsilon^2\left(\mu_2^2-\varepsilon\right)+\varepsilon^3(\mu_2^2-\varepsilon)d^{-2}\right)\right).
\end{aligned}
\end{equation}
And similar to the estimate of (\ref{J10-g1}), using (\ref{derivative}), we can get
\begin{equation}\label{JR-2}
\begin{aligned}
&-2\int_{\mathcal{D}} \left(P\mathcal Q_2\delta_2+\varepsilon \mathcal Q_2g^{(2)} \right)
\cdot \left[\left(\delta_1\right)_x
\wedge \left(\sum_{l=1}^2\left(a_{1,l}\varphi_{-1,l}\left(z,\xi\right)+p_{1, l}\varphi_{2,l}\left(z,\xi\right)\right)\right)_y\right.
\\&\qquad\qquad\qquad\qquad\qquad\qquad\qquad\left.
+\left(\sum_{l=1}^2\left(a_{1,l}\varphi_{-1,l}\left(z,\xi\right)+p_{1, l}\varphi_{2,l}\left(z,\xi\right)\right)\right)_x
\wedge \left(\delta_1\right)_y\right]
\\&=O\left((a_{1,1}+a_{1,2})\left(\left(\mu_2^2|\log \mu_2|+\left(\mu_2^2-\varepsilon\right)d^{-2}\right)\mu_1^3+\left(\mu_2^2+\varepsilon\right)\mu_1^3|\log\mu_1|d^{-2}\right)\right)
\\&\quad+O\left((p_{1,1}+p_{1,2})\left(\left(\mu_2^2|\log \mu_2|+\left(\mu_2^2-\varepsilon\right)d^{-2}\right)\mu_1^6+\left(\mu_2^2+\varepsilon\right)\mu_1^6|\log\mu_1|d^{-5}\right)\right)
\\&=O\left((a_{1,1}+a_{1,2})\left(\varepsilon^{\frac{3}{2}}\left(\mu_2^2-\varepsilon\right)d^{-2}+\varepsilon^{\frac{5}{2}}|\log\varepsilon|d^{-2}\right)+(p_{1,1}+p_{1,2})\left(\varepsilon^3\left(\mu_2^2-\varepsilon\right)d^{-2}+\varepsilon^4|\log\varepsilon|d^{-5}\right)\right).
\end{aligned}
\end{equation}
Using (\ref{Ju25-6}), (\ref{M28a1-1}),(\ref{M28a1-2}), (\ref{M30p1-1}) and (\ref{M30p1-2}), it is direct to check that
\begin{equation}\label{JR-1-1}
\begin{aligned}
&-2\int_{\mathcal{D}} \left(P\mathcal Q_2\delta_2+\varepsilon \mathcal Q_2g^{(2)} \right)
\cdot \left[\left(\varphi_1(z,\xi)-\varepsilon g^{(1)}\right)_x
\wedge \left(P\mathcal{L}_{\mathcal{A}_1}\right)_y
+\left(P\mathcal{L}_{\mathcal{A}_1}\right)_x
\wedge \left(\varphi_1(z,\xi)-\varepsilon g^{(1)}\right)_y\right]
\\&=O\left((a_{1,1}+a_{1,2})\left(\left(\mu_2^2|\log \mu_2|+\left(\mu_2^2-\varepsilon\right)d^{-2}\right)\mu_1\left(\mu_1^2-\varepsilon\right)+\left(\mu_2^2+\varepsilon\right)\mu_1\left(\mu_1^2-\varepsilon\right)d^{-5}\right)\right)
\\&\quad+O\left((p_{1,1}+p_{1,2})\left(\mu_2^2|\log \mu_2|+\left(\mu_2^2-\varepsilon\right)d^{-2}\right)\mu_1^4\left(\mu_1^2-\varepsilon\right)\right)
\\&\quad+O\left((p_{1,1}+p_{1,2})\left(\mu_2^2+\varepsilon\right)\left(\mu_1^2\left(\mu_1^2-\varepsilon\right)d^{-3}+\mu_1^4\left(\mu_1^2-\varepsilon\right)d^{-8}\right)\right)
\\&=O\left((a_{1,1}+a_{1,2})\left(\varepsilon^{\frac{3}{2}}|\log \varepsilon|+\varepsilon^{\frac{1}{2}}(\mu_2^2-\varepsilon)d^{-2}+\varepsilon^{\frac{3}{2}}d^{-5}\right)\left(\mu_1^2-\varepsilon\right)\right)
\\&\quad+O\left((p_{1,1}+p_{1,2})\left(\varepsilon^3|\log \varepsilon|+\varepsilon^2\left(\mu_2^2-\varepsilon\right)d^{-2}+\varepsilon^2d^{-3}+\varepsilon^3d^{-8}\right)\left(\mu_1^2-\varepsilon\right)\right).
\end{aligned}
\end{equation}
Thus
\begin{equation*}
\begin{aligned}
\mathcal{M}^1_{\delta_1,\delta_2,Z_{-1,l},Z_{2,l}}
&=(\ref{Jan17-1})+(\ref{Jan17-2})+(\ref{Jan17-3})+(\ref{Jan17-4})+(\ref{Jan17-5})+ (\ref{JR-2})+(\ref{JR-1-1})
\\&=F^{(1)}_{\mathcal{D}}(\delta_1,\delta_2,Id,\mathcal{Q}_2,a_{1,l},p_{1,l})+\mathcal{E}^{(1)}\left(\varepsilon,a_{1,l},p_{1,l}\right),
\end{aligned}
\end{equation*}
where $F^{(1)}_{\mathcal{D}}(\delta_1,\delta_2,Id,\mathcal{Q}_2,a_{1,l},p_{1,l})$ and $\mathcal{E}^{(1)}\left(\varepsilon,a_{1,l},p_{1,l}\right)$ are defined in (\ref{J7-R1}) and (\ref{J7-R2}).

Note that $\mathcal{M}^2_{\delta_1,\delta_2,Z_{-1,l},Z_{2,l}}$, i.e. (\ref{2-remainder}) can be addressed similarly, and the result stated in Lemma \ref{lemma-5.2} holds, we omit the details here.

Now, we proceed to estimate the terms in $\mathcal{M}^3_{\delta_1,\delta_2,Z_{-1,l},Z_{2,l}}$. For simplicity, we will only present the estimate for the first term in (\ref{4-remainder}), as the second term follows a similar approach.
Using (\ref{M27-2}), (\ref{M27-3}), (\ref{Ma30-a1a2}) and reasoning as (\ref{JR-1-1}), we deduce
\begin{equation}\label{Jan18-1}
\begin{aligned}
&2\int_{\mathcal{D}}\left(P\mathcal Q_2\delta_2+\varepsilon \mathcal Q_2g^{(2)}\right)\cdot \left[\left(\sum_{l=1}^2a_{1,l}PZ_{-1,l}\left(\frac{z-\xi}{\mu_1}\right)\right)_x  \wedge \left(\sum_{l=1}^2a_{1,l}PZ_{-1,l}\left(\frac{z-\xi}{\mu_1}\right) \right)_y\right]
\\&=O\left(\left(a_{1,1}+a_{1,2}\right)^2\left(\varepsilon^2|\log \varepsilon|+\varepsilon(\mu_2^2-\varepsilon)d^{-2}+\varepsilon^2d^{-4}\right)\right)
\\&=O\left(\left(a_{1,1}+a_{1,2}\right)^2\varepsilon^2d^{-4}\right).
\end{aligned}
\end{equation}
By (\ref{M30p1-1})-(\ref{M2-p2p2}) and (\ref{M2-p1p2}), it is straightforward to deduce that
\begin{equation}\label{Jan18-2}
\begin{aligned}
&2\sum_{l=1}^2p_{1,l}^2\int_{\mathcal{D}} \left(P\mathcal Q_2\delta_2+\varepsilon \mathcal Q_2g^{(2)}\right)\cdot \left[\left(PZ_{2,l}\left(\frac{z-\xi}{\mu_1}\right)\right)_x \wedge\left(PZ_{2,l}\left(\frac{z-\xi}{\mu_1}\right)\right)_y\right]
\\&=O\left(\left(p_{1,1}+p_{1,2}\right)^2\left(\varepsilon^5|\log \varepsilon|+\varepsilon^4(\mu_2^2-\varepsilon)d^{-2}+\varepsilon^2+\varepsilon^4d^{-5}+\varepsilon^5d^{-10}\right)\right)
\\&=O\left(\left(p_{1,1}+p_{1,2}\right)^2\left(\varepsilon^2+\varepsilon^5d^{-10}\right)\right).
\end{aligned}
\end{equation}
As for the mixed terms of $a_{1,k}$ and $p_{1,l}$, $k,l=1,2$, by (\ref{A27-a1p1}), (\ref{A27-a2p1}), (\ref{A27-a1p2}) and (\ref{A27-a2p2}), there holds
\begin{equation}\label{Jan18-3}
\begin{aligned}
&2\int_{\mathcal{D}} \left(P\mathcal Q_2\delta_2+\varepsilon \mathcal Q_2g^{(2)}\right)\cdot \left[\left(\sum_{l=1}^2a_{1,l}PZ_{-1,l}\left(\frac{z-\xi}{\mu_1}\right)\right)_x  \wedge \left(\sum_{l=1}^2p_{1, l}PZ_{2,l}\left(\frac{z-\xi}{\mu_1}\right)\right)_y\right.
\\&\qquad\qquad\qquad\qquad\qquad\qquad \left.
+\left(\sum_{l=1}^2p_{1, l}PZ_{2,l}\left(\frac{z-\xi}{\mu_1}\right)\right)_x  \wedge \left(\sum_{l=1}^2a_{1,l}PZ_{-1,l}\left(\frac{z-\xi}{\mu_1}\right) \right)_y\right]
\\&=O\left(\left(a_{1,1}+a_{1,2}\right)\left(p_{1,1}+p_{1,2}\right)\left(\varepsilon^{\frac{7}{2}}|\log \varepsilon|+\varepsilon^{\frac{5}{2}}(\mu_2^2-\varepsilon)d^{-2}+\varepsilon^{\frac{3}{2}}+\varepsilon^{\frac{5}{2}}d^{-2}+\varepsilon^{\frac{7}{2}}d^{-7}\right)\right)
\\&=O\left(\left(a_{1,1}+a_{1,2}\right)\left(p_{1,1}+p_{1,2}\right)\left(\varepsilon^{\frac{3}{2}}+\varepsilon^{\frac{7}{2}}d^{-7}\right)\right).
\end{aligned}
\end{equation}
A parallel result applies to the second term in (\ref{4-remainder}). Therefore,
\begin{equation*}
\begin{aligned}
\mathcal{M}^3_{\delta_1,\delta_2,Z_{-1,l},Z_{2,l}}=
\mathcal{E}^{(3)}\left(\varepsilon,a_{1,l},a_{2,l},p_{1,l},p_{2,l}\right),
\end{aligned}
\end{equation*}
where $\mathcal{E}^{(3)}\left(\varepsilon,a_{1,l},a_{2,l},p_{1,l},p_{2,l}\right)$ is as depicted in (\ref{J10-R3}).

Now, we deal with the terms in $\mathcal{M}^4_{\delta_1,\delta_2,Z_{-1,l},Z_{2,l}}$.
Using the expansion of $Z_{-1,l}$ and $Z_{2,l}$, $l=1,2$ as well as the corresponding derivatives provided in  Appendix \ref{Appendix-F}, the first term in (\ref{5-remainder}) can be estimated as
\begin{align*}
&-2\int_{\mathcal{D}} P\mathcal Q_2\mathcal{L}_{\mathcal{A}_2}
\cdot \left[\left(\delta_1\right)_x
\wedge \left(\sum_{l=1}^2\left(a_{1,l}\varphi_{-1,l}\left(z,\xi\right)+p_{1, l}\varphi_{2,l}\left(z,\xi\right)\right)\right)_y\right.
\\&\qquad\qquad\qquad\qquad\quad\left.
+\left(\sum_{l=1}^2\left(a_{1,l}\varphi_{-1,l}\left(z,\xi\right)+p_{1, l}\varphi_{2,l}\left(z,\xi\right)\right)\right)_x
\wedge \left(\delta_1\right)_y\right]
\\&=O\left(\sum_{k,l=1}^2a_{2,k}a_{1,l}\left(\varepsilon^2d^{-1}+\varepsilon^2|\log \varepsilon|d^{-2}\right)\right)
+O\left(\sum_{k,l=1}^2a_{2,k}p_{1,l}\left(\varepsilon^{\frac{7}{2}}d^{-1}+\varepsilon^{\frac{7}{2}}|\log \varepsilon|d^{-5}\right)\right)
\\&\quad+O\left(\sum_{k,l=1}^2p_{2,k}a_{1,l}\left(\varepsilon^{\frac{5}{2}}+\varepsilon^{\frac{7}{2}}d^{-4}+\varepsilon^{\frac{7}{2}}|\log \varepsilon|d^{-2}\right)\right)
+O\left(\sum_{k,l=1}^2p_{2,k}p_{1,l}\left(\varepsilon^{4}+\varepsilon^{5}d^{-4}+\varepsilon^{5}|\log \varepsilon|d^{-5}\right)\right).
\end{align*}
And similar to (\ref{JR-1-1}), there holds
\begin{align*}
&-2\int_{\mathcal{D}} P\mathcal Q_2\mathcal{L}_{\mathcal{A}_2}
\cdot \left[ \left(\varphi_1-\varepsilon g^{(1)}\right)_x
\wedge \left(P\mathcal{L}_{\mathcal{A}_1}\right)_y
+\left(P\mathcal{L}_{\mathcal{A}_1}\right)_x
\wedge  \left(\varphi_1-\varepsilon g^{(1)}\right)_y\right]
\\&=O\left(\sum_{k,l=1}^2a_{2,k}a_{1,l}\left(\varepsilon \left(\mu_1^2-\varepsilon\right)d^{-5}\right)+a_{2,k}p_{1,l}\left(\varepsilon^{\frac{5}{2}} \left(\mu_1^2-\varepsilon\right)d^{-8}\right)\right)
\\&\quad+O\left(\sum_{k,l=1}^2p_{2,k}a_{1,l}\left(\varepsilon^{\frac{3}{2}}\left(\mu_1^2-\varepsilon\right)+\varepsilon^{\frac{5}{2}} \left(\mu_1^2-\varepsilon\right)d^{-4}\right)+p_{2,k}p_{1,l}\left(\varepsilon^4\left(\mu_1^2-\varepsilon\right)d^{-8}\right)\right).
\end{align*}

A similar estimate can be applied to the last term in (\ref{5-remainder}), proceeding with the necessary adjustments.  It is important to note that compared to the second term in
(\ref{5-remainder}), where cancellations do not occur as described above, we need to estimate some additional terms taking the following form
\begin{equation*}
\begin{aligned}
&2\int_{\mathcal{D}} P\mathcal{L}_{\mathcal{A}_1}\cdot \left[\left(\mathcal Q_2\sum_{l=1}^2\left(a_{2,l}Z_{-1,l}\left(\frac{z-\zeta}{\mu_2}\right)+p_{2, l}Z_{2,l}\left(\frac{z-\zeta}{\mu_2}\right)\right)\right)_x
\wedge \left(\mathcal{Q}_2\delta_2\right)_y\right.
\\&\qquad\qquad\qquad\left.
+\left(\mathcal{Q}_2\delta_2\right)_x
\wedge \left(\mathcal Q_2\sum_{l=1}^2\left(a_{2,l}Z_{-1,l}\left(\frac{z-\zeta}{\mu_2}\right)+p_{2, l}Z_{2,l}\left(\frac{z-\zeta}{\mu_2}\right)\right)\right)_y\right].
\end{aligned}
\end{equation*}
Using (\ref{Ju24-1})-(\ref{Ju24-5}), (\ref{expansion6-1}), (\ref{J9-a2}), (\ref{A26-5}) and (\ref{A26-6}), similar to the calculation in (\ref{Mar22-1}), (\ref{M27-a2}), (\ref{A30-p3-a}) and (\ref{A30-p3-b}), we have
\begin{align*}
&2\int_{\mathcal{D}}  P\mathcal{L}_{\mathcal{A}_1}
\cdot \left[\left(\mathcal Q_2\sum_{l=1}^2\left(a_{2,l}Z_{-1,l}\left(\frac{z-\zeta}{\mu_2}\right)+p_{2, l}Z_{2,l}\left(\frac{z-\zeta}{\mu_2}\right)\right)\right)_x
\wedge \left(\mathcal{Q}_2\delta_2\right)_y\right.
\\&\qquad\qquad\qquad\quad\left.
+\left(\mathcal{Q}_2\delta_2\right)_x
\wedge \left(\mathcal Q_2\sum_{l=1}^2\left(a_{2,l}Z_{-1,l}\left(\frac{z-\zeta}{\mu_2}\right)+p_{2, l}Z_{2,l}\left(\frac{z-\zeta}{\mu_2}\right)\right)\right)_y\right]
\\&=O\left(\sum_{k,l=1}^2a_{1,k}a_{2,l}\left(\varepsilon+\varepsilon^2d^{-1}\right)\right)
+O\left(\sum_{k,l=1}^2p_{1,k}a_{2,l}\left(\varepsilon^{\frac{5}{2}}+\varepsilon^{\frac{7}{2}}d^{-4}\right)\right)
\\&\quad+O\left(\sum_{k,l=1}^2a_{1,k}p_{2,l}\left(\varepsilon^{\frac{5}{2}}+\varepsilon^{\frac{7}{2}}d^{-1}\right)\right)+O\left(\sum_{k,l=1}^2p_{1,k}p_{2,l}\left(\varepsilon^{4}+\varepsilon^{5}d^{-4}\right)\right).
\end{align*}
Thus we conclude that
\begin{equation*}
\begin{aligned}
\mathcal{M}^4_{\delta_1,\delta_2,Z_{-1,l},Z_{2,l}}
=\mathcal{E}^{(4)}\left(\varepsilon,a_{1,l},a_{2,l},p_{1,l},p_{2,l}\right),
\end{aligned}
\end{equation*}
where $\mathcal{E}^{(4)}\left(\varepsilon,a_{1,l},a_{2,l},p_{1,l},p_{2,l}\right)$ appears in (\ref{J10-R4}).

Next, we estimate the terms in $\mathcal{M}^5_{\delta_1,\delta_2,Z_{-1,l},Z_{2,l}}$, i.e. (\ref{6-remainder}).
And similar to the estimate of $\mathcal{M}^3_{\delta_1,\delta_2,Z_{-1,l},Z_{2,l}}$, refer to (\ref{Jan18-1}), (\ref{Jan18-2}) and (\ref{Jan18-3}), we deduce
\begin{align*}
2\int_{\mathcal{D}}  P\mathcal Q_2\mathcal{L}_{\mathcal{A}_2}\cdot \left[\left(P\mathcal{L}_{\mathcal{A}_1}\right)_x\wedge \left(P\mathcal{L}_{\mathcal{A}_1}\right)_y\right]
&=O\left(\left(a_{2,1}+a_{2,2}\right)\left(a_{1,1}+a_{1,2}\right)^2\left(\varepsilon^{\frac{3}{2}}|\log \varepsilon|+\varepsilon^{\frac{3}{2}}d^{-4}\right)\right)
\\&\quad+O\left(\left(a_{2,1}+a_{2,2}\right)\left(a_{1,1}+a_{1,2}\right)\left(p_{1,1}+p_{1,2}\right)\left(\varepsilon+\varepsilon^{3}d^{-7}\right)\right)
\\&\quad+O\left(\left(a_{2,1}+a_{2,2}\right)\left(p_{1,1}+p_{1,2}\right)^2\left(\varepsilon^{\frac{3}{2}}+\varepsilon^{\frac{9}{2}}d^{-10}\right)\right)
\\&\quad+O\left(\left(p_{2,1}+p_{2,2}\right)\left(a_{1,1}+a_{1,2}\right)^2\left(\varepsilon^3|\log \varepsilon|+\varepsilon^{3}d^{-4}\right)\right)
\\&\quad+O\left(\left(p_{2,1}+p_{2,2}\right)\left(a_{1,1}+a_{1,2}\right)\left(p_{1,1}+p_{1,2}\right)\left(\varepsilon^{\frac{5}{2}}+\varepsilon^{\frac{9}{2}}d^{-7}\right)\right)
\\&\quad+O\left(\left(p_{2,1}+p_{2,2}\right)\left(p_{1,1}+p_{1,2}\right)^2\left(\varepsilon^{3}+\varepsilon^{6}d^{-10}\right)\right).
\end{align*}
The first term in $\mathcal{M}^5_{\delta_1,\delta_2,Z_{-1,l},Z_{2,l}}$ can be estimated in a similar manner. Thus, we conclude that
$$\mathcal{M}^5_{\delta_1,\delta_2,Z_{-1,l},Z_{2,l}}=
\mathcal{E}^{(5)}\left(\varepsilon,a_{1,l},a_{2,l},p_{1,l},p_{2,l}\right)$$
as stated in (\ref{Ju30-M5}).

This completes the proof of Lemma \ref{lemma-5.2}.
\end{proof}

Now we consider the interaction associated with the higher order expansion $\mathcal{R}_{\mathcal{A}_i}$ in $z_{\mathcal{A}_i}$, $i=1,2$. Specifically, for $i,j=1,2$ and $i\neq j$, we need to estimate terms of the following form
\begin{align*}
\mathcal{M}^6_{\delta_1,\delta_2,Z_{-1,l},Z_{2,l}}
:&=\int_{\mathcal{D}}\nabla\left( P\mathcal Q_i\delta_i+P\mathcal Q_i\mathcal{L}_{\mathcal{A}_i}\right)\cdot\nabla P\mathcal Q_j\mathcal{R}_{\mathcal{A}_j}+\int_{\mathcal{D}}\nabla P\mathcal Q_i\mathcal{R}_{\mathcal{A}_i}\cdot\nabla P\mathcal Q_j\mathcal{R}_{\mathcal{A}_j}
\\&\quad+2\int_{\mathcal{D}} \left( P\mathcal Q_i\delta_i+P\mathcal Q_i\mathcal{L}_{\mathcal{A}_i}+\varepsilon \mathcal Q_i g^{(i)}\right)\cdot \left[\left(P\mathcal Q_j\mathcal{R}_{\mathcal{A}_j}\right)_x
\wedge \left(P\mathcal Q_j\delta_j+P\mathcal Q_j\mathcal{L}_{\mathcal{A}_j}\right)_y\right.
\\&\qquad\qquad\qquad\qquad\qquad\qquad\qquad\qquad\qquad\quad\left.
+\left(P\mathcal Q_j\delta_j+P\mathcal Q_j\mathcal{L}_{\mathcal{A}_j}\right)_x
\wedge \left(P\mathcal Q_j\mathcal{R}_{\mathcal{A}_j}\right)_y\right]
\\&\quad+2\int_{\mathcal{D}} \left( P\mathcal Q_i\delta_i+P\mathcal Q_i\mathcal{L}_{\mathcal{A}_i}+\varepsilon \mathcal Q_i g^{(i)}\right)\cdot \left[\left(P\mathcal Q_j\mathcal{R}_{\mathcal{A}_j}\right)_x
\wedge \left(P\mathcal Q_j\mathcal{R}_{\mathcal{A}_j}\right)_y\right]
\\&\quad+2\int_{\mathcal{D}} P\mathcal Q_i\mathcal{R}_{\mathcal{A}_i}
\cdot\left[ \left(P\mathcal Q_j\delta_j+P\mathcal Q_j\mathcal{L}_{\mathcal{A}_j}+P\mathcal Q_j\mathcal{R}_{\mathcal{A}_j}\right)_x
\wedge \left(P\mathcal Q_j\delta_j+P\mathcal Q_j\mathcal{L}_{\mathcal{A}_j}+P\mathcal Q_j\mathcal{R}_{\mathcal{A}_j}\right)_y\right]
\\&\quad+2\varepsilon\int_{\mathcal{D}} \left( P\mathcal Q_i\delta_i+P\mathcal Q_i\mathcal{L}_{\mathcal{A}_i}\right)
\cdot\left[ \left(P\mathcal Q_j\mathcal{R}_{\mathcal{A}_j}\right)_x
\wedge \left(\mathcal Q_j g^{(j)}\right)_y+\left(\mathcal Q_j g^{(j)}\right)_x\wedge \left(P\mathcal Q_j\mathcal{R}_{\mathcal{A}_j}\right)_y\right]
\\&\quad+2\varepsilon\int_{\mathcal{D}} P\mathcal Q_i\mathcal{R}_{\mathcal{A}_i}
\cdot\left[ \left( P\mathcal Q_j\delta_j+P\mathcal Q_j\mathcal{L}_{\mathcal{A}_j}+P\mathcal Q_j\mathcal{R}_{\mathcal{A}_j}\right)_x
\wedge \left(\mathcal Q_j g^{(j)}\right)_y\right.
\\&\qquad\qquad\qquad\qquad\qquad\left.
+\left(\mathcal Q_j g^{(j)}\right)_x\wedge  \left( P\mathcal Q_j\delta_j+P\mathcal Q_j\mathcal{L}_{\mathcal{A}_j}+P\mathcal Q_j\mathcal{R}_{\mathcal{A}_j}\right)_y\right]
\\&\quad+ 2\varepsilon^2\int_{\mathcal{D}} \mathcal Q_i g^{(i)}
\cdot\left[ \left(P\mathcal Q_j\mathcal{R}_{\mathcal{A}_j}\right)_x
\wedge \left(\mathcal Q_j g^{(j)}\right)_y+\left(\mathcal Q_j g^{(j)}\right)_x\wedge \left(P\mathcal Q_j\mathcal{R}_{\mathcal{A}_j}\right)_y\right]
\\&\quad+ 2\varepsilon^2\int_{\mathcal{D}} P\mathcal Q_i\mathcal{R}_{\mathcal{A}_i}
\cdot\left[\left(\mathcal Q_j g^{(j)}\right)_x
\wedge \left(\mathcal Q_j g^{(j)}\right)_y\right]
.
\end{align*}
Similar to the estimate of $\mathcal{M}^i_{\delta_1,\delta_2,Z_{-1,l},Z_{2,l}}$, $i=1,\dots,5$, we conclude that
\begin{equation}\label{Mar15-1}
\begin{aligned}
\mathcal{M}^6_{\delta_1,\delta_2,Z_{-1,l},Z_{2,l}}=\mathcal{E}^{(3)}\left(\varepsilon,a_{1,l},a_{2,l},p_{1,l},p_{2,l}\right)+\mathcal{E}^{(5)}\left(\varepsilon,a_{1,l},a_{2,l},p_{1,l},p_{2,l}\right).
\end{aligned}
\end{equation}

\subsection{Expansion for $k$ bubbles}\label{subsection5.4}
Given two bubbles $\delta_i=\mathcal Q_{i}\delta_{\mu_i,\xi^{(i)}}$ and $\delta_j=\mathcal Q_{j}\delta_{\mu_j,\xi^{(j)}}$, where $\delta_{\mu_i,\xi^{(i)}}$ is defined in (\ref{A18-1}) and $\mathcal Q_{i}$, $\mathcal Q_{j}\in SO(3)$. By rotation invariance, we know the interaction between $\delta_i$ and $\delta_j$ is the same as that between $\delta_{\mu_i,\xi^{(i)}}$ and $\mathcal Q_{i}^{-1}\delta_j=\mathcal Q_{i}^{-1}\mathcal Q_{j}\delta_{\mu_j,\xi^{(j)}}$. Thus, for $k=2$, combining Proposition \ref{prop 4.1} and Lemma \ref{lemma 6.3}, Lemma \ref{lemma-5.2} and (\ref{Mar15-1}), we have
\begin{equation}\label{Ju30-energy}
\begin{aligned}
I_{\varepsilon}(u)
=&~\frac{16}{3}\pi+F_{\mathcal{D},g^{(1)}}(\mu_1,\xi,Id,a_{1,l},p_{1,l})+F_{\mathcal{D},g^{(2)}}(\mu_2,\zeta,\mathcal{Q}_2,a_{2,l},p_{2,l})
\\&+F_{\mathcal{D}}(\delta_1,\delta_2,Id,\mathcal{Q}_2)
+F^{(1)}_{\mathcal{D}}(\delta_1,\delta_2,Id,\mathcal{Q}_2,a_{1,l},p_{1,l})+F^{(2)}_{\mathcal{D}}(\delta_1,\delta_2,Id,\mathcal{Q}_2,a_{2,l},p_{2,l})
\\&
+\tilde{e}^{(1)}(\mu_1,a_{1,l},p_{1,l})+\tilde{e}^{(2)}(\varepsilon,\mu_1,a_{1,l},p_{1,l})+\tilde{e}^{(1)}(\mu_2,a_{2,l},p_{2,l})+\tilde{e}^{(2)}(\varepsilon,\mu_2,a_{2,l},p_{2,l})
+\mathcal{E}\left(\mu_1,\mu_2\right)
\\&
+\mathcal{E}^{(1)}\left(\varepsilon,a_{1,l},p_{1,l}\right)+\mathcal{E}^{(2)}\left(\varepsilon,a_{2,l},p_{2,l}\right)+\mathcal{E}^{(3)}\left(\varepsilon,a_{1,l},a_{2,l},p_{1,l},p_{2,l}\right)
\\&+\mathcal{E}^{(4)}\left(\varepsilon,a_{1,l},a_{2,l},p_{1,l},p_{2,l}\right)
+\mathcal{E}^{(5)}\left(\varepsilon,a_{1,l},a_{2,l},p_{1,l},p_{2,l}\right).
\end{aligned}
\end{equation}

Now, we consider the case of $k>2$.
From the explicit form of the Euler functional $I_{\varepsilon}$, we observe that $I_{\varepsilon}$ is cubic in $u$, then in the case of $k$ bubbles, by using integration by parts, we may find mixed terms of the form
\begin{align}
\notag
&2\int_{\mathcal{D}}z_{\mathcal{A}_i}\cdot \left[(z_{\mathcal{A}_j})_x\wedge (z_{\mathcal{A}_m})_y+(z_{\mathcal{A}_m})_x\wedge (z_{\mathcal{A}_j})_y\right]
\\ \notag&+2\varepsilon \int_{\mathcal{D}}z_{\mathcal{A}_i}\cdot \left[(z_{\mathcal{A}_j})_x\wedge (g^{(m)})_y+(g^{(m)})_x\wedge (z_{\mathcal{A}_j})_y\right]+2\varepsilon \int_{\mathcal{D}}z_{\mathcal{A}_i}\cdot \left[(g^{(j)})_x\wedge (z_{\mathcal{A}_m})_y+(g^{(j)})_x\wedge (z_{\mathcal{A}_m})_y\right]
\\ \notag&+2\varepsilon \int_{\mathcal{D}}z_{\mathcal{A}_j}\cdot \left[(g^{(i)})_x\wedge (z_{\mathcal{A}_m})_y+(z_{\mathcal{A}_m})_x\wedge (g^{(i)})_y\right]
\\ \notag& +2\varepsilon^2 \int_{\mathcal{D}}z_{\mathcal{A}_i}\cdot \left[(g^{(j)})_x\wedge (g^{(m)})_y+(g^{(m)})_x\wedge (g^{(j)})_y\right]+2\varepsilon^2 \int_{\mathcal{D}}z_{\mathcal{A}_j}\cdot \left[(g^{(i)})_x\wedge (g^{(m)})_y+(g^{(m)})_x\wedge (g^{(i)})_y\right]
\\ \label{Mar4-1} &+2\varepsilon^2 \int_{\mathcal{D}}z_{\mathcal{A}_m}\cdot \left[(g^{(i)})_x\wedge (g^{(j)})_y+(g^{(j)})_x\wedge (g^{(m)})_y\right],
\end{align}
where $i,j,m$ are distinct. Recall that in (\ref{solution-set}), we assume that the distances between the points $\xi^{(i)}$, $\xi^{(j)}$ and $\xi^{(m)}$ are uniformly bounded from below. Using a similar approach as in the proof of Lemma \ref{lemma 6.3}, choosing
$ \gamma \leq \frac{1}{4}\bar{C}$, for the terms that contain $\delta_i$, $\delta_j$, $\delta_m$, it is direct to see that
\begin{equation}\label{interaction-55}
\begin{aligned}
&\int_{\mathcal{D}}P\mathcal Q_i\delta_i\cdot \left[(P\mathcal Q_j\delta_j)_x\wedge (P\mathcal Q_m\delta_m)_y+(P\mathcal Q_m\delta_m)_x\wedge (P\mathcal Q_j\delta_j)_y\right]
\\&= \int_{B_{\gamma}(\xi^{(i)})\cup B_{\gamma}(\xi^{(j)})\cup B_{\gamma}(\xi^{(l)})}P\mathcal Q_i\delta_i\cdot \left[(P\mathcal Q_j\delta_j)_x\wedge (P\mathcal Q_m\delta_m)_y+(P\mathcal Q_m\delta_m)_x\wedge (P\mathcal Q_j\delta_j)_y\right]
\\&\quad+\int_{\mathcal{D} \backslash \left(B_{\gamma}(\xi^{(i)})\cup B_{\gamma}(\xi^{(j)})\cup B_{\tau}(\xi^{(m)})\right)}P\mathcal Q_i\delta_i\cdot \left[(P\mathcal Q_j\delta_j)_x\wedge (P\mathcal Q_m\delta_m)_y+(P\mathcal Q_m\delta_m)_x\wedge (P\mathcal Q_j\delta_j)_y\right]
\\&=O
\left(\varepsilon^3|\log\varepsilon|+\varepsilon^3d^{-3}\right).
\end{aligned}
\end{equation}
For the terms involving $P\mathcal Q_m\mathcal{L}_{\mathcal{A}_m}$, we need to deal with them carefully. Specifically, we need to estimate the following terms
\begin{equation}\label{Mar3-1}
\begin{aligned}
&\int_{\mathcal{D}}P\mathcal Q_i\delta_i\cdot \left[(P\mathcal Q_j\delta_j+\varepsilon g^{(j)})_x\wedge \left(P\mathcal Q_m\mathcal{L}_{\mathcal{A}_m}\right)_y+\left(P\mathcal Q_m\mathcal{L}_{\mathcal{A}_m}\right)_x\wedge (P\mathcal Q_j\delta_j+\varepsilon g^{(j)})_y\right]
\\&+\int_{\mathcal{D}}P\mathcal Q_m\mathcal{L}_{\mathcal{A}_m}\cdot \left[(P\mathcal Q_j\delta_j+\varepsilon g^{(j)})_x\wedge (\varepsilon g^{(i)})_y+(\varepsilon g^{(i)})_x\wedge (P\mathcal Q_j\delta_j+\varepsilon g^{(j)})_y\right].
\end{aligned}
\end{equation}
We can then proceed with similar estimations, and (\ref{Mar3-1}) can be bounded by
\begin{equation}\label{Jan20-1}
\begin{aligned}
O\left(\sum_{l=1}^2a_{m,l}
\left(\varepsilon^{\frac{5}{2}}d^{-2}+(\mu_j^2-\varepsilon)\varepsilon^{\frac{3}{2}}d^{-3}\right)\right)
+O\left(\sum_{l=1}^2p_{m,l}\left(\varepsilon^{3}+\varepsilon^{4}d^{-5}\right)\right).
\end{aligned}
\end{equation}
Similarly,
\begin{equation}\label{Jan20-2}
\begin{aligned}
&\int_{\mathcal{D}}P\mathcal Q_i\delta_i\cdot \left[\left(P\mathcal Q_j\mathcal{L}_{\mathcal{A}_j}\right)_x
\wedge \left(P\mathcal Q_m\mathcal{L}_{\mathcal{A}_m}\right)_y+\left(P\mathcal Q_m\mathcal{L}_{\mathcal{A}_m}\right)_x
\wedge \left(P\mathcal Q_j\mathcal{L}_{\mathcal{A}_j}\right)_y\right]
\\&+2\varepsilon \int_{\mathcal{D}}P\mathcal Q_j\mathcal{L}_{\mathcal{A}_j}\cdot \left[(g^{(i)})_x\wedge \left(P\mathcal Q_m\mathcal{L}_{\mathcal{A}_m}\right)_y+\left(P\mathcal Q_m\mathcal{L}_{\mathcal{A}_m}\right)_x\wedge (g^{(i)})_y\right]
\\&=O
\left(\sum_{l=1}^2a_{j,l} a_{m,l}\left(\varepsilon^{2}|\log\varepsilon|+\varepsilon^{2} d^{-2}\right)\right)
+O
\left(\sum_{l=1}^2a_{j,l} p_{m,l}\left(\varepsilon^{\frac{5}{2}}+\varepsilon^{\frac{7}{2}}d^{-5}
\right)\right)
\\&\quad+O
\left(\sum_{l=1}^2p_{j,l} a_{m,l}\left(\varepsilon^{\frac{5}{2}}+\varepsilon^{\frac{7}{2}}d^{-5}\right)\right)
+O
\left(\sum_{l=1}^2p_{j,l} p_{m,l}\left(\varepsilon^{4}+\varepsilon^{5}d^{-5}\right)\right),
\end{aligned}
\end{equation}

\begin{equation}\label{Ju28-1}
\begin{aligned}
&\int_{\mathcal{D}}\left(P\mathcal Q_i\mathcal{L}_{\mathcal{A}_i}\right)
\cdot \left[\left(P\mathcal Q_j\mathcal{L}_{\mathcal{A}_j}\right)_x
\wedge \left(P\mathcal Q_m\mathcal{L}_{\mathcal{A}_m}\right)_y+\left(P\mathcal Q_m\mathcal{L}_{\mathcal{A}_m}\right)_x
\wedge \left(P\mathcal Q_j\mathcal{L}_{\mathcal{A}_j}\right)_y\right]
\\&=O
\left(\sum_{l=1}^2a_{i,l}a_{j,l} a_{m,l}\left(\varepsilon^{\frac{3}{2}}|\log\varepsilon|+\varepsilon^{\frac{3}{2}}d^{-2}\right)+a_{i,l}a_{j,l} p_{m,l}\left(\varepsilon^{2}+\varepsilon^{3}d^{-5}
\right)\right)
\\&\quad+O
\left(\sum_{l=1}^2a_{i,l}p_{j,l} a_{m,l}\left(\varepsilon^{2}+\varepsilon^{3}d^{-5}\right)+a_{i,l}p_{j,l} p_{m,l}\left(\varepsilon^{\frac{7}{2}}+\varepsilon^{\frac{9}{2}}d^{-5}\right)+p_{i,l}a_{j,l} a_{m,l}\left(\varepsilon^{2}+\varepsilon^{3}d^{-4}
\right)\right)
\\&\quad+O
\left(\sum_{l=1}^2p_{i,l}a_{j,l} p_{m,l}\left(\varepsilon^{\frac{7}{2}}+\varepsilon^{\frac{9}{2}}d^{-5}
\right)+p_{i,l}p_{j,l} a_{m,l}\left(\varepsilon^{\frac{7}{2}}+\varepsilon^{\frac{9}{2}}d^{-5}
\right)+p_{i,l}p_{j,l} p_{m,l}\left(\varepsilon^{5}+\varepsilon^{6}d^{-5}
\right)\right),
\end{aligned}
\end{equation}

\begin{align}
\notag&\int_{\mathcal{D}}P\mathcal Q_i\delta_i\cdot \left[(P\mathcal Q_j\delta_j+\varepsilon g^{(j)})_x\wedge \left(P\mathcal Q_m\mathcal{R}_{\mathcal{A}_m}\right)_y+\left(P\mathcal Q_m\mathcal{        R}_{\mathcal{A}_m}\right)_x\wedge (P\mathcal Q_j\delta_j+\varepsilon g^{(j)})_y\right]
\\\notag&+\int_{\mathcal{D}}P\mathcal Q_m\mathcal{R}_{\mathcal{A}_m}\cdot \left[(P\mathcal Q_j\delta_j+\varepsilon g^{(j)})_x\wedge (\varepsilon g^{(i)})_y+(\varepsilon g^{(i)})_x\wedge (P\mathcal Q_j\delta_j+\varepsilon g^{(j)})_y\right]
\\\notag&=O
\left((a_{m,1}+a_{m,2})^2
\left(\varepsilon^{\frac{5}{2}}+\varepsilon^{3}d^{-3}\right)\right)
+O
\left((a_{m,1}+a_{m,2})(p_{m,1}+p_{m,2})\left(\varepsilon^{\frac{5}{2}}+\varepsilon^{\frac{7}{2}}d^{-4}\right)\right)
\\ \label{Jan20-3}&\quad+O
\left((p_{m,1}+p_{m,2})^2\left(\varepsilon^{\frac{5}{2}}+\varepsilon^{5}d^{-7}\right)\right),
\end{align}

\begin{equation}\label{Jan20-4}
\begin{aligned}
&\int_{\mathcal{D}}P\mathcal Q_i\delta_i\cdot \left[(P\mathcal Q_j\mathcal{L}_{\mathcal{A}_j})_x\wedge \left(P\mathcal Q_m\mathcal{R}_{\mathcal{A}_m}\right)_y+\left(P\mathcal Q_m\mathcal{R}_{\mathcal{A}_m}\right)_x\wedge (P\mathcal Q_j\mathcal{L}_{\mathcal{A}_j})_y\right]
\\&+2\varepsilon \int_{\mathcal{D}}P\mathcal Q_j\mathcal{L}_{\mathcal{A}_j}\cdot \left[(g^{(i)})_x\wedge \left(P\mathcal Q_m\mathcal{R}_{\mathcal{A}_m}\right)_y+\left(P\mathcal Q_m\mathcal{R}_{\mathcal{A}_m}\right)_x\wedge (g^{(i)})_y\right]
\\&=O
\left((a_{j,1}+a_{j,2})(a_{m,1}+a_{m,2})^2
\left(\varepsilon^2+\varepsilon^{\frac{5}{2}}d^{-3}\right)\right)
+O
\left((a_{j,1}+a_{j,2})(a_{m,1}+a_{m,2})(p_{m,1}+p_{m,2})\left(\varepsilon^2+\varepsilon^{3}d^{-4}\right)\right)
\\&\quad+O
\left((a_{j,1}+a_{j,2})(p_{m,1}+p_{m,2})^2\left(\varepsilon^2+\varepsilon^{\frac{9}{2}}d^{-7}\right)\right)
+O
\left((p_{j,1}+p_{j,2})(a_{m,1}+a_{m,2})^2
\left(\varepsilon^{\frac{5}{2}}+\varepsilon^{4}d^{-5}\right)\right)
\\&\quad+O
\left((p_{j,1}+p_{j,2})(a_{m,1}+a_{m,2})(p_{m,1}+p_{m,2})\left(\varepsilon^3+\varepsilon^{\frac{9}{2}}d^{-5}\right)\right)
+O
\left((p_{j,1}+p_{j,2})(p_{m,1}+p_{m,2})^2\left(\varepsilon^{\frac{7}{2}}+\varepsilon^{6}d^{-7}\right)\right).
\end{aligned}
\end{equation}

For simplicity, we denote the errors in (\ref{interaction-55})-(\ref{Jan20-4}) as $\mathcal{E}(\delta_i,\delta_j,\delta_m, a_i,a_j,a_m,p_i,p_j,p_m)$.
The estimation of the remaining terms in (\ref{Mar4-1}) follows a similar approach. These terms are bounded by $\mathcal{E}(\delta_i,\delta_j,\delta_m, a_i,a_j,a_m,p_i,p_j,p_m)$, with the detailed derivation omitted here.

From (\ref{interaction-55})-(\ref{Jan20-4}), we observe that the interaction between three distinct bubbles in (\ref{solution-set}) is negligible compared to the interactions with the boundary function $g$ and the domain $\mathcal{D}$.

Parallel to (\ref{Ju30-energy}), by combining (\ref{interaction-55})-(\ref{Jan20-4}), for general $k$ bubbles, we establish the following proposition.
\begin{prop}\label{prop 6.4}
Let $\bar C>0$, $k\in \mathbb{N}_+$, for $i\neq j$, $l=1,2$, $\mathcal Q_{i}$, $\mathcal Q_{j}\in SO(3)$, set $\delta_i=Id\delta_{\mu_i,\xi^{(i)}}$, $\delta_j=\mathcal Q_{i}^{-1}\mathcal Q_{j}\delta_{\mu_j,\xi^{(j)}}$, $\xi^{(j)}-\xi^{(i)}=\sigma_{ij}=(\sigma_{ij,1},\sigma_{ij,2})$ and
\begin{equation*}
\begin{aligned}
&\vec{\mu}=(\mu_1,\cdots,\mu_k),\quad \vec{\xi}=\left(\xi^{(1)},\cdots,\xi^{(k)}\right),\quad \vec{Q}=\left(\mathcal Q_{1},\cdots,\mathcal Q_{k}\right),\quad \vec{a}=\left(a_{1},\cdots,a_{k}\right),\quad \vec{p}=\left(p_{1},\cdots,p_{k}\right),
\end{aligned}
\end{equation*}
\begin{equation*}
\begin{aligned}
\Sigma_{\mathcal{D},g}(\varepsilon,\vec{\mu},\vec{\xi},\vec{Q},\vec{a},\vec{p})
=&\sum_{i=1}^k F_{\mathcal{D},g^{(i)}}(\mu_i,\xi^{(i)},\mathcal Q_{i},a_{i,l},p_{i,l})+\sum_{i<j} F_{\mathcal{D}}(\delta_i,\delta_j,Id,\mathcal Q_{i}^{-1}\mathcal Q_{j})
\\&+\sum_{i<j}F^{(1)}_{\mathcal{D}}(\delta_i,\delta_j,Id,\mathcal Q_{i}^{-1}\mathcal Q_{j}, a_{i,l},p_{i,l})
+\sum_{i<j}F^{(2)}_{\mathcal{D}}(\delta_i,\delta_j,Id,\mathcal Q_{i}^{-1}\mathcal Q_{j}, a_{j,l},p_{j,l}),
\end{aligned}
\end{equation*}
and
\begin{equation*}
\begin{aligned}
\mathcal{E}(\varepsilon,\vec{\mu},\vec{a},\vec{p})
=&\sum_{i=1}^k\tilde{e}^{(1)}(\mu_i,a_{i,l},p_{i,l})+\tilde{e}^{(2)}(\varepsilon,\mu_i,a_{i,l},p_{i,l})
+\sum_{i,j=1}^k \mathcal{E}(\mu_i,\mu_j)
+\sum_{i=1}^k\mathcal{E}^{(i)}\left(\varepsilon,a_{i,l},p_{i,l}\right)
\\&+\sum_{i<j}\mathcal{E}^{(3)}\left(\varepsilon,a_{i,l},a_{j,l},p_{i,l},p_{j,l}\right)+\mathcal{E}^{(4)}\left(\varepsilon,a_{i,l},a_{j,l},p_{i,l},p_{j,l}\right)
+\mathcal{E}^{(5)}\left(\varepsilon,a_{i,l},a_{j,l},p_{i,l},p_{j,l}\right)
\\&+\sum_{i<j<m}\mathcal{E}(\delta_i,\delta_j,\delta_m, a_i,a_j,a_m,p_i,p_j,p_m),
\end{aligned}
\end{equation*}
where $F_{\mathcal{D},g^{(i)}}(\mu_i,\xi^{(i)},\mathcal Q_{i},a_{i,l},p_{i,l})$ is defined in Proposition \ref{prop 4.1}, $F_{\mathcal{D}}(\delta_i,\delta_j,Id,\mathcal Q_{i}^{-1}\mathcal Q_{j})$ and $\mathcal{E}(\mu_i,\mu_j)$ are defined in Lemma \ref{lemma 6.3}.
$F^{(1)}_{\mathcal{D}}(\delta_i,\delta_j,Id,\mathcal Q_{i}^{-1}\mathcal Q_{j}, a_{i,l},p_{i,l})$, $F^{(2)}_{\mathcal{D}}(\delta_i,\delta_j,Id,\mathcal Q_{i}^{-1}\mathcal Q_{j}, a_{j,l},p_{j,l})$, $\mathcal{E}^{(i)}\left(\varepsilon,a_{i,l},p_{i,l}\right),~i=1,\dots,k$, $\mathcal{E}^{(3)}\left(\varepsilon,a_{i,l},a_{j,l},p_{i,l},p_{j,l}\right)$,
$\mathcal{E}^{(4)}\left(\varepsilon,a_{i,l},a_{j,l},p_{i,l},p_{j,l}\right)$,
$\mathcal{E}^{(5)}\left(\varepsilon,a_{i,l},a_{j,l},p_{i,l},p_{j,l}\right)$ are defined in Lemma \ref{lemma-5.2}.
Then there holds
\begin{equation*}
\begin{aligned}
I_{\varepsilon}(u)=\frac{8k\pi}{3}+\Sigma_{\mathcal{D},g}(\varepsilon,\vec{\mu},\vec{\xi},\vec{Q},\vec{a},\vec{p})+\mathcal{E}(\varepsilon,\vec{\mu},\vec{a},\vec{p})
\end{aligned}
\end{equation*}
and
\begin{equation*}
\begin{aligned}
& \frac{\partial I_{\varepsilon}(u)}{\partial \mu_i}=\frac{\partial \Sigma_{\mathcal{D},g}}{\partial \mu_i}+\frac{1}{\mu_i}\mathcal{E}(\varepsilon,\vec{\mu},\vec{a},\vec{p}),
\quad
\frac{\partial I_{\varepsilon}\left(u\right)}{\partial \xi^{(i)}}=\frac{\partial \Sigma_{\mathcal{D},g}}{\partial \xi^{(i)}}+\frac{\pp\mathcal{E}(\varepsilon,\vec{\mu},\vec{a},\vec{p})}{\pp\xi^{(i)}},\quad \frac{\partial I_{\varepsilon}(u)}{\partial \mathcal Q_{i}}=\frac{\partial \Sigma_{\mathcal{D},g}}{\partial \mathcal Q_{i}}+\mathcal{E}(\varepsilon,\vec{\mu},\vec{a},\vec{p}),
\\& \frac{\partial I_{\varepsilon}\left(u\right)}{\partial a_{i,l}}=\frac{\partial \Sigma_{\mathcal{D},g}}{\partial a_{i,l}}+\frac{\pp\mathcal{E}(\varepsilon,\vec{\mu},\vec{a},\vec{p})}{\pp a_{i,l}},
\quad \frac{\partial I_{\varepsilon}\left(u\right)}{\partial p_{i,l}}=\frac{\partial \Sigma_{\mathcal{D},g}}{\partial p_{i,l}}+\frac{\pp\mathcal{E}(\varepsilon,\vec{\mu},\vec{a},\vec{p})}{\pp p_{i,l}},
\end{aligned}
\end{equation*}
where $u=\sum_{i=1}^kz_{\mathcal{A}_i}=\sum_{i=1}^kP\mathcal Q_i\delta_{\mu_i,\xi^{(i)},a_i, p_i}\in Z.$
\end{prop}

\subsection{Main order of the parameters}\label{remark4.1}
Combining Propositions \ref{prop 4.1} and \ref{prop 6.4}, we formally derive the main order and the next order of the parameters $\mu$, $\xi_1$, $\xi_2$, $\mathcal Q$, $a_1$, $a_2$, $p_1$, $p_2$ for two bubbles.

{\bf $\bullet$ The main order of the parameters $\mu$, $\xi_1$, $\xi_2$,  $\mathcal Q$.}
Recall that from Proposition \ref{prop 4.1}, we have
\begin{equation*}
I_{\varepsilon}\left(P\mathcal{Q}\delta_{\mu,\xi,a,p}\right)=\frac{8}{3}\pi+F_{\mathcal{D},g}(\mu,\xi,\mathcal{Q},a,p)+\tilde{e}^{(1)}(\varepsilon,\mu,a,p)+\tilde{e}^{(2)}(\varepsilon,\mu,a,p),
\end{equation*}
where
\begin{align*}
&F_{\mathcal{D},g}(\mu,\xi,\mathcal{Q},a,p)
\\&:=8\pi\mu^4\frac{\partial^2 h_2^{(1)}}{\partial x \partial y}(\xi,\xi)
-4\pi\varepsilon \mu^2\left[\frac{\partial^2 \left(\mathcal Q^{-1} g\right)_1}{\partial x^2}(\xi,\omega)+\frac{\partial^2 \left(\mathcal Q^{-1} g\right)_2}{\partial x \partial y}(\xi,\omega)\right]
\\&\quad-16\pi a_1^2\mu^2\frac{\partial h_1^{(-1,1)}}{\partial x}(\xi,\xi)
-16\pi a_2^2\mu^2\frac{\partial h_2^{(-1,2)}}{\partial x}(\xi,\xi)
\\&\quad-\frac{4\pi}{3}p_1^2\mu^8\frac{\partial^4 h^{(2, 1)}_1}{\partial x^4}(\xi, \xi) +\frac{4\pi}{3}p_2^2\mu^8\frac{\partial^4 h^{(2, 2)}_2}{\partial x^4}(\xi, \xi)
\\&\quad-32\pi p_1a_1\mu^5\frac{\partial h_1^{(2,1)}}{\partial x}(\xi,\xi)-32\pi p_1a_2\mu^5\frac{\partial h_2^{(2,1)}}{\partial x}(\xi,\xi)
-32\pi p_2a_1\mu^5\frac{\partial h_1^{(2,2)}}{\partial x}(\xi,\xi)
-32\pi p_2a_2\mu^5\frac{\partial h_2^{(2,2)}}{\partial x}(\xi,\xi).
\end{align*}
For simplicity, in the following, we write
\begin{equation*}
\begin{aligned}
H(\xi)=2\frac{\partial^2 h_2^{(1)}}{\partial x \partial y}(\xi,\xi),\qquad d_{\mathcal{Q}} g(\xi)=\frac{\partial^2 (\mathcal{Q} g)_1}{\partial x^2}(\xi,\omega)+\frac{\partial^2(\mathcal{Q} g)_2}{\partial x \partial y}(\xi,\omega), ~\quad\xi,~ \omega \in \mathcal{D}.
\end{aligned}
\end{equation*}
Then differentiating $F_{\mathcal{D},g}(\mu,\xi,\mathcal{Q},a,p)$ with respect to the parameters $\mu$, $\xi$ and $\mathcal Q$,  while neglecting higher order terms, specifically those involving $a_l$, $p_l$ for $l=1,2$, we obtain
\begin{equation}\label{derivative1}
\left\{
\begin{aligned}
&\frac{\partial F_{\mathcal{D},g}(\mu,\xi,\mathcal{Q},a,p)}{\partial \mu}=0~\Longleftrightarrow~ 2\mu^2 H(\xi)-\varepsilon d_{\mathcal{Q}^{-1}}g(\xi)=0,
\\&\nabla_{\xi}F_{\mathcal{D},g}(\mu,\xi,\mathcal{Q},a,p)=0~\Longleftrightarrow~ \mu^2\nabla_{\xi} H(\xi)-\varepsilon \nabla_{\xi} d_{\mathcal{Q}^{-1}}g(\xi)=0,
\\&\frac{\partial F_{\mathcal{D},g}(\mu,\xi,\mathcal{Q},a,p)}{\partial \mathcal{Q}}=0~\Longleftrightarrow~ \frac{\partial}{\partial \mathcal{Q}}d_{\mathcal{Q}^{-1}}g(\xi)=0.
\end{aligned}
\right.
\end{equation}
Then $\mathcal{Q}$ is the identity matrix at main order. Specifically, we parameterize the rotation matrix by its Euler angles $\theta$, $\psi$ and $\phi$ as follows
\begin{equation*}
\mathcal Q^{-1} =\begin{pmatrix}
\cos\psi\cos\phi-\cos\theta\sin\phi\sin\psi & \cos\psi\sin\phi+\cos\theta\cos\phi\sin\psi & \sin\psi\sin\theta\\
-\sin\theta\sin\phi & \sin\theta\cos\phi & -\cos\theta\\
-\sin\psi\cos\phi-\cos\theta\sin\phi\cos\psi & -\sin\psi\sin\phi+\cos\theta\cos\phi\cos\psi & \cos\psi\sin\theta
\end{pmatrix}.
\end{equation*}
When $\theta = \frac{\pi}{2}$, $\phi = 0$ and $\psi = 0$, $\mathcal Q^{-1}$ reduces to the identity matrix.
In the case where $\phi = 0$ and $\psi = 0$, we have
\begin{equation*}
\mathcal Q^ {-1}=\begin{pmatrix}
1&0&0\\
0&\sin \theta & -\cos \theta \\
0&\cos \theta & \sin \theta
\end{pmatrix}
\end{equation*}
and
\begin{equation*}
\begin{aligned}
\frac{\partial}{\partial \mathcal{Q}}d_{\mathcal{Q}^{-1}}g(\xi)
&=\frac{\partial}{\partial \mathcal{Q}}\left[\frac{\partial^2 \left(\mathcal Q^{-1} g\right)_1}{\partial x^2}(\xi,\omega)+\frac{\partial^2 \left(\mathcal Q^{-1} g\right)_2}{\partial x \partial y}(\xi,\omega)\right]
\\& = \frac{\partial }{\partial \theta}\left(\frac{\partial^2}{\partial x^2}g_1(\xi, \omega) + \sin\theta \frac{\partial^2}{\partial x\partial y}g_2(\xi, \omega)\right)
\\&= \cos\theta \frac{\partial^2}{\partial x\partial y}g_2(\xi, \omega)
\\&= 0.
\end{aligned}
\end{equation*}
Therefore, the main order of $\theta$
is $\theta_0 = \frac{\pi}{2}$. Similarly, we find that $\phi_0= 0$ and $\psi_0 = 0$.

From (\ref{derivative1}), we have the following relations
$$\frac{\partial}{\partial \mathcal{Q}}d_{\mathcal{Q}^{-1}}g(\xi)=0,\quad
H(\xi)=\frac{\varepsilon d_{\mathcal{Q}^{-1}}g(\xi)}{2\mu^2}, \quad\nabla_{\xi} \log H(\xi)=2\nabla_{\xi} \log \left(d_{\mathcal{Q}^{-1}}g(\xi)\right).$$
From the definition of $d_{\mathcal{Q}^{-1}}g$, it holds that
$$d_{\mathcal{Q}^{-1}}g(\xi)=\frac{\partial^2 g}{\partial x^2}(\xi)\cdot (\mathcal{Q}e_1)
+\frac{\partial^2 g}{\partial x \partial y}(\xi)\cdot (\mathcal{Q}e_2),$$
where $\{\vec{e}_1, \vec{e}_2, \vec{e}_3\}$ is the standard basis of $\mathbb{R}^3$: $\vec{e}_1=(1,0,0)^{\mathrm{T}} $, $\vec{e}_2=(0,1,0)^{\mathrm{T}} $, $\vec{e}_3=(0,0,1)^{\mathrm{T}}$.
Note that $ d_{\mathcal{Q}^{-1}}g$ is positive. Using  \cite[Lemma 3.1.2]{Takeshi2000}, we have
\begin{equation*}
\begin{aligned}
\frac{\partial}{\partial \mathcal{Q}}d_{\mathcal{Q}^{-1}}g(\xi)=0 ~\Longrightarrow ~d_{\mathcal{Q}^{-1}}g(\xi)=\left(\left|\frac{\partial^2 g}{\partial x^2}\right|^2+\left|\frac{\partial^2 g}{\partial x \partial y}\right|^2 \pm 2\left|\left(\frac{\partial^2 g}{\partial x^2}\right)\wedge \left(\frac{\partial^2 g}{\partial x \partial y}\right)\right|\right)^{\frac{1}{2}}(\xi).
\end{aligned}
\end{equation*}
Thus, under the condition that $\left|\frac{\partial^2 g}{\partial x^2}\right|^2+\left|\frac{\partial^2 g}{\partial x \partial y}\right|^2\neq 0$, the system (\ref{derivative1}) becomes
\begin{equation*}
\begin{aligned}
&\frac{\partial}{\partial \mathcal{Q}}d_{\mathcal{Q}^{-1}}g(\xi)=0,\quad
\mu=\left(\frac{\varepsilon d_{\mathcal{Q}^{-1}}g(\xi)}{2H (\xi)}\right)^{\frac{1}{2}},
\quad
\nabla_{\xi} \left(\frac{ \left|\frac{\partial^2 g}{\partial x^2}\right|^2+\left|\frac{\partial^2 g}{\partial x \partial y}\right|^2 \pm 2\left|\left(\frac{\partial^2 g}{\partial x^2}\right)\wedge \left(\frac{\partial^2 g}{\partial x \partial y}\right)\right|}{H}\right)^{\frac{1}{2}}(\xi)=0.
\end{aligned}
\end{equation*}
and then we obtain the main order $\mu\sim\mu_0=\sqrt \varepsilon$, $\xi\sim \xi_0=\omega$.

{\bf $\bullet$ The next order of $\mu_i$.}
For $k=2$, we have (\ref{Ju30-energy}). Let us now consider the terms in (\ref{Ju30-energy}) that only contains $\mu_1$, $\mu_2$, $\xi$, $\zeta$ and $\varepsilon$, i.e.
\begin{align*}
\tilde I_\varepsilon(\mu_1, \mu_2, \xi, \zeta) : = &~ 4\pi\mu^4_1H(\xi)
-4\pi\varepsilon \mu^2_1 d_{Id}g(\xi) + 4\pi\mu^4_2H(\zeta)
-4\pi\varepsilon \mu^2_2d_{\mathcal{Q}_2^{-1}}g(\zeta)
\\& + q_{11}\left[\frac{-48\pi \left(\sigma _1^4-6 \sigma _2^2 \sigma _1^2+\sigma _2^4\right)}{|\sigma|^8}\mu_1^2\mu_2^2-4\pi \varepsilon\mu_2^2\frac{\partial^2g_1^{(1)}}{\partial x^2}(\zeta,\omega^{(1)})\right]
\\&+q_{12}\left[\frac{-192\pi\sigma_1\sigma_2(\sigma_1^2-\sigma_2^2)\mu_1^2\mu_2^2}{|\sigma|^8}
-4\pi \varepsilon\mu_2^2\frac{\partial^2 g^{(1)}_2}{\partial x\pp y}(\zeta,\omega^{(1)})\right]
\\&
+q_{21}\left[\frac{-192\pi\sigma_1\sigma_2(\sigma_1^2-\sigma_2^2)\mu_1^2\mu_2^2}{|\sigma|^8}
-4\pi \varepsilon\mu_2^2\frac{\partial^2 g^{(1)}_1}{\partial x^2}(\zeta,\omega^{(1)})\right]
\\&+q_{22}\left[\frac{48\pi\left(\sigma _1^4-6 \sigma _2^2 \sigma _1^2+\sigma _2^4\right)}{|\sigma|^8}\mu_1^2\mu_2^2
-4\pi \varepsilon\mu_2^2\frac{\partial^2 g^{(1)}_2}{\partial x \partial y}(\zeta,\omega^{(1)}) \right].
\end{align*}
Then by taking derivatives with respect to $\mu_1$ and $\mu_2$, we obtain that the critical points of $\tilde I_\varepsilon(\mu_1, \mu_2, \xi, \zeta)$ satisfy the following systems of $\mu_1$ and $\mu_2$,
\begin{align*}
& 16\pi\mu^2_1H(\xi)
-8\pi\varepsilon d_{Id}g(\xi)
\\&+ q_{11}\left[\frac{-48\pi \left(\sigma _1^4-6 \sigma _2^2 \sigma _1^2+\sigma _2^4\right)}{|\sigma|^8}2\mu_2^2\right]
+q_{12}\left[\frac{-192\pi\sigma_1\sigma_2(\sigma_1^2-\sigma_2^2)}{|\sigma|^8}2\mu_2^2\right]
\\& + q_{21}\left[\frac{-192\pi\sigma_1\sigma_2(\sigma_1^2-\sigma_2^2)}{|\sigma|^8}2\mu_2^2\right]+q_{22}\left[\frac{48\pi\left(\sigma _1^4-6 \sigma _2^2 \sigma _1^2+\sigma _2^4\right)}{|\sigma|^8}2\mu_2^2\right] = 0,
\end{align*}
and
\begin{align*}
& 16\pi\mu^2_2H(\zeta)
-8\pi\varepsilon d_{\mathcal{Q}_2^{-1}}g(\zeta)
\\&+ q_{11}\left[\frac{-48\pi \left(\sigma _1^4-6 \sigma _2^2 \sigma _1^2+\sigma _2^4\right)}{|\sigma|^8}2\mu_1^2-8\pi\varepsilon\frac{\partial^2 g^{(1)}_1}{\partial x^2}(\zeta,\omega^{(1)})\right]
\\& +q_{12}\left[\frac{-192\pi\sigma_1\sigma_2(\sigma_1^2-\sigma_2^2)}{|\sigma|^8}2\mu_1^2
-8\pi\varepsilon\frac{\partial^2 g^{(1)}_1}{\partial x\partial y}(\zeta,\omega^{(1)})\right]
+q_{21}\left[\frac{-192\pi\sigma_1\sigma_2(\sigma_1^2-\sigma_2^2)}{|\sigma|^8}2\mu_1^2
-8\pi\varepsilon\frac{\partial^2 g^{(1)}_2}{\partial x^2}(\zeta,\omega^{(1)})\right]
\\& +q_{22}\left[\frac{48\pi\left(\sigma _1^4-6 \sigma _2^2 \sigma _1^2+\sigma _2^4\right)}{|\sigma|^8}2\mu_1^2
-8\pi\varepsilon\frac{\partial^2 g^{(1)}_2}{\partial x\partial y}(\zeta,\omega^{(1)})\right] = 0.
\end{align*}
In the matrix form, we have
\begin{equation*}
\begin{pmatrix}
16\pi H(\xi) & A_{\xi, \zeta}\\
A_{\xi, \zeta} & 16\pi H(\zeta)
\end{pmatrix} \begin{pmatrix}
\mu_1^2\\
\mu_2^2
\end{pmatrix}= 8\pi\varepsilon
\begin{pmatrix}
d_{Id}g(\xi) \\
 d_{\mathcal{Q}_2^{-1}}g(\zeta)
\end{pmatrix}+8\pi\varepsilon
\begin{pmatrix}
B_{\xi, \zeta}\\
B_{\zeta, \xi}
\end{pmatrix},
\end{equation*}
where the constants $A_{\xi, \zeta}$, $B_{\xi, \zeta}$ and $B_{\zeta, \xi}$ are defined as
\begin{equation*}
\begin{aligned}
 A_{\xi, \zeta}: =
&~ q_{11}\left[2\frac{-48\pi \left(\sigma _1^4-6 \sigma _2^2 \sigma _1^2+\sigma _2^4\right)}{|\sigma|^8}\right]+q_{12}\left[2\frac{-192\pi\sigma_1\sigma_2(\sigma_1^2-\sigma_2^2)}{|\sigma|^8}\right]
\\&
+q_{21}\left[2\frac{-192\pi\sigma_1\sigma_2(\sigma_1^2-\sigma_2^2)}{|\sigma|^8}\right]+q_{22}\left[2\frac{48\pi\left(\sigma _1^4-6 \sigma _2^2 \sigma _1^2+\sigma _2^4\right)}{|\sigma|^8}\right]
\end{aligned}
\end{equation*}

\begin{equation*}
\begin{aligned}
&B_{\xi, \zeta}: = 0,\quad B_{\zeta, \xi}: = q_{11}\frac{\partial^2 g^{(1)}_1}{\partial x^2}(\zeta,\omega^{(1)})+q_{12}\frac{\partial^2 g^{(1)}_1}{\partial x\partial y}(\zeta,\omega^{(1)})
+q_{21}\frac{\partial^2 g^{(1)}_2}{\partial x^2}(\zeta,\omega^{(1)})
+q_{22}\frac{\partial^2 g^{(1)}_2}{\partial x \partial y}(\zeta,\omega^{(1)}).
\end{aligned}
\end{equation*}
An important observation is that when the concentration points $\xi$, $\zeta$ are sufficiently close to the boundary, the matrix
$$
\begin{pmatrix}
16\pi H(\xi) & A_{\xi, \zeta}\\
A_{\xi, \zeta} & 16\pi H(\zeta)
\end{pmatrix}=\begin{pmatrix}
32\pi\frac{\partial^2 h_2^{(1)}}{\partial x \partial y}(\xi,\xi) & A_{\xi, \zeta}\\
A_{\xi, \zeta} & 32\pi\frac{\partial^2 h_2^{(1)}}{\partial x \partial y}(\zeta,\zeta)
\end{pmatrix}
$$
is almost diagonal and positive definite. Therefore, we have $\mu_1^2 = b_1^2\varepsilon$, $\mu_2^2 = b_2^2\varepsilon$ with $b_1-1 = O((1-|\xi|^2)^4)$ and $b_2-1 = O((1-|\zeta|^2)^4)$ from Lemma \ref{evalues-at-xi-2}.

{\bf $\bullet$ The next order of $\xi^{(i)}$}.
Let us differentiate $\tilde I_\varepsilon(\mu_1, \mu_2, \xi, \zeta)$ with respect to $\xi$, we obtain
\begin{equation*}
\begin{aligned}
&4\pi\mu^4_1\nabla_\xi H(\xi)
-4\pi\varepsilon \mu^2_1\nabla_\xi d_{Id}g(\xi)
\\&+ q_{11}\left[8\pi\mu_1^2\mu_2^2\nabla_\xi\frac{\partial^2h_1^{(1)}}{\partial x^2}(\xi,\zeta)-4\pi\varepsilon \mu_1^2\nabla_\xi\frac{\partial^2 g^{(2)}_1}{\partial x^2}(\xi,\omega^{(2)})\right]
\\&+q_{12}\left[8\pi\mu_1^2\mu_2^2\nabla_\xi\frac{\partial^2h_2^{(1)}}{\partial x^2}(\xi,\zeta)-4\pi\varepsilon \mu_1^2\nabla_\xi\frac{\partial^2 g^{(2)}_2}{\partial x^2}(\xi,\omega^{(2)})\right]
\\ & +q_{21}\left[8\pi\mu_1^2\mu_2^2\nabla_\xi\frac{\partial^2h_1^{(1)}}{\partial x \partial y}(\xi,\zeta)-4\pi\varepsilon \mu_1^2\nabla_\xi\frac{\partial^2 g^{(2)}_1}{\partial x \partial y}(\xi,\omega^{(2)})\right]
\\&+q_{22}\left[8\pi\mu_1^2\mu_2^2\nabla_\xi\frac{\partial^2 h_2^{(1)}}{\partial x\partial y}(\xi,\zeta) -4\pi\varepsilon \mu_1^2\nabla_\xi\frac{\partial^2 g^{(2)}_2}{\partial x \partial y}(\xi,\omega^{(2)})\right] = 0.
\end{aligned}
\end{equation*}
Note that $d_{Id}g(\xi)=4\frac{\pp^2h_2^{(1)}}{\pp x\pp y}(\xi,\omega^{(1)})$, then we have the following approximate equation
\begin{align*}
& \left(8\pi\mu^4_1-16\pi\varepsilon\mu_1^2\right)\nabla_\xi^2\frac{\partial^2 h_2^{(1)}}{\partial x \partial y}(\omega^{(1)},\omega^{(1)})(\xi-\omega^{(1)})
\\& + q_{11}\left[8\pi\mu_1^2\mu_2^2\nabla_\xi\frac{\partial^2h_1^{(1)}}{\partial x^2}(\xi,\zeta)-4\pi\varepsilon \mu_1^2\nabla_\xi\frac{\partial^2 g^{(2)}_1}{\partial x^2}(\xi,\omega^{(2)})\right]
\\&+q_{12}\left[8\pi\mu_1^2\mu_2^2\nabla_\xi\frac{\partial^2h_2^{(1)}}{\partial x^2}(\xi,\zeta)-4\pi\varepsilon \mu_1^2\nabla_\xi\frac{\partial^2 g^{(2)}_2}{\partial x^2}(\xi,\omega^{(2)})\right]
\\&+q_{21}\left[8\pi\mu_1^2\mu_2^2\nabla_\xi\frac{\partial^2h_1^{(1)}}{\partial x \partial y}(\xi,\zeta)-4\pi\varepsilon \mu_1^2\nabla_\xi\frac{\partial^2 g^{(2)}_1}{\partial x \partial y}(\xi,\omega^{(2)})\right]
\\&+q_{22}\left[8\pi\mu_1^2\mu_2^2\nabla_\xi\frac{\partial^2 h_2^{(1)}}{\partial x\partial y}(\xi,\zeta) -4\pi\varepsilon \mu_1^2\nabla_\xi\frac{\partial^2 g^{(2)}_2}{\partial x \partial y}(\xi,\omega^{(2)})\right] = 0.
\end{align*}
Using Lemma \ref{evalues-at-xi-2}, we have
\begin{equation*}
\begin{aligned}
\xi = \omega^{(1)} + O((1-|\omega^{(1)}|^2)^6).
\end{aligned}
\end{equation*}
Similarly, we have
\begin{equation*}
\begin{aligned}
\zeta = \omega^{(2)} + O((1-|\omega^{(2)}|^2)^6).
\end{aligned}
\end{equation*}

{\bf $\bullet$ The next order of $\theta_i$, $\phi_i$, $\psi_i$}. We consider the following function involving $\mu_1$, $\mu_2$, $\xi$, $\zeta$, $\mathcal Q_1$ and $\mathcal Q_2$. Without loss of generality, we assume that $\mathcal Q_1=Id$.
\begin{align*}
&\hat I_\varepsilon(\mu_1, \mu_2, \xi, \zeta, \mathcal Q_1, \mathcal Q_2)
\\&:=4\pi\mu^4_1H(\xi)
-4\pi\varepsilon \mu^2_1 d_{Id}g(\xi) + 4\pi\mu^4_2H(\zeta)
-4\pi\varepsilon \mu^2_2d_{\mathcal{Q}_2^{-1}}g(\zeta)
\\&\quad + q_{11}\left[\frac{-48\pi \left(\sigma _1^4-6 \sigma _2^2 \sigma _1^2+\sigma _2^4\right)}{|\sigma|^8}\mu_1^2\mu_2^2+8\pi\mu_1^2\mu_2^2\frac{\partial^2h_1^{(1)}}{\partial x^2}(\xi,\zeta) -4\pi\varepsilon \mu_1^2\frac{\partial^2 g^{(2)}_1}{\partial x^2}(\xi,\omega^{(2)}) -4\pi\varepsilon \mu_2^2\frac{\partial^2 g^{(1)}_1}{\partial x^2}(\zeta,\omega^{(1)})\right]
\\&\quad+q_{12}\left[\frac{-192\pi\sigma_1\sigma_2(\sigma_1^2-\sigma_2^2)\mu_1^2\mu_2^2}{|\sigma|^8}
+8\pi\mu_1^2\mu_2^2\frac{\partial^2h_2^{(1)}}{\partial x^2}(\xi,\zeta)-4\pi\varepsilon \mu_1^2\frac{\partial^2 g^{(2)}_2}{\partial x^2}(\xi,\omega^{(2)})-4\pi\varepsilon \mu_2^2\frac{\partial^2g^{(1)}_1}{\partial x\partial y}(\zeta,\omega^{(1)})\right]
\\&
\quad+q_{21}\left[\frac{-192\pi\sigma_1\sigma_2(\sigma_1^2-\sigma_2^2)\mu_1^2\mu_2^2}{|\sigma|^8}
+8\pi\mu_1^2\mu_2^2\frac{\partial^2h_1^{(1)}}{\partial x \partial y}(\xi,\zeta) -4\pi\varepsilon \mu_1^2\frac{\partial^2g^{(2)}_1}{\partial x \partial y}(\xi,\omega^{(2)})-4\pi\varepsilon \mu_2^2\frac{\partial^2g^{(1)}_2}{\partial x^2}(\zeta,\omega^{(1)})\right]
\\&\quad+q_{22}\left[\frac{48\pi\left(\sigma _1^4-6 \sigma _2^2 \sigma _1^2+\sigma _2^4\right)}{|\sigma|^8}\mu_1^2\mu_2^2
+8\pi\mu_1^2\mu_2^2\frac{\partial^2 h_2^{(1)}}{\partial x\partial y}(\xi,\zeta)-4\pi\varepsilon \mu_1^2\frac{\partial^2 g^{(2)}_2}{\partial x \partial y}(\xi,\omega^{(2)}) -4\pi\varepsilon \mu_2^2\frac{\partial^2g^{(1)}_2}{\partial x\partial y}(\zeta,\omega^{(1)})\right].
\end{align*}
Then take derivative with respect to $\phi_2$ and set $\theta_2 = \frac{\pi}{2}$, $\psi_2 = 0$,
\begin{equation}\label{Mar25-1}
\begin{aligned}
\frac{\partial }{\partial \phi_2}\left(-4\pi\varepsilon \mu^2_2d_{\mathcal{Q}_2^{-1}}g(\zeta)\right)+ C_{\mathcal Q_2}= 0,
\end{aligned}
\end{equation}
here
\begin{align*}
&C_{\mathcal Q_2}:=\Bigg\{\frac{\partial q_{11}}{\partial \phi_2}\left[\frac{-48\pi \left(\sigma _1^4-6 \sigma _2^2 \sigma _1^2+\sigma _2^4\right)}{|\sigma|^8}\mu_1^2\mu_2^2+8\pi\mu_1^2\mu_2^2\frac{\partial^2h_1^{(1)}}{\partial x^2}(\xi,\zeta)\right.
\\&\qquad\qquad\qquad\qquad\left. -4\pi\varepsilon \mu_1^2\frac{\partial^2 g^{(2)}_1}{\partial x^2}(\xi,\omega^{(2)}) -4\pi\varepsilon \mu_2^2\frac{\partial^2 g^{(1)}_1}{\partial x^2}(\zeta,\omega^{(1)})\right]
\\&\qquad\qquad\quad +\frac{\partial q_{12} }{\partial \phi_2}\left[\frac{-192\pi\sigma_1\sigma_2(\sigma_1^2-\sigma_2^2)\mu_1^2\mu_2^2}{|\sigma|^8}
+8\pi\mu_1^2\mu_2^2\frac{\partial^2h_2^{(1)}}{\partial x^2}(\xi,\zeta)\right.
\\&\qquad\qquad\qquad\qquad\quad\left. -4\pi\varepsilon \mu_1^2\frac{\partial^2 g^{(2)}_2}{\partial x^2}(\xi,\omega^{(2)})-4\pi\varepsilon \mu_2^2\frac{\partial^2g^{(1)}_1}{\partial x\partial y}(\zeta,\omega^{(1)})\right]
\\&\qquad\qquad\quad +\frac{\partial q_{21}}{\partial \phi_2}\left[\frac{-192\pi\sigma_1\sigma_2(\sigma_1^2-\sigma_2^2)\mu_1^2\mu_2^2}{|\sigma|^8}
+8\pi\mu_1^2\mu_2^2\frac{\partial^2h_1^{(1)}}{\partial x \partial y}(\xi,\zeta)\right.\\
&\qquad\qquad\qquad\qquad\quad\left. -4\pi\varepsilon \mu_1^2\frac{\partial^2g^{(2)}_1}{\partial x \partial y}(\xi,\omega^{(2)})-4\pi\varepsilon \mu_2^2\frac{\partial^2g^{(1)}_2}{\partial x^2}(\zeta,\omega^{(1)})\right]
\\&\qquad\qquad\quad +\frac{\partial q_{22} }{\partial \phi_2}\left[\frac{48\pi\left(\sigma_1^4-6 \sigma_2^2 \sigma_1^2+\sigma_2^4\right)}{|\sigma|^8}\mu_1^2\mu_2^2
+8\pi\mu_1^2\mu_2^2\frac{\partial^2 h_2^{(1)}}{\partial x\partial y}(\xi,\zeta)\right.\\
&\qquad\qquad\qquad\qquad\quad\left. -4\pi\varepsilon \mu_1^2\frac{\partial^2 g^{(2)}_2}{\partial x \partial y}(\xi,\omega^{(2)}) -4\pi\varepsilon \mu_2^2\frac{\partial^2g^{(1)}_2}{\partial x\partial y}(\zeta,\omega^{(1)})\right]\Bigg\}.
\end{align*}
(\ref{Mar25-1}) can be rewritten as
\begin{equation*}
\begin{aligned}
&4\pi\varepsilon \mu^2_2\sin\phi_2\left[\frac{\partial^2 g^{(1)}_1}{\partial x^2}(\zeta,\omega^{(2)})+\frac{\partial^2 g^{(1)}_2}{\partial x \partial y}(\zeta,\omega^{(2)})\right]-4\pi\varepsilon \mu^2_2\cos\phi_2\left[\frac{\partial^2 g^{(1)}_2}{\partial x^2}(\zeta,\omega^{(2)})-\frac{\partial^2 g^{(1)}_1}{\partial x \partial y}(\zeta,\omega^{(2)})\right]+ C_{\mathcal Q_2}= 0.
\end{aligned}
\end{equation*}
From Lemma \ref{evalues-at-xi-2}, the term $-4\pi\varepsilon \mu^2_2\cos\phi_2\left[\frac{\partial^2 g^{(1)}_2}{\partial x^2}(\zeta,\omega^{(2)})-\frac{\partial^2 g^{(1)}_1}{\partial x \partial y}(\zeta,\omega^{(2)})\right]$ is negligible, therefore, we have the main order
\begin{equation*}
\begin{aligned}
&\phi_2 \approx \sin\phi_2 = \frac{-C_{\mathcal Q_2}}{4\pi\varepsilon \mu^2_2\left[\frac{\partial^2 g^{(1)}_1}{\partial x^2}(\zeta,\omega^{(2)})+\frac{\partial^2 g^{(1)}_2}{\partial x \partial y}(\zeta,\omega^{(2)})\right]} = O\left((1-|\omega^{(2)}|^2)^4\right) .
\end{aligned}
\end{equation*}
Similarly, we have
\begin{equation*}
\begin{aligned}
\theta_2 = \frac{\pi}{2} + O\left((1-|\omega^{(2)}|^2)^4\right),\quad
\psi_2 = O\left((1-|\omega^{(2)}|^2)^4\right),
\end{aligned}
\end{equation*}
and
\begin{equation*}
\begin{aligned}
\phi_1 = O\left((1-|\omega^{(1)}|^2)^4\right), \quad\theta_1 = \frac{\pi}{2} + O\left((1-|\omega^{(1)}|^2)^4\right),\quad \psi_1 = O\left((1-|\omega^{(1)}|^2)^4\right).
\end{aligned}
\end{equation*}

{\bf $\bullet$ The order of $a_{i,l}$ and $p_{i, l}$, $i,~l= 1, 2$.} Recalling the definitions in Proposition \ref{prop 4.1} and \eqref{J7-R1}, and differentiating
$$F_{\mathcal{D},g^{(1)}}(\mu_1,\xi^{(1)},\mathcal{Q}_1,a_{1,l},p_{1,l})+F^{(1)}_{\mathcal{D}}(\delta_1,\delta_2,Id,\mathcal{Q}_2,a_{1,l},p_{1,l})$$
with respect to the parameters $a_{1,1}$, $a_{1,2}$, $p_{1,1}$, $p_{1,2}$, the following relations hold
\begin{align*}
&\frac{\partial F_{\mathcal{D},g^{(1)}}(\mu_1,\xi^{(1)},\mathcal{Q}_1,a_{1,l},p_{1,l})}{\partial a_{1,1}}+\frac{\partial F^{(1)}_{\mathcal{D}}(\delta_1,\delta_2,Id,\mathcal{Q}_2,a_{1,l},p_{1,l})}{\partial a_{1,1}}=0
\\&\Longleftrightarrow
 -a_{1,1}\mu_1^2\frac{\partial h_1^{(-1,1)}}{\partial x}(\xi,\xi)-p_{1,1}\mu_1^5\frac{\partial h_1^{(2,1)}}{\partial x}(\xi,\xi)
-p_{1,2}\mu_1^5\frac{\partial h_1^{(2,2)}}{\partial x}(\xi,\xi)+\mu_1\mu_2^2\bar A_{1,1}=0,
\\&
\frac{\partial F_{\mathcal{D},g^{(1)}}(\mu_1,\xi^{(1)},\mathcal{Q}_1,a_{1,l},p_{1,l})}{\partial a_{1,2}}+\frac{\partial F^{(1)}_{\mathcal{D}}(\delta_1,\delta_2,Id,\mathcal{Q}_2,a_{1,l},p_{1,l})}{\partial a_{1,2}}=0
\\&\Longleftrightarrow
 -a_{1,2}\mu_1^2\frac{\partial h_1^{(-1,1)}}{\partial x}(\xi,\xi)-p_{1,1}\mu_1^5\frac{\partial h_2^{(2,1)}}{\partial x}(\xi,\xi)
-p_{1,2}\mu_1^5\frac{\partial h_2^{(2,2)}}{\partial x}(\xi,\xi)+\mu_1\mu_2^2\bar A_{1,2}=0,
\\&
\frac{\partial F_{\mathcal{D},g^{(1)}}(\mu_1,\xi^{(1)},\mathcal{Q}_1,a_{1,l},p_{1,l})}{\partial p_{1,1}}+\frac{\partial F^{(1)}_{\mathcal{D}}(\delta_1,\delta_2,Id,\mathcal{Q}_2,a_{1,l},p_{1,l})}{\partial p_{1,1}}=0
\\&\Longleftrightarrow-p_{1,1}\mu_1^8\frac{\partial^4 h^{(2, 1)}_1}{\partial x^4}(\xi, \xi)-12 a_{1,1}\mu_1^5\frac{\partial h_1^{(2,1)}}{\partial x}(\xi,\xi)-12a_{1,2}\mu_1^5\frac{\partial h_2^{(2,1)}}{\partial x}(\xi,\xi)=0,
\\&
\frac{\partial F_{\mathcal{D},g^{(1)}}(\mu_1,\xi^{(1)},\mathcal{Q}_1,a_{1,l},p_{1,l})}{\partial p_{1,2}}+\frac{\partial F^{(1)}_{\mathcal{D}}(\delta_1,\delta_2,Id,\mathcal{Q}_2,a_{1,l},p_{1,l})}{\partial p_{1,2}}=0
\\&
\Longleftrightarrow
p_{1,2}\mu_1^8\frac{\partial^4 h^{(2, 2)}_2}{\partial x^4}(\xi, \xi) -12a_{1,1}\mu_1^5\frac{\partial h_1^{(2,2)}}{\partial x}(\xi,\xi)-12 a_{1,2}\mu_1^5\frac{\partial h_2^{(2,2)}}{\partial x}(\xi,\xi)=0.
\end{align*}
Here
\begin{equation*}
\begin{aligned}
\bar A_{1,2}: =q_{11}\frac{-32\pi\sigma_1 \left(\sigma_1^2-3 \sigma_2^2\right)}{|\sigma|^6}+q_{12}\frac{-32\pi\sigma_2\left(3\sigma_1^2-\sigma_2^2\right)}{|\sigma|^6}+q_{21}\frac{-32\pi\sigma_2\left(3\sigma_1^2- \sigma_2^2\right)}{|\sigma|^6}
+q_{22}\frac{32\pi\sigma_1\left(\sigma_1^2-3\sigma_2^2\right)}{|\sigma|^6},
\end{aligned}
\end{equation*}
\begin{equation*}
\begin{aligned}
\bar A_{1,2}: =q_{11}\frac{32\pi \sigma_2 \left(3\sigma_1^2-\sigma _2^2\right)}{|\sigma|^6}+q_{12}\frac{-32\pi \sigma_1 \left(\sigma_1^2-3\sigma_2^2\right)}{|\sigma|^6}+q_{21}\frac{-32\pi\sigma_1\left(\sigma_1^2-3 \sigma_2^2\right)}{|\sigma|^6}+q_{22}\frac{-32\pi\sigma_2\left(3\sigma_1^2- \sigma_2^2\right)}{|\sigma|^6}.
\end{aligned}
\end{equation*}
Observe that the coefficient matrix of this system
\begin{equation*}
\begin{pmatrix}
-\mu_1^2\frac{\partial h_1^{(-1,1)}}{\partial x}(\xi,\xi) & 0 & -\mu_1^5\frac{\partial h_1^{(2,1)}}{\partial x}(\xi,\xi) & -\mu_1^5\frac{\partial h_1^{(2,2)}}{\partial x}(\xi,\xi) \\
0 & -\mu_1^2\frac{\partial h_1^{(-1,1)}}{\partial x}(\xi,\xi) & -\mu_1^5\frac{\partial h_2^{(2,1)}}{\partial x}(\xi,\xi) &
-\mu_1^5\frac{\partial h_2^{(2,2)}}{\partial x}(\xi,\xi)\\
-12\mu_1^5\frac{\partial h_1^{(2,1)}}{\partial x}(\xi,\xi) & -12\mu_1^5\frac{\partial h_2^{(2,1)}}{\partial x}(\xi,\xi) & -\mu_1^8\frac{\partial^4 h^{(2, 1)}_1}{\partial x^4}(\xi, \xi) & 0\\
-12\mu_1^5\frac{\partial h_1^{(2,2)}}{\partial x}(\xi,\xi) & -12\mu_1^5\frac{\partial h_2^{(2,2)}}{\partial x}(\xi,\xi) & 0 & \mu_1^8\frac{\partial^4 h^{(2, 2)}_2}{\partial x^4}(\xi, \xi)
\end{pmatrix}
\end{equation*}
is positive definite, solving this system, we obtain that
$$
a_{1, 1} = O\left(\mu_0(1-|\omega^{(1)}|^2)^2\right),\quad a_{1, 2} =O\left(\mu_0(1-|\omega^{(1)}|^2)^2\right),
$$
$$
p_{1, 1} = O\left(\frac{(1-|\omega^{(1)}|^2)^5}{\mu_0^2}\right),\quad p_{1, 2} = O\left(\frac{(1-|\omega^{(1)}|^2)^5}{\mu_0^2}\right).
$$
The same results hold for $a_{2,l}$ and $p_{2,l}$ with $l=1,2$.

\section{Proof of Theorem \ref{thm}}\label{proof-of-the-main-theorem}
In this section, we give the proof of Theorem \ref{thm}. We begin with a special case of Lemma \ref{lemma7.1a} with $\omega = (\rho, 0)$.
\begin{lemma}\label{lemma7.1}
Assume $\rho\in (0, 1)$, $\tilde g_{\rho}:\partial \mathcal{D}\to \mathbb{R}^3$ is the function defined by
\begin{equation*}
\tilde g_{\rho}(x,y) =2\left(\frac{(x-\rho)^2-y^2}{\left((x-\rho)^2 + y^2\right)^2}, \frac{2\left(x-\rho\right) y}{\left((x-\rho)^2 + y^2\right)^2}, 0\right),\quad (x,y)\in \partial \mathcal{D}.
\end{equation*}
Then the harmonic extension on $\mathcal{D}$ of $\tilde g_\rho$ is given by
\begin{equation*}
g_\rho(x, y) = 2\left(\frac{\left(\rho\left(x^2+y^2\right)-x+y\right)\left(\rho\left(x^2+y^2\right)-x-y\right)}{\left(\rho^2\left(x^2+y^2\right)-2\rho x +1\right)^2}, \frac{2\left(x-\rho\left(x^2+y^2\right)\right)y}{\left(\rho^2\left(x^2+y^2\right)-2\rho x +1\right)^2}, 0\right),\quad (x,y)\in \mathcal{D}.
\end{equation*}
\end{lemma}
Recall from Subsection \ref{remark4.1} that the critical value of $d_{\mathcal{Q}^{-1}} g$ is
\begin{equation*}
d_{\mathcal{Q}^{-1}} g=\left(\left|\frac{\partial^2 g}{\partial x^2}\right|^2+\left|\frac{\partial^2 g}{\partial x \partial y}\right|^2 \pm 2\left|\left(\frac{\partial^2 g}{\partial x^2}\right)\wedge \left(\frac{\partial^2 g}{\partial x \partial y}\right)\right|\right)^{\frac{1}{2}}.
\end{equation*}
For the function $g_\rho$ and from direct computations we know
\begin{equation*}
\left(\left|\frac{\partial^2 g_\rho}{\partial x^2}\right|^2+\left|\frac{\partial^2 g_\rho}{\partial x \partial y}\right|^2 + 2\left|\left(\frac{\partial^2 g_\rho}{\partial x^2}\right)\wedge \left(\frac{\partial^2 g_\rho}{\partial x \partial y}\right)\right|\right)^{\frac{1}{2}}(x, y) = \frac{8\sqrt{4\rho\left(\rho\left(x^2+y^2\right)+x\right)+1}}{\left(\rho^2\left(x^2+y^2\right)-2x \rho+1\right)^2},
\end{equation*}
and also it holds that
\begin{equation*}
\left(\left|\frac{\partial^2 g_\rho}{\partial x^2}\right|^2+\left|\frac{\partial^2 g_\rho}{\partial x \partial y}\right|^2 - 2\left|\left(\frac{\partial^2 g_\rho}{\partial x^2}\right)\wedge \left(\frac{\partial^2 g_\rho}{\partial x \partial y}\right)\right|\right)^{\frac{1}{2}}(x, y) = 0.
\end{equation*}
Define
\begin{equation*}
W_\rho: = \left(\frac{\left|\frac{\partial^2 g_\rho}{\partial x^2}\right|^2+\left|\frac{\partial^2 g_\rho}{\partial x \partial y}\right|^2 + 2\left|\left(\frac{\partial^2 g_\rho}{\partial x^2}\right)\wedge\left(\frac{\partial^2 g_\rho}{\partial x \partial y}\right)\right|}{H}\right)^{\frac{1}{2}},
\end{equation*}
then we have
\begin{equation*}
W_\rho(\xi) = \frac{8\sqrt{4\rho\left(\rho\left(\xi_1^2+\xi_2^2\right)+\xi_1\right)+1}}{\left(\rho^2\left(\xi_1^2+\xi_2^2\right)-2\rho\xi_1+1\right)^2\sqrt H}.
\end{equation*}
Since the domain we are considering is the unit disk, it holds that\begin{equation*}
H(\xi)=2\frac{\partial^2 h_2^{(1)}}{\partial x \partial y}(\xi,\xi)=\frac{4(1+2|\xi|^2)}{(1-|\xi|^2)^4},
\end{equation*}
and
\begin{equation*}
W_{\rho}(\xi) = \frac{4\sqrt{4\rho\left(\rho\left(\xi_1^2+\xi_2^2\right)+\xi_1\right)+1}(1-|\xi|^2)^2}{\left(\rho^2\left(\xi_1^2+\xi_2^2\right)-2\rho\xi_1+1\right)^2\sqrt{1+2|\xi|^2}}.
\end{equation*}
Through direct computations, $W_{\rho}$ attains a non-degenerate global maximum at $\left(\rho,0\right)$ and
$$W_{\rho}\left(\rho,0\right)=\frac{4\left(\rho^2-1\right)^2\sqrt{4\rho \left(\rho^3+\rho\right)+1}}{\sqrt{2\rho^2+1}\left(\rho^4-2\rho^2+1\right)^2}.
$$
The Hessian of $W_{\rho}$ at the point $\left(\rho,0\right)$ is given by
\begin{equation*}
\begin{pmatrix}
-\frac{24\sqrt{\left(2\rho^2+1\right)^2}\left(3\rho^4+2\rho^2+1\right)}{\left(\rho^2-1\right)^4 \left(2\rho^2+1\right)^{5/2}} & 0 \\
0 &-\frac{24\sqrt{\left(2\rho^2+1\right)^2}\left(3\rho^4+2\rho^2+1\right)}{\left(\rho^2-1\right)^4 \left(2\rho^2+1\right)^{5/2}}\\
\end{pmatrix}.
\end{equation*}
Observe that when $\rho\in (0,1)$ and $\rho$ is close to 1, both the eigenvalues of the Hessian for $W_{\rho}$ at its critical point $\left(\rho,0\right)$ are quite negative, so it is natural that the interaction of this datum with the bubble will be stronger than the interaction among different bubbles. This fact will be crucial in our construction of gluing $k$ different bubbles.

To proceed, we parameterize the rotation matrix by its Euler angles $\theta$, $\psi$ and $\phi$ as follows,
\begin{equation*}
\mathcal Q =\begin{pmatrix}
\cos\psi\cos\phi-\cos\theta\sin\phi\sin\psi & -\sin\theta\sin\phi
& -\sin\psi\cos\phi-\cos\theta\sin\phi\cos\psi \\
\cos\psi\sin\phi+\cos\theta\cos\phi\sin\psi & \sin\theta\cos\phi &-\sin\psi\sin\phi+\cos\theta\cos\phi\cos\psi \\
\sin\psi\sin\theta & -\cos\theta  & \cos\psi\sin\theta
\end{pmatrix}.
\end{equation*}
Note that when $\theta = \frac{\pi}{2}$, $\phi = 0$ and $\psi = 0$, $\mathcal Q$ is the identity matrix.
 Consequently, the function $d_{\mathcal Q^{-1}} g(\xi)$ can be expressed as follows
\begin{align*}
d_{\mathcal Q^{-1}} g(\xi)& =\frac{\partial^2g}{\partial x^2}\cdot\left( \mathcal Q e_1\right)+\frac{\partial^2g}{\partial x\partial y}\cdot \left(\mathcal Qe_2\right)
\\& = \frac{\partial^2g}{\partial x^2}\cdot
\begin{pmatrix}
\cos\psi\cos\phi-\cos\theta\sin\phi\sin\psi\\
\cos\psi\sin\phi+\cos\theta\cos\phi\sin\psi \\
\sin\psi\sin\theta
\end{pmatrix}
+\frac{\partial^2g}{\partial x\partial y}\cdot
\begin{pmatrix}
-\sin\theta\sin\phi\\
\sin\theta\cos\phi\\
 -\cos\theta
\end{pmatrix}
\\& = \frac{\partial^2g_1}{\partial x^2}\left(\cos\psi\cos\phi-\cos\theta\sin\phi\sin\psi\right) + \frac{\partial^2g_2}{\partial x^2}\left(\cos\psi\sin\phi+\cos\theta\cos\phi\sin\psi\right)
\\&\quad +\frac{\partial^2g_1}{\partial x\partial y}\left(-\sin\theta\sin\phi\right) + \frac{\partial^2g_2}{\partial x\partial y}\left(\sin\theta\cos\phi\right).\qquad\qquad\qquad
\end{align*}
From direct computations, we have
\begin{equation*}
\begin{aligned}
\frac{\partial^2g_1}{\partial x^2}(x, y) &= \frac{\tilde f_1(x, y)}{\left(\rho^2\left(x^2+y^2\right)-2\rho x+1\right)^4},\qquad\
\frac{\partial^2g_2}{\partial x^2}(x, y) = \frac{\tilde f_2(x, y)}{\left(\rho^2\left(x^2+y^2\right)-2\rho x +1\right)^4},\\
\frac{\partial^2g_1}{\partial x \partial y}(x, y)& = -\frac{\tilde f_2(x, y)}{\left(\rho^2\left(x^2+y^2\right)-2\rho x+1\right)^4},\qquad
\frac{\partial^2g_2}{\partial x \partial y}(x, y) = \frac{\tilde f_1(x, y)}{\left(\rho^2\left(x^2+y^2\right)-2\rho x+1\right)^4},
\end{aligned}
\end{equation*}
for $g = g_\rho$, the functions $\tilde f_1(x, y)$ and $\tilde f_2(x, y)$ are defined as follows
\begin{equation*}
\begin{aligned}
&\tilde f_1(x, y)
\\&: = 4\rho\left(2\rho^4x\left(x^2-3y^2\right)\left(x^2+y^2\right)+8\rho^2x\left(x^2+3y^2\right)-2 \rho\left(x^2+7y^2\right)-\rho^3\left(7x^4+6x^2y^2-9y^4\right)-2x\right)+4,
\\&\tilde f_2(x, y) : = 8\rho y\left(2\rho^3x\left(3x^2+5y^2\right)+\rho^4\left(-3x^4-2x^2y^2+y^4\right)-6\rho x-8\rho^2y^2+3\right).
\end{aligned}
\end{equation*}
Therefore, we obtain
\begin{align*}
d_{\mathcal Q^{-1}} g(\xi)
=&~\frac{\tilde f_1(\xi_1, \xi_2)}{\left(\rho^2\left(\xi_1^2+\xi_2^2\right)-2\rho\xi_1+1\right)^4}\left(\cos\psi\cos\phi-\cos\theta\sin\phi\sin\psi\right)
\\&+ \frac{\tilde f_2(\xi_1, \xi_2)}{\left(\rho^2\left(\xi_1^2+\xi_2^2\right)-2\rho\xi_1+1\right)^4}\left(\cos\psi\sin\phi+\cos\theta\cos\phi\sin\psi\right)
\\& -\frac{\tilde f_2(\xi_1, \xi_2)}{\left(\rho^2\left(\xi_1^2+\xi_2^2\right)-2\rho\xi_1+1\right)^4}\left(-\sin\theta\sin\phi\right)
+ \frac{\tilde f_1(\xi_1, \xi_2)}{\left(\rho^2\left(\xi_1^2+\xi_2^2\right)-2\rho\xi_1+1\right)^4}\left(\sin\theta\cos\phi\right).
\end{align*}
Based on these facts, we have the following lemma.
\begin{lemma}
Let $\rho\in (0,1)$ be close to 1, and let $g_\rho$ be the harmonic function defined in Lemma \ref{lemma7.1}. Suppose $\lambda$ is a large positive constant, then there exists a constant $C_0>0$, independent of $\varepsilon$, $\mu$ and $\rho$ such that
\begin{equation*}
\nabla_{\chi} F_{\mathcal D, g_\rho}(\chi)\cdot \left(\chi^{(1)}-\chi_0\right) \geq
C_0\varepsilon^2\lambda^{2}(1-\rho)^{4}\quad \text{ on }\ \partial^{(1)} \mathcal T_\lambda,
\end{equation*}
\begin{equation*}
\nabla_{\chi} F_{\mathcal D, g_\rho}(\chi)\cdot \left(\chi^{(2)}-\chi_0\right) \geq
C_0\varepsilon^2\lambda^{2}(1-\rho)^{2}\quad\text{ on }\ \partial^{(2)} \mathcal T_\lambda,
\end{equation*}
and the topological degree satisfies $\text{deg}\left(\nabla_\chi F_{\mathcal D, g_\rho},\mathcal T_\lambda, 0\right) = 1$. Here,  $F_{\mathcal D,g_\rho}$ is the function defined in Proposition \ref{prop 4.1}.
The parameter set $\mathcal T_\lambda$ and its boundary components are defined as follows
\begin{equation*}
\begin{aligned}
\mathcal T_\lambda:=
 \Bigg\{\chi:=&( \mu,\xi_1, \xi_2, \theta, \phi, \psi, a_1, a_2, p_1, p_2)\ |\ \left|\mu - \mu_0\right|\leq\lambda(1-\rho)^4\mu_0,\ \left|\xi_1-\rho\right|\leq\lambda(1-\rho)^{5},
\\&\qquad\qquad  |\xi_2|\leq\lambda(1-\rho)^{5},\
|\theta-\frac{\pi}{2}|\leq\lambda(1-\rho)^4,\ |\phi|\leq\lambda(1-\rho)^4,
\ |\psi|\leq\lambda(1-\rho)^4,
\\&\qquad\qquad |a_1|\leq \lambda(1-\rho)^2\mu_0,\ |a_2|\leq \lambda(1-\rho)^2\mu_0,\ |p_1|\leq\frac{\lambda(1-\rho)^5}{\mu_0^2},\ |p_2|\leq \frac{\lambda(1-\rho)^5}{\mu_0^2}\Bigg\},
\end{aligned}
\end{equation*}

\begin{align*}
\partial^{(1)} \mathcal T_\lambda&:= \partial \mathcal T_\lambda\setminus \partial^{(2)}\mathcal T_\lambda, \\
\partial^{(2)} \mathcal T_\lambda&:= \Bigg\{\chi\in \mathcal T_\lambda\ |\ |a_1| =\lambda(1-\rho)^2\mu_0\text{ or } |a_2| =\lambda(1-\rho)^2\mu_0\text{ or }|p_1|=\frac{\lambda(1-\rho)^5}{\mu_0^2}\text{ or }|p_2|=\frac{\lambda(1-\rho)^5}{\mu_0^2}\Bigg\}.
\end{align*}
And the variables sets are
\begin{align*}
&\chi^{(1)} := (\mu, \xi_1, \xi_2, \theta, \phi, \psi,0,0,0,0), \quad \chi^{(2)} :=\chi= ( \mu,\xi_1, \xi_2, \theta, \phi, \psi, a_1, a_2, p_1, p_2),
\\ &\chi_0 := \left(\mu_0, \rho, 0, \tfrac{\pi}{2}, 0, 0,0, 0, 0, 0\right),
\end{align*}
where $\mu_0 := \sqrt{\varepsilon}$.
\end{lemma}
\begin{proof}
Recall the definition of $F_{\mathcal D,g_\rho}$ in Proposition \ref{prop 4.1}, we have
\begin{align*}
&F_{\mathcal{D},g_\rho}(\mu,\xi,\mathcal{Q},a,p)
\\&=8\pi\mu^4\frac{\partial^2 h_2^{(1)}}{\partial x \partial y}(\xi,\xi)
-4\pi\varepsilon \mu^2\left[\frac{\partial^2 \left(\mathcal Q^{-1} g\right)_1}{\partial x^2}(\xi,\omega)+\frac{\partial^2 \left(\mathcal Q^{-1} g\right)_2}{\partial x \partial y}(\xi,\omega)\right]
\\&\quad-16\pi a_1^2\mu^2\frac{\partial h_1^{(-1,1)}}{\partial x}(\xi,\xi)
-16\pi a_2^2\mu^2\frac{\partial h_2^{(-1,2)}}{\partial x}(\xi,\xi)
-\frac{4\pi}{3}p_1^2\mu^8\frac{\partial^4 h^{(2, 1)}_1}{\partial x^4}(\xi, \xi) +\frac{4\pi}{3}p_2^2\mu^8\frac{\partial^4 h^{(2, 2)}_2}{\partial x^4}(\xi, \xi)
\\&\quad-32\pi p_1a_1\mu^5\frac{\partial h_1^{(2,1)}}{\partial x}(\xi,\xi)-32\pi p_1a_2\mu^5\frac{\partial h_2^{(2,1)}}{\partial x}(\xi,\xi)
-32\pi p_2a_1\mu^5\frac{\partial h_1^{(2,2)}}{\partial x}(\xi,\xi)
-32\pi p_2a_2\mu^5\frac{\partial h_2^{(2,2)}}{\partial x}(\xi,\xi)
\\&=4\pi\mu^4H(\xi)
-4\pi\varepsilon \mu^2d_{\mathcal Q^{-1}} g(\xi)
-16\pi a_1^2\mu_0^2\frac{\partial h_1^{(-1,1)}}{\partial x}(\omega,\omega)
-16\pi a_2^2\mu_0^2\frac{\partial h_2^{(-1,2)}}{\partial x}(\omega,\omega)
\\&\quad-\frac{4\pi}{3}p_1^2\mu_0^8\frac{\partial^4 h^{(2, 1)}_1}{\partial x^4}(\omega, \omega) +\frac{4\pi}{3}p_2^2\mu_0^8\frac{\partial^4 h^{(2, 2)}_2}{\partial x^4}(\omega, \omega)
-32\pi p_1a_1\mu_0^5\frac{\partial h_1^{(2,1)}}{\partial x}(\omega, \omega)-32\pi p_1a_2\mu_0^5\frac{\partial h_2^{(2,1)}}{\partial x}(\omega, \omega)
\\&\quad-32\pi p_2a_1\mu_0^5\frac{\partial h_1^{(2,2)}}{\partial x}(\omega, \omega)
-32\pi p_2a_2\mu_0^5\frac{\partial h_2^{(2,2)}}{\partial x}(\omega, \omega)+\text{h.o.t.}
\\& =4\pi\mu^4H(\xi)-4\pi\varepsilon \mu^2\left[\frac{\tilde f_1(\xi_1, \xi_2)}{\left(\rho^2\left(\xi_1^2+\xi_2^2\right)-2\rho\xi_1+1\right)^4}\left(\cos\psi\cos\phi-\cos\theta\sin\phi\sin\psi\right)\right.
\\&\qquad\qquad\qquad\qquad \qquad
+ \frac{\tilde f_2(\xi_1, \xi_2)}{\left(\rho^2\left(\xi_1^2+\xi_2^2\right)-2\rho\xi_1+1\right)^4}\left(\cos\psi\sin\phi+\cos\theta\cos\phi\sin\psi\right)
\\&\qquad\qquad\qquad\qquad\qquad \left. -\frac{\tilde f_2(\xi_1, \xi_2)}{\left(\rho^2\left(\xi_1^2+\xi_2^2\right)-2\rho\xi_1+1\right)^4}\left(-\sin\theta\sin\phi\right)
 + \frac{\tilde f_1(\xi_1, \xi_2)}{\left(\rho^2\left(\xi_1^2+\xi_2^2\right)-2\rho\xi_1+1\right)^4}\left(\sin\theta\cos\phi\right)\right]
\\&\quad-16\pi a_1^2\mu_0^2\frac{\partial h_1^{(-1,1)}}{\partial x}(\omega,\omega)
-16\pi a_2^2\mu_0^2\frac{\partial h_2^{(-1,2)}}{\partial x}(\omega,\omega)
-\frac{4\pi}{3}p_1^2\mu_0^8\frac{\partial^4 h^{(2, 1)}_1}{\partial x^4}(\omega, \omega) +\frac{4\pi}{3}p_2^2\mu_0^8\frac{\partial^4 h^{(2, 2)}_2}{\partial x^4}(\omega, \omega)
\\&\quad-32\pi p_1a_1\mu_0^5\frac{\partial h_1^{(2,1)}}{\partial x}(\omega, \omega)-32\pi p_1a_2\mu_0^5\frac{\partial h_2^{(2,1)}}{\partial x}(\omega, \omega)
\\&\quad-32\pi p_2a_1\mu_0^5\frac{\partial h_1^{(2,2)}}{\partial x}(\omega, \omega)
-32\pi p_2a_2\mu_0^5\frac{\partial h_2^{(2,2)}}{\partial x}(\omega, \omega)+\text{h.o.t.}
\\&:= \tilde F_{\mathcal D, g_\rho}(\mu,\xi,\mathcal{Q},a,p)+\text{h.o.t.}
\end{align*}
Then by direct computations, we have
\begin{equation*}
\nabla_\chi \tilde F_{\mathcal D, g_\rho}\left( \mu_0, \rho, 0,\frac{\pi}{2}, 0, 0, 0, 0, 0, 0\right) = \vec{0}
\end{equation*}
and
\begin{equation*}
\nabla_\chi^2\tilde F_{\mathcal D, g_\rho}\left(\mu_0, \rho, 0, \frac{\pi}{2}, 0, 0, 0, 0, 0, 0\right) = A_{\rho, \varepsilon},
\end{equation*}
where
\begin{equation*}
A_{\rho, \varepsilon} =
\begin{pmatrix}
A_{1,1} & A_{1,2} & 0 & 0 & 0 & 0 & 0 & 0 & 0 & 0\\
A_{2,1} & A_{2,2} & 0 & 0 & 0 & 0 & 0 & 0 & 0 & 0\\
0 & 0 & A_{3,3} & 0 & A_{3,5} & 0 & 0 & 0 & 0 & 0\\
0 & 0 & 0 & A_{4,4} & 0 & 0 & 0 & 0 & 0 & 0\\
0 & 0 & A_{5,3} & 0 & A_{5,5} & 0 & 0 & 0 & 0 & 0\\
0 & 0 & 0 & 0 & 0 & A_{6,6} & 0 & 0 & 0 & 0\\
0 & 0 & 0 & 0 & 0 & 0 & A_{7,7} & 0 & A_{7,9} & 0\\
0 & 0 & 0 & 0 & 0 & 0 & 0 & A_{8,8} & 0 & A_{8,10}\\
0 & 0 & 0 & 0 & 0 & 0 & A_{9,7} & 0 & A_{9,9} & 0\\
0 & 0 & 0 & 0 & 0 & 0 & 0 & A_{10,8} & 0 & A_{10,10}\\
\end{pmatrix}
\end{equation*}
with
\begin{align*}
A_{1,1} & =\frac{128\pi\left(2\rho^2+1\right)\varepsilon}{\left(\rho^2-1\right)^4},
\qquad
A_{1,2} =-\frac{384\pi\left(\rho^3+\rho\right)\varepsilon^{3/2}}{\left(\rho^2-1\right)^5},\qquad
A_{2,1} = -\frac{384\pi\left(\rho^3+\rho\right)\varepsilon^{3/2}}{\left(\rho^2-1\right)^5} ,\\
A_{2,2} &= \frac{192\pi\left(3\rho^4+6\rho^2+1\right)\varepsilon^2}{\left(\rho^2-1\right)^6},\qquad
A_{3,3}=\frac{192\pi\left(3\rho^4+6\rho^2+1\right)\varepsilon^2}{\left(\rho^2-1\right)^6},\qquad A_{3,5} =\frac{192\pi\rho\left(\rho^2+1\right)\varepsilon^2}{\left(\rho^2-1\right)^5},\\
A_{4,4} & =\frac{16\pi\left(2\rho^2+1\right)\varepsilon^2}{\left(\rho^2-1\right)^4},\quad
A_{5,3} = \frac{192\pi\rho\left(\rho^2+1\right)\varepsilon^2}{\left(\rho^2-1\right)^5},\quad
A_{5,5} =\frac{32\pi\left(2\rho^2+1\right)\varepsilon^2}{\left(\rho^2-1\right)^4}, \quad
A_{6,6} =\frac{16\pi\left(2\rho^2+1\right)\varepsilon^2}{\left(\rho^2-1\right)^4},\\
A_{7,7} &= \frac{64\pi\varepsilon}{(1-\rho^2)^2},\qquad  A_{7,9} = \frac{128\pi\varepsilon^{\frac{5}{2}}\rho^3}{(1-\rho^2)^5},\qquad
A_{8,8} = \frac{64\pi\varepsilon}{(1-\rho^2)^2}, \qquad A_{8,10} = \frac{-128\pi\varepsilon^{\frac{5}{2}}\rho^3}{(1-\rho^2)^5},\\
A_{9,7} &= \frac{128\pi\varepsilon^{\frac{5}{2}}\rho^3}{(1-\rho^2)^5},\qquad A_{9,9} = \frac{2240\pi\varepsilon^4}{(1-\rho^2)^8}, \qquad A_{10,8} = \frac{-128\pi\varepsilon^{\frac{5}{2}}\rho^3}{(1-\rho^2)^5}, \qquad A_{10,10} = \frac{2240\pi\varepsilon^4}{(1-\rho^2)^8}.
\end{align*}
Note that the following estimates hold
$$\begin{aligned}
A_{1,1} &\sim \frac{\varepsilon}{(1-\rho)^4}, \quad
A_{1,2}\sim\frac{\varepsilon^\frac{3}{2}}{(1-\rho)^{5}},\quad A_{2,1} \sim \frac{\varepsilon^\frac{3}{2}}{(1-\rho)^{5}},\quad A_{2,2} \sim \frac{\varepsilon^2}{(1-\rho)^6},\quad
A_{3,3} \sim \frac{\varepsilon^2}{(1-\rho)^6}, \quad
,\\
A_{3,5}&\sim\frac{\varepsilon^2}{(1-\rho)^5},\quad
A_{4,4} \sim \frac{\varepsilon^2}{(1-\rho)^4},\quad
A_{5,3} \sim \frac{\varepsilon^2}{(1-\rho)^5},\quad
A_{5,5} \sim \frac{\varepsilon^2}{(1-\rho)^4},\quad
A_{6,6} \sim \frac{\varepsilon^2}{(1-\rho)^4},\\
A_{7,7} &\sim \frac{\varepsilon}{(1-\rho)^2},\quad A_{7,9} \sim \frac{\varepsilon^{\frac{5}{2}}}{(1-\rho)^5},\quad
A_{8,8} \sim \frac{\varepsilon}{(1-\rho)^2},\quad A_{8,10} \sim \frac{\varepsilon^{\frac{5}{2}}}{(1-\rho)^5},\quad A_{9,7} \sim \frac{\varepsilon^{\frac{5}{2}}}{(1-\rho)^5},\\
A_{9,9} & \sim\frac{\varepsilon^4}{(1-\rho)^8},\quad A_{10,8} \sim \frac{\varepsilon^{\frac{5}{2}}}{(1-\rho)^5},\quad
A_{10,10}\sim\frac{\varepsilon^4}{(1-\rho)^8}.\end{aligned}$$
Furthermore, we have the following expansion
\begin{equation*}
\left|\nabla_\chi \tilde F_{\mathcal D, g_\rho}(\chi)\cdot\left(\chi-\chi_0\right) - \left(\chi-\chi_0\right)\cdot A_{\rho, \varepsilon}\cdot \left(\chi-\chi_0\right)\right|\leq C\varepsilon^{2}\lambda^3(1-\rho)^{8}\quad \text{for}~\chi\in \mathcal T_\lambda.
\end{equation*}
Through detailed computations, the matrix $A_{\rho, \varepsilon}$ is verified to be positive definite when $\rho$ is close to 1. The critical point $$\chi_0 = \left( \mu_0,\rho, 0, \frac{\pi}{2}, 0, 0, 0, 0, 0, 0\right)$$ of $\tilde F_{\mathcal D, g_\rho}$ leads to the lower bound
\begin{equation*}
\nabla_{\chi} \tilde F_{\mathcal D, g_\rho}(\chi)\cdot \left(\chi^{(1)}-\chi_0\right) \geq
2C_0\varepsilon^2\lambda^{2}(1-\rho)^{4}\quad \text{ on }\ \partial^{(1)} \mathcal T_\lambda,
\end{equation*}
\begin{equation*}
\nabla_{\chi} \tilde F_{\mathcal D, g_\rho}(\chi)\cdot \left(\chi^{(2)}-\chi_0\right) \geq
2C_0\varepsilon^2\lambda^{2}(1-\rho)^{2}\quad \text{ on }\ \partial^{(2)} \mathcal T_\lambda,
\end{equation*}
with the topological degree $\text{deg}\left(\nabla \tilde F_{\mathcal D, g_\rho}, \mathcal T_\lambda, 0\right) = 1$. Additionally, since
$$
\left|\left(\nabla_\chi F_{\mathcal D, g_\rho}-\nabla_\chi\tilde F_{\mathcal D, g_\rho}\right)(\chi)\cdot (\chi-\chi_0)\right|\leq C\varepsilon^2\lambda^3(1-\rho)^6\quad \text{ for }~\chi\in T_\lambda,
$$
we still have
\begin{equation*}
\nabla_{\chi} F_{\mathcal D, g_\rho}(\chi)\cdot \left(\chi^{(1)}-\chi_0\right) \geq
C_0\varepsilon^2\lambda^{2}(1-\rho)^{4}\quad \text{ on }\ \partial^{(1)} \mathcal T_\lambda,
\end{equation*}
\begin{equation*}
\nabla_{\chi} F_{\mathcal D, g_\rho}(\chi)\cdot \left(\chi^{(2)}-\chi_0\right) \geq
C_0\varepsilon^2\lambda^{2}(1-\rho)^{2}\quad\text{ on }\ \partial^{(2)} \mathcal T_\lambda,
\end{equation*}
and the topological degree $\text{deg}\left(\nabla_\chi F_{\mathcal D, g_\rho}, \mathcal T_\lambda, 0\right) = 1$ when $\rho = \rho(\lambda, \varepsilon)$ is chosen sufficiently close to 1.
\end{proof}

\noindent {\bf Proof of Theorem \ref{thm}.} Let $\mathcal{S} = \{S_1\cup S_2\cup\cdots \cup S_k\}$ be a collection of $k$ spheres with centers $v_1, v_2, \dots, v_k$. Since each sphere has radius $1$ and passes through the origin, we have $|v_j| = 1$ for all $j =1,\dots, k$. Choose rotation matrices $\mathcal R_1, \dots, \mathcal R_k\in SO(3)$ such that
$$
\mathcal R_j(0, 0, -1) = v_j,\quad j= 1,\dots, k.
$$
For $\rho\in (0, 1)$ we define $\tilde G_{k, \rho}:\partial \mathcal D\to \mathbb{R}^3$ by
$$
\tilde G_{k, \rho}(x, y) = \sum_{j=1}^k\mathcal R_j \tilde g_{j, \varepsilon, \rho}(x, y), \quad (x, y)\in \partial \mathcal D,
$$
where
$$
\tilde g_{j, \varepsilon, \rho}(x, y)  = \tilde g_{\varepsilon, \rho}\left(\left(\cos\frac{2\pi j}{k}\right)x+\left(\sin\frac{2\pi j}{k}\right)y, -\left(\sin\frac{2\pi j}{k}\right)x+\left(\cos\frac{2\pi j}{k}\right)y\right)
$$
and
$$
\tilde g_{\varepsilon, \rho} = \tilde g_\rho + \left(0, 0, \frac{-2\varepsilon}{\left((x-\rho)^2 + y^2\right)^2}\right).
$$
Then the harmonic extension $G_{k, \rho}$ of $\tilde G_{k, \rho}$ to the unit disk $\mathcal D$ can be written as
\begin{equation*}
\begin{aligned}
G_{k, \rho}(x, y)  = \sum_{j=1}^k\mathcal R_j g_{j, \varepsilon, \rho}= \sum_{j=1}^k\mathcal R_j  g_{\varepsilon, \rho}\left(\left(\cos\frac{2\pi j}{k}\right)x+\left(\sin\frac{2\pi j}{k}\right)y, -\left(\sin\frac{2\pi j}{k}\right)x+\left(\cos\frac{2\pi j}{k}\right)y\right),
\end{aligned}
\end{equation*}
where $g_{\varepsilon, \rho}$ is the harmonic extension of $\tilde g_{\varepsilon, \rho}$.
Fix $j\in\{ 1, 2, \dots, k\}$, let $\theta_j$, $\psi_j$, $\phi_j$ be the Euler angle of the rotation matrix $\mathcal Q^{-1}_j\mathcal R_j$, that is to say, we have
\begin{equation*}
\mathcal Q^{-1}_j\mathcal R_j =\begin{pmatrix}
\cos\psi_j\cos\phi_j-\cos\theta_j\sin\phi_j\sin\psi_j & \cos\psi_j\sin\phi_j+\cos\theta_j\cos\phi_j\sin\psi_j & \sin\psi_j\sin\theta_j\\
-\sin\theta_j\sin\phi_j & \sin\theta_j\cos\phi_j & -\cos\theta_j\\
-\sin\psi_j\cos\phi_j-\cos\theta_j\sin\phi_j\cos\psi_j & -\sin\psi_j\sin\phi_j+\cos\theta_j\cos\phi_j\cos\psi_j & \cos\psi_j\sin\theta_j
\end{pmatrix}.
\end{equation*}
Furthermore, let us define the parameter sets as
\begin{align*}
\mathcal T_\lambda^j   =&\Bigg\{\chi_j:=\left(\mu_j, \xi_1^{(j)}, \xi_2^{(j)}, \theta_j, \phi_j, \psi_j, a_{j,1}, a_{j,2}, p_{j,1}, p_{j,2}\right)\ |\ \left|\mu_j-\mu_0\right|\leq\sqrt{\varepsilon}\lambda(1-\rho)^4,
\\&\qquad\qquad \left|\left(\xi_1^{(j)}, \xi_2^{(j)}\right)-\rho\left(\cos\frac{2\pi j }{k},\sin\frac{2\pi j}{k}\right)\right|\leq\lambda(1-\rho)^{5},\ |\theta_j-\frac{\pi}{2}|\leq \lambda (1-\rho)^4, |\phi_j|\leq\lambda(1-\rho)^4,
\\&\qquad\quad|\psi_j|\leq\lambda(1-\rho)^4,\
 |a_{j, 1}|\leq \lambda\mu_0(1-\rho)^2,|a_{j, 2}|\leq \lambda\mu_0(1-\rho)^2,
\ |p_{j, 1}|\leq\frac{\lambda(1-\rho)^5}{\mu_0^2},\ |p_{j, 2}|\leq\frac{\lambda(1-\rho)^5}{\mu_0^2}\Bigg\},
\end{align*}
with
\begin{align*}
&\chi_j^{(1)}:=\left(\mu_j, \xi_1^{(j)}, \xi_2^{(j)}, \theta_j, \phi_j, \psi_j,0,0,0,0\right),\qquad\qquad\qquad
\chi_j^{(2)}:=\chi_j=\left(\mu_j, \xi_1^{(j)}, \xi_2^{(j)}, \theta_j, \phi_j, \psi_j, a_{j,1}, a_{j,2}, p_{j,1}, p_{j,2}\right),
\\&\chi_{0,j}:=\left(\mu_0,\rho\left(\cos\frac{2\pi j }{k},\sin\frac{2\pi j}{k}\right),\frac{\pi}{2},0,0,0,0,0,0\right).
\end{align*}
By Proposition \ref{prop 6.4}, we have
\begin{align*}
&\Sigma_{\mathcal D,G_{k,\rho}}+\mathcal{E}(\varepsilon,\vec{\mu},\vec{a},\vec{p}) \\&= \sum_{i=1}^k F_{\mathcal{D},\mathcal R_i g_{i,\varepsilon, \rho}}(\mu_i,\xi^{(i)},\mathcal Q_{i},a_{i,l},p_{i,l})+\sum_{i<s} F_{\mathcal{D}}(\delta_i,\delta_s,Id,\mathcal Q_{i}^{-1}\mathcal Q_{s})\\
&\quad +\sum_{i<s}F^{(1)}_{\mathcal{D}}(\delta_i,\delta_s,Id,\mathcal Q_{i}^{-1}\mathcal Q_{s}, a_{i,l},p_{i,l})
+\sum_{i<s}F^{(2)}_{\mathcal{D}}(\delta_i,\delta_s,Id,\mathcal Q_{i}^{-1}\mathcal Q_{s}, a_{s,l},p_{s,l})+\mathcal{E}(\varepsilon,\vec{\mu},\vec{a},\vec{p})\\
& = F_{\mathcal{D},\mathcal R_j g_{j,\varepsilon, \rho}}(\mu_j,\xi^{(j)},\mathcal Q_{j},a_{j,l},p_{j,l}) + \sum_{i\neq j} F_{\mathcal{D},\mathcal R_i g_{i,\varepsilon, \rho}}(\mu_i,\xi^{(i)},\mathcal Q_{i},a_{i,l},p_{i,l})+\sum_{i<s} F_{\mathcal{D}}(\delta_i,\delta_s,Id,\mathcal Q_{i}^{-1}\mathcal Q_{s})\\
&\quad +\sum_{i<s}F^{(1)}_{\mathcal{D}}(\delta_i,\delta_s,Id,\mathcal Q_{i}^{-1}\mathcal Q_{s}, a_{i,l},p_{i,l})
+\sum_{i<s}F^{(2)}_{\mathcal{D}}(\delta_i,\delta_s,Id,\mathcal Q_{i}^{-1}\mathcal Q_{s}, a_{s,l},p_{s,l})+\mathcal{E}(\varepsilon,\vec{\mu},\vec{a},\vec{p})\\
& = F_{\mathcal{D}, g_{j, \varepsilon, \rho}}(\mu_j,\xi^{(j)},\mathcal R_j^{-1} \mathcal Q_{j},a_{j,l},p_{j,l}) + \sum_{i\neq j} F_{\mathcal{D},\mathcal R_i g_{i,\varepsilon, \rho}}(\mu_i,\xi^{(i)},\mathcal Q_{i},a_{i,l},p_{i,l})+\sum_{i<s} F_{\mathcal{D}}(\delta_i,\delta_s,Id,\mathcal Q_{i}^{-1}\mathcal Q_{s})\\
&\quad +\sum_{i<s}F^{(1)}_{\mathcal{D}}(\delta_i,\delta_s,Id,\mathcal Q_{i}^{-1}\mathcal Q_{s}, a_{i,l},p_{i,l})
+\sum_{i<s}F^{(2)}_{\mathcal{D}}(\delta_i,\delta_s,Id,\mathcal Q_{i}^{-1}\mathcal Q_{s}, a_{s,l},p_{s,l})+\mathcal{E}(\varepsilon,\vec{\mu},\vec{a},\vec{p}).
\end{align*}
From the results in Section \ref{expansion-multi-bubbles}, one has
\begin{align*}
&\left|\frac{\partial }{\partial \mu_j}\sum_{i\neq j} F_{\mathcal{D},\mathcal R_i g_{i,\varepsilon, \rho}}(\mu_i,\xi^{(i)},\mathcal Q_{i},a_{i,l},p_{i,l})+\frac{\partial }{\partial \mu_j}\sum_{i<s} F_{\mathcal{D}}(\delta_i,\delta_s,Id,\mathcal Q_{i}^{-1}\mathcal Q_{s})\right.\\
&\ \left. +\frac{\partial }{\partial \mu_j}\sum_{i<s}F^{(1)}_{\mathcal{D}}(\delta_i,\delta_s,Id,\mathcal Q_{i}^{-1}\mathcal Q_{s}, a_{i,l},p_{i,l})
+\frac{\partial }{\partial \mu_j}\sum_{i<s}F^{(2)}_{\mathcal{D}}(\delta_i,\delta_s,Id,\mathcal Q_{i}^{-1}\mathcal Q_{s}, a_{s,l},p_{s,l})+\frac{\partial }{\partial \mu_j}\mathcal{E}(\varepsilon,\vec{\mu},\vec{a},\vec{p})\right|\leq C\varepsilon^{\frac{3}{2}},
\end{align*}
thus
\begin{align*}
\left|\frac{\partial }{\partial \mu_j}\Sigma_{\mathcal D, G_{k, \rho}}+\frac{\partial }{\partial \mu_j}\mathcal{E}(\varepsilon,\vec{\mu},\vec{a},\vec{p}) -\frac{\partial }{\partial \mu_j} F_{\mathcal{D}, g_{j, \varepsilon, \rho}}(\mu_j,\xi^{(j)}, \mathcal R_j^{-1}\mathcal Q_{j},a_{j,l},p_{j,l})\right|\leq C\varepsilon^\frac{3}{2}.
\end{align*}
Analogous bounds hold for the gradients with respect to other variables in $\chi_j$. Indeed, let $\tilde{\mathcal{A}}$ denote the set of variables $\left(\xi_1^{(j)}, \xi_2^{(j)},\theta_j, \phi_j, \psi_j, a_{j, 1}, a_{j, 2}, p_{j, 1}, p_{j, 2}\right)$, we derive
\begin{align*}
&\nabla_{\tilde{\mathcal{A}}}\Sigma_{\mathcal D, G_{k, \rho}}+\nabla_{\tilde{\mathcal{A}}}\mathcal{E}(\varepsilon,\vec{\mu},\vec{a},\vec{p}) -\nabla_{\tilde{\mathcal{A}}}F_{\mathcal{D}, g_{j, \varepsilon, \rho}}(\mu_j,\xi^{(j)}, \mathcal R_j^{-1}\mathcal Q_{j},a_{j,l},p_{j,l})
\\&=\left(O\left(\varepsilon^3 d^{-4}\right),O\left(\varepsilon^3 d^{-4}\right),O\left(\varepsilon^2\right),O\left(\varepsilon^2\right),O\left(\varepsilon^2\right),O\left(\varepsilon^\frac{3}{2}\right),O\left(\varepsilon^\frac{3}{2}\right),O\left(\varepsilon^4|\log\varepsilon|d^{-5}\right),O\left(\varepsilon^4|\log\varepsilon|d^{-5}\right)\right).
\end{align*}
Combining these estimates, we obtain
\begin{align*}
\nabla_{\chi_j}I_{\varepsilon, G_{k,\rho}}(z_{\mathcal A})\cdot \left(\chi_j^{(1)}-\chi_{0,j}\right)
&= \left(\nabla_{\chi_j}\Sigma_{\mathcal D, G_{k, \rho}}+\nabla_{\chi_j}\mathcal{E}(\varepsilon,\vec{\mu},\vec{a},\vec{p})\right)\cdot \left(\chi_j^{(1)}-\chi_{0,j}\right)
\\&\geq C_0\varepsilon^2\lambda^{2}(1-\rho)^{4} -C\varepsilon^2\lambda(1-\rho)^{4}  \quad \text{ on }\ \partial^{(1)} T_\lambda^j,
\end{align*}
\begin{align*}
\nabla_{\chi_j}I_{\varepsilon, G_{k,\rho}}(z_{\mathcal A})\cdot \left(\chi_j^{(2)}-\chi_{0,j}\right)
&= \left(\nabla_{\chi_j}\Sigma_{\mathcal D, G_{k, \rho}}+\nabla_{\chi_j}\mathcal{E}(\varepsilon,\vec{\mu},\vec{a},\vec{p})\right)\cdot \left(\chi_j^{(2)}-\chi_{0,j}\right)
\\&\geq C_0\varepsilon^2\lambda^{2}(1-\rho)^{2} - C\varepsilon^2\lambda(1-\rho)^{2} \quad \text{ on }\ \partial^{(2)} T_\lambda^j,
\end{align*}
here $I_{\varepsilon, G_{k,\rho}}$ is the Euler-Lagrange functional (\ref{Euler-functional}) corresponding to the boundary datum $G_{k,\rho}$, $\mathcal{A}=\left(\mathcal{A}_1^{\mathrm{T}},\cdots,\mathcal{A}_k^{\mathrm{T}}\right)$, $\mathcal A_j\in \mathcal T^j_\lambda$, $j=1,\dots,k$.
Therefore, for sufficiently large $\lambda$ which is independent of $\rho$, there holds
\begin{align*}
\nabla_{\chi_j}I_{\varepsilon, G_{k,\rho}}(z_{\mathcal A})\cdot \left(\chi_j^{(1)}-\chi_{0,j}\right) \geq \frac{C_0}{2}\varepsilon^2\lambda^{2}(1-\rho)^{4} \quad \text{ on }\ \partial^{(1)} T_\lambda^j,
\end{align*}
\begin{align*}
\nabla_{\chi_j}I_{\varepsilon, G_{k,\rho}}(z_{\mathcal A})\cdot \left(\chi_j^{(2)}-\chi_{0,j}\right) \geq \frac{C_0}{2}\varepsilon^2\lambda^{2}(1-\rho)^{2} \quad \text{ on }\ \partial^{(2)} T_\lambda^j,
\end{align*}
and
$$
\text{ deg}(\nabla_{\chi_j}I_{\varepsilon, G_{k,\rho}}(z_{\mathcal A}), \mathcal T_\lambda^j, 0) = 1.
$$
From Proposition \ref{prop-3.6}, we have
\begin{align*}
\left|\left(\nabla_{\chi_j}I_{\varepsilon, G_{k,\rho}}(z_{\mathcal A}+ w_\varepsilon(z_{\mathcal A}))-\nabla_{\chi_j}I_{\varepsilon, G_{k,\rho}}(z_{\mathcal A})\right)\cdot \left(\chi_j^{(1)}-\chi_{0,j}\right)\right|
\leq C  o(1) \lambda\varepsilon^{2}(1-\rho)^{4},
\end{align*}
\begin{align*}
\left|\left(\nabla_{\chi_j}I_{\varepsilon, G_{k,\rho}}(z_{\mathcal A}+ w_\varepsilon(z_{\mathcal A}))-\nabla_{\chi_j}I_{\varepsilon, G_{k,\rho}}(z_{\mathcal A})\right)\cdot \left(\chi_j^{(2)}-\chi_{0,j}\right)\right|
\leq C  o(1) \lambda \varepsilon^{2}(1-\rho)^{2},
\end{align*}
where $o(1)\to 0$ as $\varepsilon\to 0^{+}$.
Then we have
$$
\nabla_{\chi_{j}}I_{\varepsilon, G_{k,\rho}}(z_{\mathcal A}+ w_\varepsilon(z_{\mathcal A}))\cdot \left(\chi_{j}^{(1)}-\chi_{0,j}\right) \geq \frac{C_0}{4}\varepsilon^2\lambda^2(1-\rho)^{4} \quad \text{ on }\ \partial^{(1)} T_{\lambda}^j,
$$
and
$$
\nabla_{\chi_{j}}I_{\varepsilon, G_{k,\rho}}(z_{\mathcal A}+ w_\varepsilon(z_{\mathcal A}))\cdot \left(\chi_{j}^{(2)}-\chi_{0,j}\right) \geq \frac{C_0}{4}\varepsilon^2\lambda^2(1-\rho)^{2} \quad \text{ on }\ \partial^{(2)} T_{\lambda}^j.
$$
Consequently, we establish
$$
\nabla_{\chi_j}I_{\varepsilon, G_{k,\rho}}(z_{\mathcal A}+ w_\varepsilon(z_{\mathcal A}))\neq \vec{0}\quad  \text{ for all }~ \chi_j\in \partial T_\lambda^j,$$
and
$$\text{ deg}\left(\nabla_{\chi_j}I_{\varepsilon, G_{k,\rho}}(z_{\mathcal A}+ w_\varepsilon(z_{\mathcal A})), \mathcal T_\lambda^j, 0 \right) = 1\quad \text{ for all }j.
$$
Therefore, we find a weak solution $u_{\varepsilon}$ of (\ref{e:main2}) with the form
\begin{equation*}
u_{\varepsilon}=\sum_{i=1}^kP\mathcal Q_i\delta_{\mu_i,\xi^{(i)},a_i, p_i}+w,
\end{equation*}
and $\|w\|_{H^1_0(\mathcal{D})}\to 0$ as $\varepsilon\to 0^{+}$. Using the Jacobian structure for equation (\ref{e:main2}), along with the regularity result in \cite{Chanillo-Li} and a minor modification of \cite[Lemma 8.5]{chanillomalchiodi2005cagasymptotic} we obtain $\|w\|_{L^\infty(\mathcal{D})}\to 0$ as $\varepsilon\to 0^{+}$. Furthermore, up to a sub-sequence and for a fixed large $\lambda > 0$, we have $u_{\varepsilon}(\mathcal D)$ converges to $\mathcal{S} = \{S_1\cup S_2\cup\cdots \cup S_k\}$ in the Hausdorff sense as $\varepsilon\to 0^+$. This concludes the proof of Theorem \ref{thm}.

\appendix

\section{Computations of mixed terms in Section \ref{expansion-for-one-bubble}}\label{Computations of mixed terms}
In this appendix,  we estimate the terms in $\mathcal{R}_{\mu,\xi,a,p,\varepsilon}$.
From (\ref{Mar22-2}) and (\ref{Au10-I1}), we know that
\begin{align}
\notag
\mathcal{R}_{\mu,\xi,a,p,\varepsilon}=&
-a_1a_2\int_{\mathcal{D}}L_\delta\left[P Z_{-1,2}\right]\cdot \left( PZ_{-1,1}\right)
-p_1p_2\int_{\mathcal{D}}L_\delta\left[P Z_{2,2}\right]\cdot \left(PZ_{2,1}\right)
\\ \notag&
-2\int_{\mathcal{D}}\varphi_1
\cdot \left[\left(P\mathcal{L}_{\mathcal{A}}\right)_x\wedge \left(P\mathcal{L}_{\mathcal{A}}\right)_y\right]
\\ \notag&+\frac{1}{2}\int_{\mathcal{D}}\nabla \left(P\mathcal{R}_{\mathcal{A}}\right)\cdot\nabla \left(P\mathcal{R}_{\mathcal{A}}\right)+2\int_{\mathcal{D}}P\delta\cdot \left[\left(P\mathcal{R}_{\mathcal{A}}\right)_x
 \wedge \left(P\mathcal{R}_{\mathcal{A}}\right)_y\right]
\\ \notag&-2\int_{\mathcal{D}}\varphi_1
\cdot \left[\delta_x\wedge \left(P\left(\mathcal{L}_{\mathcal{A}}+\mathcal{R}_{\mathcal{A}}\right)\right)_y+\left(P\left(\mathcal{L}_{\mathcal{A}}+\mathcal{R}_{\mathcal{A}}\right)\right)_x\wedge \delta_y\right]
\\ \notag&+2\int_{\mathcal{D}} \left(P\left(\mathcal{L}_{\mathcal{A}}+\mathcal{R}_{\mathcal{A}}\right)\right)\cdot\left[(\varphi_1)_x\wedge (\varphi_1)_y\right]
\\ \notag&-2\int_{\mathcal{D}}\left(P\mathcal{R}_{\mathcal{A}}\right)
\cdot \left[\delta_x\wedge \left(\mathcal{L}_{\mathcal{A}}\right)_y +\left(\mathcal{L}_{\mathcal{A}}\right)_x\wedge\delta_y\right]+2\int_{\mathcal{D}}\left(P\mathcal{R}_{\mathcal{A}}\right)
\cdot \left[\left(P\delta\right)_x\wedge \left(P\mathcal{L}_{\mathcal{A}}\right)_y+\left(P\mathcal{L}_{\mathcal{A}}\right)_x\wedge \left(P\delta\right)_y \right]
\\ \notag &+\frac{2}{3}\int_{\mathcal{D}}\left(P\left(\mathcal{L}_{\mathcal{A}}+\mathcal{R}_{\mathcal{A}}\right)\right)\cdot \left[\left(P\left(\mathcal{L}_{\mathcal{A}}+\mathcal{R}_{\mathcal{A}}\right)\right)_x
 \wedge \left(P\left(\mathcal{L}_{\mathcal{A}}+\mathcal{R}_{\mathcal{A}}\right)\right)_y\right]
\\ \notag &-2\varepsilon\int_{\mathcal{D}}g\cdot \left[(\varphi_1)_x\wedge \left(P\mathcal{L}_{\mathcal{A}}\right)_y+\left(P\mathcal{L}_{\mathcal{A}}\right)_x\wedge (\varphi_1)_y\right]
\\ \notag &+2\varepsilon \int_{\mathcal{D}} g \cdot \left[\left(P\mathcal{L}_{\mathcal{A}}\right)_x\wedge \left(P\mathcal{L}_{\mathcal{A}}\right)_y\right]
+2\varepsilon\int_{\mathcal{D}}g\cdot  \left[\left(P\delta\right)_x\wedge \left(P\left(\mathcal{L}_{\mathcal{A}}+\mathcal{R}_{\mathcal{A}}\right)\right)_y+\left(P\left(\mathcal{L}_{\mathcal{A}}+\mathcal{R}_{\mathcal{A}}\right)\right)_x\wedge \left(P\delta\right)_y\right]
\\ \notag &+2\varepsilon \int_{\mathcal{D}} g \cdot \left[\left(P\mathcal{R}_{\mathcal{A}}\right)_x\wedge \left(P\mathcal{R}_{\mathcal{A}}\right)_y\right]
+2\varepsilon\int_{\mathcal{D}}g\cdot  \left[\left(P\mathcal{L}_{\mathcal{A}}\right)_x\wedge \left(P\mathcal{R}_{\mathcal{A}}\right)_y+\left(P\mathcal{R}_{\mathcal{A}}\right)_x\wedge \left(P\mathcal{L}_{\mathcal{A}}\right)_y\right]
\\ \label{Au12-R1}&+2\varepsilon^2\int_{\mathcal{D}}\left(P\mathcal{L}_{\mathcal{A}}+P\mathcal{R}_{\mathcal{A}}\right)\cdot (g_x\wedge g_y).
\end{align}

We point out that there are some interesting and subtle cancellations among the terms in $\mathcal{R}_{\mu,\xi,a,p,\varepsilon}$.
For the sake of convenience and clarity, we categorize and organize the terms in $\mathcal{R}_{\mu,\xi,a,p,\varepsilon}$ as follows.
Write
$$\mathcal{R}_{\mu,\xi,a,p,\varepsilon}=\mathcal{R}^{(1)}_{\mu,\xi,a,p,\varepsilon}+\mathcal{R}^{(2)}_{\mu,\xi,a,p,\varepsilon}+\mathcal{R}^{(3)}_{\mu,\xi,a,p,\varepsilon}+\mathcal{R}^{(4)}_{\mu,\xi,a,p,\varepsilon},$$
where
\begin{equation}\label{Jan6-1}
\begin{aligned}
\mathcal{R}^{(1)}_{\mu,\xi,a,p,\varepsilon}:=&-2\int_{\mathcal{D}}\left(\varphi_1-\varepsilon g\right)
\cdot \left[\delta_x\wedge (P\mathcal{L}_{\mathcal{A}})_y+(P\mathcal{L}_{\mathcal{A}})_x\wedge \delta_y\right]
\\&+2\int_{\mathcal{D}}P\mathcal{L}_{\mathcal{A}}
\cdot \left[(\varphi_1)_x\wedge (\varphi_1)_y\right]-2\varepsilon\int_{\mathcal{D}}g\cdot \left[(\varphi_1)_x\wedge \left(P\mathcal{L}_{\mathcal{A}}\right)_y+\left(P\mathcal{L}_{\mathcal{A}}\right)_x\wedge (\varphi_1)_y\right]
\\&+2\varepsilon^2\int_{\mathcal{D}}P\mathcal{L}_{\mathcal{A}}\cdot (g_x\wedge g_y),
\end{aligned}
\end{equation}

\begin{equation}\label{Jan2-6}
\begin{aligned}
\mathcal{R}^{(2)}_{\mu,\xi,a,p,\varepsilon}:=&-a_1a_2\int_{\mathcal{D}}L_\delta\left[P Z_{-1,2}\right]\cdot \left( PZ_{-1,1}\right)
-p_1p_2\int_{\mathcal{D}}L_\delta\left[P Z_{2,2}\right]\cdot \left(PZ_{2,1}\right)
\\&-2\int_{\mathcal{D}}\varphi_1
\cdot \left[\left(P\mathcal{L}_{\mathcal{A}}\right)_x\wedge \left(P\mathcal{L}_{\mathcal{A}}\right)_y\right]
-2\int_{\mathcal{D}}\varphi_1
\cdot \left[\delta_x\wedge \left(P\mathcal{R}_{\mathcal{A}}\right)_y+\left(P\mathcal{R}_{\mathcal{A}}\right)_x\wedge \delta_y\right]
\\&+2\int_{\mathcal{D}} \left(P\mathcal{R}_{\mathcal{A}}\right)\cdot\left[(\varphi_1)_x\wedge (\varphi_1)_y\right]
+2\varepsilon  \int_{\mathcal{D}} g \cdot \left[\left(P\mathcal{L}_{\mathcal{A}}\right)_x\wedge \left(P\mathcal{L}_{\mathcal{A}}\right)_y\right]
\\&+2\varepsilon\int_{\mathcal{D}}g\cdot  \left[\left(P\delta\right)_x\wedge \left(P\mathcal{R}_{\mathcal{A}}\right)_y+\left(P\mathcal{R}_{\mathcal{A}}\right)_x\wedge \left(P\delta\right)_y\right]
+2\varepsilon^2\int_{\mathcal{D}}\left(P\mathcal{R}_{\mathcal{A}}\right)\cdot (g_x\wedge g_y),
\end{aligned}
\end{equation}

\begin{equation}\label{Jan2-11}
\begin{aligned}
\mathcal{R}^{(3)}_{\mu,\xi,a,p,\varepsilon}:=&-2\int_{\mathcal{D}}\left(P\mathcal{R}_{\mathcal{A}}\right)
\cdot \left[\delta_x\wedge \left(\mathcal{L}_{\mathcal{A}}\right)_y +\left(\mathcal{L}_{\mathcal{A}}\right)_x\wedge\delta_y\right]
+2\int_{\mathcal{D}}\left(P\mathcal{R}_{\mathcal{A}}\right)
\cdot \left[\left(P\delta\right)_x\wedge \left(P\mathcal{L}_{\mathcal{A}}\right)_y+\left(P\mathcal{L}_{\mathcal{A}}\right)_x\wedge \left(P\delta\right)_y \right]
\\&+\frac{2}{3}\int_{\mathcal{D}}\left(P\mathcal{L}_{\mathcal{A}}\right)\cdot \left[\left(P\mathcal{L}_{\mathcal{A}}\right)_x
 \wedge \left(P\mathcal{L}_{\mathcal{A}}\right)_y\right]
\\&+2\varepsilon\int_{\mathcal{D}}g\cdot  \left[\left(P\mathcal{L}_{\mathcal{A}}\right)_x\wedge \left(P\mathcal{R}_{\mathcal{A}}\right)_y+\left(P\mathcal{R}_{\mathcal{A}}\right)_x\wedge \left(P\mathcal{L}_{\mathcal{A}}\right)_y\right],
\end{aligned}
\end{equation}

\begin{equation}\label{Jan7-2}
\begin{aligned}
\mathcal{R}^{(4)}_{\mu,\xi,a,p,\varepsilon}:=&~\frac{1}{2}\int_{\mathcal{D}}\nabla \left(P\mathcal{R}_{\mathcal{A}}\right)\cdot\nabla \left(P\mathcal{R}_{\mathcal{A}}\right)+2\int_{\mathcal{D}}P\delta\cdot \left[\left(P\mathcal{R}_{\mathcal{A}}\right)_x
 \wedge \left(P\mathcal{R}_{\mathcal{A}}\right)_y\right]
\\&+2\int_{\mathcal{D}}\left(P\mathcal{R}_{\mathcal{A}}\right)\cdot \left[\left(P\mathcal{L}_{\mathcal{A}}\right)_x
 \wedge \left(P\mathcal{L}_{\mathcal{A}}\right)_y\right]
+2\int_{\mathcal{D}}\left(P\mathcal{L}_{\mathcal{A}}\right)\cdot \left[\left(P\mathcal{R}_{\mathcal{A}}\right)_x
 \wedge \left(P\mathcal{R}_{\mathcal{A}}\right)_y\right]
 \\&+\frac{2}{3}\int_{\mathcal{D}}\left(P\mathcal{R}_{\mathcal{A}}\right)\cdot \left[\left(P\mathcal{R}_{\mathcal{A}}\right)_x
 \wedge \left(P\mathcal{R}_{\mathcal{A}}\right)_y\right].
 \end{aligned}
\end{equation}
In particular, we provide the cancellations associated with the terms involving $a$ and $p$.

{\bf$\bullet$ Cancellations of the terms containing $p_k$ or $p_kp_l$.}
Firstly, we estimate the terms that contain $p_1$ in $\mathcal{R}^{(1)}_{\mu,\xi,a,p,\varepsilon}$. We start with $-2\sum_{l=1}^2 p_l\int_{\mathcal{D}}\left(\varphi_1-\varepsilon g\right)
\cdot \left[\delta_x\wedge (PZ_{2,l})_y+(PZ_{2,l})_x\wedge \delta_y\right]$.
Using (\ref{A26-5}) and (\ref{A26-6}), we compute $-2\sum_{l=1}^2 p_l\int_{\mathcal{D}}\left(\varphi_1-\varepsilon g\right)
\cdot \left[\delta_x\wedge \left(Z_{2,l}\right)_y+\left(Z_{2,l}\right)_x\wedge \delta_y\right].$
Similar to (\ref{A30-p3-a}) and (\ref{A30-p3-b}) we have
\begin{align*}
&-2\sum_{l=1}^2 p_l\int_{\mathcal{D}}\varphi_1
\cdot \left[\delta_x\wedge (Z_{2,l})_y+(Z_{2,l})_x\wedge \delta_y\right]
\\&=-4p_1\mu^2\left[\frac{\pi^2}{4}\mu^2\frac{\partial^2h_1^{(1)}}{\partial x^2}(\xi,\xi)+\frac{\pi^2}{4}\mu^2\frac{\partial^2h_1^{(1)}}{\partial y^2}(\xi,\xi) + \frac{\pi}{6}\mu^4\frac{\partial^4h_1^{(1)}}{\partial x^4}(\xi,\xi) + \frac{\pi}{6}\mu^4\frac{\partial^4h_1^{(1)}}{\partial y^4}(\xi,\xi)\right]
\\&\quad-4p_1\mu^2\left[\frac{\pi}{6}\mu^4\frac{\partial^4h_2^{(1)}}{\partial x^3\partial y}(\xi,\xi)-\frac{\pi}{6}\mu^4\frac{\partial^4h_2^{(1)}}{\partial x\partial y^3}(\xi,\xi)\right]
-4p_2\mu^2 \left[\frac{\pi}{6}\mu^4\frac{\partial^4 h^{(1)}_1}{\partial x^3 \partial y}(\xi, \xi) -\frac{\pi}{6}\mu^4\frac{\partial^4 h^{(1)}_1}{\partial x\pp y^3 }(\xi, \xi)\right]
\\&\quad-4p_2\mu^2\left[\frac{\pi^2}{4}\mu^2\frac{\partial^2h_2^{(1)}}{\partial x^2}(\xi,\xi)+\frac{\pi^2}{4}\mu^2\frac{\partial^2h_2^{(1)}}{\partial y^2}(\xi,\xi)\right]
\\&\quad-4p_2\mu^2\left[+\frac{\pi}{12}\mu^4 \frac{\partial^4 h^{(1)}_2}{\partial x^4}(\xi, \xi)+
\frac{\pi}{12}\mu^4 \frac{\partial^4 h^{(1)}_2}{\partial y^4}(\xi, \xi)+\frac{\pi}{2}\mu^4 \frac{\partial^4 h^{(1)}_2}{\partial x^2\partial y^2}(\xi, \xi)\right]+O\left((p_1+p_2)\mu^8d^{-8}\right)
\\&=-\frac{8\pi}{3}p_1\mu^6\frac{\partial^4h_1^{(1)}}{\partial x^4}(\xi,\xi)+\frac{8\pi}{3}p_2\mu^6\frac{\partial^4 h^{(1)}_2}{\partial x^4}(\xi, \xi)
+O\left((p_1+p_2)\mu^8d^{-8}\right),
\end{align*}
where in the last equality, we have used the fact that $\frac{\partial^2h_1^{(1)}}{\partial x^2}(\xi,\xi)+\frac{\partial^2h_1^{(1)}}{\partial y^2}(\xi,\xi)=0$. Recall that
$$g(z,\omega)=\left(g_1(z,\omega),g_2(z,\omega),g_3(z,\omega)\right)=\left(2h_1^{(1)}(z,\omega),2h_2^{(1)}(z,\omega),-2\varepsilon h_3^{(1)}(z,\omega)\right),$$
then we have
\begin{equation*}
\begin{aligned}
&2\sum_{l=1}^2 p_l\int_{\mathcal{D}}\varepsilon g
\cdot \left[\delta_x\wedge (Z_{2,l})_y+(Z_{2,l})_x\wedge \delta_y\right]
\\&=\frac{8\pi}{3}p_1\varepsilon \mu^4\frac{\partial^4h_1^{(1)}}{\partial x^4}(\xi,\omega)+\frac{8\pi}{3}p_2\varepsilon\mu^4\frac{\partial^4 h^{(1)}_2}{\partial x^4}(\xi, \omega)
+O\left((p_1+p_2)\left(\varepsilon\mu^6+\varepsilon^2\mu^4\right)d^{-8}\right).
\end{aligned}
\end{equation*}
Since
$$\left|\frac{\partial^4 h_1^{(1)}}{\partial x^4}(\xi,\omega)-\frac{\partial^4 h_1^{(1)}}{\partial x^4}(\xi,\xi)\right|=O\left(d^{-7}|\xi-\omega|\right),\qquad \left|\frac{\partial^4 h_2^{(1)}}{\pp x^4}(\xi,\omega)-\frac{\partial^4 h_2^{(1)}}{\pp x^4}(\xi,\xi)\right|=O\left(d^{-7}|\xi-\omega|\right) ,$$
one has
\begin{equation*}
\begin{aligned}
&-2\sum_{l=1}^2 p_l\int_{\mathcal{D}}\left(\varphi_1-\varepsilon g\right)
\cdot \left[\delta_x\wedge (PZ_{2,l})_y+(PZ_{2,l})_x\wedge \delta_y\right]
\\&=-\frac{8\pi}{3}p_1\left(\mu^2-\varepsilon\right)\mu^4\frac{\partial^4h_1^{(1)}}{\partial x^4}(\xi,\xi)+\frac{8\pi}{3}p_2\left(\mu^2-\varepsilon\right)\mu^4\frac{\partial^4 h^{(1)}_2}{\partial x^4}(\xi, \xi)
\\&\quad+O\left((p_1+p_2)\left(\varepsilon\mu^4d^{-7}|\xi-\omega|+\left(\mu^2-\varepsilon\right)\mu^4\left(\mu^2+\varepsilon\right)d^{-8}\right)\right)
\\&=O\left((p_1+p_2)\left(\mu^2-\varepsilon\right)\mu^4d^{-6}\right).
\end{aligned}
\end{equation*}
Here we employ the derivatives of $h^{(1)}$ established in Lemma \ref{evalues-at-xi-2}, and leverage the asymptotic relation $\varepsilon d^{-1}|\xi-\omega|=O\left(\mu^2-\varepsilon\right)$, which persists throughout subsequent analysis. This allows us to bound the term $\varepsilon\mu^4d^{-7}|\xi-\omega|$ as $O\left((\mu^2-\varepsilon)\mu^4d^{-6}\right)$. Similar to (\ref{Jan5-2}), we have
\begin{equation*}
\begin{aligned}
&2\sum_{l=1}^2p_l\int_{\mathcal{D}}\left(\varphi_1-\varepsilon g\right)
\cdot \left[\delta_x\wedge \left(\varphi_{2,l}\right)_y+\left(\varphi_{2,l}\right)_x\wedge \delta_y\right]
\\&=O\left((p_1+p_2)\left(\left(\mu^2-\varepsilon\right)\mu^{8}|\log\mu|d^{-10}+\varepsilon\mu^8d^{-11}|\log\mu||\xi-\omega|+\left(\mu^2-\varepsilon\right)\mu^6d^{-8}+\varepsilon\mu^6d^{-9}|\xi-\omega|\right)\right).
\end{aligned}
\end{equation*}
Thus we obtain
\begin{equation}\label{M1p1}
\begin{aligned}
-2\sum_{l=1}^2 p_l\int_{\mathcal{D}}\left(\varphi_1-\varepsilon g\right)
\cdot \left[\delta_x\wedge (PZ_{2,l})_y+(PZ_{2,l})_x\wedge \delta_y\right]
=O\left((p_1+p_2)\left(\mu^2-\varepsilon\right)\mu^4d^{-6}\right).
\end{aligned}
\end{equation}
For the remaining terms, we deduce
\begin{equation}\label{M1p3}
\begin{aligned}
&2p_1\int_{\mathcal{D}}PZ_{2,1}\cdot\left[(\varphi_1)_x\wedge (\varphi_1)_y\right]
\\&=-8p_1\mu^8\int_{\mathbb{R}^2}\frac{-2(x^4-6x^2y^2+y^4-1)}{\left(\left(x^2+y^2\right)^2+1\right)^2}
\left[\frac{\pp h_2^{(1)}}{\pp x}(\xi,\xi)\frac{\pp h_3^{(1)}}{\pp y}(\xi,\xi)-\frac{\pp h_3^{(1)}}{\pp x}(\xi,\xi)\frac{\pp h_2^{(1)}}{\pp y}(\xi,\xi)+O\left(\mu^{2}d^{-10}\right)\right]
\\&\quad+O\left(p_1\mu^{10}d^{-12}\right)
\\&=-4\pi^2p_1\mu^8\left[\frac{\pp h_2^{(1)}}{\pp x}(\xi,\xi)\frac{\pp h_3^{(1)}}{\pp y}(\xi,\xi)-\frac{\pp h_3^{(1)}}{\pp x}(\xi,\xi)\frac{\pp h_2^{(1)}}{\pp y}(\xi,\xi)\right]
+O\left(p_1\mu^{10}d^{-12}\right).
\end{aligned}
\end{equation}
Using (\ref{M30p1-1}), (\ref{M30p1-2}), (\ref{M30p2-1}) and (\ref{M30p2-2}), it is easy to check that
\begin{align}
\notag&-2\varepsilon p_1\int_{\mathcal{D}}g\cdot\left[\left(PZ_{2,1}\right)_x\wedge (\varphi_1)_y+(\varphi_1)_x\wedge\left(PZ_{2,1}\right)_y\right]
\\\notag&=2\pi^2\varepsilon p_1\mu^6\left[\frac{\partial g_2}{\partial x}(\xi, \omega)\frac{\partial h^{(1)}_3}{\partial y}(\xi, \xi)-\frac{\partial h^{(1)}_3}{\partial x}(\xi, \xi)\frac{\partial g_2}{\partial y}(\xi, \omega)\right]
\\\notag&\quad-2\pi^2\varepsilon p_1\mu^4\left[\frac{\partial h^{(1)}_2}{\partial x}(\xi, \xi)\frac{\partial g_3}{\partial y}(\xi, \omega)-\frac{\partial g_3}{\partial x}(\xi, \omega)\frac{\partial h^{(1)}_2}{\partial y}(\xi, \xi)\right]
+O\left(p_1\left(\varepsilon \mu^8d^{-12}+\varepsilon^2 \mu^6d^{-12}\right)\right)
\\\notag&=4\pi^2\varepsilon p_1\mu^6\left[\frac{\pp h_2^{(1)}}{\pp x}(\xi,\xi)\frac{\pp h_3^{(1)}}{\pp y}(\xi,\xi)-\frac{\pp h_3^{(1)}}{\pp x}(\xi,\xi)\frac{\pp h_2^{(1)}}{\pp y}(\xi,\xi)\right]
\\\notag&\quad+4\pi^2\varepsilon^2 p_1\mu^4\left[\frac{\partial h^{(1)}_2}{\partial x}(\xi, \xi)\frac{\partial h^{(1)}_3}{\partial y}(\xi, \omega)-\frac{\partial h^{(1)}_3}{\partial x}(\xi, \omega)\frac{\partial h^{(1)}_2}{\partial y}(\xi, \xi)\right]
\\ \label{D27-3}\href{}{}&\quad+O\left(p_1\left(\varepsilon \mu^6d^{-9}|\xi-\omega|+\varepsilon \mu^8d^{-12}+\varepsilon^2 \mu^6d^{-12}\right)\right).
\end{align}
In the last equality, we use the following estimate
\begin{equation}\label{F22-1}
\begin{aligned}
\left|\frac{\partial h_2^{(1)}}{\partial x}(\xi,\xi)-\frac{\partial h_2^{(1)}}{\partial x}(\xi,\omega)\right|=O\left(d^{-4}|\xi-\omega|\right),\quad \left|\frac{\partial h_2^{(1)}}{\partial y}(\xi,\xi)-\frac{\partial h_2^{(1)}}{\partial y}(\xi,\omega)\right|=O\left(d^{-4}|\xi-\omega|\right).
\end{aligned}
\end{equation}
Similar to (\ref{M1p3}), using (\ref{F22-1}), we obtain
\begin{equation}\label{Jan4-1}
\begin{aligned}
2\varepsilon^2p_1\int_{\mathcal{D}}PZ_{2,1}\cdot\left(g_x\wedge g_y\right)
&=\pi^2p_1\varepsilon^2\mu^2\left[\frac{\pp g_2}{\pp x}(\xi,\omega)\frac{\pp g_3}{\pp y}(\xi,\omega)-\frac{\pp g_3}{\pp x}(\xi,\omega)\frac{\pp g_2}{\pp y}(\xi,\omega)\right]
\\&\quad+O\left(p_1\left(\varepsilon^2\mu^{6}d^{-12}+\varepsilon^3\mu^{4}d^{-12}\right)\right)
\\&=-4\pi^2\varepsilon^3 p_1\mu^2\left[\frac{\partial h^{(1)}_2}{\partial x}(\xi, \xi)\frac{\partial h^{(1)}_3}{\partial y}(\xi, \omega)-\frac{\partial h^{(1)}_3}{\partial x}(\xi, \omega)\frac{\partial h^{(1)}_2}{\partial y}(\xi, \xi)\right]
\\&\quad+O\left(p_1\varepsilon^3\mu^2 d^{-9}|\xi-\omega|+p_1\left(\varepsilon^2\mu^{6}d^{-12}+\varepsilon^3\mu^{4}d^{-12}\right)\right).
\end{aligned}
\end{equation}
Combining (\ref{M1p3}), (\ref{D27-3}), (\ref{Jan4-1}) and using the fact that $\varepsilon d^{-1}|\xi-\omega|=O\left(\mu^2-\varepsilon\right)$, we obtain
\begin{align}
\notag&2p_1\int_{\mathcal{D}}PZ_{2,1}\cdot\left[(\varphi_1)_x\wedge (\varphi_1)_y\right]-2\varepsilon p_1\int_{\mathcal{D}}g\cdot\left[\left(PZ_{2,1}\right)_x\wedge (\varphi_1)_y+(\varphi_1)_x\wedge\left(PZ_{2,1}\right)_y\right]+2\varepsilon^2p_1\int_{\mathcal{D}}PZ_{2,1}\cdot\left(g_x\wedge g_y\right)
\\\notag&=-4\pi^2p_1(\mu^2-\varepsilon)\mu^6\left[\frac{\pp h_2^{(1)}}{\pp x}(\xi,\xi)\frac{\pp h_3^{(1)}}{\pp y}(\xi,\xi)-\frac{\pp h_3^{(1)}}{\pp x}(\xi,\xi)\frac{\pp h_2^{(1)}}{\pp y}(\xi,\xi)\right]
\\\notag&\quad+4\pi^2 p_1(\mu^2-\varepsilon)\varepsilon^2\mu^2\left[\frac{\partial h^{(1)}_2}{\partial x}(\xi, \xi)\frac{\partial h^{(1)}_3}{\partial y}(\xi, \omega)-\frac{\partial h^{(1)}_3}{\partial x}(\xi, \omega)\frac{\partial h^{(1)}_2}{\partial y}(\xi, \xi)\right]
\\\notag&\quad+O\left(p_1(\varepsilon \mu^6+\varepsilon^3\mu^2)d^{-9}|\xi-\omega|+p_1(\mu^2-\varepsilon)\left(\mu^8+\varepsilon\mu^6+\varepsilon^2\mu^4\right)d^{-12}\right)
\\\notag&\quad+O\left(p_1\left(\varepsilon \mu^8+\varepsilon^2 \mu^6+\varepsilon^3 \mu^4\right)d^{-13}|\xi-\omega|\right)
\\ \label{Jan4-2}&=O\left(p_1(\mu^2-\varepsilon)\left((\mu^6+\varepsilon^2\mu^2)d^{-8}+\left(\mu^8+\varepsilon\mu^6+\varepsilon^2\mu^4\right)d^{-12}\right)\right).
\end{align}
Similarly, for the corresponding terms that involve $p_2$, there holds
\begin{equation}\label{Jan6-2}
\begin{aligned}
&2p_2\int_{\mathcal{D}}PZ_{2,2}\cdot\left[(\varphi_1)_x\wedge (\varphi_1)_y\right]-2\varepsilon p_2\int_{\mathcal{D}}g\cdot\left[\left(PZ_{2,2}\right)_x\wedge (\varphi_1)_y+(\varphi_1)_x\wedge\left(PZ_{2,2}\right)_y\right]+2\varepsilon^2p_2\int_{\mathcal{D}}PZ_{2,2}\cdot\left(g_x\wedge g_y\right)
\\&=O\left(p_2(\mu^2-\varepsilon)\left((\mu^6+\varepsilon^2\mu^2)d^{-8}+\left(\mu^8+\varepsilon\mu^6+\varepsilon^2\mu^4\right)d^{-12}\right)\right).
\end{aligned}
\end{equation}
From (\ref{M1p1}), (\ref{Jan4-2}) and (\ref{Jan6-2}), we conclude that compared to the terms containing $p_l$, $l=1,2$ in Subsection \ref{subsection4.1}, the corresponding terms in $\mathcal{R}^{(1)}_{\mu,\xi,a,p,\varepsilon}$ are of higher order.

Secondly, we consider the terms containing  $p_kp_l$ in $\mathcal{R}^{(2)}_{\mu,\xi,a,p,\varepsilon}$. Using (\ref{A26-5}) and reasoning as (\ref{A30-p3}) and (\ref{A30-p3-a}), we can get
\begin{equation}\label{Au11-6}
\begin{aligned}
-p_1p_2\int_{\mathcal{D}}L_\delta\left[P Z_{2,2}\right]\cdot \left(PZ_{2,1}\right)
 &=-2p_1p_2\int_{\mathcal{D}}\varphi_{2,2}\cdot \left[\delta_x\wedge(Z_{2,1})_y+(Z_{2,1})_x \wedge\delta_y\right]
+O\left(p_1p_2\mu^{10}d^{-10}\right)
\\&=-\frac{8\pi}{3}p_1p_2\mu^8\frac{\partial^4 h^{(2, 2)}_1}{\partial x^4}(\xi, \xi)+O\left(p_1p_2\mu^{10}d^{-10}\right)
\\&=O\left(p_1p_2\mu^{10}d^{-10}\right),
\end{aligned}
\end{equation}
where we use the fact that $\frac{\partial^4 h^{(2, 2)}_1}{\partial x^4}(\xi, \xi)=0$.
Next, we need to estimate the following terms
\begin{align}
\notag&-2\sum_{k,l=1}^{2}p_kp_l\int_{\mathcal{D}}\left(\varphi_1-\varepsilon g\right)
\cdot \left[(PZ_{2,k})_x\wedge (PZ_{2,l})_y+\left(\delta_x\wedge \left(PZ^{2,2}_{k,l}\right)_y+\left(PZ^{2,2}_{k,l}\right)_x\wedge \delta_y\right)\right]
\\\notag&+2\sum_{k,l=1}^2p_kp_l\int_{\mathcal{D}}PZ^{2,2}_{k,l}
\cdot \left[(\varphi_1)_x\wedge (\varphi_1)_y\right]
\\& \label{D27-2}-2\varepsilon\sum_{k,l=1}^2 p_kp_l\int_{\mathcal{D}}g
\cdot \left[(\varphi_1)_x\wedge\left(PZ^{2,2}_{k,l}\right)_y+\left(PZ^{2,2}_{k,l}\right)_x\wedge (\varphi_1)_y\right]+2\varepsilon^2 \sum_{k,l=1}^2p_kp_l\int_{\mathcal{D}}PZ^{2,2}_{k,l}\cdot \left[g_x\wedge g_y\right].
\end{align}
Recall that
$$Z^{2,2}_{k,l}:=\frac{\pp^2\delta_{\mu, \xi, a, p} }{\pp p_k\pp p_l}\Bigg|_{a=0,p=0}.$$
We begin by considering the terms containing $p_1^2$ in (\ref{D27-2}).
Specifically, we have
\begin{equation}\label{D30-1}
\begin{aligned}
&-2p_1^2\int_{\mathcal{D}}\left(\varphi_1-\varepsilon g\right)
\cdot \left[(PZ_{2,1})_x\wedge (PZ_{2,1})_y+\left(\delta_x\wedge (PZ^{2,2}_{1,1})_y+(PZ^{2,2}_{1,1})_x\wedge \delta_y\right)\right]
\\&=-2p_1^2\int_{\mathcal{D}}\left(\varphi_1-\varepsilon g\right)
\cdot \left[(Z_{2,1})_x\wedge (Z_{2,1})_y+\left(\delta_x\wedge (Z^{2,2}_{1,1})_y+(Z^{2,2}_{1,1})_x\wedge \delta_y\right)\right]
\\&\quad+2p_1^2\int_{\mathcal{D}}\left(\varphi_1-\varepsilon g\right)
\cdot \left[(Z_{2,1})_x\wedge (\varphi_{2,1})_y+(\varphi_{2,1})_x\wedge (Z_{2,1})_y\right]
\\&\quad+O\left(p_1^2(\mu^2-\varepsilon)\left(\mu^8+\mu^6(\mu^2+\varepsilon)\right)d^{-12}+p_1^2\varepsilon\mu^8d^{-13}|\xi-\omega|\right).
\end{aligned}
\end{equation}

Using (\ref{D27-1}), we obtain
\begin{align*}
&-2p_1^2\int_{\mathcal{D}}\varphi_1
\cdot \left[(Z_{2,1})_x\wedge (Z_{2,1})_y+\left(\delta_x\wedge (Z^{2,2}_{1,1})_y+(Z^{2,2}_{1,1})_x\wedge \delta_y\right)\right]
\\&=-4p_1^2\mu^2\int_{\mathbb{R}^2} \left[h_1^{(1)}(\xi,\xi)+\frac{\mu^2}{2}\left(x^2\frac{\partial^2 h_1^{(1)}}{\partial x^2}(\xi,\xi)+y^2\frac{\partial^2 h_1^{(1)}}{\partial y^2}(\xi,\xi)\right)\right.
\\&\left.\qquad\qquad\qquad\quad+\frac{\mu^4}{24}\left(x^4\frac{\pp^4 h_1^{(1)}}{\pp x^4}(\xi,\xi)+6x^2y^2\frac{\pp^4 h_1^{(1)}}{\pp x^2\pp y^2}(\xi,\xi)+y^4\frac{\pp^4 h_1^{(1)}}{\pp y^4}(\xi,\xi)\right)
\right.
\\&\left.\qquad\qquad\qquad\quad+\frac{\mu^6}{720}\left(x^6\frac{\pp^6 h_1^{(1)}}{\pp x^6}(\xi,\xi)+15x^4y^2\frac{\pp^6 h_1^{(1)}}{\pp x^4\pp y^2}(\xi,\xi)+15x^2y^4\frac{\pp^6 h_1^{(1)}}{\pp x^2\pp y^4}(\xi,\xi)+y^6\frac{\pp^6 h_1^{(1)}}{\pp y^6}(\xi,\xi)\right)\right]
\\&\qquad\qquad\qquad\times\frac{-96 (x-y) (x+y) \left(x^2+y^2\right) \left(5 x^4-22 x^2 y^2+5 y^4-3\right)}{\left(\left(x^2+y^2\right)^2+1\right)^5}
\\&\quad -4p_1^2\mu^2\int_{\mathbb{R}^2}\left[\mu^2xy \frac{\partial^2 h_2^{(1)}}{\partial x \partial y}(\xi,\xi)+\frac{\mu^4}{24}\left(4x^3y\frac{\pp^4 h_2^{(1)}}{\pp x^3\pp y}(\xi,\xi)+4xy^3\frac{\pp^4 h_2^{(1)}}{\pp x\pp y^3}(\xi,\xi)\right)
\right.
\\&\left.\qquad\qquad\qquad\quad+\frac{\mu^6}{720}\left(6x^5y\frac{\pp^6 h_2^{(1)}}{\pp x^5\pp y}(\xi,\xi)+20x^3y^3\frac{\pp^6 h_2^{(1)}}{\pp x^3\pp y^3}(\xi,\xi)+6xy^5\frac{\pp^6 h_2^{(1)}}{\pp x\pp y^5}(\xi,\xi)\right)\right]
\\&\qquad\qquad\qquad\times \frac{-192 x y \left(x^2+y^2\right) \left(7 x^4-18
   x^2 y^2+7 y^4-1\right)}{\left(\left(x^2+y^2\right)^2+1\right)^5}
\\&\quad-2p_1^2 \int_{\mathbb{R}^2}\left[1-2\mu^4\left[h_3^{(1)}(\xi,\xi)+\frac{\mu^2}{2}\left(x^2\frac{\partial^2 h_3^{(1)}}{\partial x^2}(\xi,\xi)+y^2\frac{\partial^2 h_3^{(1)}}{\partial y^2}(\xi,\xi)\right)\right. \right.
\\&\left.\left.\qquad\qquad\qquad\qquad\qquad+\frac{\mu^4}{24}\left(x^4\frac{\pp^4 h_3^{(1)}}{\pp x^4}(\xi,\xi)+6x^2y^2\frac{\pp^4 h_3^{(1)}}{\pp x^2\pp y^2}(\xi,\xi)+y^4\frac{\pp^4 h_3^{(1)}}{\pp y^4}(\xi,\xi)\right)\right]\right]
\\&\qquad\qquad\quad\times\frac{-32 \left(x^2+y^2\right) \left(5 x^8-4 x^6 y^2-x^4 \left(18 y^4+17\right)+x^2 \left(38 y^2-4 y^6\right)+5
   y^8-17 y^4+2\right)}{\left(\left(x^2+y^2\right)^2+1\right)^5}
\\&\quad+O\left(p_1^2\mu^{10}d^{-10}\right)
\\&=\frac{\pi^2}{2}p_1^2\mu^6\left[\frac{\partial^2 h_3^{(1)}}{\partial x^2}(\xi,\xi)+\frac{\partial^2 h_3^{(1)}}{\partial y^2}(\xi,\xi)\right]+
\frac{3\pi^2}{32}p_1^2\mu^6\left[\frac{\partial^4 h_1^{(1)}}{\partial x^4}(\xi,\xi)-\frac{\partial^4 h_1^{(1)}}{\partial y^4}(\xi,\xi)\right]
\\&\quad-
\frac{\pi^2}{16}p_1^2\mu^6\left[\frac{\partial^4 h_2^{(1)}}{\partial x^3\pp y}(\xi,\xi)+\frac{\partial^4 h_2^{(1)}}{\partial x\pp y^3}(\xi,\xi)\right]
\\&\quad+\frac{\pi}{30}
 p_1^2\mu^8\left[\frac{\partial^6 h_1^{(1)}}{\partial x^6}(\xi,\xi)-\frac{\partial^6 h_1^{(1)}}{\pp y^6}(\xi,\xi)+\frac{\partial^6 h_2^{(1)}}{\partial x^5\pp y}(\xi,\xi)+\frac{\partial^6 h_2^{(1)}}{\partial x\pp y^5}(\xi,\xi)\right]
+O\left(p_1^2\mu^{10}d^{-10}\right)
\\&=\frac{\pi}{15}
 p_1^2\mu^8\left[\frac{\partial^6 h_1^{(1)}}{\partial x^6}(\xi,\xi)-\frac{\partial^6 h_1^{(1)}}{\pp y^6}(\xi,\xi)\right]
+O\left(p_1^2\mu^{10}d^{-10}\right).
\end{align*}
Since $g(z,\omega)=\left(2h_1^{(1)}(z,\omega),2h_2^{(1)}(z,\omega),-2\varepsilon h_3^{(1)}(z,\omega)\right)$, we deduce
\begin{equation}\label{D28-3}
\begin{aligned}
&-2p_1^2\int_{\mathcal{D}}\left(\varphi_1-\varepsilon g\right)
\cdot \left[(Z_{2,1})_x\wedge (Z_{2,1})_y+\left(\delta_x\wedge (Z^{2,2}_{1,1})_y+(Z^{2,2}_{1,1})_x\wedge \delta_y\right)\right]
\\&=\frac{\pi}{15}
 p_1^2(\mu^2-\varepsilon)\mu^6\left[\frac{\partial^6 h_1^{(1)}}{\partial x^6}(\xi,\xi)-\frac{\partial^6 h_1^{(1)}}{\pp y^6}(\xi,\xi)\right]
+O\left(p_1^2\varepsilon\mu^6d^{-9}|\xi-\omega|+p_1^2(\mu^2-\varepsilon)\left(\mu^8+\mu^6\left(\mu^2+\varepsilon\right)\right)d^{-10}\right)
\\&=O\left(p_1^2(\mu^2-\varepsilon)\mu^6d^{-8}\right).
\end{aligned}
\end{equation}
For the second term in (\ref{D30-1}), similar to the calculation of (\ref{D27-3}), one has
\begin{align}
\notag&2p_1^2\int_{\mathcal{D}}\left(\varphi_1-\varepsilon g\right)
\cdot \left[(Z_{2,1})_x\wedge (\varphi_{2,1})_y+(\varphi_{2,1})_x\wedge (Z_{2,1})_y\right]
\\\notag&=-4\pi^2p_1^2\left(\mu^2-\varepsilon \right)\mu^8\left[\frac{\partial h^{(1)}_2}{\partial x}(\xi, \xi)\frac{\partial h^{(2,1)}_3}{\partial y}(\xi, \xi)-\frac{\partial h^{(2,1)}_3}{\partial x}(\xi, \xi)\frac{\partial h^{(1)}_2}{\partial y}(\xi, \xi)\right]
\\ \notag&\quad-4\pi^2p_1^2\left(\mu^2-\varepsilon \right)\mu^6\left(\mu^2+\varepsilon \right)\left[\frac{\partial h^{(2,1)}_2}{\partial x}(\xi, \xi)\frac{\partial h^{(1)}_3}{\partial y}(\xi, \xi)-\frac{\partial h^{(1)}_3}{\partial x}(\xi, \xi)\frac{\partial h^{(2,1)}_2}{\partial y}(\xi, \xi)\right]
\\\notag&\quad+O\left(p_1^2\left(\mu^8\varepsilon d^{-11}|\xi-\omega|+\mu^6\varepsilon^2 d^{-11}|\xi-\omega|+\left(\mu^2-\varepsilon \right)\left(\mu^{10}+\mu^8\left(\mu^2+\varepsilon\right)\right)d^{-12}\right)\right)
\\ \label{Jan2-7}&=O\left(p_1^2\left(\mu^2-\varepsilon \right)\mu^6(\mu^2+\varepsilon)d^{-10}\right).
\end{align}
Moreover, similar to the estimate of (\ref{Jan4-2}), there holds
\begin{equation}\label{D28-1}
\begin{aligned}
&2p_1^2\int_{\mathcal{D}}PZ^{2,2}_{1,1}
\cdot \left[(\varphi_1)_x\wedge (\varphi_1)_y\right]
\\&-2\varepsilon p_1^2\int_{\mathcal{D}}g
\cdot \left[(\varphi_1)_x\wedge\left(PZ^{2,2}_{1,1}\right)_y+\left(PZ^{2,2}_{1,1}\right)_x\wedge (\varphi_1)_y\right]+2\varepsilon^2 p_1^2\int_{\mathcal{D}}PZ^{2,2}_{1,1}\cdot \left[g_x\wedge g_y\right]
\\&=2\pi^2p_1^2(\mu^2-\varepsilon)\mu^4\left[\frac{\partial h^{(1)}_1}{\partial x}(\xi, \xi)\frac{\partial h^{(1)}_2}{\partial y}(\xi, \xi)-\frac{\partial h^{(1)}_1}{\partial y}(\xi, \xi)\frac{\partial h^{(1)}_2}{\partial x}(\xi, \xi)\right]
\\&\quad-2\pi^2p_1^2(\mu^2-\varepsilon)\varepsilon\mu^2\left[\frac{\partial h^{(1)}_1}{\partial x}(\xi, \xi)\frac{\partial h^{(1)}_2}{\partial y}(\xi, \omega)-\frac{\partial h^{(1)}_1}{\partial y}(\xi, \omega)\frac{\partial h^{(1)}_2}{\partial x}(\xi, \xi)\right]
\\&\quad+O\left(p_1^2(\mu^2-\varepsilon)\left(\left(\mu^{10}+\mu^8(\mu^2+\varepsilon)\right)d^{-14}+\varepsilon\mu^8d^{-12}+\varepsilon^2\mu^6d^{-10}\right)\right)
\\&=O\left(p_1^2(\mu^2-\varepsilon)\left((\mu^4+\varepsilon\mu^2)d^{-6}+(\mu^8(\mu^2+\varepsilon)+\varepsilon^2\mu^6)d^{-14}\right)\right).
\end{aligned}
\end{equation}
Therefore, combining (\ref{D30-1}), (\ref{D28-3}), (\ref{Jan2-7}) and (\ref{D28-1}), we conclude that as $|\xi|\to 1$, with $\mu^2-\varepsilon=O\left(\varepsilon d_{\varepsilon}^4\right)$ and $d_{\varepsilon}=O(\mu^{\frac{2}{3}-\alpha})$ for $\alpha>0$ small, compared to (\ref{A30-p3-a}), the terms in $\mathcal{R}^{(2)}_{\mu,\xi,a,p,\varepsilon}$ that contain $p_1^2$ are of higher order.

Follow the same path, using (\ref{D27-4}) and (\ref{D27-5}), we obtain
\begin{align*}
&-2p_2^2\int_{\mathcal{D}}\varphi_1
\cdot \left[(PZ_{2,2})_x\wedge (PZ_{2,2})_y+\left(\delta_x\wedge (PZ^{2,2}_{2,2})_y+(PZ^{2,2}_{2,2})_x\wedge \delta_y\right)\right]
\\&=\frac{\pi^2}{2}p_2^2\mu^6\left[\frac{\partial^2 h_3^{(1)}}{\partial x^2}(\xi,\xi)+\frac{\partial^2 h_3^{(1)}}{\partial y^2}(\xi,\xi)\right]+
\frac{\pi^2}{32}p_2^2\mu^6\left[\frac{\partial^4 h_1^{(1)}}{\partial x^4}(\xi,\xi)-\frac{\partial^4 h_1^{(1)}}{\partial y^4}(\xi,\xi)\right]
\\&\quad+
\frac{3\pi^2}{16}p_2^2\mu^6\left[\frac{\partial^4 h_2^{(1)}}{\partial x^3\pp y}(\xi,\xi)+\frac{\partial^4 h_2^{(1)}}{\partial x\pp y^3}(\xi,\xi)\right]
\\&\quad+\frac{\pi}{120}
 p_2^2\mu^8\left[\frac{\partial^6 h_1^{(1)}}{\partial x^6}(\xi,\xi)+5\frac{\partial^6 h_1^{(1)}}{\partial x^4y^2}(\xi,\xi)-5\frac{\partial^6 h_1^{(1)}}{\partial x^2y^4}(\xi,\xi)-\frac{\partial^6 h_1^{(1)}}{\pp y^6}(\xi,\xi)\right]
\\&\quad+\frac{\pi}{20}
 p_2^2\mu^8\left[\frac{\partial^6 h_2^{(1)}}{\partial x^5\pp y}(\xi,\xi)+\frac{10}{3}\frac{\partial^6 h_2^{(1)}}{\partial x^3\pp y^3}(\xi,\xi)+\frac{\partial^6 h_2^{(1)}}{\partial x\pp y^5}(\xi,\xi)\right]
+O\left(p_2^2\mu^{10}d^{-10}+p_1^2\mu^{10}d^{-12}\right)
\\&=-\frac{\pi}{15}
 p_2^2\mu^8\left[\frac{\partial^6 h_1^{(1)}}{\partial x^6}(\xi,\xi)-\frac{\partial^6 h_1^{(1)}}{\pp y^6}(\xi,\xi)\right]
+O\left(p_2^2\mu^{10}d^{-10}+p_2^2\mu^{10}d^{-12}\right)
\\&=O\left(p_2^2\mu^8d^{-8}+p_2^2\mu^{10}d^{-12}\right),
\end{align*}
then we deduce
\begin{equation}\label{Jan2-8}
\begin{aligned}
&-2p_2^2\int_{\mathcal{D}}\left(\varphi_1-\varepsilon g\right)
\cdot \left[(PZ_{2,2})_x\wedge (PZ_{2,2})_y+\left(\delta_x\wedge (PZ^{2,2}_{2,2})_y+(PZ^{2,2}_{2,2})_x\wedge \delta_y\right)\right]
\\&=O\left(p_2^2(\mu^2-\varepsilon)\left(\mu^6d^{-8}+\mu^6(\mu^2+\varepsilon)d^{-12}\right)\right).
\end{aligned}
\end{equation}
And
\begin{align*}
&-2p_1p_2\int_{\mathcal{D}} \varphi_1\cdot \left[\left(PZ_{2,1}\right)_x\wedge \left(PZ_{2,2}\right)_y+\left(P Z_{2,2}\right)_x \wedge\left(PZ_{2,1}\right)_y+\delta_x\wedge (PZ^{2,2}_{1,2})_y+(PZ^{2,2}_{1,2})_x\wedge \delta_y\right]
\\&=\frac{\pi^2}{8}p_1p_2\mu^6\left[\frac{\partial^4 h_1^{(1)}}{\partial x^3\pp y}(\xi,\xi)+\frac{\partial^4 h_1^{(1)}}{\partial x\pp y^3}(\xi,\xi)\right]+\frac{\pi^2}{16}p_1p_2\mu^6\left[\frac{\partial^4 h_2^{(1)}}{\partial x^4}(\xi,\xi)-\frac{\partial^4 h_2^{(1)}}{\pp y^4}(\xi,\xi)\right]
\\&\quad+\frac{\pi}{15}
 p_1 p_2\mu^8\left[\frac{\partial^6 h_1^{(1)}}{\partial x^5\pp y}(\xi,\xi)+\frac{\partial^6 h_1^{(1)}}{\partial xy^5}(\xi,\xi)\right]
\\&\quad+\frac{\pi}{60}
 p_1 p_2\mu^8\left[\frac{\partial^6 h_2^{(1)}}{\partial x^6}(\xi,\xi)+5\frac{\partial^6 h_2^{(1)}}{\partial x^4\pp y^2}(\xi,\xi)-5\frac{\partial^6 h_2^{(1)}}{\partial x^2\pp y^4}(\xi,\xi)-\frac{\partial^6 h_2^{(1)}}{\pp y^6}(\xi,\xi)\right]
\\&\quad+O\left(p_1p_2\mu^{10}d^{-10}+p_1p_2\mu^{10}d^{-12}\right)
\\&=-\frac{\pi}{15}
  p_1 p_2\mu^8\left[\frac{\partial^6 h_1^{(1)}}{\partial x^6}(\xi,\xi)-\frac{\partial^6 h_1^{(1)}}{\pp y^6}(\xi,\xi)\right]+O\left(p_1p_2\mu^{10}d^{-10}+p_1p_2\mu^{10}d^{-12}\right)
\\&=O\left(p_1p_2\mu^{8}d^{-8}+p_1p_2\mu^{10}d^{-12}\right),
\end{align*}
then
\begin{equation}\label{Jan2-9}
\begin{aligned}
&-2p_1p_2\int_{\mathcal{D}} \left(\varphi_1-\varepsilon g\right)\cdot \left[\left(PZ_{2,1}\right)_x\wedge \left(PZ_{2,2}\right)_y+\left(P Z_{2,2}\right)_x \wedge\left(PZ_{2,1}\right)_y+\delta_x\wedge (PZ^{2,2}_{1,2})_y+(PZ^{2,2}_{1,2})_x\wedge \delta_y\right]
\\&=O\left(p_1p_2(\mu^2-\varepsilon)\left(\mu^6d^{-8}+\mu^6(\mu^2+\varepsilon)d^{-12}\right)\right).
\end{aligned}
\end{equation}
And similar to (\ref{D28-1}), one has
\begin{equation}\label{D28-4}
\begin{aligned}
&2p_2^2\int_{\mathcal{D}}PZ^{2,2}_{2,2}
\cdot \left[(\varphi_1)_x\wedge (\varphi_1)_y\right]+2p_1p_2\int_{\mathcal{D}}PZ^{2,2}_{1,2}
\cdot \left[(\varphi_1)_x\wedge (\varphi_1)_y\right]
\\&-2\varepsilon p_2^2\int_{\mathcal{D}}g
\cdot \left[\left(PZ^{2,2}_{2,2}\right)_x\wedge (\varphi_1)_y+(\varphi_1)_x\wedge\left(PZ^{2,2}_{2,2}\right)_y\right]+2\varepsilon^2 p_2^2\int_{\mathcal{D}}PZ^{2,2}_{2,2}\cdot \left[g_x\wedge g_y\right]
\\&-2\varepsilon p_1p_2\int_{\mathcal{D}}g
\cdot \left[\left(PZ^{2,2}_{1,2}\right)_x\wedge (\varphi_1)_y+(\varphi_1)_x\wedge\left(PZ^{2,2}_{1,2}\right)_y\right]
+2\varepsilon^2 p_1p_2\int_{\mathcal{D}}PZ^{2,2}_{1,2}\cdot \left[g_x\wedge g_y\right]
\\&=O\left(\left(p_2^2+p_1p_2\right)(\mu^2-\varepsilon)\left((\mu^4+\varepsilon\mu^2)d^{-6}+(\mu^8(\mu^2+\varepsilon)+\varepsilon^2\mu^6)d^{-14}\right)\right).
\end{aligned}
\end{equation}
Therefore, compared to the results in Subsection \ref{subsection4.1}, the terms in $\mathcal{R}^{(2)}_{\mu,\xi,a,p,\varepsilon}$ that contain $p_2^2$ and $p_1p_2$ are also of higher order.

{\bf$\bullet$ Cancellations of the terms containing $a_k$ or $a_ka_l$.}
We now turn to the terms in (\ref{Jan6-1}) that contain $a_l$, $l=1,2$. We begin with
$-2\sum_{l=1}^2 a_l\int_{\mathcal{D}}\left(\varphi_1-\varepsilon g\right)
\cdot \left[\delta_x\wedge (PZ_{-1,l})_y+(PZ_{-1,l})_x\wedge \delta_y\right].$
Similar to (\ref{Jan5-2}), we have
\begin{equation*}
\begin{aligned}
&2\sum_{l=1}^2a_l\int_{\mathcal{D}}\left(\varphi_1-\varepsilon g\right)
\cdot \left[\delta_x\wedge \left(\varphi_{-1,l}\right)_y+\left(\varphi_{-1,l}\right)_x\wedge \delta_y\right]
\\&=O\left((a_1+a_2)\left(\left(\mu^2-\varepsilon\right)\mu^3d^{-5}+\varepsilon\mu^3 d^{-6}|\xi-\omega|\right)\right).
\end{aligned}
\end{equation*}
For $-2\sum_{l=1}^2 a_l\int_{\mathcal{D}}\left(\varphi_1-\varepsilon g\right)
\cdot \left[\delta_x\wedge \left(Z_{-1,l}\right)_y+\left(Z_{-1,l}\right)_x\wedge \delta_y\right]$, we omit the detailed calculation, as it closely resembles the calculations in (\ref{Mar22-1}) and (\ref{M27-a2}). The key difference lies in the definition of $\varphi_1$ and $\varepsilon g$
and the associated changes in the constants. Thus we conclude that
\begin{equation}\label{A30-a5}
\begin{aligned}
&-2\sum_{l=1}^2 a_l\int_{\mathcal{D}}\left(\varphi_1-\varepsilon g\right)
\cdot \left[\delta_x\wedge (PZ_{-1,l})_y+(PZ_{-1,l})_x\wedge \delta_y\right]
\\&=-32\pi a_1\left(\mu^2-\varepsilon\right)\mu\frac{\partial h_1^{(1)}}{\partial x}(\xi,\xi)
-32\pi a_2\left(\mu^2-\varepsilon\right)\mu\frac{\partial h_2^{(1)}}{\partial x}(\xi,\xi)
+O\left((a_1+a_2)\left(\varepsilon\mu d^{-4}|\xi-\omega|+\left(\mu^2-\varepsilon\right)\mu^3d^{-5}\right)\right)
\\&=O\left((a_1+a_2)\left(\left(\mu^2-\varepsilon\right)\mu d^{-3}\right)\right).
\end{aligned}
\end{equation}
And similar to (\ref{Jan4-2}), we can obtain
\begin{equation}\label{Jan4-3}
\begin{aligned}
&2\int_{\mathcal{D}}\left(\sum_{l=1}^2 a_lPZ_{-1,l}\right)\cdot [(\varphi_1)_x\wedge (\varphi_1)_y]-2\varepsilon \sum_{l=1}^2 a_l\int_{\mathcal{D}}g\cdot\left[\left(PZ_{-1,l}\right)_x\wedge (\varphi_1)_y+(\varphi_1)_x\wedge\left(PZ_{-1,l}\right)_y\right]
\\&+2\varepsilon^2\sum_{l=1}^2 a_l\int_{\mathcal{D}}PZ_{-1,l}\cdot\left(g_x\wedge g_y\right)
\\&=O\left((a_1+a_2)(\mu^2-\varepsilon)\left(\mu^3\left(\mu^2+\varepsilon\right)+\varepsilon^2\mu\right)d^{-9}\right).
\end{aligned}
\end{equation}

For the terms that contain $a_ka_l$, $k,l=1,2$, using (\ref{expansion6-1}), computing analogously with (\ref{A30-a1}), we have
\begin{equation}\label{D24-2}
\begin{aligned}
-a_1a_2\int_{\mathcal{D}}L_\delta\left[P Z_{-1,2}\right]\cdot \left( PZ_{-1,1}\right)
&=-2a_1a_2\int_{\mathcal{D}}\varphi_{-1,2}\cdot\left[\delta_x\wedge \left( PZ_{-1,1}\right)_y+\left(PZ_{-1,1}\right)_x \wedge\delta_y\right]
\\&=-32\pi a_1a_2\mu^2\frac{\partial h^{(-1,2)}_1}{\partial x}(\xi,\xi)
+O\left(a_1a_2\mu^4d^{-4}\right)
\\&=O\left(a_1a_2\mu^4d^{-4}\right),
\end{aligned}
\end{equation}
where we use the facts stated in Lemma \ref{evalues-at-xi-2}, i.e. $\frac{\partial h^{(-1,2)}_1}{\partial x}(\xi,\xi)=0.$

Moreover, from (\ref{D31-1}), (\ref{D31-2}) and (\ref{D31-3}), it is easy to see that
\begin{align*}
&-2\sum_{k,l=1}^2a_ka_l\int_{\mathcal{D}}\varphi_1
\cdot \left[(P Z_{-1,k})_x\wedge (P Z_{-1,l})_y+\left(\delta_x\wedge (PZ^{-1,-1}_{k,l})_y+(PZ^{-1,-1}_{k,l})_x\wedge \delta_y\right)\right]
\\&=-4a_1^2\mu^2\int_{\mathbb{R}^2}\left[h_1^{(1)}(\xi,\xi)+O\left(\mu^2d^{-4}\right)\right]
 \frac{-192 \left(x^2+y^2\right)f_3(x,y)}{\left(\left(x^2+y^2\right)^2+1\right)^5}
\\&\quad-4a_2^2\mu^2\int_{\mathbb{R}^2}\left[h_1^{(1)}(\xi,\xi)+O\left(\mu^2d^{-4}\right)\right]
\frac{192 \left(x^2+y^2\right)f_6(x,y)}{\left(\left(x^2+y^2\right)^2+1\right)^5}
\\&\quad-4a_1a_2\mu^2 \int_{\mathbb{R}^2}\left[h^{(1)}_2(\xi,\xi)+O\left(\mu^2d^{-4}\right)\right]\frac{-192 \left(x^2+y^2\right) f_9(x,y)}{\left(\left(x^2+y^2\right)^2+1\right)^5}
+O\left(\left(a_1+a_2\right)^2\mu^6d^{-8}\right)
\\&=O\left(\left(a_1+a_2\right)^2\mu^4d^{-4}+\mu^6d^{-8}\right),
\end{align*}
thus
\begin{align}
\notag&-2\sum_{k,l=1}^2a_ka_l\int_{\mathcal{D}}\left(\varphi_1-\varepsilon g\right)
\cdot \left[(P Z_{-1,k})_x\wedge (P Z_{-1,l})_y+\left(\delta_x\wedge (PZ^{-1,-1}_{k,l})_y+(PZ^{-1,-1}_{k,l})_x\wedge \delta_y\right)\right]
\\ \notag&=O\left(\left(a_1+a_2\right)^2\left(\left(\mu^2-\varepsilon\right)\left(\mu^2d^{-4}+\mu^2(\mu^2+\varepsilon)d^{-8}\right)+\mu^2\varepsilon d^{-5}|\xi-\omega|+\mu^2\varepsilon^2d^{-9}|\xi-\omega|\right)\right)
\\ \label{A30-a2}&=O\left(\left(a_1+a_2\right)^2\left(\mu^2-\varepsilon\right)\mu^2(\mu^2+\varepsilon)d^{-8}\right).
\end{align}
And similar to (\ref{D28-1}), there holds
\begin{equation}\label{Jan2-10}
\begin{aligned}
&2\sum_{k,l=1}^2a_ka_l\int_{\mathcal{D}}PZ^{-1,-1}_{k,l}
\cdot \left[(\varphi_1)_x\wedge (\varphi_1)_y\right]
\\&-2\varepsilon\sum_{k,l=1}^2 a_ka_l\int_{\mathcal{D}}g
\cdot \left[(\varphi_1)_x\wedge\left(PZ^{-1,-1}_{k,l}\right)_y+\left(PZ^{-1,-1}_{k,l}\right)_x\wedge (\varphi_1)_y\right]+2\varepsilon^2 \sum_{k,l=1}^2a_ka_l\int_{\mathcal{D}}PZ^{-1,-1}_{k,l}\cdot \left[g_x\wedge g_y\right]
\\&=O\left(\left(a_1+a_2\right)^2\left(\mu^2-\varepsilon\right)\left((\mu^6+\varepsilon^2 \mu^2)d^{-8}+(\mu^4|\log\mu|+\varepsilon \mu^2|\log\mu|)d^{-6}\right)\right)
\\&\quad+O\left(\left(a_1+a_2\right)^2(\mu^2-\varepsilon)\left(\mu^6(\mu^2+\varepsilon)+\varepsilon^2 \mu^4\right)d^{-12}\right).
\end{aligned}
\end{equation}
Thus, compared to the results in Subsection \ref{subsection4.1}, the terms in $\mathcal{R}^{(1)}_{\mu,\xi,a,p,\varepsilon}$ and $\mathcal{R}^{(2)}_{\mu,\xi,a,p,\varepsilon}$ that involve $a_k$ and $a_ka_l$, $k,l=1,2$ are of higher order.

{\bf$\bullet$ Cancellations of the terms containing $a_kp_l$.}
Now, we deal with the mixed terms that contain $a_kp_l$ for $k,l=1,2$ in $\mathcal{R}^{(2)}_{\mu,\xi,a,p,\varepsilon}$.
Firstly, using (\ref{D31-6}), we obtain
\begin{equation*}
\begin{aligned}
&-2a_1p_1\int_{\mathcal{D}}\left(\varphi_1-\varepsilon g\right)\cdot \left[\left(P Z_{-1,1}\right)_x\wedge \left( PZ_{2,1}\right)_y+\left( PZ_{2,1}\right)_x \wedge\left(P Z_{-1,1}\right)_y+\delta_x\wedge \left(P Z^{-1,2}_{1,1}\right)_y+\left(P Z^{-1,2}_{1,1}\right)_x\wedge \delta_y\right]
\\&=-2a_1p_1\int_{\mathcal{D}}\left(\varphi_1-\varepsilon g\right)\cdot \left[\left(Z_{-1,1}\right)_x\wedge \left(Z_{2,1}\right)_y+\left(Z_{2,1}\right)_x \wedge\left(Z_{-1,1}\right)_y+\delta_x\wedge \left(Z^{-1,2}_{1,1}\right)_y+\left(Z^{-1,2}_{1,1}\right)_x\wedge \delta_y\right]
\\&\quad+O\left(a_1p_1\left(\mu^2-\varepsilon \right)\left(\mu^3d^{-5}+\mu^5(\mu^2+\varepsilon)d^{-11}\right)\right)
\\&=O\left(a_1p_1\left(\mu^2-\varepsilon \right)\left(\mu^3d^{-5}+\mu^5(\mu^2+\varepsilon)d^{-11}\right)\right).
\end{aligned}
\end{equation*}
The remaining terms  can be estimated in a similar manner, and we omit the details here for brevity. There holds
\begin{equation}\label{Jan1-1}
\begin{aligned}
&-2\int_{\mathcal{D}}\left(\varphi_1-\varepsilon g\right)\cdot \left[\left(\sum_{l=1}^2a_lP Z_{-1,l}\right)_x\wedge \left(\sum_{l=1}^2p_l PZ_{2,l}\right)_y+\left(\sum_{l=1}^2p_l PZ_{2,l}\right)_x \wedge\left(\sum_{l=1}^2a_lP Z_{-1,l}\right)_y\right]
\\&-2\int_{\mathcal{D}}\left(\varphi_1-\varepsilon g\right)
\cdot \left[\delta_x\wedge \left(\sum_{k,l=1}^2a_kp_lP Z^{-1,2}_{k,l}\right)_y+\left(\sum_{k,l=1}^2a_kp_lP Z^{-1,2}_{k,l}\right)_x\wedge \delta_y\right]
\\&=O\left((a_1+a_2)(p_1+p_2)\left(\mu^2-\varepsilon \right)\left(\mu^3d^{-5}+\mu^5(\mu^2+\varepsilon)d^{-11}\right)\right).
\end{aligned}
\end{equation}
And
\begin{equation}\label{D31-4}
\begin{aligned}
&2\sum_{k,l=1}^2a_kp_l\int_{\mathcal{D}}P Z^{-1,2}_{k,l}
\cdot \left[(\varphi_1)_x\wedge (\varphi_1)_y\right]
\\&-2\varepsilon\sum_{k,l=1}^2 a_kp_l\int_{\mathcal{D}}g
\cdot \left[(\varphi_1)_x\wedge\left(P Z^{-1,2}_{k,l}\right)_y+\left(P Z^{-1,2}_{k,l}\right)_x\wedge (\varphi_1)_y\right]+2\varepsilon^2 \sum_{k,l=1}^2a_kp_l\int_{\mathcal{D}}P Z^{-1,2}_{k,l}\cdot \left[g_x\wedge g_y\right]
\\&=O\left((a_1+a_2)(p_1+p_2)\left(\mu^2-\varepsilon \right)\left((\mu^6+\varepsilon^2 \mu^2)d^{-8}+(\mu^4+\varepsilon \mu^2)d^{-6}\right)\right)
\\&\quad+O\left((a_1+a_2)(p_1+p_2)(\mu^2-\varepsilon)\left(\mu^5(\mu^2+\varepsilon)+\varepsilon^2 \mu^3\right)d^{-11}\right).
\end{aligned}
\end{equation}

Recall that $\mathcal{R}_{\mathcal{A}}$ denotes the higher order expansions of $\delta_{\mu,\xi, a,p}$, in (\ref{D30-1})-(\ref{D28-4}) and (\ref{A30-a2})-(\ref{D31-4}), we have estimated the corresponding terms that of second order expansion. Similar to the reasoning above, the remaining terms in
\begin{equation*}
\begin{aligned}
&-2\int_{\mathcal{D}}\varphi_1
\cdot \left[\delta_x\wedge \left(P\mathcal{R}_{\mathcal{A}}\right)_y+\left(P\mathcal{R}_{\mathcal{A}}\right)_x\wedge \delta_y\right]
+2\int_{\mathcal{D}} \left(P\mathcal{R}_{\mathcal{A}}\right)\cdot\left[(\varphi_1)_x\wedge (\varphi_1)_y\right]
\\&+2\varepsilon\int_{\mathcal{D}}g\cdot  \left[\left(P\delta\right)_x\wedge \left(P\mathcal{R}_{\mathcal{A}}\right)_y+\left(P\mathcal{R}_{\mathcal{A}}\right)_x\wedge \left(P\delta\right)_y\right]
+2\varepsilon^2\int_{\mathcal{D}}\left(P\mathcal{R}_{\mathcal{A}}\right)\cdot (g_x\wedge g_y)
\end{aligned}
\end{equation*}
can be estimated as
\begin{equation}\label{F24-1}
\begin{aligned}
O\left(\sum_{k,l,m=1}^2\left(a_ka_la_m+a_ka_lp_m+a_kp_lp_m+p_kp_lp_m\right)(\mu^2-\varepsilon)\left(1+(\mu^4+\varepsilon\mu^2)d^{-6}+\mu(\mu^2+\varepsilon)d^{-7}\right)\right).
\end{aligned}
\end{equation}

Therefore, we get the estimate of $\mathcal{R}^{(1)}_{\mu,\xi,a,p,\varepsilon}$ and $\mathcal{R}^{(2)}_{\mu,\xi,a,p,\varepsilon}$. Indeed, combining the fact that $|\xi|\to 1$, $d=d_{\varepsilon}=O(\mu^{\frac{2}{3}-\alpha})$ for $\alpha>0$ small, we conclude that
\begin{equation}\label{Jan8-4}
\begin{aligned}
\mathcal{R}^{(1)}_{\mu,\xi,a,p,\varepsilon}
&=(\ref{M1p1})+(\ref{Jan4-2})+(\ref{Jan6-2})+(\ref{A30-a5})+(\ref{Jan4-3})
\\&=O\left(\left(p_1+p_2\right)(\mu^2-\varepsilon)\left(\mu^4d^{-6}+(\mu^8+\varepsilon\mu^6+\varepsilon^2\mu^4)d^{-12}\right)\right)
\\&\quad+O\left((a_1+a_2)(\mu^2-\varepsilon)\left(\mu d^{-3}+\left(\mu^3\left(\mu^2+\varepsilon\right)+\varepsilon^2\mu\right)d^{-9}\right)\right).
\end{aligned}
\end{equation}
And
\begin{align}
\mathcal{R}^{(2)}_{\mu,\xi,a,p,\varepsilon}
\notag&=(\ref{Au11-6})+(\ref{D30-1})+(\ref{D28-3})+(\ref{Jan2-7})+(\ref{D28-1})+(\ref{Jan2-8})+(\ref{Jan2-9})+(\ref{D28-4})
\\ \notag&\quad+(\ref{D24-2})+(\ref{A30-a2})+(\ref{Jan2-10})+(\ref{Jan1-1})+(\ref{D31-4})+(\ref{F24-1})
\\ \notag&=O\left(p_1p_2\mu^{10}d^{-10}+\sum_{k,l=1}^2p_kp_l(\mu^2-\varepsilon)\left(\mu^6(\mu^2+\varepsilon)d^{-12}+(\mu^4+\varepsilon\mu^2)d^{-6}+\varepsilon^2\mu^6d^{-14}\right)\right)
\\ \notag&\quad+O\left(a_1a_2\mu^4d^{-4}+\sum_{k,l=1}^2a_ka_l\left(\mu^2-\varepsilon\right)\left(\mu^2(\mu^2+\varepsilon)d^{-8}++\varepsilon^2\mu^4d^{-12}\right)\right)
\\ \notag&\quad+O\left(\sum_{k,l=1}^2a_kp_l(\mu^2-\varepsilon)\left(\mu^3d^{-5}+\left(\mu^5(\mu^2+\varepsilon)+\varepsilon^2 \mu^3\right)d^{-11}\right)\right)
\\ \label{Jan2-13}&\quad+O\left(\sum_{k,l,m=1}^2\left(a_ka_la_m+a_ka_lp_m+a_kp_lp_m+p_kp_lp_m\right)(\mu^2-\varepsilon)\left(1+(\mu^4+\varepsilon\mu^2)d^{-6}+\mu(\mu^2+\varepsilon)d^{-7}\right)\right).
\end{align}

{\bf$\bullet$ The estimates of $\mathcal{R}^{(3)}_{\mu,\xi,a,p,\varepsilon}$ and $\mathcal{R}^{(4)}_{\mu,\xi,a,p,\varepsilon}$.}
Now, we estimate $\mathcal{R}^{(3)}_{\mu,\xi,a,p,\varepsilon}$ (\ref{Jan2-11}). Integrating by parts, we can obtain
\begin{align}
\mathcal{R}^{(3)}_{\mu,\xi,a,p,\varepsilon}
\notag&=-2\int_{\mathcal{D}}\left(\varphi_1-\varepsilon g\right)\cdot  \left[\left(P\mathcal{L}_{\mathcal{A}}\right)_x\wedge \left(P\mathcal{R}_{\mathcal{A}}\right)_y+\left(P\mathcal{R}_{\mathcal{A}}\right)_x\wedge \left(P\mathcal{L}_{\mathcal{A}}\right)_y\right]
\\ \notag&\quad-2\int_{\mathcal{D}}\left(\sum_{k=1}^2a_k\varphi_{-1,k}+p_k\varphi_{2,k}\right)\cdot  \left[\delta_x\wedge \left(P\mathcal{R}_{\mathcal{A}}\right)_y+\delta_x\wedge \left(P\mathcal{R}_{\mathcal{A}}\right)_y\right]
\\ \notag&\quad+2\int_{\mathcal{D}}\left(\sum_{k=1}^2a_kP Z_{-1,k}\right)\cdot \left[\left( \sum_{l=1}^2p_l PZ_{2,l}\right)_x\wedge \left( \sum_{m=1}^2p_m PZ_{2,m}\right)_y\right]
\\ \notag&\quad+2\int_{\mathcal{D}}\left( \sum_{k=1}^2p_k PZ_{2,k}\right)\cdot \left[\left(\sum_{l=1}^2a_lP Z_{-1,l}\right)_x\wedge \left(\sum_{k=1}^2a_kP Z_{-1,k}\right)_y\right]
\\ \notag&\quad+\frac{2}{3}\int_{\mathcal{D}}\left( \sum_{k=1}^2a_k PZ_{-1,k}\right)\cdot \left[\left(\sum_{l=1}^2a_l PZ_{-1,l}\right)_x\wedge \left( \sum_{m=1}^2a_m PZ_{-1,m}\right)_y\right]
\\ \label{Au13-7}&\quad+\frac{2}{3}\int_{\mathcal{D}}\left( \sum_{k=1}^2p_k PZ_{2,k}\right)\cdot \left[\left(\sum_{l=1}^2p_l PZ_{2,l}\right)_x\wedge \left( \sum_{m=1}^2p_m PZ_{2,m}\right)_y\right].
\end{align}
We start with the terms that contain $a_kp_lp_m$ for $k,l,m=1,2$.
For the first term in (\ref{Au13-7}), one has
\begin{align*}
&-2\int_{\mathcal{D}}\varphi\cdot \left[\left(\sum_{k=1}^2a_kPZ_{-1,k}\right)_x \wedge\left(\sum_{l,m=1}^2p_lp_mP Z^{2,2}_{l,m}\right)_y+\left(\sum_{l,m=1}^2p_lp_mP Z^{2,2}_{l,m}\right)_x\wedge \left(\sum_{k=1}^2a_kPZ_{-1,k}\right)_y\right]
\\ & =\frac{8\pi}{3}a_1 p^2_1\mu^3\left[\frac{\partial h^{(1)}_1}{\partial x}- \frac{\partial h^{(1)}_2}{\partial y}\right](\xi, \xi) -\frac{8\pi}{3}a_2 p^2_1\mu^3\left[\frac{\partial h^{(1)}_1}{\partial y} + \frac{\partial h^{(1)}_2}{\partial x}\right](\xi, \xi)
\\&\quad-\frac{8\pi}{3}a_1 p^2_2\mu^3\left[\frac{\partial h^{(1)}_1}{\partial x} - \frac{\partial h^{(1)}_2}{\partial y}\right](\xi, \xi)+\frac{8\pi}{3}a_2 p^2_2\mu^3\left[\frac{\partial h^{(1)}_1}{\partial y} +\frac{\partial h^{(1)}_2}{\partial x}\right](\xi, \xi)
\\&\quad +\frac{16\pi}{3}a_1 p_1p_2\mu^3\left[\frac{\partial h^{(1)}_1}{\partial y}+ \frac{\partial h^{(1)}_2}{\partial x}\right](\xi, \xi)+\frac{16\pi}{3}a_2 p_1p_2\mu^3\left[\frac{\partial h^{(1)}_1}{\partial x}-\frac{\partial h^{(1)}_2}{\partial y}\right](\xi, \xi)
\\&\quad+O\left(\sum_{k,l,m=1}^2a_kp_lp_m\left(\mu^5d^{-5}+\mu^9d^{-13}\right)\right)
\\& =O\left(\sum_{k,l,m=1}^2a_kp_lp_m\left(\mu^2-\varepsilon\right)\left(\mu^5d^{-5}+\mu^9d^{-13}\right)\right),
\end{align*}
where we use the fact that $\frac{\partial h^{(1)}_1}{\partial x}- \frac{\partial h^{(1)}_2}{\partial y}=0$ and $\frac{\partial h^{(1)}_1}{\partial y}+ \frac{\partial h^{(1)}_2}{\partial x}=0$.
And
\begin{align*}
&-2\int_{\mathcal{D}}\varphi_1
\cdot \left[\left(\sum_{m=1}^2p_mPZ_{2,m}\right)_x \wedge\left(\sum_{k,l=1}^2a_kp_lP Z^{-1,2}_{k,l}\right)_y+\left(\sum_{k,l=1}^2a_kp_lP Z^{-1,2}_{k,l}\right)_x\wedge \left(\sum_{m=1}^2p_mPZ_{2,m}\right)_y\right]
\\&=\frac{-16\pi}{3}a_1p_1^2\mu^3\left[\frac{\pp h_1^{(1)}}{\pp x}(\xi,\xi)+3\frac{\pp h_2^{(1)}}{\pp y}(\xi,\xi)\right]
+\frac{16\pi}{3}a_2p_1^2\mu^3\left[\frac{\pp h_1^{(1)}}{\pp y}(\xi,\xi)-3\frac{\pp h_2^{(1)}}{\pp x}(\xi,\xi)\right]
\\&\quad-\frac{16\pi}{3}a_1p_2^2\mu^3\left[3\frac{\pp h_1^{(1)}}{\pp x}(\xi,\xi)+\frac{\pp h_2^{(1)}}{\pp y}(\xi,\xi)\right]
+\frac{16\pi}{3}a_2p_2^2\mu^3\left[3\frac{\pp h_1^{(1)}}{\pp y}(\xi,\xi)-\frac{\pp h_2^{(1)}}{\pp x}(\xi,\xi)\right]
\\&\quad+O\left(\sum_{k,l,m=1}^2a_kp_lp_m\left(\mu^2-\varepsilon\right)\left(\mu^5d^{-5}+\mu^9d^{-13}\right)\right)
\\&=O\left(\sum_{k,l,m=1}^2a_kp_lp_m\left(\mu^3d^{-3}+\mu^9d^{-13}\right)\right).
\end{align*}
Thus,
\begin{align*}
&-2\int_{\mathcal{D}}\left(\varphi_1-\varepsilon g\right) \cdot \left[\left(\sum_{k=1}^2a_kPZ_{-1,k}\right)_x \wedge\left(\sum_{l,m=1}^2p_lp_mP Z^{2,2}_{l,m}\right)_y+\left(\sum_{l,m=1}^2p_lp_mP Z^{2,2}_{l,m}\right)_x\wedge \left(\sum_{k=1}^2a_kPZ_{-1,k}\right)_y\right]
\\ & -2\int_{\mathcal{D}}\left(\varphi_1-\varepsilon g\right)
\cdot \left[\left(\sum_{m=1}^2p_mPZ_{2,m}\right)_x \wedge\left(\sum_{k,l=1}^2a_kp_lP Z^{-1,2}_{k,l}\right)_y+\left(\sum_{k,l=1}^2a_kp_lP Z^{-1,2}_{k,l}\right)_x\wedge \left(\sum_{m=1}^2p_mPZ_{2,m}\right)_y\right]
\\& =O\left(\sum_{k,l,m=1}^2a_kp_lp_m\left(\left(\mu^2-\varepsilon\right)\mu d^{-3}+\mu\varepsilon d^{-4}|\xi-\omega|\right)\right).
\end{align*}
And by (\ref{A15-Z}), (\ref{2expansion-near-boundary}), (\ref{M2-p1p1}), (\ref{M2-p2p2}) and (\ref{M2-p1p2}), it is direct to see that
\begin{align*}
&2\int_{\mathcal{D}}\left(\sum_{k=1}^2a_kP Z_{-1,k}\right)\cdot \left[\left( \sum_{l=1}^2p_l PZ_{2,l}\right)_x\wedge \left( \sum_{m=1}^2p_m PZ_{2,m}\right)_y\right]
\\&=-2\int_{\mathcal{D}}\left(\sum_{k=1}^2a_k\varphi_{-1,k}\right)\cdot \left[\left( \sum_{l=1}^2p_l Z_{2,l}\right)_x\wedge \left( \sum_{m=1}^2p_m Z_{2,m}\right)_y\right]
+O\left(\sum_{k,l,m=1}^2a_kp_lp_m\left(\mu^5d^{-5}+\mu^{11}d^{-13}\right)\right).
\end{align*}
Combining the corresponding terms that contain $a_kp_lp_m$, $k,l,m=1,2$ in $$-2\int_{\mathcal{D}}\left(\sum_{k=1}^2a_k\varphi_{-1,k}+p_k\varphi_{2,k}\right)\cdot  \left[\delta_x\wedge \left(P\mathcal{R}_{\mathcal{A}}\right)_y+\delta_x\wedge \left(P\mathcal{R}_{\mathcal{A}}\right)_y\right],$$
and reasoning as (\ref{D28-3}), (\ref{Jan2-8}) and (\ref{Jan2-9}), we deduce
\begin{equation*}
\begin{aligned}
&-2\int_{\mathcal{D}}\left(\sum_{k=1}^2a_k\varphi_{-1,k}\right)\cdot  \left[\left( \sum_{l=1}^2p_l Z_{2,l}\right)_x\wedge \left( \sum_{m=1}^2p_m Z_{2,m}\right)_y\right.
\\&\qquad\qquad\qquad\qquad\qquad\qquad\left.+\delta_x\wedge \left(\sum_{l,m=1}^2p_lp_mP Z^{2,2}_{l,m}\right)_y+\left(\sum_{l,m=1}^2p_lp_mP Z^{2,2}_{l,m}\right)_x \wedge  \delta_y\right]
\\&= O\left(\sum_{k,l,m=1}^2a_kp_lp_m\mu^7d^{-7}\right).
\end{aligned}
\end{equation*}
And
\begin{equation*}
\begin{aligned}
&-2\int_{\mathcal{D}}\left(\sum_{k=1}^2p_k\varphi_{2,k}\right)\cdot  \left[\delta_x\wedge \left(\sum_{l,m=1}^2a_lp_mP Z^{-1,2}_{l,m}\right)_y+\left(\sum_{l,m=1}^2a_lp_mP Z^{-1,2}_{l,m}\right)_x \wedge  \delta_y\right]
\\&= O\left(\sum_{k,l,m=1}^2a_kp_lp_m\mu^5d^{-5}\right).
\end{aligned}
\end{equation*}
Thus, combining the above estimates along with the fact that $\varepsilon d^{-1}|\xi-\omega|=O\left(\mu^2-\varepsilon\right)$, we conclude that the terms in
$\mathcal{R}^{(3)}_{\mu,\xi,a,p,\varepsilon}$ that contain $a_kp_lp_m$ can be estimated as
\begin{equation*}
\begin{aligned}
O\left(\sum_{k,l,m=1}^2a_kp_lp_m\left(\mu^5d^{-5}+\mu^{11}d^{-13}+\left(\mu^2-\varepsilon\right)\mu d^{-3}\right)\right).
\end{aligned}
\end{equation*}

For the terms that contain $p_kp_lp_m$, $k,l,m=1,2$, there holds
\begin{align*}
&-2\int_{\mathcal{D}}\varphi_1 \cdot \left[\left(\sum_{k=1}^2p_kPZ_{2,k}\right)_x \wedge\left(\sum_{l,m=1}^2p_lp_mP Z^{2,2}_{l,m}\right)_y+\left(\sum_{l,m=1}^2p_lp_mP Z^{2,2}_{l,m}\right)_x\wedge \left(\sum_{k=1}^2p_kPZ_{2,k}\right)_y\right]
\\&=\frac{-\pi^2}{8}p_1^3\mu^4\left[\frac{\pp^2h_1^{(1)}}{\pp x^2}(\xi,\xi)+\frac{\pp^2h_1^{(1)}}{\pp y^2}(\xi,\xi)\right]-\frac{\pi^2}{8}p_1^2p_2\mu^4\left[\frac{\pp^2h_2^{(1)}}{\pp x^2}(\xi,\xi)+\frac{\pp^2h_2^{(1)}}{\pp y^2}(\xi,\xi)\right]
\\&\quad-\frac{\pi^2}{8}p_1p_2^2\mu^4\left[\frac{\pp^2h_1^{(1)}}{\pp x^2}(\xi,\xi)+\frac{\pp^2h_1^{(1)}}{\pp y^2}(\xi,\xi)\right]-\frac{\pi^2}{8}p_2^3\mu^4\left[\frac{\pp^2h_2^{(1)}}{\pp x^2}(\xi,\xi)+\frac{\pp^2h_2^{(1)}}{\pp y^2}(\xi,\xi)\right]
\\&\quad+O\left(\sum_{k,l,m=1}^2p_kp_lp_m\left(\mu^6d^{-6}+\mu^{14}d^{-16}\right)\right)
\\&=O\left(\sum_{k,l,m=1}^2p_kp_lp_m\left(\mu^6d^{-6}+\mu^{14}d^{-16}\right)\right),
\end{align*}
thus
\begin{align*}
&-2\int_{\mathcal{D}}\left(\varphi_1-\varepsilon g\right) \cdot \left[\left(\sum_{k=1}^2p_kPZ_{2,k}\right)_x \wedge\left(\sum_{l,m=1}^2p_lp_mP Z^{2,2}_{l,m}\right)_y+\left(\sum_{l,m=1}^2p_lp_mP Z^{2,2}_{l,m}\right)_x\wedge \left(\sum_{k=1}^2p_kPZ_{2,k}\right)_y\right]
\\&=O\left(\sum_{k,l,m=1}^2p_kp_lp_m\left(\left(\mu^2-\varepsilon\right)\mu^4d^{-6}+\varepsilon\mu^{4}d^{-7}|\xi-\omega|\right)\right).
\end{align*}
Moreover,
\begin{equation*}
\begin{aligned}
-2\int_{\mathcal{D}}\left(\sum_{k=1}^2p_k\varphi_{2,k}\right)\cdot  \left[\delta_x\wedge \left(\sum_{l,m=1}^2p_lp_mP Z^{2,2}_{l,m}\right)_y+\left(\sum_{l,m=1}^2p_lp_mP Z^{2,2}_{l,m}\right)_x \wedge  \delta_y\right]
= O\left(\sum_{k,l,m=1}^2p_kp_lp_m\mu^6d^{-6}\right),
\end{aligned}
\end{equation*}
and
 \begin{equation*}
\begin{aligned}
&\frac{2}{3}\int_{\mathcal{D}}\left( \sum_{l=1}^2p_l PZ_{2,l}\right)\cdot \left[\left(\sum_{l=1}^2p_l PZ_{2,l}\right)_x\wedge \left( \sum_{l=1}^2p_l PZ_{2,l}\right)_y\right]
\\&=\frac{32\pi}{9}p_1^3\mu^6\frac{\pp^2h_2^{(2,1)}}{\pp x\pp y}(\xi,\xi)+\frac{32\pi}{9}p_1p_2^2\mu^6\frac{\pp^2h_2^{(2,1)}}{\pp x\pp y}(\xi,\xi)
\\&\quad+\frac{32\pi}{9}p_2p_1^2\mu^6\frac{\pp^2h_2^{(2,2)}}{\pp x\pp y}(\xi,\xi)+\frac{32\pi}{9}p_2^3\mu^6\frac{\pp^2h_2^{(2,2)}}{\pp x\pp y}(\xi,\xi)
+O\left((p_1+p_2)^3\left(\mu^8d^{-8}+\mu^{12}d^{-14}\right)\right)
\\&=O\left((p_1+p_2)^3\left(\mu^6d^{-6}+\mu^{12}d^{-14}\right)\right).
\end{aligned}
\end{equation*}
Thus, we conclude that the terms in
$\mathcal{R}^{(3)}_{\mu,\xi,a,p,\varepsilon}$ that contain $p_kp_lp_m$ can be estimated as
\begin{equation*}
\begin{aligned}
O\left(\sum_{k,l,m=1}^2p_kp_lp_m\left(\mu^6d^{-6}+\mu^{12}d^{-14}+\left(\mu^2-\varepsilon\right)\mu^4d^{-6}\right)\right).
\end{aligned}
\end{equation*}

For the remaining terms in
$\mathcal{R}^{(3)}_{\mu,\xi,a,p,\varepsilon}$, they can be estimated in a similar way. Omitting extensive and cumbersome calculations, we present the results here. There holds
\begin{align}
\mathcal{R}^{(3)}_{\mu,\xi,a,p,\varepsilon}
=\notag&~O\left(\sum_{k,l,m=1}^2a_ka_la_m\left(\mu^3d^{-3}+\mu^5d^{-7}+\left(\mu^2-\varepsilon\right)\left(\mu d^{-3}+\mu^3d^{-7}+\mu^3\left(\mu^2+\varepsilon\right)d^{-9}\right)\right)\right)
\\ \notag&+O\left(\sum_{k,l,m=1}^2a_ka_lp_m\left(\mu^2d^{-2}+\mu^{14}d^{-16}+\left(\mu^2-\varepsilon\right)\left(\mu^2d^{-4}+\mu^4\left(\mu^2+\varepsilon\right)d^{-10}\right)\right)\right)
\\ \notag&+O\left(\sum_{k,l,m=1}^2a_kp_lp_m\left(\mu^5d^{-5}+\mu^{11}d^{-13}+\left(\mu^2-\varepsilon\right)\mu d^{-3}\right)\right)
\\ \notag&+O\left(\sum_{k,l,m=1}^2p_kp_lp_m\left(\mu^6d^{-6}+\mu^{12}d^{-14}+\left(\mu^2-\varepsilon\right)\mu^4d^{-6}\right)\right)
\\ \label{Jan2-3}&+O\left(\left(a_1+a_2+p_1+p_2\right)^4\left(\mu d^{-1}+(\mu^2-\varepsilon)\mu^2d^{-6}\right)\right).
\end{align}

Now, we analyze the terms in $\mathcal{R}^{(4)}_{\mu,\xi,a,p,\varepsilon}$ from (\ref{Jan7-2}). We begin by estimating the following terms
\begin{equation}\label{Au12-1}
\begin{aligned}
\frac{1}{2}\int_{\mathcal{D}}\nabla \left(P\mathcal{R}_{\mathcal{A}}\right)\cdot\nabla \left(P\mathcal{R}_{\mathcal{A}}\right)+2\int_{\mathcal{D}}P\delta\cdot \left[\left(P\mathcal{R}_{\mathcal{A}}\right)_x
 \wedge \left(P\mathcal{R}_{\mathcal{A}}\right)_y\right].
 \end{aligned}
\end{equation}
Omitting the extensive and cumbersome calculations, we present the estimates for the terms in (\ref{Au12-1}). Specifically, we have
\begin{equation}\label{Au12-5}
\begin{aligned}
&\sum_{j,k=1}^2a_j^2a_k^2\left[\frac{1}{2}\int_{\mathcal{D}}\left|\nabla (PZ^{-1,-1}_{j,k})\right|^2+2
\int_{\mathcal{D}} P\delta
\cdot\left[(PZ^{-1,-1}_{j,k})_x\wedge (PZ^{-1,-1}_{j,k})_y\right]\right]
\\&=\frac{256\pi}{5}a_1^2a_2^2+O\left(\sum_{j,k=1}^2a_j^2a_k^2\left(\mu^4d^{-4}+\mu^4d^{-6}+\mu^8d^{-10}+\mu^{12}d^{-14}\right)\right).
\end{aligned}
\end{equation}
For the term involving $a_j^2p_k^2$, $j,k=1,2$, there holds
\begin{equation}\label{Au10-1}
\begin{aligned}
&\sum_{j,k=1}^2a_j^2p_k^2\left[\frac{1}{2}\int_{\mathcal{D}}\left|\nabla (PZ^{-1,2}_{j,k})\right|^2+2
\int_{\mathcal{D}} P\delta
\cdot\left[(PZ^{-1,2}_{j,k})_x\wedge (PZ^{-1,2}_{j,k})_y\right]\right]
\\&=6\pi^2\left(a_1^2p_1^2+a_1^2p_2^2+a_2^2p_1^2+a_2^2p_2^2\right)+O\left(\sum_{j,k=1}^2a_j^2p_k^2\left(\mu^4d^{-4}+\mu^6d^{-8}+\mu^{10}d^{-12}\right)\right).
\end{aligned}
\end{equation}
We also consider the term with $p_j^2p_k^2$, $j,k=1,2$, we have
\begin{equation}\label{Au12-3}
\begin{aligned}
&\sum_{j,k=1}^2p_j^2p_k^2\left[\frac{1}{2}\int_{\mathcal{D}}\left|\nabla (PZ^{2,2}_{j,k})\right|^2+2
\int_{\mathcal{D}} P\delta
\cdot\left[(PZ^{2,2}_{j,k})_x\wedge (PZ^{2,2}_{j,k})_y\right]\right]
\\&=\frac{64\pi}{5} p_1^2p_2^2+O\left(\sum_{j,k=1}^2p_j^2p_k^2\left(\mu^4d^{-4}+\mu^{12}d^{-14}+\mu^{16}d^{-18}\right)\right).
\end{aligned}
\end{equation}
For $(j,k)\neq (l,m)$ and $(j,k)\neq (m,l)$, one has
\begin{equation}\label{Au12-6}
\begin{aligned}
&\sum_{j,k,l,m=1}^2a_ja_ka_la_m\left[\int_{\mathcal{D}}\nabla (PZ^{-1,-1}_{j,k})\cdot \nabla (PZ^{-1,-1}_{l,m})\right.
\\&\qquad\qquad\qquad\qquad\quad \left.
+2
\int_{\mathcal{D}} P\delta
\cdot\left[(PZ^{-1,-1}_{j,k})_x\wedge (PZ^{-1,-1}_{l,m})_y+(PZ^{-1,-1}_{l,m})_x\wedge (PZ^{-1,-1}_{j,k})_y \right]\right]
\\&=\frac{-256\pi}{5}a_1^2a_2^2+O\left(\sum_{j,k,l,m=1}^2a_ja_ka_la_m\left(\mu^4d^{-4}+\mu^4d^{-6}+\mu^{8}d^{-10}+\mu^{12}d^{-14}\right)\right),
\end{aligned}
\end{equation}

\begin{equation}
\begin{aligned}
&\sum_{j,k,l,m=1}^2a_ja_ka_lp_m\left[\int_{\mathcal{D}}\nabla (PZ^{-1,-1}_{j,k})\cdot \nabla (PZ^{-1,2}_{l,m})\right.
\\&\qquad\qquad\qquad\qquad\quad\left.+2
\int_{\mathcal{D}} P\delta
\cdot\left[(PZ^{-1,-1}_{j,k})_x\wedge (PZ^{-1,2}_{l,m})_y+(PZ^{-1,2}_{l,m})_x\wedge (PZ^{-1,-1}_{j,k})_y \right]\right]
\\&=O\left(\sum_{j,k,l,m=1}^2a_ja_ka_lp_m\left(\mu^3d^{-3}+\mu^{7}d^{-9}+\mu^{11}d^{-13}\right)\right),
\end{aligned}
\end{equation}

\begin{equation}\label{Au10-2}
\begin{aligned}
&\sum_{j,k,m,l=1}^2a_ja_kp_lp_m\left[\int_{\mathcal{D}}\nabla (PZ^{-1,-1}_{j,k})\cdot \nabla (PZ^{2,2}_{l,m})\right.
\\&\qquad\qquad\qquad\qquad\quad \left.
+2
\int_{\mathcal{D}} P\delta
\cdot\left[(PZ^{-1,-1}_{j,k})_x\wedge (PZ^{2,2}_{l,m})_y+(PZ^{2,2}_{l,m})_x\wedge(PZ^{-1,-1}_{j,k})_y\right]\right]
\\&=-6\pi^2\left(a_1^2p_1^2+a_2^2p_2^2+a_1^2p_2^2+a_2^2p_1^2\right)+O\left(\sum_{j,k,m,l=1}^2a_ja_kp_lp_m\left(\mu^4d^{-4}+\mu^{10}d^{-12}+\mu^{12}d^{-14}+\mu^{14}d^{-16}\right)\right),
\end{aligned}
\end{equation}

\begin{equation}
\begin{aligned}
&\sum_{j,k,l,m=1}^2a_jp_kp_lp_m\left[\int_{\mathcal{D}}\nabla (PZ^{-1,2}_{j,k})\cdot \nabla (PZ^{2,2}_{l,m})
+2
\int_{\mathcal{D}} P\delta
\cdot\left[(PZ^{-1,2}_{j,k})_x\wedge (PZ^{2,2}_{l,m})_y+(PZ^{2,2}_{l,m})_x\wedge (PZ^{-1,2}_{j,k})_y \right]\right]
\\&=O\left(\sum_{j,k,l,m=1}^2a_jp_kp_lp_m\left(\mu^3d^{-3}+\mu^{9}d^{-11}+\mu^{13}d^{-15}\right)\right),
\end{aligned}
\end{equation}
and
\begin{equation}\label{Au12-4}
\begin{aligned}
&\sum_{j,k,l,m=1}^2p_jp_kp_lp_m\left[\int_{\mathcal{D}}\nabla (PZ^{2,2}_{j,k})\cdot \nabla (PZ^{2,2}_{l,m})+2
\int_{\mathcal{D}} P\delta
\cdot\left[(PZ^{2,2}_{j,k})_x\wedge (PZ^{2,2}_{l,m})_y+(PZ^{2,2}_{l,m})_x\wedge (PZ^{2,2}_{j,k})_y\right]\right]
\\&=-\frac{64\pi}{5}p_1^2p_2^2+O\left(\sum_{j,k,l,m=1}^2p_jp_kp_lp_m\left(\mu^4d^{-4}+\mu^{12}d^{-14}+\mu^{16}d^{-18}\right)\right).
\end{aligned}
\end{equation}

We observe that the leading-order terms cancel out in the sum of (\ref{Au12-5})-(\ref{Au12-4}).
The remaining terms in (\ref{Au12-1}) can be estimated in a similar manner. We conclude that
\begin{equation*}
\begin{aligned}
&\frac{1}{2}\int_{\mathcal{D}}\nabla \left(P\mathcal{R}_{\mathcal{A}}\right)\cdot\nabla \left(P\mathcal{R}_{\mathcal{A}}\right)+2\int_{\mathcal{D}}P\delta\cdot \left[\left(P\mathcal{R}_{\mathcal{A}}\right)_x
 \wedge \left(P\mathcal{R}_{\mathcal{A}}\right)_y\right]
\\&=O\left(\left(a_1+a_2+p_1+p_2\right)^3\left(a_1+a_2\right)\left(\mu^3d^{-3}+\mu^4d^{-6}\right)+\left(p_1+p_2\right)^4\left(\mu^4d^{-4}+\mu^{12}d^{-14}\right)\right).
 \end{aligned}
\end{equation*}

Next, we estimate $2\int_{\mathcal{D}}\left(P\mathcal{R}_{\mathcal{A}}\right)\cdot \left[\left(P\mathcal{L}_{\mathcal{A}}\right)_x
 \wedge \left(P\mathcal{L}_{\mathcal{A}}\right)_y\right]$. We start with the terms that contain $a_k^2p_l^2$, $k,l=1,2$. Indeed, there holds
\begin{equation*}
\begin{aligned}
&2\int_{\mathcal{D}}\left(\sum_{k,l=1}^2a_ka_lP Z^{-1,-1}_{k,l}\right)\cdot \left[\left( \sum_{l=1}^2p_l PZ_{2,l}\right)_x\wedge \left( \sum_{l=1}^2p_l PZ_{2,l}\right)_y\right]
\\&=6\pi^2\left(a_1^2p_1^2+a_1^2p_2^2+a_2^2p_1^2+a_2^2p_2^2\right)+O\left(\sum_{j,k=1}^2a_j^2p_k^2\left(\mu^4d^{-4}+\mu^{10}d^{-12}+\mu^{14}d^{-16}\right)\right),
\end{aligned}
\end{equation*}

\begin{equation*}
\begin{aligned}
&2\int_{\mathcal{D}}\left(\sum_{k,l=1}^2p_kp_lP Z^{2,2}_{k,l}\right)\cdot\left[\left( \sum_{l=1}^2a_lP Z_{-1,l}\right)_x\wedge \left( \sum_{l=1}^2a_lP Z_{-1,l}\right)_y\right]
\\&=6\pi^2\left(a_1^2p_1^2+a_1^2p_2^2+a_2^2p_1^2+a_2^2p_2^2\right)+O\left(\sum_{j,k=1}^2a_j^2p_k^2\left(\mu^2d^{-2}+\mu^{10}d^{-12}\right)\right),
\end{aligned}
\end{equation*}

\begin{equation*}
\begin{aligned}
&2\int_{\mathcal{D}}\left(\sum_{k,l=1}^2a_kp_lP Z^{-1,2}_{k,l}\right)\cdot \left[\left( \sum_{l=1}^2a_lP Z_{-1,l}\right)_x\wedge \left( \sum_{l=1}^2p_lP Z_{2,l}\right)_y+\left( \sum_{l=1}^2p_lP Z_{2,l}\right)_x\wedge \left( \sum_{l=1}^2a_lP Z_{-1,l}\right)_y\right]
\\&=-12\pi^2\left(a_1^2p_1^2+a_1^2p_2^2+a_2^2p_1^2+a_2^2p_2^2\right)+O\left(\sum_{j,k=1}^2a_j^2p_k^2\left(\mu^2d^{-2}+\mu^{10}d^{-12}\right)\right).
\end{aligned}
\end{equation*}
As for the terms that contain $p_jp_kp_lp_m$, $j,k,l,m=1,2$, there holds
\begin{equation*}
\begin{aligned}
&2\int_{\mathcal{D}}\left(\sum_{k,l=1}^2p_k p_l P Z^{2,2}_{k,l}\right)\cdot \left[\left( \sum_{l=1}^2p_l P Z_{2,l}\right)_x\wedge \left( \sum_{l=1}^2p_l P Z_{2,l}\right)_y\right]
=O\left(\sum_{j,k,l,m=1}^2p_jp_kp_lp_m\left(\mu^5d^{-5}+\mu^{16}d^{-18}\right)\right).
\end{aligned}
\end{equation*}
The remaining terms in $2\int_{\mathcal{D}}\left(P\mathcal{R}_{\mathcal{A}}\right)\cdot \left[\left(P\mathcal{L}_{\mathcal{A}}\right)_x
 \wedge \left(P\mathcal{L}_{\mathcal{A}}\right)_y\right]$ are of higher order. We can conclude that
\begin{equation*}
\begin{aligned}
&2\int_{\mathcal{D}}\left(P\mathcal{R}_{\mathcal{A}}\right)\cdot \left[\left(P\mathcal{L}_{\mathcal{A}}\right)_x
 \wedge \left(P\mathcal{L}_{\mathcal{A}}\right)_y\right]
\\&=O\left(\left(a_1+a_2+p_1+p_2\right)^3\left(a_1+a_2\right)\left(\mu^2d^{-2}+\mu^4d^{-6}\right)+\left(p_1+p_2\right)^4\left(\mu^5d^{-5}+\mu^{16}d^{-18}\right)\right).
\end{aligned}
\end{equation*}
Finally, it is easy to check that
\begin{equation*}
\begin{aligned}
&\frac{2}{3}\int_{\mathcal{D}}\left(P\mathcal{R}_{\mathcal{A}}\right)\cdot \left[\left(P\mathcal{R}_{\mathcal{A}}\right)_x
 \wedge \left(P\mathcal{R}_{\mathcal{A}}\right)_y\right]
+2\int_{\mathcal{D}}\left(P\mathcal{L}_{\mathcal{A}}\right)\cdot \left[\left(P\mathcal{R}_{\mathcal{A}}\right)_x
 \wedge \left(P\mathcal{R}_{\mathcal{A}}\right)_y\right]
\\&=O\left(\left(a_1+a_2+p_1+p_2\right)^5\left(1+\mu^7d^{-9}\right)\right).
\end{aligned}
\end{equation*}
Thus, $\mathcal{R}^{(4)}_{\mu,\xi,a,p,\varepsilon}$ can be estimated as
\begin{equation}\label{Au12-2}
\begin{aligned}
\mathcal{R}^{(4)}_{\mu,\xi,a,p,\varepsilon}
&=O\left(\left(a_1+a_2+p_1+p_2\right)^3\left(a_1+a_2\right)\left(\mu^2d^{-2}+\mu^4d^{-6}\right)+\left(p_1+p_2\right)^4\left(\mu^4d^{-4}+\mu^{12}d^{-14}\right)\right)
\\&\quad+O\left(\left(a_1+a_2+p_1+p_2\right)^5\right).
 \end{aligned}
\end{equation}
The details are omitted here for brevity.

Therefore,  we obtain
\begin{align}
\mathcal{R}_{\mu,\xi,a,p,\varepsilon}
\notag&=\mathcal{R}^{(1)}_{\mu,\xi,a,p,\varepsilon}+\mathcal{R}^{(2)}_{\mu,\xi,a,p,\varepsilon}+\mathcal{R}^{(3)}_{\mu,\xi,a,p,\varepsilon}+\mathcal{R}^{(4)}_{\mu,\xi,a,p,\varepsilon}
\\ \notag&=(\ref{Jan8-4})+(\ref{Jan2-13})+(\ref{Jan2-3})+(\ref{Au12-2})
\\ \label{Jan8-R}&=\tilde{e}^{(2)}(\varepsilon,\mu,a,p),
\end{align}
where $\tilde{e}^{(2)}(\varepsilon,\mu,a,p)$ is given in (\ref{Ju28-error}).

\section{Various derivatives}\label{Appendix-F}
In this appendix, we provide a summary of the key calculations used in Section \ref{expansion-for-one-bubble} and Appendix \ref{Computations of mixed terms}.
Here we list some facts which will be frequently used in the remaining part of the paper. We have
\begin{equation}\label{derivative}
\begin{aligned}
&\delta_x
=\left(\frac{4 x \left(-x^4+2 x^2 y^2+3 y^4+1\right)}{\left(\left(x^2+y^2\right)^2+1\right)^2},\quad
\frac{4 y \left(-3 x^4-2 x^2 y^2+y^4+1\right)}{\left(\left(x^2+y^2\right)^2+1\right)^2},\quad
\frac{8 x
   \left(x^2+y^2\right)}{\left(\left(x^2+y^2\right)^2+1\right)^2}\right),
\\&
\delta_y=\left(\frac{4 y \left(-3 x^4-2 x^2 y^2+y^4-1\right)}{\left(\left(x^2+y^2\right)^2+1\right)^2},\quad
\frac{4 x \left(x^4-2 x^2 y^2-3 y^4+1\right)}{\left(\left(x^2+y^2\right)^2+1\right)^2},\quad
\frac{8 y
   \left(x^2+y^2\right)}{\left(\left(x^2+y^2\right)^2+1\right)^2}\right),
\end{aligned}
\end{equation}

\begin{equation}\label{cross}
\begin{aligned}
\delta_x\wedge \delta_y=\left(-\frac{32\left(x^4-y^4\right)}{\left(1+|z|^4\right)^3},\quad
-\frac{64xy\left|z\right|^2}{\left(1+|z|^4\right)^3},\quad
\frac{16\left(1-|z|^4\right)\left|z\right|^2}{\left(1+|z|^4\right)^3}\right),
\end{aligned}
\end{equation}
\begin{equation}\label{expansion6-1}
(Z_{-1,1})_x\wedge \delta_y + \delta_x\wedge (Z_{-1,1})_y =
\begin{pmatrix}
    \frac{128 x \left(x^2+y^2\right) \left(x^2 \left(\left(x^2+y^2\right)^2-2\right)+3 y^2\right)}{\left(\left(x^2+y^2\right)^2+1\right)^4}\\
    \frac{64 y \left(x^2+y^2\right) \left(3 x^6+7 x^4 y^2+x^2 \left(5
   y^4-9\right)+y^6+y^2\right)}{\left(\left(x^2+y^2\right)^2+1\right)^4}\\
    \frac{32 x \left(x^2+y^2\right) \left(\left(x^2+y^2\right)^4-8
   \left(x^2+y^2\right)^2+3\right)}{\left(\left(x^2+y^2\right)^2+1\right)^4}
\end{pmatrix},
\end{equation}

\begin{equation}\label{J9-a2}
(Z_{-1,2})_x\wedge \delta_y + \delta_x\wedge (Z_{-1,2})_y  =
\begin{pmatrix}
    \frac{-128 y \left(x^2+y^2\right) \left(\left(x^4-2\right) y^2+2 x^2 y^4+3 x^2+y^6\right)}{\left(\left(x^2+y^2\right)^2+1\right)^4}\\
   \frac{64 x \left(x^2+y^2\right) \left(x^6+5 x^4 y^2+x^2 \left(7 y^4+1\right)+3 y^2
   \left(y^4-3\right)\right)}{\left(\left(x^2+y^2\right)^2+1\right)^4}\\
   \frac{32 y \left(x^2+y^2\right) \left(\left(x^2+y^2\right)^4-8 \left(x^2+y^2\right)^2+3\right)}{\left(\left(x^2+y^2\right)^2+1\right)^4}
\end{pmatrix},
\end{equation}

\begin{equation}\label{M27-2}
\begin{aligned}
&(Z_{-1,1})_x\wedge (Z_{-1,1})_y
\\&=\begin{pmatrix}
   -\frac{32 \left(x^2+y^2\right) \left(\left(x^2+y^2\right)^2-3\right) \left(3 x^8+10 x^6 y^2+x^4 \left(12 y^4-5\right)+6 x^2 y^2
   \left(y^4+2\right)+y^8+y^4\right)}{\left(\left(x^2+y^2\right)^2+1\right)^5}\\
   -\frac{64 x y \left(x^2+y^2\right)
   \left(\left(x^2+y^2\right)^2-3\right) \left(x^6+3 x^4 y^2+x^2 \left(3
   y^4-7\right)+y^6+y^2\right)}{\left(\left(x^2+y^2\right)^2+1\right)^5}\\
  -\frac{16 \left(x^2+y^2\right)^2
   \left(\left(x^2+y^2\right)^2-3\right) \left(\left(x^2+y^2\right) \left(x^6+3 x^4 y^2+3 x^2 \left(y^4-4\right)+y^6+4
   y^2\right)+3\right)}{\left(\left(x^2+y^2\right)^2+1\right)^5}
\end{pmatrix},
\end{aligned}
\end{equation}

\begin{equation}\label{M27-3}
\begin{aligned}
&(Z_{-1,2})_x\wedge (Z_{-1,2})_y
\\&=\begin{pmatrix}
  \frac{32 \left(x^2+y^2\right) \left(x^8+\left(12 x^4-5\right) y^4+x^4+10 x^2 y^6+6 \left(x^4+2\right) x^2 y^2+3 y^8\right)
   \left(\left(x^2+y^2\right)^2-3\right)}{\left(\left(x^2+y^2\right)^2+1\right)^5}\\
  -\frac{64 x y \left(x^2+y^2\right) \left(x^6+3 x^4 y^2+3
   x^2 y^4+x^2+y^6-7 y^2\right) \left(\left(x^2+y^2\right)^2-3\right)}{\left(\left(x^2+y^2\right)^2+1\right)^5}\\
  -\frac{16
   \left(x^2+y^2\right)^2 \left(\left(x^2+y^2\right)^2-3\right) \left(\left(x^2+y^2\right) \left(3 \left(x^4-4\right) y^2+3 x^2 y^4+x^2
   \left(x^4+4\right)+y^6\right)+3\right)}{\left(\left(x^2+y^2\right)^2+1\right)^5}
\end{pmatrix},
\end{aligned}
\end{equation}

\begin{equation}\label{M28a1-1}
(Z_{-1,1})_x=
\begin{pmatrix}
\frac{4 (x-y) (x+y) \left(\left(x^2+y^2\right) \left(x^6+3 x^4 y^2+3 x^2 \left(y^4-4\right)+y^6+4
   y^2\right)+3\right)}{\left(\left(x^2+y^2\right)^2+1\right)^3}
\\
\frac{8 x y \left(\left(x^2+y^2\right) \left(x^6+3 x^4 y^2+3 x^2
   \left(y^4-4\right)+y^6+4 y^2\right)+3\right)}{\left(\left(x^2+y^2\right)^2+1\right)^3}
\\
 -\frac{8 \left(x^2+y^2\right) \left(3 x^6+5 x^4
   y^2+x^2 \left(y^4-5\right)-y^2 \left(y^4+1\right)\right)}{\left(\left(x^2+y^2\right)^2+1\right)^3}
   \end{pmatrix},
\end{equation}

\begin{equation}\label{M28a1-2}
(Z_{-1,1})_y=
\begin{pmatrix}
\frac{8 x y \left(\left(x^2+y^2\right) \left(3 \left(x^4+2\right) y^2+3 x^2 y^4+x^2
   \left(x^4-10\right)+y^6\right)-3\right)}{\left(\left(x^2+y^2\right)^2+1\right)^3}
\\
\frac{4 \left(2 \left(x^4-1\right) y^6+3 x^2 y^8-3
   \left(x^8+10 x^4+1\right) y^2-2 x^2 \left(x^4+17\right) y^4+x^2 \left(-x^8+2
   x^4+3\right)+y^{10}\right)}{\left(\left(x^2+y^2\right)^2+1\right)^3}
\\
-\frac{32 x y \left(x^2+y^2\right)
   \left(\left(x^2+y^2\right)^2-1\right)}{\left(\left(x^2+y^2\right)^2+1\right)^3}
   \end{pmatrix},
\end{equation}

\begin{equation}\label{M28a2-1}
(Z_{-1,2})_x=
\begin{pmatrix}
-\frac{8 x y \left(\left(x^2+y^2\right) \left(x^6+3 x^4 y^2+3 x^2 \left(y^4+2\right)+y^6-10
   y^2\right)-3\right)}{\left(\left(x^2+y^2\right)^2+1\right)^3}
\\
\frac{4 \left(x^{10}+3 x^8 y^2+2 x^6 \left(y^4-1\right)-2 x^4 y^2
   \left(y^4+17\right)-3 x^2 \left(y^8+10 y^4+1\right)+y^2 \left(-y^8+2
   y^4+3\right)\right)}{\left(\left(x^2+y^2\right)^2+1\right)^3}
\\
-\frac{32 x y \left(x^2+y^2\right)
   \left(\left(x^2+y^2\right)^2-1\right)}{\left(\left(x^2+y^2\right)^2+1\right)^3}
   \end{pmatrix},
\end{equation}

\begin{equation}\label{M28a2-2}
(Z_{-1,2})_y=
\begin{pmatrix}
\frac{4 (x-y) (x+y) \left(\left(x^2+y^2\right) \left(3 \left(x^4-4\right) y^2+3 x^2 y^4+x^2
   \left(x^4+4\right)+y^6\right)+3\right)}{\left(\left(x^2+y^2\right)^2+1\right)^3}
\\
\frac{8 x y \left(\left(x^2+y^2\right) \left(3
   \left(x^4-4\right) y^2+3 x^2 y^4+x^2 \left(x^4+4\right)+y^6\right)+3\right)}{\left(\left(x^2+y^2\right)^2+1\right)^3}
\\
\frac{8
   \left(x^2+y^2\right) \left(x^6-x^4 y^2+x^2 \left(1-5 y^4\right)-3 y^6+5 y^2\right)}{\left(\left(x^2+y^2\right)^2+1\right)^3}
\end{pmatrix},
\end{equation}

\begin{equation}\label{A26-5}
(Z_{2,1})_x\wedge \delta_y + \delta_x\wedge (Z_{2,1})_y =
\begin{pmatrix}
    \frac{32 \left(x^2+y^2\right) \left(5 x^4-14 x^2 y^2+5 y^4-1\right)}{\left(\left(x^2+y^2\right)^2+1\right)^4}\\
   \frac{384 x y \left(x^4-y^4\right)}{\left(\left(x^2+y^2\right)^2+1\right)^4}\\
   \frac{64 (x-y) (x+y) \left(x^2+y^2\right)
   \left(\left(x^2+y^2\right)^2-2\right)}{\left(\left(x^2+y^2\right)^2+1\right)^4}
\end{pmatrix},
\end{equation}
\begin{equation}\label{A26-6}
(Z_{2,2})_x\wedge \delta_y + \delta_x\wedge (Z_{2,2})_y=
\begin{pmatrix}
 \frac{384 x y \left(x^4-y^4\right)}{\left(\left(x^2+y^2\right)^2+1\right)^4}\\
   -\frac{32 \left(x^2+y^2\right) \left(x^4-22 x^2 y^2+y^4+1\right)}{\left(\left(x^2+y^2\right)^2+1\right)^4}\\
   \frac{128 x y \left(x^2+y^2\right)
   \left(\left(x^2+y^2\right)^2-2\right)}{\left(\left(x^2+y^2\right)^2+1\right)^4}
\end{pmatrix},
\end{equation}

\begin{equation}\label{M30p1-1}
(Z_{2,1})_x=
\begin{pmatrix}
\frac{8 x \left(x^6-9 x^4 y^2-x^2 \left(5 y^4+3\right)+5 y^6+y^2\right)}{\left(\left(x^2+y^2\right)^2+1\right)^3}
\\
\frac{8 y \left(5 x^6-5 x^4 y^2-3 x^2 \left(3 y^4+1\right)+y^6+y^2\right)}{\left(\left(x^2+y^2\right)^2+1\right)^3}
\\
  \frac{8 x \left(-3
   x^4+2 x^2 y^2+5 y^4+1\right)}{\left(\left(x^2+y^2\right)^2+1\right)^3}
   \end{pmatrix},
\end{equation}

\begin{equation}\label{M30p1-2}
(Z_{2,1})_y=
\begin{pmatrix}
\frac{8 y \left(5 x^6-5 x^4 y^2-9 x^2 y^4+x^2+y^6-3 y^2\right)}{\left(\left(x^2+y^2\right)^2+1\right)^3}
\\
-\frac{8 x \left(x^6-9 x^4 y^2-5 x^2 y^4+x^2+5 y^6-3 y^2\right)}{\left(\left(x^2+y^2\right)^2+1\right)^3}
\\
 -\frac{8 y\left(5 x^4 +2 x^2 y^2-3
   y^4+1\right)}{\left(\left(x^2+y^2\right)^2+1\right)^3}
   \end{pmatrix},
\end{equation}

\begin{equation}\label{M30p2-1}
(Z_{2,2})_x=
\begin{pmatrix}
\frac{8 y \left(5 x^6-5 x^4 y^2-3 x^2 \left(3 y^4+1\right)+y^6+y^2\right)}{\left(\left(x^2+y^2\right)^2+1\right)^3}
\\
-\frac{8 x \left(x^6-9 x^4 y^2+x^2 \left(1-5 y^4\right)+5
   \left(y^6+y^2\right)\right)}{\left(\left(x^2+y^2\right)^2+1\right)^3}
\\
 \frac{8 y \left(-7 x^4-6 x^2 y^2+y^4+1\right)}{\left(\left(x^2+y^2\right)^2+1\right)^3}
   \end{pmatrix},
\end{equation}

\begin{equation}\label{M30p2-2}
(Z_{2,2})_y=
\begin{pmatrix}
-\frac{8 x \left(x^6-9 x^4 y^2-5 x^2 y^4+x^2+5 y^6-3 y^2\right)}{\left(\left(x^2+y^2\right)^2+1\right)^3}
\\
-\frac{8 y \left(5 x^6-5 x^4 y^2+x^2 \left(5-9 y^4\right)+y^6+y^2\right)}{\left(\left(x^2+y^2\right)^2+1\right)^3}
\\
\frac{8 x \left(x^4-6 x^2
   y^2-7 y^4+1\right)}{\left(\left(x^2+y^2\right)^2+1\right)^3}
\end{pmatrix},
\end{equation}

\begin{equation}\label{M2-p1p1}
 \begin{aligned}
(Z_{2,1})_x\wedge (Z_{2,1})_y=
\begin{pmatrix}
 -\frac{64 (x-y) (x+y) \left(x^2+y^2\right) \left(3 x^4-10 x^2 y^2+3 y^4-1\right)}{\left(\left(x^2+y^2\right)^2+1\right)^5}\\
-\frac{128
   x y \left(x^2+y^2\right) \left(5 x^4-6 x^2 y^2+5 y^4+1\right)}{\left(\left(x^2+y^2\right)^2+1\right)^5}\\
-\frac{64 \left(x^2+y^2\right)
   \left(x^8+4 x^6 y^2+x^4 \left(6 y^4-3\right)+2 x^2 y^2 \left(2 y^4+5\right)+y^4
   \left(y^4-3\right)\right)}{\left(\left(x^2+y^2\right)^2+1\right)^5}
\end{pmatrix},
\end{aligned}
\end{equation}
\begin{equation}\label{M2-p2p2}
  \begin{aligned}
(Z_{2,2})_x\wedge (Z_{2,2})_y=
\begin{pmatrix}
 -\frac{64 (x-y) (x+y) \left(x^2+y^2\right) \left(x^4+18 x^2 y^2+y^4+1\right)}{\left(\left(x^2+y^2\right)^2+1\right)^5}\\
\frac{128 x y
   \left(x^2+y^2\right) \left(x^4-14 x^2 y^2+y^4+1\right)}{\left(\left(x^2+y^2\right)^2+1\right)^5}\\
-\frac{64 \left(x^2+y^2\right)
   \left(x^8+4 x^6 y^2+x^4 \left(6 y^4+1\right)+2 x^2 y^2 \left(2
   y^4-7\right)+y^8+y^4\right)}{\left(\left(x^2+y^2\right)^2+1\right)^5}
\end{pmatrix},
\end{aligned}
\end{equation}

\begin{equation}\label{Ma30-a1a2}
\begin{aligned}
&
(Z_{-1,1})_x\wedge (Z_{-1,2})_y+(Z_{-1,2})_x\wedge (Z_{-1,1})_y
\\&=
\begin{pmatrix}
\frac{512 x y (x-y) (x+y) \left(x^2+y^2\right)
   \left(\left(x^2+y^2\right)^2-3\right)}{\left(\left(x^2+y^2\right)^2+1\right)^5}\\
-\frac{128 \left(x^2+y^2\right)
   \left(\left(x^2+y^2\right)^2-3\right) \left(x^8+4 x^6 y^2+x^4 \left(6 y^4+1\right)+x^2 \left(4 y^6-6
   y^2\right)+y^8+y^4\right)}{\left(\left(x^2+y^2\right)^2+1\right)^5}\\
\frac{512 x y \left(x^2+y^2\right)^3
   \left(\left(x^2+y^2\right)^2-3\right)}{\left(\left(x^2+y^2\right)^2+1\right)^5}
\end{pmatrix},
\end{aligned}
\end{equation}

\begin{equation}\label{A27-a1p1}
\begin{aligned}
&
(Z_{-1,1})_x\wedge (Z_{2,1})_y+(Z_{2,1})_x\wedge (Z_{-1,1})_y
\\&=
\begin{pmatrix}
-\frac{32 x \left(x^2+y^2\right) \left(9 x^8-8 x^6 y^2-2 x^4 \left(23 y^4+10\right)+4 x^2 y^2 \left(11-8 y^4\right)-3 y^8-32
   y^4+3\right)}{\left(\left(x^2+y^2\right)^2+1\right)^5}\\
\frac{32 y \left(x^2+y^2\right) \left(-23 x^8+2 \left(5 x^4+1\right) y^4+38 x^4+24
   x^2 y^6-8 \left(4 x^4+7\right) x^2 y^2+5 y^8-3\right)}{\left(\left(x^2+y^2\right)^2+1\right)^5}\\
\frac{64 xf_1(x,y)}{\left(\left(x^2+y^2\right)^2+1\right)^5}
\end{pmatrix},
\end{aligned}
\end{equation}
where
$$
\begin{aligned}
f_1(x,y)=& \left[-x^{12}-2 x^{10} y^2+x^8
   \left(5 y^4+9\right)+20 x^6 \left(y^6+y^2\right)+x^4 \left(25 y^8+6 y^4-6\right)
   \right.
   \\&\left.~+2 x^2 y^6 \left(7 y^4-6\right)+y^4 \left(3 y^8-7
   y^4+6\right)\right]
\end{aligned},
$$

\begin{equation}\label{A27-a2p1}
\begin{aligned}
&
(Z_{-1,2})_x\wedge (Z_{2,1})_y+(Z_{2,1})_x\wedge (Z_{-1,2})_y
\\&=
\begin{pmatrix}
\frac{32 y \left(x^2+y^2\right) \left(3 x^8+32 x^6 y^2+x^4 \left(46 y^4+32\right)+x^2 \left(8 y^6-44 y^2\right)-9 y^8+20
   y^4-3\right)}{\left(\left(x^2+y^2\right)^2+1\right)^5}\\
-\frac{32 x \left(x^2+y^2\right) \left(5 x^8+24 x^6 y^2+2 x^4 \left(5
   y^4+1\right)-8 x^2 y^2 \left(4 y^4+7\right)-23 y^8+38 y^4-3\right)}{\left(\left(x^2+y^2\right)^2+1\right)^5}\\
\frac{64 yf_2(x,y)}{\left(\left(x^2+y^2\right)^2+1\right)^5}
\end{pmatrix},
\end{aligned}
\end{equation}
where
$$
\begin{aligned}
f_2(x,y)=& \left[-\left(5
   x^4+9\right) y^8+2 x^2 y^{10}+\left(-25 x^8-6 x^4+6\right) y^4+x^4 \left(-3 x^8+7 x^4-6\right)+2 x^6 \left(6-7 x^4\right) y^2\right.
   \\&~\left.-20 x^2
   \left(x^4+1\right) y^6+y^{12}\right]
\end{aligned}.
$$

\begin{equation}\label{A27-a1p2}
\begin{aligned}
&
(Z_{-1,1})_x\wedge (Z_{2,2})_y+(Z_{2,2})_x\wedge (Z_{-1,1})_y
\\&=
\begin{pmatrix}
\frac{32 y \left(x^2+y^2\right) \left(-27 x^8+2 \left(2-7 x^4\right) y^4+40 x^4+8 x^2 y^6-4 \left(12 x^4+13\right) x^2
   y^2+y^8+3\right)}{\left(\left(x^2+y^2\right)^2+1\right)^5}\\
\frac{32 x \left(x^2+y^2\right) \left(5 x^8-10 \left(7 x^4+1\right) y^4+2
   x^4-48 x^2 y^6+8 \left(11-3 x^4\right) x^2 y^2-7 y^8-3\right)}{\left(\left(x^2+y^2\right)^2+1\right)^5}\\
\frac{64 y \left(x^2+y^2\right)
   \left(-3 x^{10}-11 x^8 y^2+x^6 \left(17-14 y^4\right)+x^4 y^2 \left(35-6 y^4\right)+x^2 \left(y^8+19
   y^4-12\right)+y^{10}+y^6\right)}{\left(\left(x^2+y^2\right)^2+1\right)^5}
\end{pmatrix},
\end{aligned}
\end{equation}

\begin{equation}\label{A27-a2p2}
\begin{aligned}
&
(Z_{-1,2})_x\wedge (Z_{2,2})_y+(Z_{2,2})_x\wedge (Z_{-1,2})_y
\\&=
\begin{pmatrix}
\frac{-32 x \left(x^2+y^2\right) \left(x^8+8 x^6 y^2+x^4 \left(4-14 y^4\right)-4 x^2 y^2 \left(12 y^4+13\right)-27 y^8+40
   y^4+3\right)}{\left(\left(x^2+y^2\right)^2+1\right)^5}\\
\frac{-32 y \left(x^2+y^2\right) \left(7 x^8+48 x^6 y^2+10 x^4 \left(7
   y^4+1\right)+8 x^2 y^2 \left(3 y^4-11\right)-5 y^8-2 y^4+3\right)}{\left(\left(x^2+y^2\right)^2+1\right)^5}\\
\frac{64 x
   \left(x^2+y^2\right) \left(\left(x^2+y^2\right)^2 \left(x^6-x^4 y^2+x^2 \left(1-5 y^4\right)-3 y^6+17 y^2\right)-12
   y^2\right)}{\left(\left(x^2+y^2\right)^2+1\right)^5}
\end{pmatrix},
\end{aligned}
\end{equation}

\begin{equation}\label{M2-p1p2}
\begin{aligned}
(Z_{2,1})_x\wedge (Z_{2,2})_y+(Z_{2,2})_x\wedge (Z_{2,1})_y
=
\begin{pmatrix}
-\frac{256 x y \left(x^2+y^2\right) \left(3 x^4-10 x^2 y^2+3 y^4-1\right)}{\left(\left(x^2+y^2\right)^2+1\right)^5}\\
\frac{128 (x-y)
   (x+y) \left(x^2+y^2\right) \left(x^4-14 x^2 y^2+y^4+1\right)}{\left(\left(x^2+y^2\right)^2+1\right)^5}\\
\frac{1024 x y
   \left(x^4-y^4\right)}{\left(\left(x^2+y^2\right)^2+1\right)^5}
\end{pmatrix},
\end{aligned}
\end{equation}

\begin{equation}\label{D31-1}
\begin{aligned}
(Z_{-1,1})_x\wedge (Z_{-1,1})_y+\delta_x\wedge (Z^{-1,-1}_{1,1})_y+(Z^{-1,-1}_{1,1})_x\wedge \delta_y
=
\begin{pmatrix}
   \frac{-192 \left(x^2+y^2\right)f_3(x,y)}{\left(\left(x^2+y^2\right)^2+1\right)^5}\\
    \frac{-192 x y \left(x^2+y^2\right)f_4(x,y)}{\left(\left(x^2+y^2\right)^2+1\right)^5}\\
 \frac{f_5(x,y)}{\left(\left(x^2+y^2\right)^2+1\right)^5}
\end{pmatrix},
\end{aligned}
\end{equation}
where
$$\begin{aligned}
f_3(x,y)=& x^{12}+5 x^{10} y^2+x^8 \left(10 y^4-9\right)+2 x^6 y^2 \left(5 y^4-7\right)
+x^4 \left(5 y^8+6\right)+x^2 y^2 \left(y^8+6
   y^4-13\right)+y^8+y^4,
\\f_4(x,y)=&x^{10}+5 x^8 y^2+2 x^6 \left(5 y^4-8\right)+2 x^4 y^2 \left(5 y^4-18\right)+x^2
   \left(5 y^8-24 y^4+15\right)+y^2 \left(y^4-5\right) \left(y^4+1\right),
\\    f_5(x,y)=&16 \left[-\left(\left(x^2+y^2\right)^2+1\right)^4+11
   \left(\left(x^2+y^2\right)^2+1\right)^3+2 \left(24 x^2 \left(x^2+y^2\right)-11\right) \left(\left(x^2+y^2\right)^2+1\right)^2\right.
   \\&\qquad\left.-12 \left(18 x^2 \left(x^2+y^2\right)-1\right)
   \left(\left(x^2+y^2\right)^2+1\right)+192 x^2 \left(x^2+y^2\right)\right]
\end{aligned}$$

\begin{equation}\label{D31-2}
\begin{aligned}
(Z_{-1,2})_x\wedge (Z_{-1,2})_y+\delta_x\wedge (Z^{-1,-1}_{2,2})_y+(Z^{-1,-1}_{2,2})_x\wedge \delta_y
=
\begin{pmatrix}
   \frac{192 \left(x^2+y^2\right)f_6(x,y)}{\left(\left(x^2+y^2\right)^2+1\right)^5}\\
   \frac{-192 x y \left(x^2+y^2\right) f_7(x,y)}{\left(\left(x^2+y^2\right)^2+1\right)^5}\\
\frac{-16 \left(x^2+y^2\right)f_8(x,y)}{\left(\left(x^2+y^2\right)^2+1\right)^5}
\end{pmatrix},
\end{aligned}
\end{equation}

where
$$\begin{aligned}
f_6(x,y)=&\left(\left(5 x^8+6\right) y^4+x^8+\left(10 x^4-9\right) y^8+x^4+5 x^2 y^{10}+2 x^2 \left(5 x^4-7\right) y^6+x^2 \left(x^8+6 x^4-13\right)
   y^2+y^{12}\right),
\\f_7(x,y)=&\left(2 \left(5 x^4-8\right) y^6+5 x^2 y^8+\left(5 x^8-24 x^4+15\right) y^2+2 x^2 \left(5
   x^4-18\right) y^4+x^2 \left(x^4-5\right) \left(x^4+1\right)+y^{10}\right),
\\    f_8(x,y)=& \left(21 x^4-55\right) y^{10}+7 x^2
   y^{12}+\left(35 x^8-358 x^4+115\right) y^6+x^2 \left(35 x^4-227\right) y^8
   \\&+\left(7 x^{12}-83 x^8+105 x^4-21\right) y^2+x^2 \left(21 x^8-262 x^4+225\right) y^4+x^2 \left(x^4+1\right)
   \left(x^8-8 x^4+3\right)+y^{14}
\end{aligned}$$

\begin{equation}\label{D31-3}
\begin{aligned}
&(Z_{-1,1})_x\wedge (Z_{-1,2})_y+(Z_{-1,2})_x\wedge (Z_{-1,1})_y +\delta_x\wedge (Z^{-1,-1}_{1,2})_y+(Z^{-1,-1}_{1,2})_x\wedge \delta_y
\\&=
\begin{pmatrix}
   \frac{768 x y (x-y) (x+y) \left(x^2+y^2\right) \left(3 \left(x^2+y^2\right)^2-5\right)}{\left(\left(x^2+y^2\right)^2+1\right)^5}\\
  \frac{-192 \left(x^2+y^2\right) f_9(x,y)}{\left(\left(x^2+y^2\right)^2+1\right)^5}\\
\frac{768 x y \left(x^2+y^2\right) \left(\left(2 \left(x^2+y^2\right)^2-5\right)
   \left(x^2+y^2\right)^2+1\right)}{\left(\left(x^2+y^2\right)^2+1\right)^5}
\end{pmatrix},
\end{aligned}
\end{equation}
where
$$\begin{aligned}
f_9(x,y)=&x^{12}+6 x^{10}
   y^2+x^8 \left(15 y^4-2\right)+4 x^6 y^2 \left(5 y^4-8\right)+3 x^4 \left(5 y^8-20 y^4-1\right)+2 x^2 y^2 \left(3 y^8-16 y^4+17\right)
   \\&+y^4 \left(y^4-3\right)
   \left(y^4+1\right)
\end{aligned}$$

\begin{equation}\label{D27-1}
\begin{aligned}
&(Z_{2,1})_x\wedge (Z_{2,1})_y+\delta_x\wedge (Z^{2,2}_{1,1})_y+(Z^{2,2}_{1,1})_x\wedge \delta_y
\\&=
\begin{pmatrix}
   -\frac{96 (x-y) (x+y) \left(x^2+y^2\right) \left(5 x^4-22 x^2 y^2+5 y^4-3\right)}{\left(\left(x^2+y^2\right)^2+1\right)^5}\\
    -\frac{192 x y \left(x^2+y^2\right) \left(7 x^4-18
   x^2 y^2+7 y^4-1\right)}{\left(\left(x^2+y^2\right)^2+1\right)^5}\\
  -\frac{32 \left(x^2+y^2\right) \left(5 x^8-4 x^6 y^2-x^4 \left(18 y^4+17\right)+x^2 \left(38 y^2-4 y^6\right)+5
   y^8-17 y^4+2\right)}{\left(\left(x^2+y^2\right)^2+1\right)^5}
\end{pmatrix},
\end{aligned}
\end{equation}

\begin{equation}\label{D27-4}
\begin{aligned}
&(Z_{2,2})_x\wedge (Z_{2,2})_y+\delta_x\wedge (Z^{2,2}_{2,2})_y+(Z^{2,2}_{2,2})_x\wedge \delta_y
\\&=
\begin{pmatrix}
 \frac{96 (x-y) (x+y) \left(x^2+y^2\right) \left(x^4-30 x^2 y^2+y^4+1\right)}{\left(\left(x^2+y^2\right)^2+1\right)^5}\\
   \frac{192 x y \left(x^2+y^2\right) \left(3 x^4-26 x^2
   y^2+3 y^4+3\right)}{\left(\left(x^2+y^2\right)^2+1\right)^5}\\
  \frac{32 \left(x^2+y^2\right) \left(x^8-\left(42 x^4+1\right) y^4-x^4-20 x^2 y^6+10 \left(7-2 x^4\right) x^2
   y^2+y^8-2\right)}{\left(\left(x^2+y^2\right)^2+1\right)^5}
\end{pmatrix},
\end{aligned}
\end{equation}

\begin{equation}\label{D27-5}
\begin{aligned}
&(Z_{2,1})_x\wedge (Z_{2,2})_y+(Z_{2,2})_x\wedge (Z_{2,1})_y+\delta_x\wedge (Z^{2,2}_{1,2})_y+(Z^{2,2}_{1,2})_x\wedge \delta_y
\\&=
\begin{pmatrix}
-\frac{384 x y \left(x^2+y^2\right) \left(7 x^4-18 x^2 y^2+7 y^4-1\right)}{\left(\left(x^2+y^2\right)^2+1\right)^5}\\
   \frac{192 (x-y) (x+y) \left(x^2+y^2\right) \left(x^4-30
   x^2 y^2+y^4+1\right)}{\left(\left(x^2+y^2\right)^2+1\right)^5}\\
 -\frac{768 x y (x-y) (x+y) \left(x^2+y^2\right)
   \left(\left(x^2+y^2\right)^2-3\right)}{\left(\left(x^2+y^2\right)^2+1\right)^5}
\end{pmatrix},
\end{aligned}
\end{equation}

\begin{equation}\label{D31-6}
\begin{aligned}
&(Z_{-1,1})_x\wedge (Z_{2,1})_y+(Z_{2,1})_x\wedge (Z_{-1,1})_y+\delta_x\wedge (Z^{-1,2}_{1,2})_y+(Z^{-1,2}_{1,2})_x\wedge \delta_y
\\&=
\begin{pmatrix}
-\frac{192 x \left(x^2+y^2\right) \left(5 x^8+4 x^6 y^2-10 x^4 \left(y^4+1\right)+4 x^2 y^2 \left(7-3 y^4\right)-3 y^4
   \left(y^4+6\right)+1\right)}{\left(\left(x^2+y^2\right)^2+1\right)^5}\\
 \frac{384 y \left(x^2+y^2\right) \left(-5 x^8-10 x^6 y^2+x^4 \left(11-4 y^4\right)+2 x^2 y^2
   \left(y^4-8\right)+y^8+y^4\right)}{\left(\left(x^2+y^2\right)^2+1\right)^5}\\
-\frac{256 x \left(x^2+y^2\right) \left(x^{10}+4 x^8 y^2+x^6 \left(6 y^4-7\right)+4 x^4 y^2 \left(y^4-2\right)+x^2
   \left(y^8+5 y^4+4\right)+6 y^2 \left(y^4-1\right)\right)}{\left(\left(x^2+y^2\right)^2+1\right)^5}
\end{pmatrix}.
\end{aligned}
\end{equation}

\section{Proof of Lemma \ref{evalues-at-xi-2}}\label{Appendix-G}
{\bf$\bullet$ Evaluation of the Robin functions $h^{(-1, 1)}$ and $h^{(-1, 2)}$.}
For simplicity, we denote $h^{(-1, 1)}_1(z,\xi)$ and $h^{(-1, 1)}_2(z,\xi)$ as $h_1(z,\xi)$ and $h_2(z,\xi)$, respectively, and set $h(z,\xi) = h_1(z,\xi)+ih_2(z,\xi)$, where $z = re^{i\theta} = x + i y$. Then, using the Poisson integral formula, we have
\begin{align*}
h(z, \xi)  & = \frac{-1}{\pi}\int_{-\pi}^\pi\frac{1-r^2}{|e^{it}-re^{i\theta}|^2}\frac{e^{it}-\xi}{|e^{it}-\xi|^2}dt\\
& = \frac{-1}{\pi}\int_{-\pi}^\pi\frac{1-r^2}{|e^{it}-re^{i\theta}|^2}\frac{e^{it}-e^{i\theta_0}\xi_0}{|e^{it}-e^{i\theta_0}\xi_0|^2}dt\\
& = \frac{-1}{\pi}\int_{-\pi}^\pi\frac{1-r^2}{|e^{it}-re^{i(\theta-\theta_0)}|^2}\frac{e^{i\theta_0}\left(e^{it}-\xi_0\right)}{|e^{it}-\xi_0|^2}dt\\
& = e^{i\theta_0} v(e^{-i\theta_0}z, e^{-i\theta_0}\xi).
\end{align*}
Here $e^{i\theta_0} = \frac{\xi_1}{\sqrt{\xi_1^2+\xi_2^2}}+i \frac{\xi_2}{\sqrt{\xi_1^2+\xi_2^2}}$, $\xi_0 = (\sqrt{\xi_1^2+\xi_2^2}, 0)$ and the function $v(z, \xi_0)$ is defined as
$$v(z, \xi_0) : = \frac{-1}{\pi}\int_{-\pi}^\pi\frac{1-r^2}{|e^{it}-re^{i\theta}|^2}\frac{e^{it}-\xi_0}{|e^{it}-\xi_0|^2}dt = v_1(z, \xi_0) + i v_2(z, \xi_0).$$ Then we have
\begin{equation*}
\begin{aligned}
h_1(x, y, \xi_1, \xi_2)  & =\frac{\xi_1}{\sqrt{\xi_1^2+\xi_2^2}} v_1\left(\frac{\xi_1}{\sqrt{\xi_1^2+\xi_2^2}}x + \frac{\xi_2}{\sqrt{\xi_1^2+\xi_2^2}}y, \frac{\xi_1}{\sqrt{\xi_1^2+\xi_2^2}}y - \frac{\xi_2}{\sqrt{\xi_1^2+\xi_2^2}}x, \sqrt{x_1^2+\xi_2^2}, 0\right)\\
& \quad -\frac{\xi_2}{\sqrt{\xi_1^2+\xi_2^2}}v_2\left(\frac{\xi_1}{\sqrt{\xi_1^2+\xi_2^2}}x + \frac{\xi_2}{\sqrt{\xi_1^2+\xi_2^2}}y, \frac{\xi_1}{\sqrt{\xi_1^2+\xi_2^2}}y - \frac{\xi_2}{\sqrt{\xi_1^2+\xi_2^2}}x, \sqrt{x_1^2+\xi_2^2}, 0\right)
\end{aligned}
\end{equation*}
and
\begin{equation*}
\begin{aligned}
h_2(x, y, \xi_1, \xi_2)  & =\frac{\xi_1}{\sqrt{\xi_1^2+\xi_2^2}}v_2\left(\frac{\xi_1}{\sqrt{\xi_1^2+\xi_2^2}}x + \frac{\xi_2}{\sqrt{\xi_1^2+\xi_2^2}}y, \frac{\xi_1}{\sqrt{\xi_1^2+\xi_2^2}}y - \frac{\xi_2}{\sqrt{\xi_1^2+\xi_2^2}}x, \sqrt{x_1^2+\xi_2^2}, 0\right)\\
&\quad +\frac{\xi_2}{\sqrt{\xi_1^2+\xi_2^2}}v_1\left(\frac{\xi_1}{\sqrt{\xi_1^2+\xi_2^2}}x + \frac{\xi_2}{\sqrt{\xi_1^2+\xi_2^2}}y, \frac{\xi_1}{\sqrt{\xi_1^2+\xi_2^2}}y - \frac{\xi_2}{\sqrt{\xi_1^2+\xi_2^2}}x, \sqrt{x_1^2+\xi_2^2}, 0\right).
\end{aligned}
\end{equation*}
Therefore, it holds that
\begin{equation}\label{derivatives-of-h0}
\begin{aligned}
\frac{\partial }{\partial x}h_1(\xi, \xi)  & =\frac{\xi_1}{\sqrt{\xi_1^2+\xi_2^2}} \frac{\xi_1}{\sqrt{\xi_1^2+\xi_2^2}}\frac{\partial }{\partial x}v_1(\xi_0, \xi_0) - \frac{\xi_1}{\sqrt{\xi_1^2+\xi_2^2}}\frac{\xi_2}{\sqrt{\xi_1^2+\xi_2^2}}\frac{\partial }{\partial y}v_1(\xi_0, \xi_0)\\
& \quad -\frac{\xi_2}{\sqrt{\xi_1^2+\xi_2^2}}\frac{\xi_1}{\sqrt{\xi_1^2+\xi_2^2}}\frac{\partial }{\partial x}v_2(\xi_0, \xi_0)+\frac{\xi_2}{\sqrt{\xi_1^2+\xi_2^2}}\frac{\xi_2}{\sqrt{\xi_1^2+\xi_2^2}}\frac{\partial }{\partial y}v_2(\xi_0, \xi_0),\\
\frac{\partial }{\partial y}h_1(\xi, \xi)  & =\frac{\xi_1}{\sqrt{\xi_1^2+\xi_2^2}} \frac{\xi_2}{\sqrt{\xi_1^2+\xi_2^2}}\frac{\partial }{\partial x}v_1(\xi_0, \xi_0) + \frac{\xi_1}{\sqrt{\xi_1^2+\xi_2^2}}\frac{\xi_1}{\sqrt{\xi_1^2+\xi_2^2}}\frac{\partial }{\partial y}v_1(\xi_0, \xi_0)\\
& \quad -\frac{\xi_2}{\sqrt{\xi_1^2+\xi_2^2}}\frac{\xi_2}{\sqrt{\xi_1^2+\xi_2^2}}\frac{\partial }{\partial x}v_2(\xi_0, \xi_0)-\frac{\xi_2}{\sqrt{\xi_1^2+\xi_2^2}}\frac{\xi_1}{\sqrt{\xi_1^2+\xi_2^2}}\frac{\partial }{\partial y}v_2(\xi_0, \xi_0),\\
\frac{\partial }{\partial x}h_2(\xi, \xi)  & =\frac{\xi_1}{\sqrt{\xi_1^2+\xi_2^2}} \frac{\xi_1}{\sqrt{\xi_1^2+\xi_2^2}}\frac{\partial }{\partial x}v_2(\xi_0, \xi_0) - \frac{\xi_1}{\sqrt{\xi_1^2+\xi_2^2}}\frac{\xi_2}{\sqrt{\xi_1^2+\xi_2^2}}\frac{\partial }{\partial y}v_2(\xi_0, \xi_0)\\
& \quad +\frac{\xi_2}{\sqrt{\xi_1^2+\xi_2^2}}\frac{\xi_1}{\sqrt{\xi_1^2+\xi_2^2}}\frac{\partial }{\partial x}v_1(\xi_0, \xi_0)-\frac{\xi_2}{\sqrt{\xi_1^2+\xi_2^2}}\frac{\xi_2}{\sqrt{\xi_1^2+\xi_2^2}}\frac{\partial }{\partial y}v_1(\xi_0, \xi_0),\\
\frac{\partial }{\partial y}h_2(\xi, \xi)  & =  \frac{\xi_1}{\sqrt{\xi_1^2+\xi_2^2}}\frac{\xi_2}{\sqrt{\xi_1^2+\xi_2^2}}\frac{\partial }{\partial x}v_2(\xi_0, \xi_0) + \frac{\xi_1}{\sqrt{\xi_1^2+\xi_2^2}}\frac{\xi_1}{\sqrt{\xi_1^2+\xi_2^2}}\frac{\partial }{\partial y}v_2(\xi_0, \xi_0)\\
& \quad +\frac{\xi_2}{\sqrt{\xi_1^2+\xi_2^2}}\frac{\xi_2}{\sqrt{\xi_1^2+\xi_2^2}}\frac{\partial }{\partial x}v_1(\xi_0, \xi_0)+\frac{\xi_2}{\sqrt{\xi_1^2+\xi_2^2}}\frac{\xi_1}{\sqrt{\xi_1^2+\xi_2^2}}\frac{\partial }{\partial y}v_1(\xi_0, \xi_0).
\end{aligned}
\end{equation}
Now differentiating under the integral sign, we can compute $\frac{\partial }{\partial x}v_1(\xi_0, \xi_0)$, $\frac{\partial }{\partial y}v_1(\xi_0, \xi_0)$, $\frac{\partial }{\partial x}v_2(\xi_0, \xi_0)$, $\frac{\partial }{\partial y}v_2(\xi_0, \xi_0)$ directly as follows,
\begin{equation*}
\begin{aligned}
\frac{\partial }{\partial x}v_1(\xi_0, \xi_0)  & = \frac{-1}{\pi}\int_{-\pi}^\pi\frac{\partial}{\partial x}\left(\frac{1-x^2-y^2}{1+x^2+y^2-2x\cos t - 2y\sin t}\right)\frac{\cos t-\sqrt{\xi_1^2+\xi_2^2}}{(\cos t-\sqrt{\xi_1^2+\xi_2^2})^2+(\sin t)^2}dt\Bigg|_{x = \sqrt{\xi_1^2+\xi_2^2}, y = 0}\\
& = \frac{-2}{(1-|\xi|^2)^2},
\end{aligned}
\end{equation*}
\begin{equation*}
\begin{aligned}
\frac{\partial }{\partial y}v_1(\xi_0, \xi_0)  & = \frac{-1}{\pi}\int_{-\pi}^\pi\frac{\partial}{\partial y}\left(\frac{1-x^2-y^2}{1+x^2+y^2-2x\cos t - 2y\sin t}\right)\frac{\cos t-\sqrt{\xi_1^2+\xi_2^2}}{(\cos t-\sqrt{\xi_1^2+\xi_2^2})^2+(\sin t)^2}dt\Bigg|_{x = \sqrt{\xi_1^2+\xi_2^2}, y = 0}\\
& = 0,
\end{aligned}
\end{equation*}
\begin{equation*}
\begin{aligned}
\frac{\partial }{\partial x}v_2(\xi_0, \xi_0)  & = \frac{-1}{\pi}\int_{-\pi}^\pi\frac{\partial}{\partial x}\left(\frac{1-x^2-y^2}{1+x^2+y^2-2x\cos t - 2y\sin t}\right)\frac{\sin t}{(\cos t-\sqrt{\xi_1^2+\xi_2^2})^2+(\sin t)^2}dt\Bigg|_{x = \sqrt{\xi_1^2+\xi_2^2}, y = 0}\\
& = 0,
\end{aligned}
\end{equation*}
\begin{equation*}
\begin{aligned}
\frac{\partial }{\partial y}v_2(\xi_0, \xi_0)  & = \frac{-1}{\pi}\int_{-\pi}^\pi\frac{\partial}{\partial y}\left(\frac{1-x^2-y^2}{1+x^2+y^2-2x\cos t - 2y\sin t}\right)\frac{\sin t}{(\cos t-\sqrt{\xi_1^2+\xi_2^2})^2+(\sin t)^2}dt\Bigg|_{x = \sqrt{\xi_1^2+\xi_2^2}, y = 0}\\
& = \frac{-2}{(1-|\xi|^2)^2}.
\end{aligned}
\end{equation*}
Using (\ref{derivatives-of-h0}), we obtain
\begin{equation*}
\begin{aligned}
&h_1^{(-1,1)}(\xi,\xi)= \frac{-2 \xi _1}{1-|\xi|^2}, \quad h_2^{(-1,1)}(\xi,\xi)=\frac{-2 \xi _2}{1-|\xi|^2},
\\& \frac{\partial h_1^{(-1,1)}}{\partial x}(\xi,\xi)=\frac{\partial h_2^{(-1,1)}}{\partial y}(\xi,\xi)=\frac{-2}{(1-|\xi|^2)^2},\qquad\frac{\partial h_1^{(-1,1)}}{\partial y}(\xi,\xi)=-\frac{\partial h_2^{(-1,1)}}{\partial x}(\xi,\xi)=0,
\\&\frac{\partial^2 h_1^{(-1,1)}}{\partial x^2}(\xi,\xi) = -\frac{\partial^2 h_1^{(-1,1)}}{\partial y^2}(\xi,\xi)=-\frac{4\xi_1}{(1-|\xi|^2)^3},\quad\ \frac{\partial^2 h_2^{(-1,1)}}{\partial x^2}(\xi,\xi) = -\frac{\partial^2 h_2^{(-1,1)}}{\partial y^2}(\xi,\xi) =\frac{4\xi_2}{(1-|\xi|^2)^3},
\end{aligned}
\end{equation*}
Note that $h_1^{(-1,1)}(z,\xi)=h_2^{(-1,2)}(z,\xi)$ and $h_2^{(-1,1)}(z,\xi)=-h_1^{(-1,2)}(z,\xi)$, the corresponding values of $h^{(-1,2)}$ in Lemma \ref{evalues-at-xi-2} can be directly obtained.

{\bf $\bullet$ Evaluation of the Robin function $h^{(1)}$.}
Similar to the computation of $h^{(-1,l)}$, we denote $h_1^{(1)}(z,\xi)$ and $h_2^{(1)}(z,\xi)$ as $h_1(z,\xi)$ and $h_2(z,\xi)$, respectively, and set $h(z,\xi) = h_1(z,\xi)+ih_2(z,\xi)$. By the Poisson integral formula, we have
\begin{equation*}
\begin{aligned}
h(z, \xi)  & = \frac{1}{2\pi}\int_{-\pi}^\pi\frac{1-r^2}{|e^{it}-re^{i\theta}|^2}\frac{\left(e^{it}-\xi\right)^2}{|e^{it}-\xi|^4}dt\\
& = \frac{1}{2\pi}\int_{-\pi}^\pi\frac{1-r^2}{|e^{it}-re^{i\theta}|^2}\frac{\left(e^{it}-e^{i\theta_0}\xi_0\right)^2}{|e^{it}-e^{i\theta_0}\xi_0|^4}dt\\
& = \frac{1}{2\pi}\int_{-\pi}^\pi\frac{1-r^2}{|e^{it}-re^{i(\theta-\theta_0)}|^2}\frac{e^{i2\theta_0}\left(e^{it}-\xi_0\right)^2}{|e^{it}-\xi_0|^4}dt\\
& = e^{i2\theta_0} v(e^{-i\theta_0}z, e^{-i\theta_0}\xi).
\end{aligned}
\end{equation*}
Here $e^{i\theta_0} = \frac{\xi_1}{\sqrt{\xi_1^2+\xi_2^2}}+i \frac{\xi_2}{\sqrt{\xi_1^2+\xi_2^2}}$,
$\xi_0 = (\sqrt{\xi_1^2+\xi_2^2}, 0)$ and the function $v(z, \xi_0)$ is defined as
$$v(z, \xi_0) : = \frac{1}{2\pi}\int_{-\pi}^\pi\frac{1-r^2}{|e^{it}-re^{i\theta}|^2}\frac{\left(e^{it}-\xi_0\right)^2}{|e^{it}-\xi_0|^4}dt = v_1(z, \xi_0) + i v_2(z, \xi_0).$$ Then we have
\begin{equation*}
\begin{aligned}
h_1(x, y, \xi_1, \xi_2)  & =\frac{\xi_1^2-\xi_2^2}{\xi_1^2+\xi_2^2} v_1\left(\frac{\xi_1}{\sqrt{\xi_1^2+\xi_2^2}}x + \frac{\xi_2}{\sqrt{\xi_1^2+\xi_2^2}}y, \frac{\xi_1}{\sqrt{\xi_1^2+\xi_2^2}}y - \frac{\xi_2}{\sqrt{\xi_1^2+\xi_2^2}}x, \sqrt{x_1^2+\xi_2^2}, 0\right)\\
& \quad -\frac{2\xi_1\xi_2}{\xi_1^2+\xi_2^2}v_2\left(\frac{\xi_1}{\sqrt{\xi_1^2+\xi_2^2}}x + \frac{\xi_2}{\sqrt{\xi_1^2+\xi_2^2}}y, \frac{\xi_1}{\sqrt{\xi_1^2+\xi_2^2}}y - \frac{\xi_2}{\sqrt{\xi_1^2+\xi_2^2}}x, \sqrt{x_1^2+\xi_2^2}, 0\right),
\end{aligned}
\end{equation*}
and
\begin{equation*}
\begin{aligned}
h_2(x, y, \xi_1, \xi_2)  & =\frac{\xi_1^2-\xi_2^2}{\xi_1^2+\xi_2^2}v_2\left(\frac{\xi_1}{\sqrt{\xi_1^2+\xi_2^2}}x + \frac{\xi_2}{\sqrt{\xi_1^2+\xi_2^2}}y, \frac{\xi_1}{\sqrt{\xi_1^2+\xi_2^2}}y - \frac{\xi_2}{\sqrt{\xi_1^2+\xi_2^2}}x, \sqrt{x_1^2+\xi_2^2}, 0\right)\\
&\quad +\frac{2\xi_1\xi_2}{\xi_1^2+\xi_2^2}v_1\left(\frac{\xi_1}{\sqrt{\xi_1^2+\xi_2^2}}x + \frac{\xi_2}{\sqrt{\xi_1^2+\xi_2^2}}y, \frac{\xi_1}{\sqrt{\xi_1^2+\xi_2^2}}y - \frac{\xi_2}{\sqrt{\xi_1^2+\xi_2^2}}x, \sqrt{x_1^2+\xi_2^2}, 0\right).
\end{aligned}
\end{equation*}
Thus
\begin{equation}\label{derivatives-of-h1}
\begin{aligned}
\frac{\partial }{\partial x}h_1(\xi, \xi)  & =\frac{\xi_1^2-\xi_2^2}{\xi_1^2+\xi_2^2} \frac{\xi_1}{\sqrt{\xi_1^2+\xi_2^2}}\frac{\partial }{\partial x}v_1(\xi_0, \xi_0) - \frac{\xi_1^2-\xi_2^2}{\xi_1^2+\xi_2^2}\frac{\xi_2}{\sqrt{\xi_1^2+\xi_2^2}}\frac{\partial }{\partial y}v_1(\xi_0, \xi_0)\\
& \quad -\frac{2\xi_1\xi_2}{\xi_1^2+\xi_2^2}\frac{\xi_1}{\sqrt{\xi_1^2+\xi_2^2}}\frac{\partial }{\partial x}v_2(\xi_0, \xi_0)+\frac{2\xi_1\xi_2}{\xi_1^2+\xi_2^2}\frac{\xi_2}{\sqrt{\xi_1^2+\xi_2^2}}\frac{\partial }{\partial y}v_2(\xi_0, \xi_0),\\
\frac{\partial }{\partial y}h_1(\xi, \xi)  & =\frac{\xi_1^2-\xi_2^2}{\xi_1^2+\xi_2^2} \frac{\xi_2}{\sqrt{\xi_1^2+\xi_2^2}}\frac{\partial }{\partial x}v_1(\xi_0, \xi_0) + \frac{\xi_1^2-\xi_2^2}{\xi_1^2+\xi_2^2}\frac{\xi_1}{\sqrt{\xi_1^2+\xi_2^2}}\frac{\partial }{\partial y}v_1(\xi_0, \xi_0)\\
& \quad -\frac{2\xi_1\xi_2}{\xi_1^2+\xi_2^2}\frac{\xi_2}{\sqrt{\xi_1^2+\xi_2^2}}\frac{\partial }{\partial x}v_2(\xi_0, \xi_0)-\frac{2\xi_1\xi_2}{\xi_1^2+\xi_2^2}\frac{\xi_1}{\sqrt{\xi_1^2+\xi_2^2}}\frac{\partial }{\partial y}v_2(\xi_0, \xi_0),\\
\frac{\partial }{\partial x}h_2(\xi, \xi)  & =\frac{\xi_1^2-\xi_2^2}{\xi_1^2+\xi_2^2} \frac{\xi_1}{\sqrt{\xi_1^2+\xi_2^2}}\frac{\partial }{\partial x}v_2(\xi_0, \xi_0) - \frac{\xi_1^2-\xi_2^2}{\xi_1^2+\xi_2^2}\frac{\xi_2}{\sqrt{\xi_1^2+\xi_2^2}}\frac{\partial }{\partial y}v_2(\xi_0, \xi_0)\\
& \quad +\frac{2\xi_1\xi_2}{\xi_1^2+\xi_2^2}\frac{\xi_1}{\sqrt{\xi_1^2+\xi_2^2}}\frac{\partial }{\partial x}v_1(\xi_0, \xi_0)-\frac{2\xi_1\xi_2}{\xi_1^2+\xi_2^2}\frac{\xi_2}{\sqrt{\xi_1^2+\xi_2^2}}\frac{\partial }{\partial y}v_1(\xi_0, \xi_0),\\
\frac{\partial }{\partial y}h_2(\xi, \xi)  & =\frac{\xi_1^2-\xi_2^2}{\xi_1^2+\xi_2^2} \frac{\xi_2}{\sqrt{\xi_1^2+\xi_2^2}}\frac{\partial }{\partial x}v_2(\xi_0, \xi_0) + \frac{\xi_1^2-\xi_2^2}{\xi_1^2+\xi_2^2}\frac{\xi_1}{\sqrt{\xi_1^2+\xi_2^2}}\frac{\partial }{\partial y}v_2(\xi_0, \xi_0)\\
& \quad +\frac{2\xi_1\xi_2}{\xi_1^2+\xi_2^2}\frac{\xi_2}{\sqrt{\xi_1^2+\xi_2^2}}\frac{\partial }{\partial x}v_1(\xi_0, \xi_0)+\frac{2\xi_1\xi_2}{\xi_1^2+\xi_2^2}\frac{\xi_1}{\sqrt{\xi_1^2+\xi_2^2}}\frac{\partial }{\partial y}v_1(\xi_0, \xi_0).
\end{aligned}
\end{equation}
By differentiating under the integral sign, we obtain
\begin{equation*}
\begin{aligned}
\frac{\partial }{\partial x}v_1(\xi_0, \xi_0)  & = \frac{1}{2\pi}\int_{-\pi}^\pi\frac{\partial}{\partial x}\left(\frac{1-x^2-y^2}{1+x^2+y^2-2x\cos t - 2y\sin t}\right)\frac{(\cos t-\sqrt{\xi_1^2+\xi_2^2})^2-(\sin t)^2}{((\cos t-\sqrt{\xi_1^2+\xi_2^2})^2+(\sin t)^2)^2}dt\Bigg|_{x = \sqrt{\xi_1^2+\xi_2^2}, y = 0}\\
& = \frac{2|\xi|}{(1-|\xi|^2)^3},
\end{aligned}
\end{equation*}
\begin{equation*}
\begin{aligned}
\frac{\partial }{\partial y}v_1(\xi_0, \xi_0)  & = \frac{1}{2\pi}\int_{-\pi}^\pi\frac{\partial}{\partial y}\left(\frac{1-x^2-y^2}{1+x^2+y^2-2x\cos t - 2y\sin t}\right)\frac{(\cos t-\sqrt{\xi_1^2+\xi_2^2})^2-(\sin t)^2}{((\cos t-\sqrt{\xi_1^2+\xi_2^2})^2+(\sin t)^2)^2}dt\Bigg|_{x = \sqrt{\xi_1^2+\xi_2^2}, y = 0}\\
& = 0,
\end{aligned}
\end{equation*}
\begin{equation*}
\begin{aligned}
\frac{\partial }{\partial x}v_2(\xi_0, \xi_0)  & = \frac{1}{2\pi}\int_{-\pi}^\pi\frac{\partial}{\partial x}\left(\frac{1-x^2-y^2}{1+x^2+y^2-2x\cos t - 2y\sin t}\right)\frac{2(\cos t-\sqrt{\xi_1^2+\xi_2^2})(\sin t)}{((\cos t-\sqrt{\xi_1^2+\xi_2^2})^2+(\sin t)^2)^2}dt\Bigg|_{x = \sqrt{\xi_1^2+\xi_2^2}, y = 0}\\
& = 0,
\end{aligned}
\end{equation*}
\begin{equation*}
\begin{aligned}
\frac{\partial }{\partial y}v_2(\xi_0, \xi_0)  & = \frac{1}{2\pi}\int_{-\pi}^\pi\frac{\partial}{\partial y}\left(\frac{1-x^2-y^2}{1+x^2+y^2-2x\cos t - 2y\sin t}\right)\frac{2(\cos t-\sqrt{\xi_1^2+\xi_2^2})(\sin t)}{((\cos t-\sqrt{\xi_1^2+\xi_2^2})^2+(\sin t)^2)^2}dt\Bigg|_{x = \sqrt{\xi_1^2+\xi_2^2}, y = 0}\\
& = \frac{2|\xi|}{(1-|\xi|^2)^3}.
\end{aligned}
\end{equation*}
Therefore, using (\ref{derivatives-of-h1}), we obtain the following evaluations at the point $\xi$,
$$\frac{\partial h_1^{(1)}}{\partial x}(\xi,\xi)=\frac{2\xi_1}{(1-|\xi|^2)^3},\quad \frac{\partial h_1^{(1)}}{ \partial y}(\xi,\xi)=-\frac{2\xi_2}{(1-|\xi|^2)^3},$$
$$\frac{\partial h_2^{(1)}}{\partial x}(\xi,\xi)=\frac{2\xi_2}{(1-|\xi|^2)^3},\quad \frac{\partial h_2^{(1)}}{ \partial y}(\xi,\xi)=\frac{2\xi_1}{(1-|\xi|^2)^3}.$$
The computations of the other terms in Lemma \ref{evalues-at-xi-2} are similar, we omit the details here.

{\bf $\bullet$ Evaluation of the Robin functions $h^{(2,l)}$.}
Similarly, we denote $h_1^{(2,1)}(z,\xi)$ and $h_2^{(2,1)}(z,\xi)$ as $h_1(z,\xi)$ and $h_2(z,\xi)$, respectively and set $h (z,\xi)= h_1(z,\xi)+ih_2(z,\xi)$. From the Poisson integral formula, we have
\begin{equation*}
\begin{aligned}
h(z, \xi)  & = -\frac{1}{2\pi}\int_{-\pi}^\pi\frac{1-r^2}{|e^{it}-re^{i\theta}|^2}\frac{\left(e^{it}-\xi\right)^4}{|e^{it}-\xi|^8}dt\\
& = -\frac{1}{2\pi}\int_{-\pi}^\pi\frac{1-r^2}{|e^{it}-re^{i\theta}|^2}\frac{\left(e^{it}-e^{i\theta_0}\xi_0\right)^4}{|e^{it}-e^{i\theta_0}\xi_0|^8}dt\\
& =- \frac{1}{2\pi}\int_{-\pi}^\pi\frac{1-r^2}{|e^{it}-re^{i(\theta-\theta_0)}|^2}\frac{e^{i4\theta_0}\left(e^{it}-\xi_0\right)^4}{|e^{it}-\xi_0|^8}dt\\
& =- e^{i4\theta_0} v(e^{-i\theta_0}z, e^{-i\theta_0}\xi).
\end{aligned}
\end{equation*}
Here $e^{i\theta_0} = \frac{\xi_1}{\sqrt{\xi_1^2+\xi_2^2}}+i \frac{\xi_2}{\sqrt{\xi_1^2+\xi_2^2}}$,
$\xi_0 = (\sqrt{\xi_1^2+\xi_2^2}, 0)$ and the function $v(z, \xi_0)$ is defined as
$$v(z, \xi_0) : = \frac{1}{2\pi}\int_{-\pi}^\pi\frac{1-r^2}{|e^{it}-re^{i\theta}|^2}\frac{\left(e^{it}-\xi_0\right)^4}{|e^{it}-\xi_0|^8}dt = v_1(z, \xi_0) + i v_2(z, \xi_0).$$ Then we have
\begin{equation*}
\begin{aligned}
h_1(x, y, \xi_1, \xi_2) & =-\frac{\xi_1^4-6\xi_1^2\xi_2^2+\xi_2^6}{(\xi_1^2+\xi_2^2)^2} v_1\left(\frac{\xi_1}{\sqrt{\xi_1^2+\xi_2^2}}x + \frac{\xi_2}{\sqrt{\xi_1^2+\xi_2^2}}y, \frac{\xi_1}{\sqrt{\xi_1^2+\xi_2^2}}y - \frac{\xi_2}{\sqrt{\xi_1^2+\xi_2^2}}x, \sqrt{x_1^2+\xi_2^2}, 0\right)\\
& \quad +\frac{4\xi_1\xi_2(\xi_1^2-\xi_2^2)}{(\xi_1^2+\xi_2^2)^2}v_2\left(\frac{\xi_1}{\sqrt{\xi_1^2+\xi_2^2}}x + \frac{\xi_2}{\sqrt{\xi_1^2+\xi_2^2}}y, \frac{\xi_1}{\sqrt{\xi_1^2+\xi_2^2}}y-\frac{\xi_2}{\sqrt{\xi_1^2+\xi_2^2}}x, \sqrt{x_1^2+\xi_2^2}, 0\right),
\end{aligned}
\end{equation*}
and
\begin{equation*}
\begin{aligned}
h_2(x, y, \xi_1, \xi_2)  & =-\frac{\xi_1^4-6\xi_1^2\xi_2^2+\xi_2^6}{(\xi_1^2+\xi_2^2)^2}v_2\left(\frac{\xi_1}{\sqrt{\xi_1^2+\xi_2^2}}x + \frac{\xi_2}{\sqrt{\xi_1^2+\xi_2^2}}y, \frac{\xi_1}{\sqrt{\xi_1^2+\xi_2^2}}y - \frac{\xi_2}{\sqrt{\xi_1^2+\xi_2^2}}x, \sqrt{x_1^2+\xi_2^2}, 0\right)\\
&\quad -\frac{4\xi_1\xi_2(\xi_1^2-\xi_2^2)}{(\xi_1^2+\xi_2^2)^2}v_1\left(\frac{\xi_1}{\sqrt{\xi_1^2+\xi_2^2}}x + \frac{\xi_2}{\sqrt{\xi_1^2+\xi_2^2}}y, \frac{\xi_1}{\sqrt{\xi_1^2+\xi_2^2}}y - \frac{\xi_2}{\sqrt{\xi_1^2+\xi_2^2}}x, \sqrt{x_1^2+\xi_2^2}, 0\right).
\end{aligned}
\end{equation*}
Similar to (\ref{derivatives-of-h0}) and (\ref{derivatives-of-h1}), by direct computation, we get
\begin{equation*}
\begin{aligned}
&\frac{\pp h_1^{(2,1)}}{\pp x}(\xi,\xi)=-\frac{4\xi_1 \left(\xi_1^2-3\xi_2^2\right)}{\left(1-|\xi|^2\right)^5},\qquad \frac{\pp h_1^{(2,1)}}{\pp y}(\xi,\xi)=-\frac{4\xi_2 \left(-3\xi_1^2+\xi_2^2\right)}{\left(1-|\xi|^2\right)^5}
\\&\frac{\pp^4 h_1^{(2,1)}}{\pp x^4}(\xi,\xi)=-\frac{24 \left(35+4 \left(|\xi|^2-1\right)^3+30 \left(|\xi|^2-1\right)^2+60 \left(|\xi|^2-1\right)\right)}{\left(1-|\xi|^2\right)^8}=\frac{\pp h_1^{(2,1)}}{\pp y^4}(\xi,\xi),
\\&\frac{\pp^4 h_2^{(2,1)}}{\pp x^4}(\xi,\xi)=0=\frac{\pp^4 h_2^{(2,1)}}{\pp x^4}(\xi,\xi).
\end{aligned}
\end{equation*}
Since $h_1^{(2,1)}(z,\xi)=-h_2^{(2,2)}(z,\xi)$, $h_2^{(2,1)}(z,\xi)=h_1^{(2,2)}(z,\xi)$, we can obtain the corresponding values of $h_1^{(2,2)}$ in Lemma \ref{evalues-at-xi-2}.

\section*{Acknowledgements}
The authors would like to express their sincere gratitude to Professor Takeshi Isobe for sending \cite{Takeshi2001a} and \cite{Takeshi2001b} to the authors. J.C. Wei is partially supported by National Key R\&D Program of China (No.2022YFA1005602), and Hong Kong General Research Fund "New frontiers in singular limits of nonlinear partial differential equations". Y. Zheng is supported by NSF of China (No.12171355). Y. Zhou is supported in part by the Fundamental Research Funds for the Central Universities.


\begin{thebibliography}{10}

\bibitem{AmbrosettiBadiale}
Antonio Ambrosetti and Marino Badiale.
\newblock Variational perturbative methods and bifurcation of bound states from
  the essential spectrum.
\newblock {\em Proc. Roy. Soc. Edinburgh Sect. A}, 128(6):1131--1161, 1998.

\bibitem{Bethuel1}
F.~Bethuel, P.~Caldiroli, and M.~Guida.
\newblock Parametric surfaces with prescribed mean curvature.
\newblock volume~60, pages 175--231 (2003). 2002.
\newblock Turin Fortnight Lectures on Nonlinear Analysis (2001).

\bibitem{bethuel1992resultat}
Fabrice Bethuel.
\newblock Un r{\'e}sultat de r{\'e}gularit{\'e} pour les solutions de
  l'{\'e}quation des surfaces {\`a} courbure moyenne prescrite.
\newblock {\em Comptes rendus de l'Acad{\'e}mie des sciences. S{\'e}rie 1,
  Math{\'e}matique}, 314(13):1003--1007, 1992.

\bibitem{BrezisCoron}
H.~Brezis and J.-M. Coron.
\newblock Convergence of solutions of {$H$}-systems or how to blow bubbles.
\newblock {\em Arch. Rational Mech. Anal.}, 89(1):21--56, 1985.

\bibitem{BrezisCoroncpam}
Ha\"{\i}m Brezis and Jean-Michel Coron.
\newblock Multiple solutions of {$H$}-systems and {R}ellich's conjecture.
\newblock {\em Comm. Pure Appl. Math.}, 37(2):149--187, 1984.

\bibitem{Caldiroli2004}
Paolo Caldiroli.
\newblock {$H$}-bubbles with prescribed large mean curvature.
\newblock {\em Manuscripta Math.}, 113(1):125--142, 2004.

\bibitem{CaldiroliMusinaCCM}
Paolo Caldiroli and Roberta Musina.
\newblock Existence of minimal {$H$}-bubbles.
\newblock {\em Commun. Contemp. Math.}, 4(2):177--209, 2002.

\bibitem{CaldiroliMusinaRMI}
Paolo Caldiroli and Roberta Musina.
\newblock Existence of {$H$}-bubbles in a perturbative setting.
\newblock {\em Rev. Mat. Iberoamericana}, 20(2):611--626, 2004.

\bibitem{caldirolimusinaduke2004}
Paolo Caldiroli and Roberta Musina.
\newblock {$H$}-bubbles in a perturbative setting: the finite-dimensional
  reduction method.
\newblock {\em Duke Math. J.}, 122(3):457--484, 2004.

\bibitem{CaldiroliMusina2006arma}
Paolo Caldiroli and Roberta Musina.
\newblock The {D}irichlet problem for {$H$}-systems with small boundary data:
  blowup phenomena and nonexistence results.
\newblock {\em Arch. Ration. Mech. Anal.}, 181(1):1--42, 2006.

\bibitem{CaldiroliMusinaJFA2007}
Paolo Caldiroli and Roberta Musina.
\newblock Weak limit and blowup of approximate solutions to {$H$}-systems.
\newblock {\em J. Funct. Anal.}, 249(1):171--198, 2007.

\bibitem{caldiroli2011bubbles}
Paolo Caldiroli and Roberta Musina.
\newblock Bubbles with prescribed mean curvature: the variational approach.
\newblock {\em Nonlinear Analysis: Theory, Methods $\&$ Applications},
  74(9):2985--2999, 2011.

\bibitem{Chanillo-Li}
Sagun Chanillo and Yan~Yan Li.
\newblock Continuity of solutions of uniformly elliptic equations in {${\bf
  R}^2$}.
\newblock {\em Manuscripta Math.}, 77(4):415--433, 1992.

\bibitem{chanillomalchiodi2005cagasymptotic}
Sagun Chanillo and Andrea Malchiodi.
\newblock Asymptotic {M}orse theory for the equation {$\Delta v=2v_x\wedge
  v_y$}.
\newblock {\em Comm. Anal. Geom.}, 13(1):187--251, 2005.

\bibitem{chenliuwei}
Guoyuan Chen, Yong Liu, and Juncheng Wei.
\newblock Nondegeneracy of harmonic maps from {$\Bbb R^2$} to {$\Bbb S^2$}.
\newblock {\em Discrete Contin. Dyn. Syst.}, 40(6):3215--3233, 2020.

\bibitem{DuzaarSteffen1}
Frank Duzaar and Klaus Steffen.
\newblock The {P}lateau problem for parametric surfaces with prescribed mean
  curvature.
\newblock In {\em Geometric analysis and the calculus of variations}, pages
  13--70. Int. Press, Cambridge, MA, 1996.

\bibitem{DuzaarSteffen}
Frank Duzaar and Klaus Steffen.
\newblock Parametric surfaces of least {$H$}-energy in a {R}iemannian manifold.
\newblock {\em Math. Ann.}, 314(2):197--244, 1999.

\bibitem{Fall2012}
Mouhamed~Moustapha Fall.
\newblock Embedded disc-type surfaces with large constant mean curvature and
  free boundaries.
\newblock {\em Commun. Contemp. Math.}, 14(6):1250037, 35, 2012.

\bibitem{felli2005note}
Veronica Felli.
\newblock A note on the existence of {$H$}-bubbles via perturbation methods.
\newblock {\em Rev. Mat. Iberoamericana}, 21(1):163--178, 2005.

\bibitem{guerra2024existence}
Andr{\'e} Guerra, Xavier Lamy, and Konstantinos Zemas.
\newblock On the existence of degenerate solutions of the two-dimensional $ h
  $-system.
\newblock {\em arXiv preprint arXiv:2409.18068}, 2024.

\bibitem{Spruck1971}
Robert Gulliver and Joel Spruck.
\newblock The {P}lateau problem for surfaces of prescribed mean curvature in a
  cylinder.
\newblock {\em Invent. Math.}, 13:169--178, 1971.

\bibitem{Spruck1972}
Robert Gulliver and Joel Spruck.
\newblock Existence theorems for parametric surfaces of prescribed mean
  curvature.
\newblock {\em Indiana Univ. Math. J.}, 22:445--472, 1972/73.

\bibitem{Heinz1954}
Erhard Heinz.
\newblock \"{U}ber die {E}xistenz einer {F}l\"{a}che konstanter mittlerer
  {K}r\"{u}mmung bei vorgegebener {B}erandung.
\newblock {\em Math. Ann.}, 127:258--287, 1954.

\bibitem{HeinzARMA}
Erhard Heinz.
\newblock On the nonexistence of a surface of constant mean curvature with
  finite area and prescribed rectifiable boundary.
\newblock {\em Arch. Rational Mech. Anal.}, 35:249--252, 1969.

\bibitem{Hildebrandt1969}
Stefan Hildebrandt.
\newblock Randwertprobleme f\"{u}r {F}l\"{a}chen mit vorgeschiebener mittlerer
  {K}r\"{u}mmung und {A}nwendungen auf die {K}apillarit\"{a}tstheorie. {I}.
  {F}est vorgebener {R}and.
\newblock {\em Math. Z.}, 112:205--213, 1969.

\bibitem{Hildebrandt1970}
Stefan Hildebrandt.
\newblock On the {P}lateau problem for surfaces of constant mean curvature.
\newblock {\em Comm. Pure Appl. Math.}, 23:97--114, 1970.

\bibitem{Takeshi2000}
Takeshi Isobe.
\newblock Classification of blow-up points and multiplicity of solutions for
  {$H$}-systems.
\newblock {\em Comm. Partial Differential Equations}, 25(7-8):1259--1325, 2000.

\bibitem{Takeshi2001a}
Takeshi Isobe.
\newblock On the asymptotic analysis of {$H$}-systems. {I}. {A}symptotic
  behavior of large solutions.
\newblock {\em Adv. Differential Equations}, 6(5):513--546, 2001.

\bibitem{Takeshi2001b}
Takeshi Isobe.
\newblock On the asymptotic analysis of {$H$}-systems. {II}. {T}he construction
  of large solutions.
\newblock {\em Adv. Differential Equations}, 6(6):641--700, 2001.

\bibitem{Laurain2012}
Paul Laurain.
\newblock Concentration of {$CMC$} surfaces in a 3-manifold.
\newblock {\em Int. Math. Res. Not. IMRN}, (24):5585--5649, 2012.

\bibitem{LenzmanSchikorra}
Enno Lenzmann and Armin Schikorra.
\newblock On energy-critical half-wave maps into {$\Bbb {S}^2$}.
\newblock {\em Invent. Math.}, 213(1):1--82, 2018.

\bibitem{muller2009boundary}
Frank M\"uller and Armin Schikorra.
\newblock Boundary regularity via {U}hlenbeck-{R}ivi\`ere decomposition.
\newblock {\em Analysis (Munich)}, 29(2):199--220, 2009.

\bibitem{MusinaJAM2004}
Roberta Musina.
\newblock The role of the spectrum of the {L}aplace operator on {$\Bbb S^2$} in
  the {$H$}-bubble problem.
\newblock {\em J. Anal. Math.}, 94:265--291, 2004.

\bibitem{MusinaAALLMA}
Roberta Musina.
\newblock On the regularity of weak solutions to {$H$}-systems.
\newblock {\em Atti Accad. Naz. Lincei Rend. Lincei Mat. Appl.},
  18(3):209--219, 2007.

\bibitem{Riviere-2008}
Tristan Rivi\`ere.
\newblock Conservation laws for conformally invariant variational problems.
\newblock {\em Invent. Math.}, 168(1):1--22, 2007.

\bibitem{Yasuhiro1995}
Yasuhiro Sasahara.
\newblock An asymptotic analysis for large solutions of {$H$}-systems.
\newblock {\em Adv. Math. Sci. Appl.}, 5(1):219--237, 1995.

\bibitem{Schikorra-Strzelecki}
Armin Schikorra and Pawe{\l} Strzelecki.
\newblock Invitation to {$H$}-systems in higher dimensions: known results, new
  facts, and related open problems.
\newblock {\em EMS Surv. Math. Sci.}, 4(1):21--42, 2017.

\bibitem{SireWeiZhengZhou-H-system}
Yannick Sire, Jun-cheng Wei, Youquan Zheng, and Yifu Zhou.
\newblock Finite-time singularity formation for the heat flow of the
  {H}-system.
\newblock {\em arXiv:2311.14336}.

\bibitem{SireWeiZheng}
Yannick Sire, Juncheng Wei, and Youquan Zheng.
\newblock Nondegeneracy of half-harmonic maps from {$\Bbb{R}$} into
  {$\Bbb{S}^1$}.
\newblock {\em Proc. Amer. Math. Soc.}, 146(12):5263--5268, 2018.

\bibitem{Steffen1976-2}
Klaus Steffen.
\newblock Isoperimetric inequalities and the problem of {P}lateau.
\newblock {\em Math. Ann.}, 222(2):97--144, 1976.

\bibitem{Steffen1976-1}
Klaus Steffen.
\newblock On the existence of surfaces with prescribed mean curvature and
  boundary.
\newblock {\em Math. Z.}, 146(2):113--135, 1976.

\bibitem{SteffenARMA}
Klaus Steffen.
\newblock On the nonuniqueness of surfaces with constant mean curvature
  spanning a given contour.
\newblock {\em Arch. Rational Mech. Anal.}, 94(2):101--122, 1986.

\bibitem{Steffen1}
Klaus Steffen.
\newblock Parametric surfaces of prescribed mean curvature.
\newblock In {\em Calculus of variations and geometric evolution problems
  ({C}etraro, 1996)}, volume 1713 of {\em Lecture Notes in Math.}, pages
  211--265. Springer, Berlin, 1999.

\bibitem{StruweActa}
Michael Struwe.
\newblock The existence of surfaces of constant mean curvature with free
  boundaries.
\newblock {\em Acta Math.}, 160(1-2):19--64, 1988.

\bibitem{Struwe1988}
Michael Struwe.
\newblock {\em Plateau's problem and the calculus of variations}, volume~35 of
  {\em Mathematical Notes}.
\newblock Princeton University Press, Princeton, NJ, 1988.

\bibitem{Struwebook}
Michael Struwe.
\newblock {\em Variational methods}.
\newblock Springer-Verlag, Berlin, 1990.
\newblock Applications to nonlinear partial differential equations and
  Hamiltonian systems.

\bibitem{strzelecki2003new}
Pawel Strzelecki.
\newblock A new proof of regularity of weak solutions of the {$H$}-surface
  equation.
\newblock {\em Calculus of Variations and Partial Differential Equations},
  16(3):227--242, 2003.

\bibitem{Wente1969}
Henry~C. Wente.
\newblock An existence theorem for surfaces of constant mean curvature.
\newblock {\em J. Math. Anal. Appl.}, 26:318--344, 1969.

\bibitem{Wente1975}
Henry~C. Wente.
\newblock The differential equation {$\Delta x=2H(x_{u}\wedge x_{v})$} with
  vanishing boundary values.
\newblock {\em Proc. Amer. Math. Soc.}, 50:131--137, 1975.

\end{thebibliography}
\end{document}